\documentclass[10pt]{article}
\usepackage[utf8]{inputenc}
\usepackage{tikz-cd}
\usepackage{tikz}
\usepackage{graphicx}
\usepackage{amsfonts}
\usepackage{url}
\usepackage{amsmath,amsthm,amssymb,enumitem}
\usepackage{mathtools}
\usepackage{geometry}
\usepackage{mathrsfs} 
\mathtoolsset{showonlyrefs}
\usepackage{stmaryrd}
\usetikzlibrary{shapes.geometric}

\usepackage{xcolor,palatino}
\definecolor{ultramarine}{RGB}{0,32,96}
\colorlet{mygreen}{green!20!gray}
\colorlet{myultramarine}{ultramarine!20!gray}

\usepackage[colorlinks=true,citecolor=mygreen, linkcolor=mygreen, urlcolor=myultramarine,breaklinks, naturalnames, pagebackref]{hyperref}

\usepackage{float}
\usepackage{comment}

\usepackage{marginnote}

\makeatletter
\newsavebox\myboxA
\newsavebox\myboxB
\newlength\mylenA

\DeclareMathOperator*{\dprime}{^{\hskip-0.5mm\prime \prime}}
\DeclareMathOperator*{\dprimeind}{^{\hskip-1.5mm\prime \prime}}
\DeclareMathOperator*{\dprimeindl}{^{\hskip-4.5mm\prime \prime}}
\DeclareMathOperator*{\dprimeindtl}{^{\hskip-8.5mm\prime \prime}}
\DeclareMathOperator*{\dprimeindttl}{^{\hskip-26mm\prime \prime}}
\DeclareMathOperator*{\tprime}{^{\hskip-1.5mm\prime \prime \prime}}

\newcommand{\upperset}[2]{%
\underset{%
        \text{\raisebox{0.5ex}{\smash{\fontsize{5}{5}$#1$}}}
              }{#2}%
                    }

\newcommand{\uppersetF}[2]{%
\underset{%
        \text{\raisebox{0.8ex}{\smash{\fontsize{5}{5}$#1$}}}
              }{#2}%
                    }
\newcommand{\oversetF}[2]{%
\overset{%
        \text{\raisebox{0.2ex}{\smash{\fontsize{5}{5}$#1$}}}
              }{#2}%
                    }
\numberwithin{equation}{section}
\numberwithin{figure}{section}
\allowdisplaybreaks[4]

\DeclareFontFamily{U}{BOONDOX-calo}{\skewchar\font=45 }
\DeclareFontShape{U}{BOONDOX-calo}{m}{n}{
  <-> s*[1.05] BOONDOX-r-calo}{}
\DeclareFontShape{U}{BOONDOX-calo}{b}{n}{
  <-> s*[1.05] BOONDOX-b-calo}{}
\DeclareMathAlphabet{\mathcalboondox}{U}{BOONDOX-calo}{m}{n}
\SetMathAlphabet{\mathcalboondox}{bold}{U}{BOONDOX-calo}{b}{n}
\DeclareMathAlphabet{\mathbcalboondox}{U}{BOONDOX-calo}{b}{n}

\newcommand{\Res}{\operatorname{Res}}
\newcommand{\Br}{\operatorname{Bar}}
\newcommand{\tot}{\operatorname{Tot}}
\newcommand{\Der}{\mathbb{D}\operatorname{er}}
\newcommand{\MA}{\mathcal{A}}
\newcommand{\MB}{\mathcal{B}}
\newcommand{\MO}{\mathcal{O}}
\newcommand{\cone}{\operatorname{cone}}
\newcommand{\lr}[1]{\{\mkern-6mu\{#1\}\mkern-6mu\}}
\newcommand{\ZZ}{{\mathbb{Z}}}
\newcommand{\NN}{{\mathbb{N}}}
\newcommand{\kk}{\Bbbk}
\newcommand{\Hom}{\operatorname{Hom}}
\newcommand{\shom}{\operatorname{hom}}
\newcommand{\Homgr}{\mathcal{H}om}
\newcommand{\Homgrgr}{\mathcalboondox{Hom}}
\newcommand{\Coder}{\operatorname{Coder}}
\newcommand{\llg}{\operatorname{lg}}
\newcommand{\nec}{\operatorname{nec}}
\newcommand{\inn }{\operatorname{inn}}
\newcommand{\out}{\operatorname{out}}
\newcommand{\llt}{\operatorname{lt}}
\newcommand{\rrt}{\operatorname{rt}}
\newcommand{\Multi}{\operatorname{Multi}}
\newcommand{\MC}{\operatorname{MC}}
\newcommand{\MCbar}{\operatorname{\underbar{MC}}}
\newcommand{\id}{\operatorname{id}}
\newcommand{\mixed}{\operatorname{mixed}}
\newcommand{\multinec}{\operatorname{multinec}}
\newcommand{\pre}{\operatorname{pre}}
\newcommand{\pCY}{\operatorname{pre-CY}_d}
\newcommand{\SpCY}{\operatorname{Spre-CY}_d}
\newcommand{\NpCY}{\operatorname{Npre-CY}_d}
\newcommand{\dPois}{\operatorname{dPois}}
\newcommand{\Ahat}{\operatorname{\widehat{A_{\infty}}_d}}
\newcommand{\cyc}{\operatorname{\operatorname{cyc}\widehat{A_{\infty}}_d}}
\newcommand{\Scyc}{\operatorname{\operatorname{Scyc}\widehat{A_{\infty}}_d}}
\makeatletter
\newcommand*{\relrelbarsep}{.386ex}
\newcommand*{\relrelbar}{%
  \mathrel{%
    \mathpalette\@relrelbar\relrelbarsep
  }%
}
\newcommand*{\@relrelbar}[2]{%
  \raise#2\hbox to 0pt{$\m@th#1\relbar$\hss}%
  \lower#2\hbox{$\m@th#1\relbar$}%
}
\providecommand*{\rightrightarrowsfill@}{%
  \arrowfill@\relrelbar\relrelbar\rightrightarrows
}
\providecommand*{\leftleftarrowsfill@}{%
  \arrowfill@\leftleftarrows\relrelbar\relrelbar
}
\providecommand*{\xrightrightarrows}[2][]{%
  \ext@arrow 0359\rightrightarrowsfill@{#1}{#2}%
}
\providecommand*{\xleftleftarrows}[2][]{%
  \ext@arrow 3095\leftleftarrowsfill@{#1}{#2}%
}
\newcommand*\xoverline[2][0.75]{%
    \sbox{\myboxA}{$\m@th#2$}%
    \setbox\myboxB\null
    \ht\myboxB=\ht\myboxA%
    \dp\myboxB=\dp\myboxA%
    \wd\myboxB=#1\wd\myboxA
    \sbox\myboxB{$\m@th\overline{\copy\myboxB}$}
    \setlength\mylenA{\the\wd\myboxA}
    \addtolength\mylenA{-\the\wd\myboxB}%
    \ifdim\wd\myboxB<\wd\myboxA%
       \rlap{\hskip 0.5\mylenA\usebox\myboxB}{\usebox\myboxA}%
    \else
        \hskip -0.5\mylenA\rlap{\usebox\myboxA}{\hskip 0.5\mylenA\usebox\myboxB}%
    \fi}
\makeatother

\newcommand*{\doubarl}[1]{\xoverline{\xoverline{#1}}} 
\newcommand*{\doubar}[1]{\bar{\bar{#1}}} 
\newcommand{\dPoisl}{\{\hskip-1mm\{}
\newcommand{\dPoisr}{\}\hskip-1mm\}}

\newcommand{\doublefleche}
{ (0.05,-0.04) -- (0.18,-0.04)
(0.05,0.04) -- (0.18,0.04)
(0,0) -- (0.07,-0.07) 
(0,0) -- (0.07,0.07)}

\newcommand{\doubleflechescindeeleft}
{(0.07,-0.07-0.1) -- (0.22,-0.07-0.1)
(0,-0.1) -- (0.22,-0.1)
(0,-0.1) -- (0.1,-0.1-0.1) }

\newcommand{\doubleflechescindeeright}
{(0,0.1) -- (0.22,0.1)
(0.07,0.07+0.1) -- (0.22,0.07+0.1)
(0,0.1) -- (0.1,0.1+0.1)
}
\newcommand{\fleche}{
[<-,> = stealth] (0,0)--(0.5,0)
}
\newcommand{\flechelong}{
[<-, > = stealth] (0,0)--(1,0)
}

\newtheorem{definition}{Definition}[section]
\newtheorem{definition-proposition}[definition]{Definition-Proposition}
\newtheorem{remark}[definition]{Remark}
\newtheorem{theorem}[definition]{Theorem}
\newtheorem{proposition}[definition]{Proposition}

\newtheorem{corollary}[definition]{Corollary}
\newtheorem{lemma}[definition]{Lemma}
\newtheorem{theoreme}[]{Theorem}

\title{Homotopy theory of pre-Calabi-Yau morphisms}
\author{Marion Boucrot}
\date{}

\begin{document}

\maketitle
\hrulefill
\begin{abstract} 
In this article we study the homotopy theory of pre-Calabi-Yau morphisms, viewing them as Maurer-Cartan elements of an $L_{\infty}$-algebra.
We give two different notions of homotopy: a notion of weak homotopy for morphisms between $d$-pre-Calabi-Yau categories whose underlying graded quivers on the domain (resp. codomain) are the same, and a notion of homotopy for morphisms between fixed pre-Calabi-Yau categories $(\MA,s_{d+1}M_{\MA})$ and $(\MB,s_{d+1}M_{\MB})$. 
Then, we show that the notion of homotopy is stable under composition and that homotopy equivalences are quasi-isomorphisms.
Finally, we prove that the functor constructed by the author in a previous article between the category of pre-Calabi-Yau categories and the partial category of $A_{\infty}$-categories of the form $\MA\oplus\MA^*[d-1]$, for $\MA$ a graded quiver, together with hat morphisms sends homotopic $d$-pre-Calabi-Yau morphisms to weak homotopic $A_{\infty}$-morphisms. 
\end{abstract}


\textbf{Keywords:} \begin{small} $A_{\infty}$-categories, pre-Calabi-Yau categories, $L_{\infty}$-algebras, homotopies
\end{small}

\hrulefill

\tableofcontents
\section{Introduction}

Pre-Calabi-Yau structures were introduced by M. Kontsevich and Y. Vlassopoulos in the last de\-cade.
These structures have also appeared
under different names, such as $V_{\infty}$-algebras in \cite{tz}, $A_{\infty}$-algebras with boundary in
\cite{seidel}, or weak Calabi-Yau structures in \cite{kontsevich} for example. These references show that pre-Calabi-Yau structures
play an important role in homological algebra, symplectic geometry, string topology, noncommutative geometry and even in Topological Quantum Field Theory. 

A pre-Calabi-Yau structure is determined by a Maurer-Cartan element of the necklace bracket introduced in \cite{ktv}. There is a strong relation between pre-Calabi-Yau structures and $A_{\infty}$-structures, given as Maurer-Cartan elements of the Gerstenhaber bracket introduced in \cite{gerstenhaber}. Actually, M. Kontsevich, A. Takeda and Y. Vlassopoulos showed in \cite{ktv} that for $d\in\ZZ$, a $d$-pre-Calabi-Yau structure on a finite dimensional vector space $A$ is equivalent to a cyclic $A_{\infty}$-structure on $A\oplus A^*[d-1]$ such that $A\hookrightarrow A\oplus A^*[d-1]$ is an $A_{\infty}$-subalgebra.

The definition of pre-Calabi-Yau morphisms first appeared in \cite{ktv} and then in \cite{lv}, in the properadic setting. Moreover, in \cite{moi}, we proved the following result, relating $d$-pre-Calabi-Yau morphisms and $A_{\infty}$-morphisms:

\begin{theoreme}
\label{thm:main-article-1}
     There is a functor $\mathcal{P} : \pCY\rightarrow \Ahat$ sending a $d$-pre-Calabi-Yau category $(\MA,s_{d+1}M_{\MA})$ to the cyclic $A_{\infty}$-category $(\MA\oplus\MA^*[d-1],sm_{\MA\oplus\MA^*})$ given in \cite{ktv} and a $d$-pre-Calabi-Yau morphism $(\Phi_0,s_{d+1}\Phi) : (\MA,s_{d+1}M_{\MA})\rightarrow (\MB,s_{d+1}M_{\MB})$ to an $A_{\infty}$-structure $sm_{\MA\oplus\MB^*}$ together with $A_{\infty}$-morphisms 
 \begin{equation}
\begin{tikzcd}
&(\MA \oplus \MB^*[d-1],sm_{\MA\oplus\MB^*}) \arrow[swap,"s\varphi_{\MA}"]{dl} \arrow[swap,"s\varphi_{\MB}"]{dr}\\
(\MA \oplus \MA^*[d-1],sm_{\MA\oplus\MA^*})&& (\MB \oplus \MB^*[d-1], sm_{\MB\oplus\MB^*}).
\end{tikzcd}
\end{equation}
\end{theoreme}

Although the homotopy theory of $A_{\infty}$-morphisms is well-known, with an explicit formula for homotopies between them given in \cite{lefevre}, there is no formula for homotopies between pre-Calabi-Yau morphisms and it is not so easy to generalize the one of the case of $A_{\infty}$-morphisms. 
However, there are two other definitions of homotopic $A_{\infty}$-morphisms, given in \cite{borisov1,borisov2}, based on the homotopy theory of Maurer-Cartan elements of $L_{\infty}$-algebras. 

$L_{\infty}$-algebras (also called strong homotopy Lie algebras or SH Lie algebras) are a generalisation of Lie algebras and were introduced in \cite{lada-stasheff,schlessinger-stasheff,stasheff-linf}. They play an important role in deformation theory, since any deformation deformation problem is governed by a DGLA (Differential Graded Lie Algebra) or more generally by an $L_{\infty}$-algebra via solutions of their Maurer-Cartan equation modulo equivalences (see \textit{e.g.} \cite{kontsevich-soibelman,manetti}). Homotopy theory of Maurer-Cartan elements of either a DGLA or an $L_{\infty}$-algebra has been studied in various articles (see \textit{e.g.} \cite{getzler, manetti, yalin}) and the equivalence relation is, in the case of DGLA, the same as the gauge equivalence (see \cite{manetti,schlessinger-stasheff}).

It is shown in \cite{borisov1} that there is an $L_{\infty}$-algebra whose Maurer-Cartan elements are in bijection with triples $(sm_B,f,sm_A)$ where $sm_B$ and $sm_A$ are $A_{\infty}$-structures on graded vector spaces $B$ and $A$ respectively and where $f : (A,sm_A)\rightarrow (B,sm_B)$ is an $A_{\infty}$-morphism. Moreover, the same author prove in \cite{borisov2} that there is also a way of seeing $A_{\infty}$-structures together with an $A_{\infty}$-morphism between them as Maurer-Cartan elements of a graded Lie algebra described in \cite{borisov2}. Therefore, one can define a homotopy between $A_{\infty}$-morphisms as a homotopy between the corresponding Maurer-Cartan elements of either the $L_{\infty}$-algebra given in \cite{borisov1} or the Lie algebra of \cite{borisov2}.

We give in this paper a construction of an $L_{\infty}$-algebra analogous to the one of \cite{borisov1} whose Maurer-Cartan elements are in bijection with triples $(s_{d+1}M_{\MB},s_{d+1}\mathbf{F},s_{d+1}M_{\MA})$ where $s_{d+1}M_{\MB}$ and $s_{d+1}M_{\MA}$ are $d$-pre-Calabi-Yau structures on graded quivers $\MB$ and $\MA$ respectively and where $s_{d+1}\mathbf{F}: (\MA,s_{d+1}M_{\MA})\rightarrow (\MB,s_{d+1}M_{\MB})$ is an $d$-pre-Calabi-Yau morphism, leading to the notion of weak homotopy of pre-Calabi-Yau morphisms.

Moreover, it is known that given $L_{\infty}$-algebras $A$ and $B$, there exists an $L_{\infty}$-algebra whose Maurer-Cartan elements are in correspondence with $L_{\infty}$-morphisms $A\rightarrow B$ (see \cite{kraft-schnitzer}). We present in this paper, in the case of $A_{\infty}$-categories and then in the case of $d$-pre-Calabi-Yau categories, a similar construction leading to the notion of homotopy of pre-Calabi-Yau morphisms. 
We also study, as it is done in \cite{kraft-schnitzer} in the case of $L_{\infty}$-morphisms, the first properties of this notion, such as the stability under composition and the link between homotopy equivalences and quasi-isomorphisms.

The main result of this paper about homotopies of $d$-pre-Calabi-Yau morphisms is the following:

\begin{theoreme}
    \label{thm:main}
    The functor $\mathcal{P}$ given in Theorem \ref{thm:main-article-1} sends homotopic $d$-pre-Calabi-Yau morphisms $(\Phi_0,s_{d+1}\Phi)$ and $(\Psi_0,s_{d+1}\Psi)$ to triples $\mathcal{P}(\Phi_0,s_{d+1}\Phi)=(sm_{\Phi},s\varphi_{\MA},s\varphi_{\MB})$, $\mathcal{P}(\Psi_0,s_{d+1}\Psi)=(sm_{\Psi},s\psi_{\MA},s\psi_{\MB})$ respectively such that $(s\varphi_{\MA},s\psi_{\MA})$ and $(s\varphi_{\MB},s\psi_{\MB})$ are pairs of weak homotopic $A_{\infty}$-morphisms.
\end{theoreme}

Let us briefly present the contents of the article.
In Section \ref{section:not-conv}, we present the notations and conventions we use throughout this paper and we recall the basics about $L_{\infty}$-algebras and their Maurer-Cartan elements in Section \ref{section:L-inf}.

In Section \ref{section:A-inf-htpy}, we present two different notions of homotopies for $A_{\infty}$-morphisms: the one obtained in \cite{borisov1} and the one which is analogous to the notion of homotopy for $L_{\infty}$-morphisms given in \cite{kraft-schnitzer}. 
We prove that the latter is stable under left and right compositions with  $A_{\infty}$-morphisms (see Propositions \ref{prop:htpic-A-inf-1} and \ref{prop:htpic-A-inf-2}), and that homotopy equivalences coincide with quasi-isomorphisms of $A_{\infty}$-categories (see Theorem \ref{thm:equiv-htpy-qiso-A-inf}).

In Section \ref{section:pCY}, we first recall the definitions of pre-Calabi-Yau categories and pre-Calabi-Yau morphisms.
We then present in subsection \ref{section:brackets} two different ways of seeing pre-Calabi-Yau morphisms as Maurer-Cartan elements of $L_{\infty}$-algebras. 
In subsubsection \ref{subsection:non-fixed-pCY}, we define an $L_{\infty}$-algebra whose Maurer-Cartan elements are triples $(s_{d+1}M_{\MB},s_{-1}(s_{d+1}\mathbf{F}),s_{d+1}M_{\MA})$, where $s_{d+1}M_{\MA}$ and $s_{d+1}M_{\MB}$ are $d$-pre-Calabi-Yau structures on graded quivers $\MA$ and $\MB$, respectively, and where $s_{d+1}\mathbf{F}$ is a $d$-pre-Calabi-Yau morphism $(\MA,s_{d+1}M_{\MA})\rightarrow (\MB,s_{d+1}M_{\MB})$ (see Definition \ref{def:ell}). 
In subsubsection \ref{subsection:fixed-pCY}, we fix $d$-pre-Calabi-Yau categories $(\MA,s_{d+1}M_{\MA})$ and $(\MB,s_{d+1}M_{\MB})$, and define an $L_{\infty}$-algebra whose Maurer-Cartan elements are $s_{-1}(s_{d+1}\mathbf{F})$, where $s_{d+1}\mathbf{F} : (\MA,s_{d+1}M_{\MA})\rightarrow (\MB,s_{d+1}M_{\MB})$ is a $d$-pre-Calabi-Yau morphism (see Definition \ref{def:bar-ell}). In subsection \ref{section:htpy}, we define the two notions of homotopy for $d$-pre-Calabi-Yau morphisms analogous two those of Section \ref{section:A-inf-htpy} using the results of subsection \ref{section:brackets}.
We then prove that the notion of homotopy for pre-Calabi-Yau morphisms is stable under left and right compositions with pre-Calabi-Yau morphisms (see Propositions \ref{prop:htpy-cp-left} and \ref{prop:htpy-cp-right}), and show that homotopy equivalences are quasi-isomor\-phisms (see Proposition \ref{prop:htpy-equiv-q-iso}).

In Section \ref{section:relation-htpy}, we relate the homotopy theory of pre-Calabi-Yau categories with the one of $A_{\infty}$-categories, using the functor $\mathcal{P}$ defined in Theorem \ref{thm:main-article-1}. 
Namely, we show that $\mathcal{P}$ sends homotopic $d$-pre-Calabi-Yau morphisms to pairs of weak homotopic $A_{\infty}$-morphisms.

\textbf{Acknowledgements.} This work is part of a PhD thesis supervised by Estanislao Herscovich and Hossein Abbaspour. The author thanks them for the useful advice and comments about this paper. She also thanks D. Calaque and B. Vallette for useful discussions about homotopy. 

This work was supported by the French National Research Agency in the framework of the ``France 2030" program (ANR-15-IDEX-0002) and the LabEx PERSYVAL-Lab (ANR-11-LABX-0025-01).

\section{Notations and conventions}\label{section:not-conv}

We fix a field $\kk$ a field of characteristic different from $2$ and $3$ and an integer $d\in\ZZ$. To simplify we will denote $\otimes$ for $\otimes_{\kk}$. We denote by $\NN = \{0,1,2,\dots \}$ the set of natural numbers and we define $\NN^*=\NN\setminus\{0\}$.
For $i,j\in\NN$, we define the interval of integers $\llbracket i,j\rrbracket=\{n\in\NN | i\leq n\leq j\}$. 

Recall that if we have a (cohomologically) graded vector space $V=\oplus_{i\in\ZZ}V^i$, we define for $n\in\ZZ$ the graded vector space $V[n]$ given by $V[n]^i=V^{n+i}$ for each $i\in\ZZ$ and the map $s_{V,n} : V\rightarrow V[n]$ whose underlying set theoretic map is the identity. 
Moreover, if $f:V\rightarrow W$ is a morphism of graded vector spaces, we define the map $f[n] : V[n]\rightarrow W[n]$ sending an element $s_{V,n}(v)$ to $s_{W,n}(f(v))$ for all $v\in V$.
We will denote $s_{V,n}$ simply by $s_n$ when there is no possible confusion, and $s_1$ just by $s$.

We now recall the Koszul sign rules, that are the ones we use to determine the signs appearing in this thesis. Given graded vector spaces $V_1,\dots,V_n$ and $\sigma\in\mathfrak{S}_n$, 
we define the map
\[
\tau^{\sigma}_{V_1,\dots,V_n} : V_1\otimes\dots\otimes V_n \rightarrow V_{\sigma^{-1}(1)}\otimes\dots\otimes V_{\sigma^{-1}(n)}
\]
by
\begin{equation}
    \tau^{\sigma}_{V_1,\dots,V_n}(v_1\otimes\dots\otimes v_n)=(-1)^{\epsilon}(v_{\sigma^{-1}(1)}\otimes\dots\otimes v_{\sigma^{-1}(n)})
\end{equation}
where
\begin{small}
    \begin{equation}
        \epsilon=\hskip-6mm\sum\limits_{\substack{i>j\\ \sigma^{-1}(i)<\sigma^{-1}(j)}}\hskip-6mm |v_{\sigma^{-1}(i)}||v_{\sigma^{-1}(j)}|
    \end{equation}
\end{small}
and where $v_i\in V_i$ is a homogeneous element of degree $|v_i|$ for every $i\in\llbracket 1,n \rrbracket$.

Throughout this paper, when we consider an element $v$ of degree $|v|$ in a graded vector space $V$, we mean a homogeneous element $v$ of $V$.
Moreover, we will denote by $\id$ the identity map of every space of morphisms, without specifying it, when there is no possible confusion.
All the products in the paper will be products in the category of graded vector spaces.

Consider a graded quiver $\MA$ with set of objects $\MO$.
We will denote \[
\bar{\MO}=\bigsqcup_{n\in\NN^*}\MO^n
\]and more generally, we will denote by $\doubar{\MO}$ the set formed by all finite tuples of elements of $\bar{\MO}$, \textit{i.e.} 
\begin{equation}
\begin{split}
\bar{\bar{\MO}}=\bigsqcup_{n\in\NN^*} \bar{\MO}^n=\bigsqcup_{n\in\NN^*}\bigsqcup_{(p_1,\dots,p_n)\in\mathcal{T}_n}\MO^{p_1}\times\dots\times \MO^{p_n}
\end{split}
\end{equation}  
where $\mathcal{T}_n=\NN^n$ for $n>1$ and $\mathcal{T}_1=\NN^*$.

Given $\bar{x}=(x_1,\dots,x_n)\in\bar{\MO}$ we define its \textbf{\textcolor{ultramarine}{length}} as $\llg(\bar{x})=n$, its \textbf{\textcolor{ultramarine}{left term}} as $\llt(\bar{x})=x_1$ and its \textbf{\textcolor{ultramarine}{right term}} as $\rrt(\bar{x})=x_n$. For $i\in\llbracket 1,n\rrbracket$, we define $\bar{x}_{\leq i}=(x_1,\dots,x_i)$, $\bar{x}_{\geq i}=(x_i,\dots,x_n)$ and for $j>i$, $\bar{x}_{\llbracket i,j\rrbracket}=(x_i,x_{i+1},\dots,x_j)$. One can similarly define $\bar{x}_{<i}$ and $\bar{x}_{>i}$. Moreover, given $\bar{x}=(x_1,\dots,x_n)\in\bar{\MO}$, we will denote 
$\MA^{\otimes \bar{x}}={}_{x_1}\MA_{x_2}\otimes {}_{x_2}\MA_{x_3}\otimes\dots\otimes {}_{x_{n-1}}\MA_{x_{n}}$
and we will often denote an element of $\MA^{\otimes\bar{x}}$ as $a_1,a_2,\dots,a_{\llg(\bar{x})-1}$ instead of $a_1\otimes a_2\otimes \dots \otimes a_{\llg(\bar{x})-1}$ for $a_i\in {}_{x_i}\MA_{x_{i+1}}$, $i\in\llbracket 1,\llg(\bar{x})-1\rrbracket$. Similarly, given $\doubar{x}=(\bar{x}^1,\dots,\bar{x}^n)\in\doubar{\MO}$ we define its \textbf{\textcolor{ultramarine}{length}} as $\llg(\doubar{x})=n$, its \textbf{\textcolor{ultramarine}{left term}} as $\llt(\doubar{x})=\bar{x}^1$, its \textbf{\textcolor{ultramarine}{right term}} as $\rrt(\doubar{x})=\bar{x}^n$, $\MA^{\otimes \doubar{x}}=\MA^{\otimes \bar{x}^1}\otimes \MA^{\otimes \bar{x}^2}\otimes\dots\otimes \MA^{\otimes \bar{x}^n}$ and we define $N(\doubar{x})=\llg(\Bar{x}^1)+\dots+\llg(\Bar{x}^n)$. We also define the inverse of a tuple $\bar{x}=(x_1,\dots,x_n)\in\bar{\MO}$ as $\bar{x}^{-1}=(x_n,x_{n-1},\dots,x_1)$.

If $\sigma\in\mathfrak{S}_n$ and $\bar{x}=(x_1,\dots,x_n)\in\MO^n$, we define $\bar{x}\cdot\sigma=(x_{\sigma^{-1}(1)},x_{\sigma^{-1}(2)},\dots,x_{\sigma^{-1}(n)})$ and similarly, given $\doubar{x}=(\bar{x}^1,\dots,\bar{x}^n)\in\bar{\MO}^n$, we define $\doubar{x}\cdot\sigma=(\bar{x}^{\sigma^{-1}(1)},\bar{x}^{\sigma^{-1}(2)},\dots,\bar{x}^{\sigma^{-1}(n)})$.
We denote by $C_n$ the subgroup of $\mathfrak{S}_n$ generated by the cycle $\sigma=(1 2\dots n)$ which sends $i\in \llbracket 1,n-1\rrbracket $ to $i+1$ and $n$ to $1$.

We will denote 
\begin{equation}
    \begin{split}
        \mathcal{P}_k^n&=\{\big(\bar{i}=(i_1,\dots,i_k),\bar{j}=(j_1,\dots,j_{n-k})\big) | i_1<\dots<i_k, j_1<\dots<j_{n-k} \\&\hspace{6cm} \text{ and }\{i_1,\dots,i_k\}\cup\{j_1,\dots,j_{n-k}\}=\llbracket 1,n\rrbracket\},
    \end{split}
\end{equation}
and, to simplify, for any $n\in\NN$ and $\sigma\in\mathfrak{S}_n$, we denote the image of $i\in\llbracket 1,n\rrbracket$ under $\sigma$ by $\sigma_i$ instead of $\sigma(i)$.
Moreover, given $\sigma\in\mathfrak{S}_n$ and $f_1,\dots,f_n$, we define \[\delta^{\sigma}(f_{n},\dots,f_{1})=\sum\limits_{\substack{i>j\\\sigma_i<\sigma_j}}|f_{\sigma_i}||f_{\sigma_j}|.\]

In this paper, we will often make the use of the diagrammatic calculus introduced in \cite{ktv}. 
We refer the reader to \cite{moi} for a detailed account of the involved definitions and terminology, which we shall follow in this article. 
\section{\texorpdfstring{A quick reminder on $L_{\infty}$-algebras and their Maurer-Cartan elements}{L-infinity algebras and their MC elements}} \label{section:L-inf}
We recall here the definition of $L_{\infty}$-algebras and their Maurer-Cartan elements as well as the basics on the homotopy theory of the latter. For a more detailed account on those notions we refer the reader to \cite{getzler,yalin}.

\begin{definition}
    Given a graded vector space $\mathcal{L}$, an \textbf{\textcolor{ultramarine}{$L_{\infty}$-bracket}} on $\mathcal{L}$ is a collection of maps $\ell=(\ell^n)_{n\in\NN^*}$ of degree $0$ that are anti-symmetric, \textit{i.e.}
    \begin{equation}
        \ell^n(l_n,\dots,l_{i},l_{i+1},\dots,l_1)=-(-1)^{|l_i||l_{i+1}|}\ell^n(l_n,\dots,l_1)
    \end{equation}
     and that satisfies the higher Jacobi identities, \textit{i.e.}
    \begin{equation}
\label{eq:Jac-l}
    \begin{split}
    \sum\limits_{k=1}^n(-1)^k\sum\limits_{(\bar{i},\bar{j})\in\mathcal{P}_k^n}(-1)^{\Delta} \ell^{n-k+1}\big(\ell^k(l_{i_k},\dots,l_{i_1}),l_{j_{n-k}},\dots,l_{j_{1}}\big)=0
    \end{split}
\end{equation}
for every $n\in\NN^*$, $i\in\llbracket 1,n\rrbracket$, $l_j\in\mathcal{L}$ for $j\in\llbracket 1,n\rrbracket$ with 
\begin{small}
\begin{equation}
    \begin{split}
        \Delta&=\sum\limits_{r=1}^k|l_{i_r}|\sum\limits_{s=i_r+1}^{n}|l_s|+\sum\limits_{1\leq r<s\leq k}|l_{i_r}||l_{i_s}|+\sum\limits_{r=1}^k\sum\limits_{s=i_r+1}^{n}1+\sum\limits_{1\leq r<s\leq k}1.
    \end{split}
\end{equation}
\end{small}
A graded vector space endowed with an $L_{\infty}$-bracket is called a \textbf{\textcolor{ultramarine}{graded $L_{\infty}$-algebra}}.
\end{definition}

\begin{remark}
\label{remark:lie-is-linfinity}
    Consider a graded Lie algebra $(L,[-,-])$. 
    Then, $(L,[-,-])$ is an $L_{\infty}$-algebra whose $L_{\infty}$-bracket $\ell=(\ell^n)_{n\in\NN^*}$ is given by $\ell^2=[-,-]$ and $\ell^n=0$ for every $n\neq 2$.
\end{remark}

\begin{definition}
    Given a product of graded vector spaces $\mathcal{L}=\prod_{i\in I}L_i$ endowed with a graded Lie bracket $\ell$ such that for every $f\in \mathcal{L}^1$ and $i\in I$ the projection $\pi_i\circ \sum_{n\geq 0}\ell^n(f,\dots,f)$ on $L_i$ is a finite sum, we define a \textbf{\textcolor{ultramarine}{Maurer-Cartan element of $(\mathcal{L},\ell)$}} as a degree $1$ element $f\in \mathcal{L}^1$ such that for every $i\in I$, $\pi_i\circ \sum_{n\geq 0}\ell^n(f,\dots,f)=0$.
\end{definition}
We now recall the definition of the equivalence relation called homotopy on the sets of Maurer-Cartan elements.
\begin{definition}
\label{def:htpic-MC-elts}
    Given a graded $L_{\infty}$-algebra $(\mathcal{L},\ell)$, two Maurer-Cartan elements $f,g\in\mathcal{L}$ are \textbf{\textcolor{ultramarine}{homotopic}} if 
    there exists a Maurer-Cartan element $a(t)+b(t)dt\in \mathcal{L}[t,dt]$ such that $a(0)=f$ and $a(1)=g$, where $\mathcal{L}[t,dt]=\mathcal{L}[t]\oplus \mathcal{L}[t]\partial t$ is an $L_{\infty}$-algebra with differential $\partial$ given by 
    \[
    \partial(f+g \partial t)=\frac{\partial f}{\partial t}\partial t
    \]
    and $L_{\infty}$-bracket given by 
    \[
    [f_1+g_1\partial t,\dots,f_n+g_n\partial t]=\ell^n(f_1,\dots,f_n)+\sum_{i=1}^n(-1)^{\sum\limits_{j=i+1}^n|f_j|}\ell^n(f_1,\dots,f_{i-1},g_i,f_{i+1},\dots,f_n)\partial t.
    \]
\end{definition}

\begin{proposition}
    \label{prop:htpy-equiv-rel}
    Being homotopic Maurer-Cartan elements is an equivalence relation.
\end{proposition}
\begin{proof}
    See for example \cite{getzler}.
\end{proof}

\begin{remark}
\label{remark:htpic-MC-elts-equiv}
    An element $a(t)+b(t)dt\in \mathcal{L}[t,dt]$ is a Maurer-Cartan element of $\mathcal{L}[t,dt]$ if and only if $a(t)$ is a Maurer-Cartan element of $\mathcal{L}$ and \[\frac{\partial a}{\partial t}(t)^{\doubar{x}}=\sum\limits_{n=2}^{\infty}\frac{1}{(n-1)!}\ell^n\big(a(t),\dots,a(t),b(t)\big)^{\doubar{x}}\] for every $t\in\Bbbk$, $\doubar{x}\in\doubar{\MO}_{\MA}$. This is a generalisation of a characterisation of Maurer-Cartan elements of Lie algebras given in \cite{manetti, yalin}.
\end{remark}

\section{\texorpdfstring{Homotopy theory of $A_{\infty}$-morphisms}{Some results for the case of A-inf categories}}
\label{section:A-inf-htpy}

In this section, we recall the way of seeing $A_{\infty}$-structures together with an $A_{\infty}$-morphism between them as Maurer-Cartan elements of a certain graded $L_{\infty}$-algebra as it is done in \cite{borisov1}.
We then present a way of seeing $A_{\infty}$-morphisms between fixed $A_{\infty}$-categories as Maurer-Cartan elements of another $L_{\infty}$-algebra, leading to a different notion of homotopy between $A_{\infty}$-morphisms analogous to the one presented in \cite{kraft-schnitzer} in the case of $L_{\infty}$-algebras. We finally prove that this notion of homotopy is stable under left and right composition with an $A_{\infty}$-morphism and that homotopy equivalences coincide with quasi-isomorphisms.
We refer the reader to \cite{lefevre} for the definitions of $A_{\infty}$-categories and their morphisms. 
\begin{definition}
\label{def:A-inf-L-inf-borisov}
    Consider graded quivers $\MA$ and $\MB$ with respective sets of objects $\MO_{\MA}$ and $\MO_{\MB}$ and a map $\Phi_0 : \MO_{\MA}\rightarrow \MO_{\MB}$. We define an $\infty$-bracket $\ell$ on $C(\MB)[1]\oplus C_{\Phi_0}(\MA,\MB)\oplus C(\MA)[1]$ given by 
    \begin{itemize}
    \item $\ell^2_{|C(\MB)[1]^{\otimes 2}}(sm,sm')=[sm,sm']_G$,
    \item  $\ell^2_{|C(\MA)[1]^{\otimes 2}}(sm,sm')=[sm,sm']_G$,
    \item $\ell^{n}(f_{n-1},\dots,f_{i+1},sm,f_{i},\dots,f_{1})=(-1)^{\epsilon_i}\sum\limits_{\sigma\in\mathfrak{S}_{n-1}}(-1)^{\delta^{\sigma}(sf_{n-1},\dots,sf_{1})}\varphi_{sm}^{\sigma}(sf_{n-1},\dots,sf_{1})$ for every $n\geq 2$, $i\in\llbracket 0,n-1\rrbracket$ with $sm\in C(\MB)[1]$ and $f_j\in C_{\Phi_0}(\MA,\MB)$ for $j\in\llbracket 1,n-1\rrbracket$,
    \item $\ell^2(f,sm)=-(-1)^{|f||sm|}\ell^2(sm,f)=\psi_{sm}(sf)$ for $sm\in C(\MA)[1]$ and $f\in C_{\Phi_0}(\MA,\MB)$,
    \item $\ell^n(l_n,\dots,l_1)=0$ if 
    \begin{itemize}
        \item  there exists $i,j\in\llbracket 1,n \rrbracket$, $i\neq j$ such that $l_i\in C(\MA)[1]$ and $l_j\in C(\MB)[1]$,
        \item $n>2$ and there exists $i\in\llbracket 1,n \rrbracket$ such that $l_i\in C(\MA)[1]$, 
        \item $n>2$ and there exists $i,j\in\llbracket 1,n \rrbracket$, $i\neq j$ such that $l_i,l_j\in C(\MB)[1]$,
        \item $l_i\in C_{\Phi_0}(\MA,\MB)$ for every $i\in\llbracket 1,n\rrbracket$,
    \end{itemize} 
\end{itemize}
where 
\begin{small}
    \begin{equation}
        \epsilon_i=i+1+i|sm|+|sm|\sum\limits_{j=i+1}^{n-1}|sf_j|+\sum\limits_{j=1}^{n-1}(j-1)|sf_j|
    \end{equation}
\end{small}
and where $\varphi_{sm}^{\sigma}(sf_{n-1},\dots,sf_{1})^{\Bar{x}}=\sum\mathcal{E}(\mathcalboondox{D})$, $\psi_{sm}(sf)^{\Bar{x}}=\sum\mathcal{E}(\mathcalboondox{D'})$ where the sums are over all the filled diagrams $\mathcalboondox{D}$ and $\mathcalboondox{D'}$ of type $\Bar{x}$ and of the form 

\begin{minipage}{21cm}
    \begin{tikzpicture}[line cap=round,line join=round,x=1.0cm,y=1.0cm]
\clip(-4,-1.8) rectangle (4,1.8);
     \draw [line width=0.5pt] (0.,0.) circle (0.5cm);
    \draw [rotate=0] [->, >= stealth, >= stealth] (0.5,0) -- (0.9,0);
    \draw [rotate=-90] [<-, >= stealth, >= stealth](0.5,0) -- (0.9,0);
    \draw [rotate=180] [<-, >= stealth, >= stealth](0.5,0) -- (0.9,0);
    \draw [rotate=135] [<-, >= stealth, >= stealth](0.5,0) -- (0.9,0);
    \draw [line width=0.5pt] (0.,-1.2) circle (0.3cm);
    \draw[rotate around={-90:(0,-1.2)},shift={(0.3,-1.2)}]\doublefleche;
    \draw [line width=0.5pt] (-1.2,0) circle (0.3cm);
     \draw[rotate around={180:(-1.2,0)},shift={(-0.9,0)}]\doublefleche;
    \draw [line width=0.5pt] (-0.84,0.84) circle (0.3cm);
    \draw[rotate around={135:(-0.84,0.84)},shift={(-0.54,0.84)}]\doublefleche;
\begin{scriptsize}
\draw [fill=black] (-0.45,-0.4) circle (0.3pt);
\draw [fill=black] (-0.3,-0.5) circle (0.3pt);
\draw [fill=black] (-0.5,-0.25) circle (0.3pt);
\end{scriptsize}
\draw(3,0.25)node[anchor=north]{and};
\draw (0,0.2) node[anchor=north] {$m_{\MB}$};
\draw (-1.2,0.28) node[anchor=north] {$f_{\sigma_2}$};
\draw (0.18,-0.92) node[anchor=north] {$f_{\sigma_{n-1}}$};
\draw (-0.84,1.12) node[anchor=north] {$f_{\sigma_1}$};
\end{tikzpicture}
\begin{tikzpicture}[line cap=round,line join=round,x=1.0cm,y=1.0cm]
\clip(-2,-1.8) rectangle (4,1.8);
    \draw [line width=0.5pt] (0.,0.) circle (0.5cm);
    \shadedraw [rotate=180, shift={(0.5cm,0cm)}] \doublefleche;
    \draw [line width=0.5pt] (1.5,0.) circle (0.5cm);
    \shadedraw[shift={(1cm,0cm)},rotate=180] \doubleflechescindeeleft;
    \shadedraw[shift={(1cm,0cm)},rotate=180] \doubleflechescindeeright;
    \shadedraw[shift={(1cm,0cm)},rotate=180] \fleche;
    \draw [rotate around={0:(1.5,0)}] [->,, >= stealth, >= stealth](2,0) -- (2.4,0);
    \draw (0,0.2) node[anchor=north] {$m_{\MA}$};
\draw (1.5,0.25) node[anchor=north] {$f$};
\end{tikzpicture}
\end{minipage}

\noindent respectively.
\end{definition}

\begin{proposition}
\label{prop:A-inf-L-inf-borisov}
    The graded quiver $C(\MB)[1]\oplus C_{\Phi_0}(\MA,\MB)\oplus C(\MA)[1]$ endowed with the $\infty$-bracket $\ell$ is a graded $L_{\infty}$-algebra whose Maurer-Cartan elements are in bijection with triples $(sm_{\MB},\mathbf{f},sm_{\MA})$ where $sm_{\MB}$ and $sm_{\MA}$ are $A_{\infty}$-structures on $\MB$ and $\MA$ respectively and where $(\Phi_0,s\mathbf{f}) : (\MA,sm_{\MA})\rightarrow (\MB,sm_{\MB})$ is an $A_{\infty}$-morphism.
\end{proposition}
\begin{proof}
    See \cite{borisov1}.
\end{proof}

\begin{definition}
\label{def:weak-htpy-A-inf}
    Consider graded quivers $\MA$ and $\MB$ with respective sets of objects $\MO_{\MA}$ and $\MO_{\MB}$ and a map $\Phi_0:\MO_{\MA}\rightarrow \MO_{\MB}$.
    Given $A_{\infty}$-structures $sm_{\MA}$ and $sm'_{\MA}$ on $\MA$ (resp. $sm_{\MB}$ and $sm'_{\MB}$ on $\MB$), we say that two $A_{\infty}$-morphisms 
    \begin{equation}
        (\Phi_0,s\mathbf{f}) : (\MA,sm_{\MA})\rightarrow (\MB,sm_{\MB}) \text{ and } (\Phi_0,s\mathbf{g}) : (\MA,sm'_{\MA})\rightarrow (\MB,sm'_{\MB})
    \end{equation} are \textbf{\textcolor{ultramarine}{weakly homotopic}} if the corresponding Maurer-Cartan elements $(sm_{\MB},\mathbf{f},sm_{\MA})$ and $(sm'_{\MB},\mathbf{g},sm'_{\MA})$ of $(C(\MB)[1]\oplus C_{\Phi_0}(\MA,\MB)\oplus C(\MA)[1],\ell)$ are homotopic in the sense of Definition \ref{def:htpic-MC-elts}.
\end{definition}

\begin{definition}
    Consider $A_{\infty}$-categories $(\MA,sm_{\MA})$, $(\MB,sm_{\MB})$ and a map $\Phi_0 : \MO_{\MA}\rightarrow \MO_{\MB}$ between the sets of objects of the underlying graded quivers. We endow the graded vector space $C_{\Phi_0}(\MA,\MB)$ with the $\infty$-bracket $\bar{\ell}$ given by 
    \[
    \Bar{\ell}^1(\mathbf{f})^{\Bar{x}}=m_{\MB}^{\llt(\Bar{x}),\rrt(\Bar{x})}\circ sf^{\Bar{x}}-(-1)^{|s\mathbf{f}|}(\mathbf{f}\upperset{G}{\circ}sm_{\MA})^{\Bar{x}} \in C^{\Bar{x}}(\MA,\MB)
    \]
    for $\mathbf{f}\in C_{\Phi_0}(\MA,\MB)$ and
    \begin{equation}
        \Bar{\ell}^n(f_n,\dots,f_1)=(-1)^{\epsilon}\sum\limits_{\sigma\in\mathfrak{S}_n}(-1)^{\delta^{\sigma}(sf_{n},\dots,sf_{1})}\psi^{\sigma}_{m_{\MB}}(sf_{n},\dots,sf_{1})
    \end{equation}
    for $n\geq 2$, with 
    \begin{small}
        \begin{equation}
            \epsilon=\sum\limits_{i=1}^n(i-1)|sf_i|
        \end{equation}
    \end{small}
    and where $\psi^{\sigma}_{m_{\MB}}(sf_{n},\dots,sf_{1})^{\Bar{x}}=\sum\mathcal{E}(\mathcalboondox{D})$ where the sum is over all the filled diagrams $\mathcalboondox{D}$ of type $\Bar{x}$ and of the form

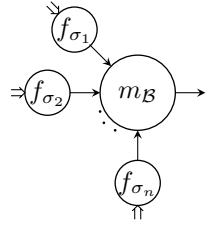
\begin{equation}
\label{eq:diag-A-inf}
\begin{tikzpicture}[line cap=round,line join=round,x=1.0cm,y=1.0cm]
\clip(-7,-1.8) rectangle (4,1.8);
     \draw [line width=0.5pt] (0.,0.) circle (0.5cm);
    \draw [rotate=0] [->, >= stealth, >= stealth] (0.5,0) -- (0.9,0);
    \draw [rotate=-90] [<-, >= stealth, >= stealth](0.5,0) -- (0.9,0);
    \draw [rotate=180] [<-, >= stealth, >= stealth](0.5,0) -- (0.9,0);
    \draw [rotate=135] [<-, >= stealth, >= stealth](0.5,0) -- (0.9,0);
    \draw [line width=0.5pt] (0.,-1.2) circle (0.3cm);
    \draw[rotate around={-90:(0,-1.2)},shift={(0.3,-1.2)}]\doublefleche;
    \draw [line width=0.5pt] (-1.2,0) circle (0.3cm);
     \draw[rotate around={180:(-1.2,0)},shift={(-0.9,0)}]\doublefleche;
    \draw [line width=0.5pt] (-0.84,0.84) circle (0.3cm);
    \draw[rotate around={135:(-0.84,0.84)},shift={(-0.54,0.84)}]\doublefleche;
\begin{scriptsize}
\draw [fill=black] (-0.45,-0.4) circle (0.3pt);
\draw [fill=black] (-0.3,-0.5) circle (0.3pt);
\draw [fill=black] (-0.5,-0.25) circle (0.3pt);
\end{scriptsize}
\draw (0,0.2) node[anchor=north] {$m_{\MB}$};
\draw (-1.2,0.28) node[anchor=north] {$f_{\sigma_2}$};
\draw (0,-0.92) node[anchor=north] {$f_{\sigma_n}$};
\draw (-0.84,1.12) node[anchor=north] {$f_{\sigma_1}$};
\end{tikzpicture}
\end{equation}
\end{definition}

\begin{lemma}
    \label{lemma:MC-A-inf}
    For every $\bar{x}\in\bar{\MO}_{\MA}$, there is a finite number of diagrams of the form \eqref{eq:diag-A-inf} and of type $\bar{x}$. In particular, given $f\in C_{\Phi_0}(\MA,\MB)$ of degree $1$, the projection $\pi_{\bar{x}}\circ \sum_{n\geq 0}\ell^n(f,\dots,f)$ on $C_{\Phi_0}^{\bar{x}}(\MA,\MB)$ is a finite sum for every $\bar{x}\in\bar{\MO}_{\MA}$.
\end{lemma}
\begin{proof}
    Consider $\bar{x}\in\bar{\MO}_{\MA}$. A diagram of type $\bar{x}$ carries $\llg(\bar{x})$ incoming arrows. Since every disc filled with an element of $C_{\Phi_0}(\MA,\MB)$ carries at least one incoming arrow, there are at most $\llg(\bar{x})$ such discs in a diagram of the form \eqref{eq:diag-A-inf} and of type $\bar{x}$. Therefore, we have a finite number of such diagrams.
\end{proof}

\begin{proposition}
\label{prop:A-inf-L-inf}
   Consider $A_{\infty}$-categories $(\MA,sm_{\MA})$ and $(\MB,sm_{\MB})$ together with a map  between the sets of objects of the underlying graded quivers $\Phi_0 : \MO_{\MA}\rightarrow \MO_{\MB}$. 
   
   \noindent Then, $(C_{\Phi_0}(\MA,\MB),\Bar{\ell})$ is a graded $L_{\infty}$-algebra whose Maurer-Cartan elements are in bijection with $A_{\infty}$-morphisms $(\Phi_0,s\mathbf{f}) : (\MA,sm_{\MA})\rightarrow (\MB,sm_{\MB})$.
\end{proposition}
\begin{proof}
    We first prove that the $\infty$-bracket $\Bar{\ell}$ is anti-symmetric. We thus consider $n>2$ and elements $f_1,\dots,f_n\in C_{\Phi_0}(\MA,\MB)$. For $i\in\llbracket 1,n\rrbracket$, we have that 
    \allowdisplaybreaks
    \begin{align*}
         &\Bar{\ell}^n(f_n,\dots,f_i,f_{i+1},\dots,f_1)
         \\&=(-1)^{\epsilon'}\sum\limits_{\sigma\in\mathfrak{S}_n}(-1)^{\delta^{\sigma}(sf_n,\dots,sf_i,sf_{i+1},\dots,sf_1)}\psi^{\sigma}_{sm_{\MB}}(sf_n,\dots,sf_i,sf_{i+1},\dots,sf_1)
         \\&=(-1)^{|sf_i||sf_{i+1}|+\epsilon'}\sum\limits_{\sigma\in\mathfrak{S}_n}(-1)^{\delta^{\sigma}(sf_n,\dots,sf_1)}\psi^{\sigma}_{sm_{\MB}}(sf_n,\dots,sf_1)
         \\&=-(-1)^{|f_i||f_{i+1}|}\Bar{\ell}^n(f_n,\dots,f_1)
    \end{align*}
where 
\begin{small}
    \begin{equation}
        \epsilon'=\sum\limits_{\substack{j=1\\j\neq i,i+1}}^n(j-1)|sf_j|+i|sf_i|+(i-1)|sf_{i+1}|.
    \end{equation}
\end{small}
We now prove that $\Bar{\ell}$ satisfies the higher Jacobi identities.
Consider $n\in\NN^*$ as well as elements $f_1,\dots,f_n\in C_{\Phi_0}(\MA,\MB)$.
We have to show that 
\begin{equation}
\label{eq:Jac-A-inf}
    \begin{split}
        \sum\limits_{k=1}^n(-1)^{k}\sum\limits_{(\Bar{i},\Bar{j})\in\mathcal{P}_k^n}(-1)^{\Delta}\Bar{\ell}^{n-k+1}(\Bar{\ell}^k(l_{i_k},\dots,l_{i_1}),l_{j_{n-k}},\dots,l_{j_1})=0
    \end{split}
\end{equation}
where
\begin{small}
\begin{equation}
    \begin{split}
        \Delta&=\sum\limits_{r=1}^k|l_{i_r}|\sum\limits_{s=i_r+1}^{n}|l_s|+\sum\limits_{1\leq r<s\leq k}|l_{i_r}||l_{i_s}|+\sum\limits_{r=1}^k\sum\limits_{s=i_r+1}^{n}1+\sum\limits_{1\leq r<s\leq k}1.
    \end{split}
\end{equation}
\end{small}
The term of \eqref{eq:Jac-A-inf} indexed by $k=1$ is 
\begin{equation}
\label{eq:Jac-A-inf-k=1}
    \begin{split}
        &-\sum\limits_{i=1}^n(-1)^{\delta}\Bar{\ell}^{n}(\Bar{\ell}^1(f_i),f_{n},\dots,\hat{f_i},\dots,f_1)
        \\&=-\sum\limits_{i=1}^n(-1)^{\delta+\delta'}\big(\sum\limits_{\sigma\in\mathfrak{S}_n}(-1)^{\delta^{\sigma}(sm_{\MB}\upperset{M}{\circ} sf_i,sf_n,\dots,\hat{f_i},\dots,sf_1)}\psi^{\sigma}_{sm_{\MB}}(sm_{\MB}\upperset{M}{\circ} sf_i,sf_n,\dots,\hat{f_i},\dots,sf_1)
        \\&\hskip3cm -\sum\limits_{\sigma\in\mathfrak{S}_n}(-1)^{|sf_i|+\delta^{\sigma}(sf_i\upperset{G}{\circ}sm_{\MA},sf_n,\dots,\hat{f_i},\dots,sf_1)}\psi^{\sigma}_{sm_{\MB}}(sf_i\upperset{G}{\circ}sm_{\MA},sf_n,\dots,\hat{f_i},\dots,sf_1)\big)
    \end{split}
\end{equation}
where 
\begin{small}
\begin{equation}
    \begin{split}
        \delta=|f_i|\sum\limits_{j=i+1}^{n}|f_j|+n-i
         \text{, }\delta'=|f_i|(n-1)+\sum\limits_{j=1}^{i-1}(j-1)|sf_j|+\sum\limits_{j=i+1}^{n}j|sf_j|
    \end{split}
\end{equation}  
\end{small}
and where 
\begin{small}
\begin{equation}
    \psi^{\sigma}_{sm_{\MB}}(sm_{\MB}\upperset{M}{\circ} sf_i,sf_n,\dots,\hat{f_i},\dots,sf_1)^{\Bar{x}}=\sum\mathcal{E}(\mathcalboondox{D_1})\text{, }\psi^{\sigma}_{sm_{\MB}}(sf_i\upperset{G}{\circ}sm_{\MA},sf_n,\dots,\hat{f_i},\dots,sf_1)^{\Bar{x}}=\sum\mathcal{E}(\mathcalboondox{D'_1})
\end{equation}
\end{small}
where the sums are over all the filled diagrams $\mathcalboondox{D_1}$ and $\mathcalboondox{D'_1}$ of type $\Bar{x}$ and of the form

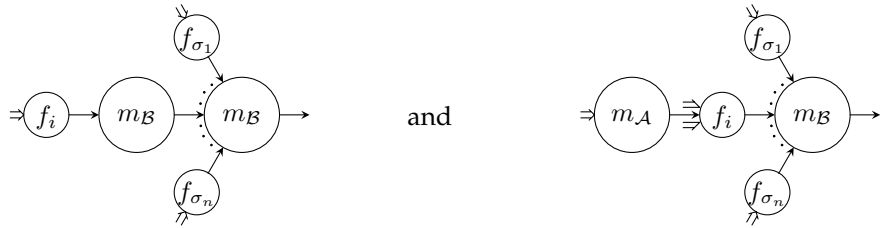
\begin{minipage}{21cm}
    \begin{tikzpicture}[line cap=round,line join=round,x=1.0cm,y=1.0cm]
\clip(-4.5,-2) rectangle (4,2);
    \draw (0,0) circle (0.5cm);
    \draw[->,>=stealth](0.5,0)--(0.9,0);
    \draw[rotate=180][<-,>=stealth] (0.5,0)--(0.9,0);
    \draw (-1.4,0) circle (0.5cm);
    \draw[rotate around={180:(-1.4,0)}][<-,>=stealth](-0.9,0)--(-0.5,0);
    \draw (-2.6,0) circle (0.3cm);
    \draw[rotate around={180:(-2.6,0)},shift={(-2.3,0)}]\doublefleche;
    \draw[rotate=120][<-,>=stealth] (0.5,0)--(0.9,0);
    \draw (-0.6,1.03) circle (0.3cm);
    \draw[rotate around={120:(-0.6,1.03)},shift={(-0.3,1.03)}]\doublefleche;
    \draw[rotate=-120][<-,>=stealth] (0.5,0)--(0.9,0);
    \draw (-0.6,-1.03) circle (0.3cm);
    \draw[rotate around={-120:(-0.6,-1.03)},shift={(-0.3,-1.03)}]\doublefleche;
\draw[fill=black] (-0.55,0.15) circle (0.3pt);
\draw[fill=black] (-0.5,0.3) circle (0.3pt);
\draw[fill=black] (-0.4,0.4) circle (0.3pt);
\draw[rotate=60][fill=black] (-0.55,0.15) circle (0.3pt);
\draw[rotate=60][fill=black] (-0.5,0.3) circle (0.3pt);
\draw[rotate=60][fill=black] (-0.4,0.4) circle (0.3pt);
\draw (2.5,0.25) node[anchor=north]{and};
\draw(0,0.2) node[anchor=north]{$m_{\MB}$};
\draw(-1.4,0.2) node[anchor=north]{$m_{\MB}$};
\draw(-2.6,0.25) node[anchor=north]{$f_{i}$};
\draw(-0.6,1.28) node[anchor=north]{$f_{\sigma_1}$};
\draw(-0.6,-0.76) node[anchor=north]{$f_{\sigma_n}$};
\end{tikzpicture}
\begin{tikzpicture}[line cap=round,line join=round,x=1.0cm,y=1.0cm]
\clip(-3.5,-2) rectangle (4.5,2);
    \draw (0,0) circle (0.5cm);
    \draw[->,>=stealth](0.5,0)--(0.9,0);
    \draw[rotate=180][<-,>=stealth] (0.5,0)--(0.9,0);
    \draw (-1.2,0) circle (0.3cm);
    \draw[rotate around={180:(-1.2,0)}][<-,>=stealth](-0.9,0)--(-0.5,0);
    \draw[rotate around={180:(-1.2,0)},shift={(-0.9,0)}]\doubleflechescindeeleft;
     \draw[rotate around={180:(-1.2,0)},shift={(-0.9,0)}]\doubleflechescindeeright;
    \draw (-2.4,0) circle (0.5cm);
    \draw[rotate around={180:(-2.4,0)},shift={(-1.9,0)}]\doublefleche;
    \draw[rotate=120][<-,>=stealth] (0.5,0)--(0.9,0);
    \draw (-0.6,1.03) circle (0.3cm);
    \draw[rotate around={120:(-0.6,1.03)},shift={(-0.3,1.03)}]\doublefleche;
    \draw[rotate=-120][<-,>=stealth] (0.5,0)--(0.9,0);
    \draw (-0.6,-1.03) circle (0.3cm);
    \draw[rotate around={-120:(-0.6,-1.03)},shift={(-0.3,-1.03)}]\doublefleche;
\draw[fill=black] (-0.55,0.15) circle (0.3pt);
\draw[fill=black] (-0.5,0.3) circle (0.3pt);
\draw[fill=black] (-0.4,0.4) circle (0.3pt);
\draw[rotate=60][fill=black] (-0.55,0.15) circle (0.3pt);
\draw[rotate=60][fill=black] (-0.5,0.3) circle (0.3pt);
\draw[rotate=60][fill=black] (-0.4,0.4) circle (0.3pt);
\draw(0,0.2) node[anchor=north]{$m_{\MB}$};
\draw(-2.4,0.2) node[anchor=north]{$m_{\MA}$};
\draw(-1.2,0.25) node[anchor=north]{$f_{i}$};
\draw(-0.6,1.28) node[anchor=north]{$f_{\sigma_1}$};
\draw(-0.6,-0.76) node[anchor=north]{$f_{\sigma_n}$};
\end{tikzpicture}
\end{minipage}

\noindent respectively, where the outgoing arrow of the disc filled with $m_{\MB}$ (resp. $f_i$) in $\mathcalboondox{D_1}$ (resp. $\mathcalboondox{D_1'}$) immediately follows counterclockwise the one of $f_{\sigma_{u-1}}$ where $u=\sigma^{-1}(1)$, whereas the term of \eqref{eq:Jac-A-inf} indexed by $k=n$ is 
\begin{equation}
\label{eq:Jac-A-inf-k=n}
    \begin{split}
        &\Bar{\ell}^{1}(\Bar{\ell}^n(f_{n},\dots,f_1))
        \\&=(-1)^{n+\sum\limits_{i=1}^n(i-1)|sf_i|}
        \sum\limits_{\sigma\in\mathfrak{S}_n}(-1)^{\delta^{\sigma}(sf_n,\dots,sf_1)}\big(sm_{\MB}\upperset{M}{\circ} \psi_{sm_{\MB}}^{\sigma}(sf_n\dots,sf_1)\big)
        \\&-(-1)^{n+\sum\limits_{i=1}^n(i-1)|sf_i|+\sum\limits_{i=1}^n|sf_i|+1}\sum\limits_{\sigma\in\mathfrak{S}_n}\sum\limits_{j=1}^n(-1)^{\delta^{\sigma}(sf_n,\dots,sf_1)+\sum\limits_{i=1}^{j-1}|sf_{\sigma_i}|}\big(\psi_{sm_{\MB}}^{\sigma,j}(sf_n\dots,sf_1)\upperset{G}{\circ}sm_{\MA}\big)
    \end{split}
\end{equation}
where $\big(sm_{\MB}\upperset{M}{\circ} \psi_{sm_{\MB}}^{\sigma}(sf_n\dots,sf_1)\big)^{\Bar{x}}=\sum\mathcal{E}(\mathcalboondox{D_2})$ and $\big(\psi_{sm_{\MB}}^{\sigma,j}(sf_n\dots,sf_1)\upperset{G}{\circ}sm_{\MA}\big)^{\Bar{x}}=\sum\mathcal{E}(\mathcalboondox{D'_2})$
where the sums are over all the filled diagrams $\mathcalboondox{D_2}$ and $\mathcalboondox{D'_2}$ of type $\Bar{x}$ an of the form

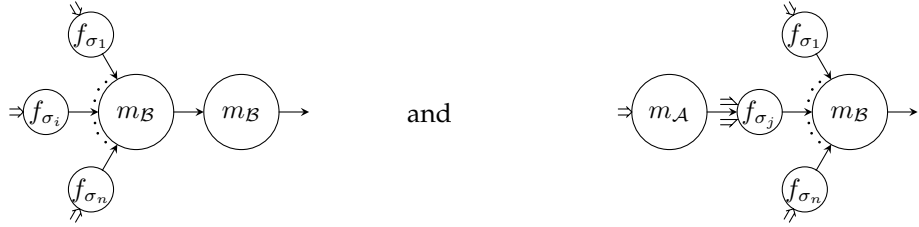
\begin{minipage}{21cm}
    \begin{tikzpicture}[line cap=round,line join=round,x=1.0cm,y=1.0cm]
\clip(-3.5,-2) rectangle (4.5,2);
    \draw (0,0) circle (0.5cm);
    \draw[->,>=stealth](0.5,0)--(0.9,0);
    \draw[rotate=180][<-,>=stealth] (0.5,0)--(0.9,0);
    \draw (-1.4,0) circle (0.5cm);
    \draw[rotate around={180:(-1.4,0)}][<-,>=stealth](-0.9,0)--(-0.5,0);
    \draw (-2.6,0) circle (0.3cm);
    \draw[rotate around={120:(-1.4,0)}][<-,>=stealth](-0.9,0)--(-0.5,0);
    \draw (-2,1.03) circle (0.3cm);
    \draw[rotate around={120:(-2,1.03)},shift={(-1.7,1.03)}]\doublefleche;
    \draw[rotate around={-120:(-1.4,0)}][<-,>=stealth](-0.9,0)--(-0.5,0);
    \draw (-2,-1.03) circle (0.3cm);
     \draw[rotate around={-120:(-2,-1.03)},shift={(-1.7,-1.03)}]\doublefleche;
    \draw[rotate around={180:(-2.6,0)},shift={(-2.3,0)}]\doublefleche;
\draw[fill=black] (-1.95,0.15) circle (0.3pt);
\draw[fill=black] (-1.9,0.3) circle (0.3pt);
\draw[fill=black] (-1.8,0.4) circle (0.3pt);
\draw[rotate around={60:(-1.4,0)}][fill=black] (-1.95,0.15) circle (0.3pt);
\draw[rotate around={60:(-1.4,0)}][fill=black] (-1.9,0.3) circle (0.3pt);
\draw[rotate around={60:(-1.4,0)}][fill=black] (-1.8,0.4) circle (0.3pt);
\draw (2.5,0.25) node[anchor=north]{and};
\draw(0,0.2) node[anchor=north]{$m_{\MB}$};
\draw(-1.4,0.2) node[anchor=north]{$m_{\MB}$};
\draw(-2.6,0.27) node[anchor=north]{$f_{\sigma_i}$};
\draw(-2,1.3) node[anchor=north]{$f_{\sigma_1}$};
\draw(-2,-0.76) node[anchor=north]{$f_{\sigma_n}$};
\end{tikzpicture}
\begin{tikzpicture}[line cap=round,line join=round,x=1.0cm,y=1.0cm]
\clip(-3.5,-2) rectangle (4.5,2);
    \draw (0,0) circle (0.5cm);
    \draw[->,>=stealth](0.5,0)--(0.9,0);
    \draw[rotate=180][<-,>=stealth] (0.5,0)--(0.9,0);
    \draw (-1.2,0) circle (0.3cm);
    \draw[rotate around={180:(-1.2,0)}][<-,>=stealth](-0.9,0)--(-0.5,0);
    \draw[rotate around={180:(-1.2,0)},shift={(-0.9,0)}]\doubleflechescindeeleft;
     \draw[rotate around={180:(-1.2,0)},shift={(-0.9,0)}]\doubleflechescindeeright;
    \draw (-2.4,0) circle (0.5cm);
    \draw[rotate around={180:(-2.4,0)},shift={(-1.9,0)}]\doublefleche;
    \draw[rotate=120][<-,>=stealth] (0.5,0)--(0.9,0);
    \draw (-0.6,1.03) circle (0.3cm);
    \draw[rotate around={120:(-0.6,1.03)},shift={(-0.3,1.03)}]\doublefleche;
    \draw[rotate=-120][<-,>=stealth] (0.5,0)--(0.9,0);
    \draw (-0.6,-1.03) circle (0.3cm);
    \draw[rotate around={-120:(-0.6,-1.03)},shift={(-0.3,-1.03)}]\doublefleche;
\draw[fill=black] (-0.55,0.15) circle (0.3pt);
\draw[fill=black] (-0.5,0.3) circle (0.3pt);
\draw[fill=black] (-0.4,0.4) circle (0.3pt);
\draw[rotate=60][fill=black] (-0.55,0.15) circle (0.3pt);
\draw[rotate=60][fill=black] (-0.5,0.3) circle (0.3pt);
\draw[rotate=60][fill=black] (-0.4,0.4) circle (0.3pt);
\draw(0,0.2) node[anchor=north]{$m_{\MB}$};
\draw(-2.4,0.2) node[anchor=north]{$m_{\MA}$};
\draw(-1.2,0.27) node[anchor=north]{$f_{\sigma_j}$};
\draw(-0.6,1.3) node[anchor=north]{$f_{\sigma_1}$};
\draw(-0.6,-0.76) node[anchor=north]{$f_{\sigma_n}$};
\end{tikzpicture}
\end{minipage}

\noindent respectively. Therefore, the term indexed by $i\in \llbracket1,n\rrbracket$ and $\sigma\in\mathfrak{S}_{n}$ in \eqref{eq:Jac-A-inf-k=1} with $sm_{\MA}$ and where $u=\sigma^{-1}(1)$ cancels with the term indexed by $\sigma'\in \mathfrak{S}_n$ and $j=u\in\llbracket 1,n\rrbracket$ in \eqref{eq:Jac-A-inf-k=n} with $sm_{\MA}$ where $\sigma'(l)=l$ if $l\neq j$ and $\sigma'(j)=i$ and we have that 

\begin{equation}
    \begin{split}
        &\sum\limits_{i=1}^n(-1)^{\delta+\delta'}\sum\limits_{\sigma\in\mathfrak{S}_n}(-1)^{|f_i|+\delta^{\sigma}(sf_i\upperset{G}{\circ}sm_{\MA},sf_n,\dots,\hat{f_i},\dots,sf_1)}\psi^{\sigma}_{sm_{\MB}}(sf_i\upperset{G}{\circ}sm_{\MA},sf_n,\dots,\hat{f_i},\dots,sf_1)\big)
        \\&+(-1)^{n+\sum\limits_{i=1}^n(i-1)|sf_i|+\sum\limits_{i=1}^n|sf_i|+1}\sum\limits_{\sigma\in\mathfrak{S}_n}\sum\limits_{j=1}^n(-1)^{\delta^{\sigma}(sf_n,\dots,sf_1)+\sum\limits_{i=1}^{j}|sf_{\sigma_i}|}\big(\psi_{sm_{\MB}}^{\sigma,j}(sf_n\dots,sf_1)\upperset{G}{\circ}sm_{\MA}\big)
        \\&=0.
    \end{split}
\end{equation}

Moreover, for $1<k<n$, we have 
\begin{equation}
    \begin{split}
        &(-1)^{k}\hskip-3mm\sum\limits_{(\Bar{i},\Bar{j})\in\mathcal{P}_k^n}(-1)^{\Delta}\Bar{\ell}^{n-k+1}(\Bar{\ell}^k(f_{i_k},\dots,f_{i_1}),f_{j_{n-k}},\dots,f_{j_1})
        \\&=(-1)^{k}\hskip-3mm\sum\limits_{(\Bar{i},\Bar{j})\in\mathcal{P}_k^n}\hskip-2mm(-1)^{\Delta+\epsilon}\hskip-1mm\sum\limits_{\sigma\in\mathfrak{S}_k}(-1)^{\delta^{\sigma}(sf_{i_k},\dots,sf_{i_1})}\Bar{\ell}^{n-k+1}(\psi_{sm_{\MB}}^{\sigma}(sf_{i_k},\dots,sf_{i_1}),f_{j_{n-k}},\dots,f_{j_1})
        \\&=(-1)^{k}\hskip-3mm\sum\limits_{(\Bar{i},\Bar{j})\in\mathcal{P}_k^n}\hskip-2mm(-1)^{\Delta+\epsilon+\epsilon'}\hskip-1mm\sum\limits_{\sigma\in\mathfrak{S}_k}(-1)^{\delta^{\sigma}(sf_{i_k},\dots,sf_{i_1})}\hskip-4mm\sum\limits_{\tau\in\mathfrak{S}_{n-k+1}}\hskip-3mm(-1)^{\delta^{\tau}(s\psi_{sm_{\MB}}^{\sigma}(sf_{i_k},\dots,sf_{i_1}),sf_{j_{n-k}},\dots,sf_{j_1})}
        \\&\hskip7.9cm\psi_{sm_{\MB}}^{\tau}(s\psi_{sm_{\MB}}^{\sigma}(sf_{i_k},\dots,sf_{i_1}),sf_{j_{n-k}},\dots,sf_{j_1})
    \end{split}
\end{equation}
with 
\begin{small}
    \begin{equation}
        \epsilon=\sum\limits_{r=1}^k(r-1)|sf_{i_r}|\text{, }\epsilon'=\sum\limits_{r=1}^{n-k}(r-1)|sf_{j_r}|+(n-k)(1+\sum\limits_{r=1}^k|sf_{i_r}|)
    \end{equation}
\end{small}
where $\psi_{sm_{\MB}}^{\tau}(s\psi_{sm_{\MB}}^{\sigma}(sf_{i_k},\dots,sf_{i_1}),sf_{j_{n-k}},\dots,sf_{j_1})^{\Bar{x}}=\sum\mathcal{E}(\mathcalboondox{D_3})$ where the sum is over all the filled diagrams of type $\Bar{x}$ and of the form

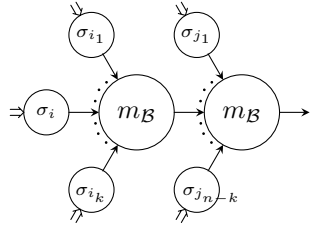
\begin{tikzpicture}[line cap=round,line join=round,x=1.0cm,y=1.0cm]
\clip(-8,-2) rectangle (4.5,2);
    \draw (0,0) circle (0.5cm);
    \draw[->,>=stealth](0.5,0)--(0.9,0);
    \draw[rotate=180][<-,>=stealth] (0.5,0)--(0.9,0);
    \draw (-1.4,0) circle (0.5cm);
    \draw[rotate around={180:(-1.4,0)}][<-,>=stealth](-0.9,0)--(-0.5,0);
    \draw (-2.6,0) circle (0.3cm);
    \draw[rotate around={120:(-1.4,0)}][<-,>=stealth](-0.9,0)--(-0.5,0);
    \draw (-2,1.03) circle (0.3cm);
    \draw[rotate around={120:(-2,1.03)},shift={(-1.7,1.03)}]\doublefleche;
    \draw[rotate around={-120:(-1.4,0)}][<-,>=stealth](-0.9,0)--(-0.5,0);
    \draw (-2,-1.03) circle (0.3cm);
     \draw[rotate around={-120:(-2,-1.03)},shift={(-1.7,-1.03)}]\doublefleche;
    \draw[rotate around={180:(-2.6,0)},shift={(-2.3,0)}]\doublefleche;
    \draw[rotate=120][<-,>=stealth] (0.5,0)--(0.9,0);
    \draw (-0.6,1.03) circle (0.3cm);
    \draw[rotate around={120:(-0.6,1.03)},shift={(-0.3,1.03)}]\doublefleche;
    \draw[rotate=-120][<-,>=stealth] (0.5,0)--(0.9,0);
    \draw (-0.6,-1.03) circle (0.3cm);
    \draw[rotate around={-120:(-0.6,-1.03)},shift={(-0.3,-1.03)}]\doublefleche;
\draw[fill=black] (-0.55,0.15) circle (0.3pt);
\draw[fill=black] (-0.5,0.3) circle (0.3pt);
\draw[fill=black] (-0.4,0.4) circle (0.3pt);
\draw[rotate=60][fill=black] (-0.55,0.15) circle (0.3pt);
\draw[rotate=60][fill=black] (-0.5,0.3) circle (0.3pt);
\draw[rotate=60][fill=black] (-0.4,0.4) circle (0.3pt);
\draw[fill=black] (-1.95,0.15) circle (0.3pt);
\draw[fill=black] (-1.9,0.3) circle (0.3pt);
\draw[fill=black] (-1.8,0.4) circle (0.3pt);
\draw[rotate around={60:(-1.4,0)}][fill=black] (-1.95,0.15) circle (0.3pt);
\draw[rotate around={60:(-1.4,0)}][fill=black] (-1.9,0.3) circle (0.3pt);
\draw[rotate around={60:(-1.4,0)}][fill=black] (-1.8,0.4) circle (0.3pt);
\draw(0,0.2) node[anchor=north]{$m_{\MB}$};
\draw(-1.4,0.2) node[anchor=north]{$m_{\MB}$};
\begin{scriptsize}
\draw(-2.6,0.16) node[anchor=north]{${\sigma_i}$};
\draw(-2,1.2) node[anchor=north]{${\sigma_{i_1}}$};
\draw(-2,-0.9) node[anchor=north]{${\sigma_{i_k}}$};
\draw(-0.6,1.2) node[anchor=north]{${\sigma_{j_1}}$};
\draw(-0.4,-0.9) node[anchor=north]{${\sigma_{j_{n-k}}}$};
\end{scriptsize}
\end{tikzpicture}

\noindent where we have filled the discs with $\sigma_j$ instead of $f_{\sigma_j}$ to simplify. Thus, since $sm_{\MB}$ satisfies the Stasheff identities, we have that 
\begin{equation}
    \begin{split}
        &\sum\limits_{k=2}^{n-1}(-1)^k\sum\limits_{(\Bar{i},\Bar{j})\in\mathcal{P}_k^n}(-1)^{\Delta}\Bar{\ell}^{n-k+1}(\Bar{\ell}^k(f_{i_k},\dots,f_{i_1}),f_{j_{n-k}},\dots,f_{j_1})
        \\&-\sum\limits_{i=1}^n(-1)^{\delta+\delta'}\sum\limits_{\sigma\in\mathfrak{S}_n}(-1)^{\delta^{\sigma}(m_{\MB}\upperset{M}{\circ} sf_i,sf_n,\dots,\hat{f_i},\dots,sf_1)}\psi^{\sigma}_{sm_{\MB}}(m_{\MB}\upperset{M}{\circ} sf_i,sf_n,\dots,\hat{f_i},\dots,sf_1)
        \\&+(-1)^{n+\sum\limits_{i=1}^n(i-1)|sf_i|}
        \sum\limits_{\sigma\in\mathfrak{S}_n}(-1)^{\delta^{\sigma}(sf_n,\dots,sf_1)}\big(sm_{\MB}\upperset{M}{\circ} \psi_{sm_{\MB}}^{\sigma}(sf_n\dots,sf_1)\big)=0.
    \end{split}
\end{equation}
Therefore, $\Bar{\ell}$ satisfies the higher Jacobi identity and $(C_{\Phi_0}(\MA,\MB),\Bar{\ell})$ is a graded $L_{\infty}$-algebra. We now show that the Maurer-Cartan elements of this $L_{\infty}$-algebra are in bijection with the $A_{\infty}$-morphisms between $(\MA,sm_{\MA})$ and $(\MB,sm_{\MB})$. 

By Lemma \ref{lemma:MC-A-inf}, one can define a Maurer-Cartan element of $(C_{\Phi_0}(\MA,\MB),\Bar{\ell})$ as a degree $1$ element $\mathbf{f}\in C_{\Phi_0}(\MA,\MB)$, \textit{i.e.} a degree $0$ element $s\mathbf{f}\in C_{\Phi_0}(\MA,\MB)[1]$, such that 
\begin{equation}
     \pi_{\bar{x}}\circ \sum\limits_{n=1}^{\infty}\frac{1}{n!}\Bar{\ell}^n(\mathbf{f},\dots,\mathbf{f})=0
\end{equation}
for every $\bar{x}\in\bar{\MO}_{\MA}$ which precisely gives that 
\[
\big((m_{\MB}\upperset{M}{\circ} s\mathbf{f})-(\mathbf{f}\upperset{G}{\circ}sm_{\MA})\big)^{\bar{x}}=0
\]
for every $\bar{x}\in\bar{\MO}_{\MA}$, \textit{i.e.} $(\Phi_0,s\mathbf{f})$ is an $A_{\infty}$-morphism. Reciprocally, given an $A_{\infty}$-morphism $(\Phi_0,s\mathbf{f}):(\MA,sm_{\MA})\rightarrow (\MB,sm_{\MB}$), $\mathbf{f}$ is clearly a Maurer-Cartan element of $(C_{\Phi_0}(\MA,\MB),\Bar{\ell})$.
\end{proof}

\begin{definition}
    Consider $A_{\infty}$-categories $(\MA,sm_{\MA})$, $(\MB,sm_{\MB})$ and a map between the sets of objects of the underlying graded quivers $\Phi_0 : \MO_{\MA}\rightarrow \MO_{\MB}$. Two $A_{\infty}$-morphisms $(\Phi_0,s\mathbf{f}),(\Phi_0,s\mathbf{g}) : (\MA,sm_{\MA}) \rightarrow (\MB,sm_{\MB})$ are \textbf{\textcolor{ultramarine}{homotopic}} if the corresponding Maurer-Cartan elements of $(C_{\Phi_0}(\MA,\MB),\Bar{\ell})$ are homotopic in the sense of Definition \ref{def:htpic-MC-elts}.
    Moreover, an $A_{\infty}$-morphism $(\Phi_0,s\mathbf{f}) : (\MA,sm_{\MA}) \rightarrow (\MB,sm_{\MB})$ is a \textbf{\textcolor{ultramarine}{homotopy equivalence}} if there exists an $A_{\infty}$-morphism $(\Psi_0,s\mathbf{g}) :(\MB,sm_{\MB}) \rightarrow (\MA,sm_{\MA})$ such that $\Psi_0$ is the inverse of $\Phi_0$, $s\mathbf{g}\circ s\mathbf{f}$ is homotopic to $\id_{\MA}$ and $s\mathbf{f}\circ s\mathbf{g}$ is homotopic to $\id_{\MB}$. 
\end{definition}

\begin{proposition}
    Consider $A_{\infty}$-categories $(\MA,sm_{\MA})$, $(\MB,sm_{\MB})$ and a map $\Phi_0 : \MO_{\MA}\rightarrow \MO_{\MB}$ between the sets of objects of the underlying graded quivers. 
    
    \noindent Then, if two $A_{\infty}$-morphisms $(\Phi_0,s\mathbf{f}),(\Phi_0,s\mathbf{g}) : (\MA,sm_{\MA}) \rightarrow (\MB,sm_{\MB})$ are homotopic, they are weakly homotopic.
\end{proposition}
\begin{proof}
    Let $(\Phi_0,s\mathbf{f}),(\Phi_0,s\mathbf{g}) : (\MA,sm_{\MA}) \rightarrow (\MB,sm_{\MB})$ be two homotopic morphisms. Following Remark \ref{remark:htpic-MC-elts-equiv}, we thus consider $a(t)+b(t)dt$ such that $sa(t)$ is an $A_{\infty}$-morphism between $(\MA,sm_{\MA})$ and $(\MB,sm_{\MB})$ for all $t\in\Bbbk$, $a(0)=\mathbf{f}$, $a(1)=\mathbf{g}$ and 
    \begin{equation}
        \frac{\partial a}{\partial t}(t)=\sum\limits_{n=1}^{\infty}\frac{1}{(n-1)!}\Bar{\ell}^n\big(a(t),\dots,a(t),b(t)\big)
    \end{equation}
    which gives by definition of $\Bar{\ell}$ that $\frac{\partial a}{\partial t}(t)^{\Bar{x}}=\sum\mathcal{E}(\mathcalboondox{D})-\sum\mathcal{E}(\mathcalboondox{D'})$ where the sums are over all the filled diagrams $\mathcalboondox{D}$ and $\mathcalboondox{D'}$ of type $\Bar{x}$ and of the form 

    \begin{minipage}{21cm}
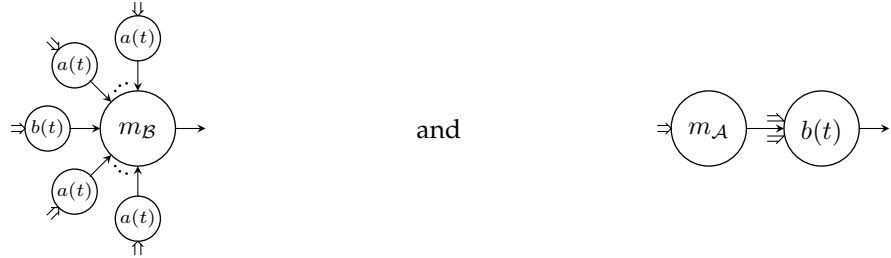
        \begin{tikzpicture}[line cap=round,line join=round,x=1.0cm,y=1.0cm]
\clip(-3,-1.8) rectangle (4.5,1.8);
     \draw [line width=0.5pt] (0.,0.) circle (0.5cm);
    \draw [rotate=0] [->, >= stealth, >= stealth] (0.5,0) -- (0.9,0);
    \draw [rotate=-90] [<-, >= stealth, >= stealth](0.5,0) -- (0.9,0);
    \draw [rotate=90] [<-, >= stealth, >= stealth](0.5,0) -- (0.9,0);
    \draw [rotate=-135] [<-, >= stealth, >= stealth](0.5,0) -- (0.9,0);
    \draw [rotate=180] [<-, >= stealth, >= stealth](0.5,0) -- (0.9,0);
    \draw [rotate=135] [<-, >= stealth, >= stealth](0.5,0) -- (0.9,0);
    \draw [line width=0.5pt] (0.,-1.2) circle (0.3cm);
    \draw[rotate around={-90:(0,-1.2)},shift={(0.3,-1.2)}]\doublefleche;
     \draw [line width=0.5pt] (0.,1.2) circle (0.3cm);
    \draw[rotate around={90:(0,1.2)},shift={(0.3,1.2)}]\doublefleche;
    \draw [line width=0.5pt] (-1.2,0) circle (0.3cm);
     \draw[rotate around={180:(-1.2,0)},shift={(-0.9,0)}]\doublefleche;
    \draw [line width=0.5pt] (-0.84,0.84) circle (0.3cm);
    \draw[rotate around={135:(-0.84,0.84)},shift={(-0.54,0.84)}]\doublefleche;
    \draw [line width=0.5pt] (-0.84,-0.84) circle (0.3cm);
    \draw[rotate around={-135:(-0.84,-0.84)},shift={(-0.54,-0.84)}]\doublefleche;
\begin{scriptsize}
\draw [fill=black] (-0.24,-0.57) circle (0.3pt);
\draw [fill=black] (-0.3,-0.5) circle (0.3pt);
\draw [fill=black] (-0.15,-0.6) circle (0.3pt);
\draw [fill=black] (-0.24,0.57) circle (0.3pt);
\draw [fill=black] (-0.3,0.5) circle (0.3pt);
\draw [fill=black] (-0.15,0.6) circle (0.3pt);
\end{scriptsize}
\draw(0,0.2)node[anchor=north]{$m_{\MB}$};
\draw(-0.84,1.09)node[anchor=north]{$\scriptstyle{a(t)}$};
\draw(-0.84,-0.59)node[anchor=north]{$\scriptstyle{a(t)}$};
\draw(0,-0.95)node[anchor=north]{$\scriptstyle{a(t)}$};
\draw(0,1.45)node[anchor=north]{$\scriptstyle{a(t)}$};
\draw(-1.2,0.25)node[anchor=north]{$\scriptstyle{b(t)}$};
\draw(4,0.25)node[anchor=north]{and};
\end{tikzpicture}
\begin{tikzpicture}[line cap=round,line join=round,x=1.0cm,y=1.0cm]
\clip(-3,-1.8) rectangle (4,1.8);
    \draw [line width=0.5pt] (0.,0.) circle (0.5cm);
    \shadedraw [rotate=180, shift={(0.5cm,0cm)}] \doublefleche;
    \draw [line width=0.5pt] (1.5,0.) circle (0.5cm);
    \shadedraw[shift={(1cm,0cm)},rotate=180] \doubleflechescindeeleft;
    \shadedraw[shift={(1cm,0cm)},rotate=180] \doubleflechescindeeright;
    \shadedraw[shift={(1cm,0cm)},rotate=180] \fleche;
    \draw [rotate around={0:(1.5,0)}] [->,, >= stealth, >= stealth](2,0) -- (2.4,0);
    \draw(1.5,0.25)node[anchor=north]{$b(t)$};
    \draw(0,0.2)node[anchor=north]{$m_{\MA}$};
\end{tikzpicture}
    \end{minipage}

    \noindent respectively.
    We now consider $A(t)=(sm_{\MB},a(t),sm_{\MA})$ and $B(t)=(0,b(t),0)$. Then, we have that $A(t)+B(t)dt$ is a weak homotopy between the Maurer-Cartan elements $(sm_{\MB},\mathbf{f},sm_{\MA})$ and $(sm_{\MB},\mathbf{g},sm_{\MA})$. Indeed, it is clear that $A(t)$ is a Maurer-Cartan element of the graded Lie algebra $C(\MB)[1]\oplus C_{\Phi_0}(\MA,\MB)\oplus C(\MA)[1]$ and that we have $A(0)=(sm_{\MB},\mathbf{f},sm_{\MA})$ and $A(1)=(sm_{\MB},\mathbf{g},sm_{\MA})$. Moreover, we have that 
    \begin{equation}
        \begin{split}
           &\sum\limits_{n=2}^{\infty}\frac{1}{(n-1)!}\ell^n\big(A(t),\dots,A(t),B(t)\big)
           \\&=\bigg(0,\sum\limits_{n=2}^{\infty}\frac{1}{(n-1)!}\sum\limits_{i=1}^{n-1}\ell^n\big(\underbrace{a(t),\dots,a(t)}_{i\text{ times }},sm_{\MB},a(t),\dots,a(t),b(t)\big)
           \\&\phantom{=}\hskip2cm+\sum\limits_{n=2}^{\infty}\frac{1}{(n-1)!}\sum\limits_{i=1}^{n-1}\ell^n\big(\underbrace{a(t),\dots,a(t)}_{i\text{ times }},sm_{\MA},a(t),\dots,a(t),b(t)\big),0\bigg)
        \end{split}
    \end{equation}
    which is by definition of $\ell$ equal to $\sum\mathcal{E}(\mathcalboondox{D})-\sum\mathcal{E}(\mathcalboondox{D'})$.
    On the other hand, $\frac{\partial A}{\partial t}(t)=(0,\frac{\partial a}{\partial t}(t),0)$ which concludes the proof. 
\end{proof}

\begin{proposition}
\label{prop:htpic-A-inf-1}
     Consider $A_{\infty}$-categories $(\MA,sm_{\MA})$, $(\MB,sm_{\MB})$ and $(\mathcal{C},sm_{\mathcal{C}})$ and maps $\Phi_0 : \MO_{\MA}\rightarrow \MO_{\MB}$ and $\Psi_0 : \MO_{\MB}\rightarrow \MO_{\mathcal{C}}$. Then, if $(\Phi_0,s\mathbf{f}),(\Phi_0,s\mathbf{g}) : (\MA,sm_{\MA}) \rightarrow (\MB,sm_{\MB})$ are homotopic $A_{\infty}$-morphisms and $(\Psi_0,s\mathbf{h}) : (\MB,sm_{\MB})\rightarrow (\mathcal{C},sm_{\mathcal{C}})$ is an $A_{\infty}$-morphism, $(\Psi_0\circ \Phi_0,s\mathbf{h}\circ s\mathbf{f})$ and $(\Psi_0\circ \Phi_0,s\mathbf{h}\circ s\mathbf{g})$ are homotopic.
\end{proposition}
\begin{proof}
    Consider a homotopy $a(t)+b(t)dt$ between the Maurer-Cartan elements associated to the $A_{\infty}$-morphisms $(\Phi_0,s\mathbf{f})$ and $(\Phi_0,s\mathbf{g})$. For any $t\in\Bbbk$ we define $sx(t)$ as the composition $s\mathbf{h} \circ sa(t)$. It is clear that this is an $A_{\infty}$-morphism and that $sx(0)=s\mathbf{h}\circ s\mathbf{f}$ and $sx(1)=s\mathbf{h}\circ s\mathbf{g}$. Moreover, we define $sy(t)^{\Bar{x}}=\sum\mathcal{E}(\mathcalboondox{D})$ where the sum is over all the filled diagrams of type $\Bar{x}$ and of the form


    
\noindent for every $\Bar{x}\in\Bar{\MO}_{\MA}$. 
    Thus, we have that 
    \begin{equation}
        \sum\limits_{n=1}^{\infty}\frac{1}{(n-1)!}\Bar{\ell}^n\big(x(t),\dots,x(t),y(t)\big)^{\Bar{x}}=\sum\mathcal{E}(\mathcalboondox{D_1})-\sum\mathcal{E}(\mathcalboondox{D_2})-\sum\mathcal{E}(\mathcalboondox{D_3})-+\sum\mathcal{E}(\mathcalboondox{D_4})
    \end{equation}
    where the sums are over all the filled diagrams $\mathcalboondox{D_i}$ for $i\in\llbracket 1,4\rrbracket$ of type $\Bar{x}$ and of the form
    \begin{minipage}{21cm}
        %
\end{minipage}
    \noindent respectively.
    
    On the other hand, $\frac{\partial x}{\partial t}(t)=\sum\mathcal{E}(\mathcalboondox{D'_1})-\sum\mathcal{E}(\mathcalboondox{D'_2})$
    where the sums are over all the filled diagrams $\mathcalboondox{D'_1}$ and $\mathcalboondox{D'_2}$ of type $\Bar{x}$ and of the form
    
    \begin{minipage}{21cm}
        %
    \end{minipage}

    \noindent respectively. We have that $\sum\mathcal{E}(\mathcalboondox{D_2})=\sum\mathcal{E}(\mathcalboondox{D'_2})$. Moreover, since $a(t)$ is an $A_{\infty}$-morphism for any $t\in\Bbbk$, we have that $-\sum\mathcal{E}(\mathcalboondox{D_3})+\sum\mathcal{E}(\mathcalboondox{D_4})=-\sum\mathcal{E}(\mathcalboondox{D'_3})+\sum\mathcal{E}(\mathcalboondox{D'_4})$
    where the last two sums are over all the filled diagrams $\mathcalboondox{D'_3}$ and $\mathcalboondox{D'_4}$ of type $\Bar{x}$ and of the form

    \begin{minipage}{21cm}
        %
    \end{minipage}
    
\noindent respectively. 

Since $(\Psi_0,s\mathbf{h})$ is an $A_{\infty}$-morphism we have $\sum\mathcal{E}(\mathcalboondox{D'_1})+\sum\mathcal{E}(\mathcalboondox{D'_3})-\sum\mathcal{E}(\mathcalboondox{D'_4})=\sum\mathcal{E}(\mathcalboondox{D_1})$ which concludes the proof.
\end{proof}

\begin{proposition}
\label{prop:htpic-A-inf-2}
    Consider $A_{\infty}$-categories $(\MA,sm_{\MA})$, $(\MB,sm_{\MB})$ and $(\mathcal{C},sm_{\mathcal{C}})$ and maps $\Phi_0 : \MO_{\MA}\rightarrow \MO_{\MB}$ and $\Psi_0 : \MO_{\MB}\rightarrow \MO_{\mathcal{C}}$. Then, if $(\Phi_0,s\mathbf{f}),(\Phi_0,s\mathbf{g}) : (\MA,sm_{\MA}) \rightarrow (\MB,sm_{\MB})$ are homotopic $A_{\infty}$-morphisms and $(\Psi_0,s\mathbf{h}) : (\mathcal{C},sm_{\mathcal{C}})\rightarrow (\MA,sm_{\MA})$ is an $A_{\infty}$-morphism, $(\Phi_0\circ \Psi_0,s\mathbf{f}\circ s\mathbf{h})$ and $(\Phi_0\circ \Psi_0,s\mathbf{g}\circ s\mathbf{h})$ are homotopic.
\end{proposition}
\begin{proof}
    Consider a homotopy $a(t)+b(t)dt$ between the Maurer-Cartan elements associated to the $A_{\infty}$-morphisms $(\Phi_0,s\mathbf{f})$ and $(\Phi_0,s\mathbf{g})$. For any $t\in\Bbbk$ we define $sx(t)=sa(t)\circ s\mathbf{h}$. It is clear that this is an $A_{\infty}$-morphism and that $sx(0)=s\mathbf{f}\circ s\mathbf{h}$ and $sx(1)=s\mathbf{g}\circ s\mathbf{h}$.
    Moreover, we define $sy(t)^{\Bar{x}}=\sum\mathcal{E}(\mathcalboondox{D})$ where the sum is over all the filled diagrams of type $\Bar{x}$ and of the form

        
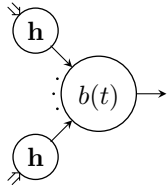
\begin{tikzpicture}[line cap=round,line join=round,x=1.0cm,y=1.0cm]
\clip(-7,-1.5) rectangle (4,1.5);
     \draw [line width=0.5pt] (0.,0.) circle (0.5cm);
    \draw [rotate=0] [->, >= stealth, >= stealth] (0.5,0) -- (0.9,0);
    \draw [rotate=-135] [<-, >= stealth, >= stealth](0.5,0) -- (0.9,0);
    \draw [rotate=135] [<-, >= stealth, >= stealth](0.5,0) -- (0.9,0);
    \draw [line width=0.5pt] (-0.84,0.84) circle (0.3cm);
    \draw[rotate around={135:(-0.84,0.84)},shift={(-0.54,0.84)}]\doublefleche;
    \draw [line width=0.5pt] (-0.84,-0.84) circle (0.3cm);
    \draw[rotate around={-135:(-0.84,-0.84)},shift={(-0.54,-0.84)}]\doublefleche;
\begin{scriptsize}
\draw [fill=black] (-0.55,-0.2) circle (0.3pt);
\draw [fill=black] (-0.55,0.2) circle (0.3pt);
\draw [fill=black] (-0.6,0) circle (0.3pt);
\end{scriptsize}
\draw(0,0.25)node[anchor=north]{$b(t)$};
\draw(-0.84,1.09)node[anchor=north]{$\mathbf{h}$};
\draw(-0.84,-0.59)node[anchor=north]{$\mathbf{h}$};
\end{tikzpicture}
    
\noindent for every $\Bar{x}\in\Bar{\MO}_{\mathcal{C}}$. 
    Thus, we have that 
    \begin{equation}
        \sum\limits_{n=1}^{\infty}\frac{1}{(n-1)!}\Bar{\ell}^n\big(x(t),\dots,x(t),y(t)\big)^{\Bar{x}}=\sum\mathcal{E}(\mathcalboondox{D_1})-\sum\mathcal{E}(\mathcalboondox{D_2})
    \end{equation}
    where the sums are over all the filled diagrams $\mathcalboondox{D_1}$ and $\mathcalboondox{D_2}$ of type $\Bar{x}$ and of the form
    
   \begin{minipage}{21cm}
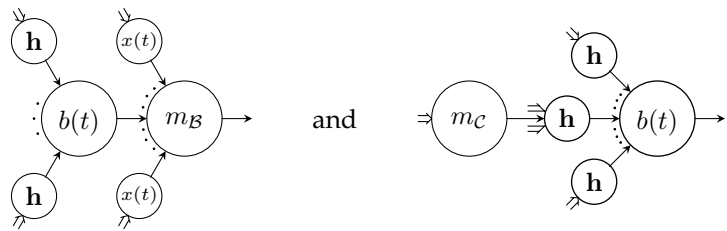
       \begin{tikzpicture}[line cap=round,line join=round,x=1.0cm,y=1.0cm]
\clip(-5,-1.8) rectangle (3,1.8);
           \draw (0,0) circle (0.5cm);
           \draw[rotate=0][->,>=stealth](0.5,0)--(0.9,0);
            \draw[rotate=180][<-,>=stealth](0.5,0)--(0.9,0);
            \draw[rotate=120][<-,>=stealth](0.5,0)--(0.9,0);
            \draw[rotate=-120][<-,>=stealth](0.5,0)--(0.9,0);
            \draw (-1.4,0) circle (0.5cm);
            \draw[rotate around={120:(-1.4,0)}][<-,>=stealth](-0.9,0)--(-0.5,0);
            \draw[rotate around={-120:(-1.4,0)}][<-,>=stealth](-0.9,0)--(-0.5,0);
        \draw (-2,1.03) circle (0.3cm);
        \draw (-2,-1.03) circle (0.3cm);
        \draw[rotate around={120:(-2,1.03)},shift={(-1.7,1.03)}]\doublefleche;
        \draw[rotate around={-120:(-2,-1.03)},shift={(-1.7,-1.03)}]\doublefleche;
        \draw (-0.6,1.03) circle (0.3cm);
        \draw[rotate around={120:(-0.6,1.03)},shift={(-0.3,1.03)}]\doublefleche;
        \draw (-0.6,-1.03) circle (0.3cm);
        \draw[rotate around={-120:(-0.6,-1.03)},shift={(-0.3,-1.03)}]\doublefleche;
        \draw (2,0.25)node[anchor=north]{and};
        \draw (0,0.2)node[anchor=north]{$m_{\mathcal{B}}$};
        \draw (-1.4,0.25)node[anchor=north]{$b(t)$};
        \draw (-2,1.28)node[anchor=north]{$\mathbf{h}$};
        \draw (-2,-0.78)node[anchor=north]{$\mathbf{h}$};
        \draw (-0.6,1.28)node[anchor=north]{$\scriptstyle{x(t)}$};
        \draw (-0.6,-0.78)node[anchor=north]{$\scriptstyle{x(t)}$};
\draw[fill=black] (-0.55,0.15) circle (0.3pt);
\draw[fill=black] (-0.5,0.3) circle (0.3pt);
\draw[fill=black] (-0.4,0.4) circle (0.3pt);
\draw[rotate=60][fill=black] (-0.55,0.15) circle (0.3pt);
\draw[rotate=60][fill=black] (-0.5,0.3) circle (0.3pt);
\draw[rotate=60][fill=black] (-0.4,0.4) circle (0.3pt);
\draw [fill=black] (-1.95,-0.2) circle (0.3pt);
\draw [fill=black] (-1.95,0.2) circle (0.3pt);
\draw [fill=black] (-2,0) circle (0.3pt);
       \end{tikzpicture}
              \begin{tikzpicture}[line cap=round,line join=round,x=1.0cm,y=1.0cm]
\clip(-3.2,-1.8) rectangle (1.3,1.8);
     \draw [line width=0.5pt] (0.,0.) circle (0.5cm);
    \draw [rotate=0] [->, >= stealth, >= stealth] (0.5,0) -- (0.9,0);
    \draw [rotate=-135] [<-, >= stealth, >= stealth](0.5,0) -- (0.9,0);
    \draw [rotate=180] [<-, >= stealth, >= stealth](0.5,0) -- (0.9,0);
    \draw [rotate=135] [<-, >= stealth, >= stealth](0.5,0) -- (0.9,0);
    \draw [line width=0.5pt] (-1.2,0) circle (0.3cm);
     \draw[rotate around={180:(-1.2,0)},shift={(-0.9,0)}]\doubleflechescindeeleft;
     \draw[rotate around={180:(-1.2,0)},shift={(-0.9,0)}]\doubleflechescindeeright;
     \draw[rotate around={180:(-1.2,0)},shift={(-0.9,0)}]\fleche;
     \draw (-2.5,0) circle (0.5cm);
     \draw[rotate around={180:(-2.5,0)},shift={(-2,0)}]\doublefleche;
    \draw [line width=0.5pt] (-0.84,0.84) circle (0.3cm);
    \draw[rotate around={135:(-0.84,0.84)},shift={(-0.54,0.84)}]\doublefleche;
    \draw [line width=0.5pt] (-0.84,-0.84) circle (0.3cm);
    \draw[rotate around={-135:(-0.84,-0.84)},shift={(-0.54,-0.84)}]\doublefleche;
\begin{scriptsize}
\draw [rotate=-45][fill=black] (-0.21,-0.55) circle (0.3pt);
\draw [rotate=-45][fill=black] (-0.3,-0.5) circle (0.3pt);
\draw [rotate=-45][fill=black] (-0.12,-0.57) circle (0.3pt);
\draw [rotate=45][fill=black] (-0.21,0.55) circle (0.3pt);
\draw [rotate=45][fill=black] (-0.3,0.5) circle (0.3pt);
\draw [rotate=45][fill=black] (-0.12,0.57) circle (0.3pt);
\end{scriptsize}
\draw(0,0.25)node[anchor=north]{$b(t)$};
\draw(-2.5,0.2)node[anchor=north]{$m_{\mathcal{C}}$};
\draw(-0.84,1.09)node[anchor=north]{$\mathbf{h}$};
\draw(-0.84,-0.59)node[anchor=north]{$\mathbf{h}$};
\draw(-1.2,0.25)node[anchor=north]{$\mathbf{h}$};
\end{tikzpicture}
   \end{minipage}

    \noindent respectively.
    
    On the other hand, $\frac{\partial x}{\partial t}(t)=\sum\mathcal{E}(\mathcalboondox{D'_1})-\sum\mathcal{E}(\mathcalboondox{D'_2})$
    where the sums are over all the filled diagrams $\mathcalboondox{D'_1}$ and $\mathcalboondox{D'_2}$ of type $\Bar{x}$ and of the form

    \begin{minipage}{21cm}
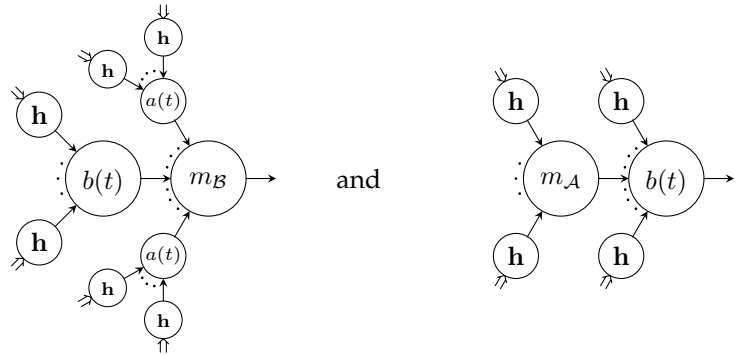
    \begin{tikzpicture}[line cap=round,line join=round,x=1.0cm,y=1.0cm]
\clip(-5,-2.8) rectangle (3,2.8);
           \draw (0,0) circle (0.5cm);
           \draw[rotate=0][->,>=stealth](0.5,0)--(0.9,0);
            \draw[rotate=180][<-,>=stealth](0.5,0)--(0.9,0);
            \draw[rotate=120][<-,>=stealth](0.5,0)--(0.9,0);
            \draw[rotate=-120][<-,>=stealth](0.5,0)--(0.9,0);
            \draw (-1.4,0) circle (0.5cm);
            \draw[rotate around={135:(-1.4,0)}][<-,>=stealth](-0.9,0)--(-0.5,0);
            \draw[rotate around={-135:(-1.4,0)}][<-,>=stealth](-0.9,0)--(-0.5,0);
        \draw (-2.25,0.85) circle (0.3cm);
        \draw (-2.25,-0.85) circle (0.3cm);
        \draw[rotate around={135:(-2.25,0.85)},shift={(-1.95,0.85)}]\doublefleche;
        \draw[rotate around={-135:(-2.25,-0.85)},shift={(-1.95,-0.85)}]\doublefleche;
        \draw (-0.6,1.03) circle (0.3cm);
        \draw[rotate around={90:(-0.6,1.03)}][<-,>=stealth](-0.3,1.03)--(0,1.03);
        \draw (-0.6,1.88) circle (0.25cm);
        \draw[rotate around={90:(-0.6,1.88)},shift={(-0.35,1.88)}] \doublefleche;
        \draw[rotate around={150:(-0.6,1.03)}][<-,>=stealth](-0.3,1.03)--(0,1.03);
        \draw (-1.34,1.455) circle (0.25cm);
        \draw[rotate around={150:(-1.34,1.455)},shift={(-1.09,1.455)}] \doublefleche;
        \draw (-0.6,-1.03) circle (0.3cm);
       \draw[rotate around={-90:(-0.6,-1.03)}][<-,>=stealth](-0.3,-1.03)--(0,-1.03);
        \draw (-0.6,-1.88) circle (0.25cm);
        \draw[rotate around={-90:(-0.6,-1.88)},shift={(-0.35,-1.88)}] \doublefleche;
        \draw[rotate around={-150:(-0.6,-1.03)}][<-,>=stealth](-0.3,-1.03)--(0,-1.03);
        \draw (-1.34,-1.455) circle (0.25cm);
        \draw[rotate around={-150:(-1.34,-1.455)},shift={(-1.09,-1.455)}] \doublefleche;
        \draw (2,0.25)node[anchor=north]{and};
        \draw (0,0.2)node[anchor=north]{$m_{\mathcal{B}}$};
        \draw (-1.4,0.25)node[anchor=north]{$b(t)$};
        \draw (-2.25,1.1)node[anchor=north]{$\mathbf{h}$};
        \draw (-2.25,-0.6)node[anchor=north]{$\mathbf{h}$};
        \draw (-0.6,1.28)node[anchor=north]{$\scriptstyle{a(t)}$};
        \draw (-0.6,-0.78)node[anchor=north]{$\scriptstyle{a(t)}$};
        \draw (-1.34,1.65)node[anchor=north]{$\scriptstyle{\mathbf{h}}$};
        \draw (-0.6,2.08)node[anchor=north]{$\scriptstyle{\mathbf{h}}$};
        \draw (-1.34,-1.25)node[anchor=north]{$\scriptstyle{\mathbf{h}}$};
        \draw (-0.6,-1.68)node[anchor=north]{$\scriptstyle{\mathbf{h}}$};
\draw[fill=black] (-0.55,0.15) circle (0.3pt);
\draw[fill=black] (-0.5,0.3) circle (0.3pt);
\draw[fill=black] (-0.4,0.4) circle (0.3pt);
\draw[rotate=60][fill=black] (-0.55,0.15) circle (0.3pt);
\draw[rotate=60][fill=black] (-0.5,0.3) circle (0.3pt);
\draw[rotate=60][fill=black] (-0.4,0.4) circle (0.3pt);
\draw [fill=black] (-1.95,-0.2) circle (0.3pt);
\draw [fill=black] (-1.95,0.2) circle (0.3pt);
\draw [fill=black] (-2,0) circle (0.3pt);
\draw [fill=black] (-0.72,1.43) circle (0.3pt);
\draw [fill=black] (-0.82,1.4) circle (0.3pt);
\draw [fill=black] (-0.88,1.32) circle (0.3pt);
\draw [fill=black] (-0.72,-1.43) circle (0.3pt);
\draw [fill=black] (-0.82,-1.4) circle (0.3pt);
\draw [fill=black] (-0.88,-1.32) circle (0.3pt);
        \end{tikzpicture}
\begin{tikzpicture}[line cap=round,line join=round,x=1.0cm,y=1.0cm]
\clip(-3,-2.8) rectangle (3,2.8);
           \draw (0,0) circle (0.5cm);
           \draw[rotate=0][->,>=stealth](0.5,0)--(0.9,0);
            \draw[rotate=180][<-,>=stealth](0.5,0)--(0.9,0);
            \draw[rotate=120][<-,>=stealth](0.5,0)--(0.9,0);
            \draw[rotate=-120][<-,>=stealth](0.5,0)--(0.9,0);
            \draw (-1.4,0) circle (0.5cm);
            \draw[rotate around={120:(-1.4,0)}][<-,>=stealth](-0.9,0)--(-0.5,0);
            \draw[rotate around={-120:(-1.4,0)}][<-,>=stealth](-0.9,0)--(-0.5,0);
        \draw (-2,1.03) circle (0.3cm);
        \draw (-2,-1.03) circle (0.3cm);
        \draw[rotate around={120:(-2,1.03)},shift={(-1.7,1.03)}]\doublefleche;
        \draw[rotate around={-120:(-2,-1.03)},shift={(-1.7,-1.03)}]\doublefleche;
        \draw (-0.6,1.03) circle (0.3cm);
        \draw[rotate around={120:(-0.6,1.03)},shift={(-0.3,1.03)}]\doublefleche;
        \draw (-0.6,-1.03) circle (0.3cm);
        \draw[rotate around={-120:(-0.6,-1.03)},shift={(-0.3,-1.03)}]\doublefleche;
        \draw (0,0.25)node[anchor=north]{$b(t)$};
        \draw (-1.4,0.2)node[anchor=north]{$m_{\mathcal{A}}$};
        \draw (-2,1.28)node[anchor=north]{$\mathbf{h}$};
        \draw (-2,-0.78)node[anchor=north]{$\mathbf{h}$};
        \draw (-0.6,1.28)node[anchor=north]{$\mathbf{h}$};
        \draw (-0.6,-0.78)node[anchor=north]{$\mathbf{h}$};
\draw[fill=black] (-0.55,0.15) circle (0.3pt);
\draw[fill=black] (-0.5,0.3) circle (0.3pt);
\draw[fill=black] (-0.4,0.4) circle (0.3pt);
\draw[rotate=60][fill=black] (-0.55,0.15) circle (0.3pt);
\draw[rotate=60][fill=black] (-0.5,0.3) circle (0.3pt);
\draw[rotate=60][fill=black] (-0.4,0.4) circle (0.3pt);
\draw [fill=black] (-1.95,-0.2) circle (0.3pt);
\draw [fill=black] (-1.95,0.2) circle (0.3pt);
\draw [fill=black] (-2,0) circle (0.3pt);
       \end{tikzpicture}
    \end{minipage}

    \noindent respectively. We have that $\sum\mathcal{E}(\mathcalboondox{D'_1})=\sum\mathcal{E}(\mathcalboondox{D_1})$.
    Moreover, since $(\Psi_0,s\mathbf{h})$ is an $A_{\infty}$-morphism we have that $\sum\mathcal{E}(\mathcalboondox{D'_2})=\sum\mathcal{E}(\mathcalboondox{D_2})$ which concludes the proof.
\end{proof}

We now recall that homotopy equivalences of $A_{\infty}$-morphisms coincide quasi-isomorphisms (see \cite{lefevre}), giving a proof similar as the one in \cite{kraft-schnitzer} for $L_{\infty}$-morphisms. 

\begin{theorem}\label{thm:equiv-htpy-qiso-A-inf}
     Consider $A_{\infty}$-categories $(\MA,sm_{\MA})$, $(\MB,sm_{\MB})$ and a map $\Phi_0 : \MO_{\MA}\rightarrow \MO_{\MB}$ between the sets of objects of the underlying graded quivers. 
     
     \noindent Then, an $A_{\infty}$-morphism $(\Phi_0,s\mathbf{f}) : (\MA,sm_{\MA}) \rightarrow (\MB,sm_{\MB})$ is a homotopy equivalence if and only if it is a quasi-isomorphism.
\end{theorem}
\begin{proof}
    Consider a homotopy equivalence $(\Phi_0,s\mathbf{f}) : (\MA,sm_{\MA}) \rightarrow (\MB,sm_{\MB})$. Then, there exists an $A_{\infty}$-morphism $(\Psi_0,s\mathbf{g}) :(\MB,sm_{\MB}) \rightarrow (\MA,sm_{\MA})$ such that $\Psi_0$ is the inverse of $\Phi_0$, $s\mathbf{g}\circ s\mathbf{f}$ is homotopic to $\id_{\MA}$ and $s\mathbf{f}\circ s\mathbf{g}$ is homotopic to $\id_{\MB}$. In particular, $sg^{x,y}\circ sf^{x,y}$ (resp. $sf^{x,y}\circ sg^{x,y}$) is chain homotopic to $\id_{{}_x\MA_y[1]}$ (resp. $\id_{{}_x\MB_y[1]}$) for any $x,y\in\MO_{\MA}$ (resp. $x,y\in\MO_{\MB}$) which means that both $(sf^{x,y})_{x,y\in\MO_{\MA}}$ and $(s
    g^{x,y})_{x,y\in\MO_{\MB}}$ are quasi-isomorphisms of graded quivers. By definition, $(\Phi_0,s\mathbf{f})$ is thus a quasi-isomorphism of $A_{\infty}$-categories.

    On the other hand, consider a quasi-isomorphism of $A_{\infty}$-categories $(\Phi_0,s\mathbf{f})$. There is an isomorphism $\MA\simeq H(\MA)\oplus L(\MA)$ (resp. $\MB\simeq H(\MB)\oplus L(\MB)$) where $L(\MA)$ (resp. $L(\MB)$) is a contractible $A_{\infty}$-algebra with differential denoted by $d_{\MA}$ (resp. $d_{\MB}$). Consider a contracting homotopy $h : L(\MA)\rightarrow L(\MA)$, \textit{i.e.} $d_{\MA}h+hd_{\MA}=\id_{L(\MA)}$. We denote by $i_{\MA}$ (resp. $p_{\MA}$) the canonical injection (resp. projection) $H(\MA)\rightarrow H(\MA)\oplus L(\MA)$ (resp. $H(\MA) \oplus L(\MA)\rightarrow H(\MA)$). We define $i_{\MB}$ and $p_{\MB}$ in a similar manner. Those are $A_{\infty}$-morphisms, due to the form of structure of the product of $A_{\infty}$-categories.
    
    Moreover, we have that $i_{\MA}\circ p_{\MA}$ is homotopic to $\id_{\MA}$. Indeed, we show that $a(t)+b(t)dt$ with
    \begin{equation}
        \begin{split}
            &a(t) : (x,y)\in H(\MA)\oplus L(\MA)\mapsto (x,ty)\in H(\MA)\oplus L(\MA),
            \\& b(t) : (x,y)\in H(\MA)\oplus L(\MA)\mapsto (0,h(y))\in H(\MA)\oplus L(\MA)
        \end{split}
    \end{equation}
    is a homotopy between $i_{\MA}\circ p_{\MA}$ and $\id_{\MA}$.
    For any $t\in\Bbbk$, $a(t)$ is an $A_{\infty}$-morphism. Moreover, $\frac{\partial a}{\partial t}(t)(x,y)=(0,y)$
    and 
    \begin{equation}
        \begin{split}
            \sum\limits_{n=2}^{\infty}\frac{1}{(n-1)!}\Bar{\ell}^n\big(a(t),\dots,a(t),b(t)\big)(x,y)
            =(0,d_{\MA}h(y))+(0,hd_{\MA}(y))
            =\frac{\partial a}{\partial t}(t)(x,y)
        \end{split}
    \end{equation}
    so $i_{\MA}\circ p_{\MA}$ is homotopic to $\id_{\MA}$. One can show in a similar manner that $i_{\MB}\circ p_{\MB}$ is homotopic to $\id_{\MB}$.

    Since $i_{\MA}$ and $p_{\MB}$ are quasi-isomorphisms, $(\Phi_0,sp_{\MB}\circ s\mathbf{f}\circ si_{\MA})$ is a quasi-isomorphism. It is even an isomorphism since $H(\MA)$ is minimal. Therefore, $(\Psi_0,s\mathbf{g})$ is a quasi-inverse of $(\Phi_0,s\mathbf{f})$ with $s\mathbf{g}=i_{\MA[1]}\circ (p_{\MB[1]}\circ s\mathbf{f}\circ i_{\MA[1]})^{-1}\circ p_{\MB[1]}$.

    Then, $s\mathbf{g}\circ s\mathbf{f}$ is homotopic to $i_{\MA[1]}\circ (p_{\MB[1]}\circ s\mathbf{f}\circ i_{\MA[1]})^{-1}\circ p_{\MB[1]}\circ s\mathbf{f}\circ i_{\MA[1]}\circ p_{\MA[1]}=i_{\MA[1]}\circ p_{\MA[1]}$ by Proposition \ref{prop:htpic-A-inf-1} which is homotopic to $\id_{\MA}$.
    Similarly, we have that $s\mathbf{f}\circ s\mathbf{g}$ is homotopic to $i_{\MB[1]}\circ p_{\MB[1]}\circ s\mathbf{f}\circ i_{\MA[1]}\circ (p_{\MB[1]}\circ s\mathbf{f}\circ i_{\MA[1]})^{-1}\circ p_{\MB[1]}=i_{\MB[1]}\circ p_{\MB[1]}$ by Proposition \ref{prop:htpic-A-inf-2} which is homotopic to $\id_{\MB}$. This shows that $(\Phi_0,s\mathbf{F})$ is a homotopy equivalence.
\end{proof}
\section{Homotopy theory of pre-Calabi-Yau morphisms}
\label{section:pCY}

\subsection{The category of pre-Calabi-Yau categories}

In this section, we recall the definition of the category of pre-Calabi-Yau categories. 
We first recall the definition of pre-Calabi-Yau categories given in \cite{ktv}.

\begin{definition}
Given graded quivers $\MA$ and $\MB$ with respective sets of objects $\MO_{\MA}$ and $\MO_{\MB}$ and a map $F_0 : \MO_{\MA}\rightarrow \MO_{\MB}$, we define the graded vector space 
\begin{small}
    \begin{equation}
    \begin{split}
&\Multi^{\bullet}_{d,F_0}(\MA,\MB)\\&\hspace{1cm}=\prod_{n\in\NN^*}\prod_{\doubar{x}\in\bar{\MO}_{\MA}^n}\Homgr_{\kk}\bigg(\bigotimes\limits_{i=1}^n\MA[1]^{\otimes \bar{x}^i},\bigotimes\limits_{i=1}^{n-1}{}_{F_0(\llt(\bar{x}^i))}\MB_{F_0(\rrt(\bar{x}^{i+1}))}[-d]\otimes{}_{F_0(\llt(\bar{x}^n))}\MB_{F_0(\rrt(\bar{x}^1))}[-d])\bigg).
    \end{split}
    \end{equation}
\end{small}
The action of $\sigma=(\sigma_n)_{n\in\NN^*}\in\prod_{n\in\NN^*} C_n$ on an element $\mathbf{F} = (F^{\doubar{x}})_{\doubar{x} \in \doubar{\MO}_{\MA}}\in\Multi^{\bullet}_{d,F_0}(\MA,\MB)$ is the element $\sigma\cdot\mathbf{F}\in\Multi^{\bullet}_{d,F_0}(\MA,\MB)$ given by 
\begin{equation}
    \begin{split}
(\sigma\cdot\mathbf{F})^{\doubar{x}}=\tau^{\sigma_n}_{{}_{F_0(\llt(\bar{x}^{1}))}\MB_{F_0(\rrt(\bar{x}^2))}[-d],\dots, {}_{F_0(\llt(\bar{x}^n))}\MB_{F_0(\rrt(\bar{x}^1))}[-d]}\circ F^{\doubar{x}\cdot\sigma_n} \circ \tau^{\sigma_n}_{\MA[1]^{\otimes \bar{x}^1},\MA[1]^{\otimes\bar{x}^2}, \dots ,\MA[1]^{\otimes \bar{x}^n}}
    \end{split}
\end{equation}
for $\doubar{x}=(\bar{x}^1,\dots,\bar{x}^n)\in\bar{\MO}_{\MA}^n$.
We denote by $\Multi^{\bullet}_{d,F_0}(\MA,\MB)^{C_{\llg(\bullet)}}$ the space of elements of $\Multi^{\bullet}_{d,F_0}(\MA,\MB)$ that are invariant under the action of $\prod_{n\in\NN^*} C_n$.

When $\MB=\MA$ and $F_0=\id$, we denote $\Multi^{\bullet}_{d,F_0}(\MA,\MB)^{C_{\llg(\bullet)}}$ simply by $\Multi^{\bullet}_{d}(\MA)^{C_{\llg(\bullet)}}$.
\end{definition}

\begin{definition}
A \textbf{\textcolor{ultramarine}{$d$-pre-Calabi-Yau structure}} on a graded quiver $\MA$ is an element
\[
s_{d+1}M_{\MA}\in\Multi^{\bullet}_d(\MA)^{C_{\llg(\bullet)}}[d+1]
\]
of degree $1$, solving the Maurer-Cartan equation 
\begin{equation}
    \label{eq:MC}
    [s_{d+1}M_{\MA},s_{d+1}M_{\MA}]_{\nec}=0.
\end{equation}
Since $s_{d+1}M_{\MA}$ has degree $1$, this is tantamount to requiring that 
\begin{equation}
\label{eq:stasheff-pCY}
    \tag{$\operatorname{SI}^{\doubar{x}}$}
    (s_{d+1}M_{\MA}\upperset{\nec}{\circ} s_{d+1}M_{\MA})^{\doubar{x}}=0
\end{equation}
for every $\doubar{x}\in\doubar{\MO}$ with $(s_{d+1}M_{\MA}\upperset{\nec}{\circ} s_{d+1}M_{\MA})^{\doubar{x}}=\sum\mathcal{E}(\mathcalboondox{D})$ where the sum is over all the filled diagrams of type $\doubar{x}$ and of the form 

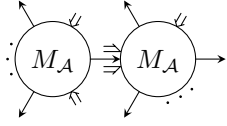
\begin{tikzpicture}[line cap=round,line join=round,x=1.0cm,y=1.0cm]
\clip(-7.4,-1) rectangle (10.586743541133627,1);
    \draw (0,0) circle (0.5cm);
    \draw[->,>=stealth] (0.5,0)--(0.9,0);
    \draw[rotate=120][->,>=stealth] (0.5,0)--(0.9,0);
    \draw[rotate=-120][->,>=stealth] (0.5,0)--(0.9,0);
    \draw[rotate=180,shift={(0.5,0)}]\doubleflechescindeeleft;
    \draw[rotate=180,shift={(0.5,0)}]\doubleflechescindeeright;
    \draw[rotate=60,shift={(0.5,0)}]\doublefleche;
    \draw (-1.4,0) circle (0.5cm);
    \draw[->,>=stealth](-0.9,0)--(-0.5,0);
    \draw[rotate around={120:(-1.4,0)}][->,>=stealth](-0.9,0)--(-0.5,0);
    \draw[rotate around={-120:(-1.4,0)}][->,>=stealth](-0.9,0)--(-0.5,0);
    \draw[rotate around={60:(-1.4,0)},shift={(-0.9,0)}]\doublefleche;
    \draw[rotate around={-60:(-1.4,0)},shift={(-0.9,0)}]\doublefleche;
    \draw[fill=black] (-2,0) circle (0.3pt);
    \draw[fill=black] (-1.95,0.2) circle (0.3pt);
    \draw[fill=black] (-1.95,-0.2) circle (0.3pt);
    \draw[fill=black] (0.3,-0.5) circle (0.3pt);
    \draw[fill=black] (0.45,-0.4) circle (0.3pt);
    \draw[fill=black] (0.13,-0.58) circle (0.3pt);
    \draw (0,0.25) node[anchor=north]{$M_{\MA}$};
    \draw (-1.4,0.25) node[anchor=north]{$M_{\MA}$};
\end{tikzpicture}

A graded quiver endowed with a $d$-pre-Calabi-Yau structure is called a \textbf{\textcolor{ultramarine}{$d$-pre-Calabi-Yau category}}. 
\end{definition}

We now recall the definition of $d$-pre-Calabi-Yau morphisms given in \cite{ktv} (cf. \cite{lv}).

\begin{definition}
    Given $d$-pre-Calabi-Yau categories $(\MA,s_{d+1}M_{\MA})$ and $(\MB,s_{d+1}M_{\MB})$, $F_0:\MO_{\MA} \rightarrow\MO_{\MB}$ between their respective sets of objects and $s_{d+1}\mathbf{F}\in \Multi^{\bullet}_{d,F_0}(\MA,\MB)[d+1]$ of degree $0$, the \textbf{\textcolor{ultramarine}{multinecklace composition of $s_{d+1}M_{\MA}$ and $s_{d+1}\mathbf{
F}$}} is the element 
    \[s_{d+1}\mathbf{F}\upperset{\multinec}{\circ} s_{d+1}M_{\MA}\in \Multi^{\bullet}_{d,F_0}(\MA,\MB)[d+1]
    \]
    given by 
    \[(s_{d+1}\mathbf{F}\upperset{\multinec}{\circ}s_{d+1}M_{\MA})^{\doubar{x}}=\sum\mathcal{E}(\mathcalboondox{D})\] for $\doubar{x}\in\MO_{\MA}$, where the sum is over all the filled diagrams $\mathcalboondox{D}$ of type $\doubar{x}$ of the form

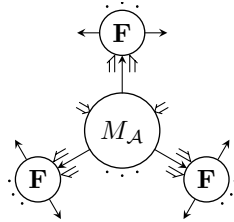
\begin{equation}
\begin{tikzpicture}[line cap=round,line join=round,x=1.0cm,y=1.0cm]
\clip(-7.4,-1.2) rectangle (10.586743541133627,1.7);
 \draw [line width=0.5pt] (0.,0.) circle (0.5cm);
     \shadedraw[rotate=30,shift={(0.5cm,0cm)}] \doublefleche;
     \shadedraw[rotate=150,shift={(0.5cm,0cm)}] \doublefleche;
     \draw [line width=0.5pt] (0,1.3) circle (0.3cm);
     \shadedraw [shift={(0cm,1cm)},rotate=-90] \doubleflechescindeeleft;
     \shadedraw [shift={(0cm,1cm)},rotate=-90] \doubleflechescindeeright;
     \shadedraw [shift={(0cm,1cm)},rotate=-90] \fleche;
     \draw [->,> = stealth] (0.3,1.3)--(0.6,1.3);
     \draw [->,> = stealth] (-0.3,1.3)--(-0.6,1.3);
     \draw [line width=0.5pt] (1.12,-0.65) circle (0.3cm);
     \shadedraw[shift={(0.86cm,-0.5cm)},rotate=150] \doubleflechescindeeleft;
     \shadedraw[shift={(0.86cm,-0.5cm)},rotate=150] \doubleflechescindeeright;
     \shadedraw[shift={(0.86cm,-0.5cm)},rotate=150] \fleche;
     \draw [rotate around ={60:(1.12,-0.65)}] [->,> = stealth] (1.43,-0.65)--(1.73,-0.65);
     \draw [rotate around ={-120:(1.12,-0.65)}] [->,> = stealth] (1.43,-0.65)--(1.73,-0.65);
     \draw [line width=0.5pt] (-1.12,-0.65) circle (0.3cm);
     \shadedraw[shift={(-0.86cm,-0.5cm)},rotate=30] \doubleflechescindeeleft;
      \shadedraw[shift={(-0.86cm,-0.5cm)},rotate=30] \doubleflechescindeeright;
       \shadedraw[shift={(-0.86cm,-0.5cm)},rotate=30] \fleche;
      \draw [rotate around ={-60:(-1.12,-0.65)}] [->,> = stealth] (-1.43,-0.65)--(-1.73,-0.65);
     \draw [rotate around ={120:(-1.12,-0.65)}] [->,> = stealth] (-1.43,-0.65)--(-1.73,-0.65);
\begin{scriptsize}
\draw [fill=black] (0,1.7) circle (0.3pt);
\draw [fill=black] (0.2,1.65) circle (0.3pt);
\draw [fill=black] (-0.2,1.65) circle (0.3pt);
\draw [fill=black] (0,-0.6) circle (0.3pt);
\draw [fill=black] (0.2,-0.55) circle (0.3pt);
\draw [fill=black] (-0.2,-0.55) circle (0.3pt);
\draw [fill=black] (1.45,-0.85) circle (0.3pt);
\draw [fill=black] (1.5,-0.67) circle (0.3pt);
\draw [fill=black] (1.33,-0.97) circle (0.3pt);
\draw [fill=black] (-1.45,-0.85) circle (0.3pt);
\draw [fill=black] (-1.5,-0.67) circle (0.3pt);
\draw [fill=black] (-1.33,-0.97) circle (0.3pt);
\end{scriptsize}
\draw (0,0.25)node[anchor=north]{$M_{\MA}$};
\draw (0,1.55)node[anchor=north]{$\mathbf{F}$};
\draw (-1.12,-0.4)node[anchor=north]{$\mathbf{F}$};
\draw (1.12,-0.4)node[anchor=north]{$\mathbf{F}$};
\end{tikzpicture}
\label{fig:morph-1}
\end{equation}
\noindent
\end{definition}

\begin{definition}
 Given $d$-pre-Calabi-Yau categories $(\MA,s_{d+1}M_{\MA})$ and $(\MB,s_{d+1}M_{\MB})$, $F_0:\MO_{\MA} \rightarrow\MO_{\MB}$ between their respective sets of objects and $s_{d+1}\mathbf{F}\in \Multi^{\bullet}_{d,F_0}(\MA,\MB)[d+1]$ of degree $0$, the \textbf{\textcolor{ultramarine}{pre composition of $s_{d+1}\mathbf{F}$ and $s_{d+1}M_{\MB}$}} is the element \[s_{d+1}M_{\MB}\upperset{\pre}{\circ}s_{d+1}\mathbf{F}\in \Multi^{\bullet}_{d,F_0}(\MA,\MB)[d+1]\] given by \[(s_{d+1}M_{\MB}\upperset{\pre}{\circ}s_{d+1}\mathbf{F})^{\doubar{x}}=\sum\mathcal{E}(\mathcal{D'})\] for $\doubar{x}\in\MO_{\MA}$, where the sum is over all the filled diagrams $\mathcalboondox{D'}$ of type $\doubar{x}$ of the form

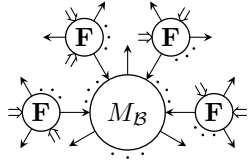
\begin{equation}
\begin{tikzpicture}[line cap=round,line join=round,x=1.0cm,y=1.0cm]
\clip(-4.6,-1) rectangle (4.579407206898253,1.5);
     \draw [line width=0.5pt] (0.,0.) circle (0.5cm);
     \draw [rotate=90] [->,> = stealth] (0.5,0)--(0.9,0);
     \draw [rotate=-30] [->,> = stealth] (0.5,0)--(0.9,0);
     \draw [rotate=-150] [->,> = stealth] (0.5,0)--(0.9,0);
     \draw [rotate=0] [<-,> = stealth] (0.5,0)--(0.9,0);
     \draw [line width=0.5pt] (1.15,0.) circle (0.25cm);
     \draw[rotate around={60:(1.15,0)}] [->,> = stealth] (1.4,0)--(1.7,0);
      \draw[rotate around={-60:(1.15,0)}] [->,> = stealth] (1.4,0)--(1.7,0);
      \shadedraw[rotate around={120:(1.15,0)}, shift={(1.4cm,0cm)}] \doublefleche;
       \shadedraw[shift={(1.4cm,0cm)}] \doublefleche;
     \draw [rotate=60] [<-,> = stealth] (0.5,0)--(0.9,0);
       \draw [rotate=180] [<-,> = stealth] (0.5,0)--(0.9,0);
      \draw [line width=0.5pt] (-1.15,0.) circle (0.25cm);
     \draw[rotate around={-60:(-1.15,0)}] [->,> = stealth] (-1.4,0)--(-1.7,0);
      \draw[rotate around={60:(-1.15,0)}] [->,> = stealth] (-1.4,0)--(-1.7,0);
      \shadedraw[rotate around={-60:(-1.15,0)}, shift={(-0.9cm,0cm)}] \doublefleche;
       \shadedraw[rotate around={180:(-1.15,0)},shift={(-0.9cm,0cm)}] \doublefleche;
        \draw [rotate=120] [<-,> = stealth] (0.5,0)--(0.9,0);
        \draw [line width=0.5pt] (0.575,1) circle (0.25cm);
        \draw[rotate around={120:(0.575,1)}] [->,> = stealth] (0.825,1)--(1.125,1);
      \draw[rotate around={0:(0.575,1)}] [->,> = stealth] (0.825,1)--(1.125,1);
      \shadedraw[rotate around={60:(0.575,1)}, shift={(0.825cm,1cm)}] \doublefleche;
       \shadedraw[rotate around={180:(0.575,1)},shift={(0.825cm,1cm)}] \doublefleche;
        \draw [line width=0.5pt] (-0.575,1) circle (0.25cm);
      \draw[rotate around={-120:(-0.575,1)}] [->,> = stealth] (-0.825,1)--(-1.125,1);
      \draw[rotate around={0:(-0.575,1)}] [->,> = stealth] (-0.825,1)--(-1.125,1);
      \shadedraw[rotate around={-120:(-0.575,1)}, shift={(-0.325cm,1cm)}] \doublefleche;
       \shadedraw[rotate around={120:(-0.575,1)},shift={(-0.325cm,1cm)}] \doublefleche;
\begin{scriptsize}
\draw [fill=black] (0.65,0.7) circle (0.3pt);
\draw [fill=black] (0.79,0.75) circle (0.3pt);
\draw [fill=black] (0.88,0.85) circle (0.3pt);
\draw [fill=black] (-0.25,1) circle (0.3pt);
\draw [fill=black] (-0.29,0.85) circle (0.3pt);
\draw [fill=black] (-0.29,1.15) circle (0.3pt);
\draw [fill=black] (-1,0.3) circle (0.3pt);
\draw [fill=black] (-0.9,0.2) circle (0.3pt);
\draw [fill=black] (-1.15,0.3) circle (0.3pt);
\draw [fill=black] (1,-0.3) circle (0.3pt);
\draw [fill=black] (1.15,-0.3) circle (0.3pt);
\draw [fill=black] (0.9,-0.2) circle (0.3pt);
\draw [fill=black] (0,-0.6) circle (0.3pt);
\draw [fill=black] (0.2,-0.55) circle (0.3pt);
\draw [fill=black] (-0.2,-0.55) circle (0.3pt);
\draw [rotate=120][fill=black] (0,-0.6) circle (0.3pt);
\draw [rotate=120][fill=black] (0.15,-0.55) circle (0.3pt);
\draw [rotate=120][fill=black] (-0.15,-0.55) circle (0.3pt);
\draw [rotate=240][fill=black] (0,-0.6) circle (0.3pt);
\draw [rotate=240][fill=black] (0.15,-0.55) circle (0.3pt);
\draw [rotate=240][fill=black] (-0.15,-0.55) circle (0.3pt);
\end{scriptsize}
\draw (0,0.25)node[anchor=north]{$M_{\MB}$};
\draw (-1.15,0.25)node[anchor=north]{$\mathbf{F}$};
\draw (1.15,0.25)node[anchor=north]{$\mathbf{F}$};
\draw (0.575,1.25)node[anchor=north]{$\mathbf{F}$};
\draw (-0.575,1.25)node[anchor=north]{$\mathbf{F}$};
\end{tikzpicture}
\label{fig:morph-2}
\end{equation}

\noindent
\end{definition}

\begin{definition}
\label{def:pcY-morphism}
Given $d$-pre-Calabi-Yau categories $(\MA,s_{d+1}M_{\MA})$ and $(\MB,s_{d+1}M_{\MB})$ with respective sets of objects $\MO_{\MA}$ and $\MO_{\MB}$ a \textbf{\textcolor{ultramarine}{$d$-pre-Calabi-Yau morphism}} is a pair $(F_0,s_{d+1}\mathbf{F})$ where $F_0$ is a map $\MO_{\MA}\rightarrow \MO_{\MB}$ and $s_{d+1}\mathbf{F}\in \Multi^{\bullet}_{d,F_0}(\MA,\MB)^{C_{\llg(\bullet)}}[d+1]$
is of degree $0$ satisfying the following equation
\begin{equation}
\label{eq:morphism-pCY}
\tag{$\operatorname{MI}^{\doubar{x}}$}
(s_{d+1}\mathbf{F}\upperset{\multinec}{\circ} s_{d+1}M_{\MA})^{\doubar{x}}= (s_{d+1}M_{\MB}\upperset{\pre}{\circ} s_{d+1}\mathbf{F})^{\doubar{x}}
\end{equation}
for all $\doubar{x}\in\doubar{\MO}_{\MA}$.
Note that the left member and right member of the previous identity belong to
\[
\Homgr_{\kk}\bigg(\bigotimes\limits_{i=1}^{n}\MA[1]^{\otimes\bar{x}^{i}}, \bigotimes\limits_{i=1}^{n-1}{}_{F_0(\llt(\bar{x}^{i}))}\MB_{F_0(\rrt(\bar{x}^{i+1}))}[-d]\otimes{}_{F_0(\llt(\bar{x}^{n}))}\MB_{F_0(\rrt(\bar{x}^{1}))}[1]\bigg).
\]
\end{definition}
We now recall how to compose $d$-pre-Calabi-Yau morphisms.
\begin{definition}
\label{def:cp-pCY}
Let $(\MA,s_{d+1}M_{\MA})$, $(\MB,s_{d+1}M_{\MB})$ and $(\mathcal{C},s_{d+1}M_{\mathcal{C}})$ be $d$-pre-Calabi-Yau categories and consider $(F_0,s_{d+1}\mathbf{F}) :(\MA,s_{d+1}M_{\MA})\rightarrow (\MB,s_{d+1}M_{\MB})$ and $(G_0,s_{d+1}\mathbf{G}) :(\MB,s_{d+1}M_{\MB})\rightarrow (\mathcal{C},s_{d+1}M_{\mathcal{C}})$ two $d$-pre-Calabi-Yau morphisms.
The \textbf{\textcolor{ultramarine}{composition of $s_{d+1}\mathbf{F}$ and $s_{d+1}\mathbf{G}$}} is the pair \[(G_0\circ F_0,s_{d+1}\mathbf{G}\circ s_{d+1}\mathbf{F})\] where
\[s_{d+1}\mathbf{G}\circ s_{d+1}\mathbf{F}\in \Multi^{\bullet}_{d,G_0\circ F_0}(\MA,\mathcal{C})^{C_{\llg(\bullet)}}[d+1]\] is of degree $0$ and is given by $(s_{d+1}\mathbf{G}\circ s_{d+1}\mathbf{F})^{\doubar{x}}=\sum\mathcal{E}(\mathcalboondox{D})$ where the sum is over all filled diagrams $\mathcalboondox{D}$ of type $\doubar{x}\in\doubar{\MO}_{\MA}$ of the form

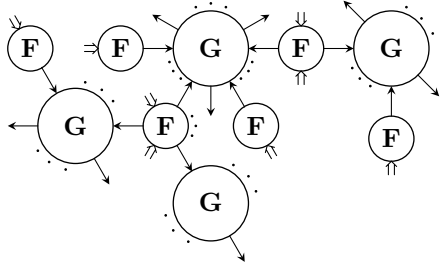
\begin{equation}
\begin{tikzpicture}[line cap=round,line join=round,x=1.0cm,y=1.0cm]
\clip(-7.5,-3) rectangle (11.248895821273946,1);
     \draw [line width=0.5pt] (0.,0.) circle (0.5cm);
     \draw [rotate=30] [->,> = stealth] (0.5,0)--(0.9,0);
     \draw [rotate=-90] [->,> = stealth] (0.5,0)--(0.9,0);
     \draw [rotate=150] [->,> = stealth] (0.5,0)--(0.9,0);
     \draw [rotate=0] [<-,> = stealth] (0.5,0)--(0.9,0);
     \draw [line width=0.5pt] (1.2,0.) circle (0.3cm);
     \draw [->,> = stealth] (1.5,0)--(1.9,0);
     \shadedraw[rotate around={90:(1.2,0)},shift={(1.5cm,0cm)}] \doublefleche;
     \shadedraw[rotate around={-90:(1.2,0)},shift={(1.5cm,0cm)}] \doublefleche;
      \draw [line width=0.5pt] (2.4,0.) circle (0.5cm);
      \draw [rotate around={-90:(2.4,0)}] [<-,> = stealth] (2.9,0)--(3.3,0);
      \draw [rotate around={-45:(2.4,0)}] [->,> = stealth] (2.9,0)--(3.3,0);
    \draw [rotate around={135:(2.4,0)}] [->,> = stealth] (2.9,0)--(3.3,0);
     \draw [line width=0.5pt] (2.4,-1.2) circle (0.3cm);
     \shadedraw[rotate around={-90:(2.4,-1.2)},shift={(2.7cm,-1.2cm)}] \doublefleche;
     \draw [rotate=-60] [<-,> = stealth] (0.5,0)--(0.9,0);
      \draw [line width=0.5pt] (0.6,-1.03) circle (0.3cm);
     \shadedraw[rotate around={-60:(0.6,-1.03)},shift={(0.9cm,-1.03cm)}] \doublefleche;
     \draw [line width=0.5pt] (-1.2,0.) circle (0.3cm);
     \draw [<-,> = stealth] (-0.5,0)--(-0.9,0);
     \shadedraw[rotate around={180:(-1.2,0)},shift={(-0.9cm,0cm)}] \doublefleche;
     \draw [rotate=-120] [<-,> = stealth] (0.5,0)--(0.9,0);
    \draw [line width=0.5pt] (-0.6,-1.03) circle (0.3cm);
     \shadedraw[rotate around={-120:(-0.6,-1.03)},shift={(-0.3cm,-1.03cm)}] \doublefleche;
     \shadedraw[rotate around={-240:(-0.6,-1.03)},shift={(-0.3cm,-1.03cm)}] \doublefleche;
     \draw[rotate around={180:(-0.6,-1.03)}][->,> = stealth] (-0.3,-1.03)--(0.1,-1.03);
     \draw[rotate around={-60:(-0.6,-1.03)}][->,> = stealth] (-0.3,-1.03)--(0.1,-1.03);
      \draw [line width=0.5pt] (-1.8,-1.03) circle (0.5cm);
      \draw [rotate around={120:(-1.8,-1.03)}] [<-,> = stealth] (-1.3,-1.03)--(-0.9,-1.03);
       \draw [rotate around={-60:(-1.8,-1.03)}] [->,> = stealth] (-1.3,-1.03)--(-0.9,-1.03);
        \draw [rotate around={180:(-1.8,-1.03)}] [->,> = stealth] (-1.3,-1.03)--(-0.9,-1.03);
     \draw [line width=0.5pt] (-2.4,0.) circle (0.3cm);
     \shadedraw[rotate around={120:(-2.4,0)},shift={(-2.1cm,0cm)}] \doublefleche;
     \draw [line width=0.5pt] (0,-2.06) circle (0.5cm);
      \draw [rotate around={-60:(0,-2.06)}] [->,> = stealth] (0.5,-2.06)--(0.9,-2.06);
\begin{scriptsize}
\draw [fill=black] (0,0.6) circle (0.3pt);
\draw [fill=black] (0.2,0.55) circle (0.3pt);
\draw [fill=black] (-0.2,0.55) circle (0.3pt);
\draw [fill=black] (0.5,-0.3) circle (0.3pt);
\draw [fill=black] (0.4,-0.4) circle (0.3pt);
\draw [fill=black] (0.55,-0.18) circle (0.3pt);
\draw [fill=black] (-0.5,-0.3) circle (0.3pt);
\draw [fill=black] (-0.4,-0.4) circle (0.3pt);
\draw [fill=black] (-0.55,-0.18) circle (0.3pt);
\draw [fill=black] (1.98,-0.4) circle (0.3pt);
\draw [fill=black] (1.88,-0.26) circle (0.3pt);
\draw [fill=black] (2.12,-0.5) circle (0.3pt);
\draw [fill=black] (2.8,0.4) circle (0.3pt);
\draw [fill=black] (2.95,0.2) circle (0.3pt);
\draw [fill=black] (2.6,0.55) circle (0.3pt);
\draw [fill=black] (-0.2,-1.03) circle (0.3pt);
\draw [fill=black] (-0.25,-0.9) circle (0.3pt);
\draw [fill=black] (-0.25,-1.16) circle (0.3pt);
\draw [fill=black] (0.3,-1.55) circle (0.3pt);
\draw [fill=black] (0.5,-1.7) circle (0.3pt);
\draw [fill=black] (0.6,-1.9) circle (0.3pt);
\draw [rotate around={180:(0,-2.06)}][fill=black] (0.6,-1.9) circle (0.3pt);
\draw [rotate around={180:(0,-2.06)}][fill=black] (0.5,-1.7) circle (0.3pt);
\draw [rotate around={180:(0,-2.06)}][fill=black] (0.3,-1.55) circle (0.3pt);
\draw [fill=black] (-1.3,-0.7) circle (0.3pt);
\draw [fill=black] (-1.45,-0.55) circle (0.3pt);
\draw [fill=black] (-1.67,-0.45) circle (0.3pt);
\draw [rotate around={180:(-1.8,-1.03)}][fill=black] (-1.3,-0.7) circle (0.3pt);
\draw [rotate around={180:(-1.8,-1.03)}][fill=black] (-1.45,-0.55) circle (0.3pt); circle (0.3pt);
\draw [rotate around={180:(-1.8,-1.03)}][fill=black] (-1.67,-0.45) circle (0.3pt);
\end{scriptsize}
\draw (0,0.25)node[anchor=north]{$\mathbf{G}$};
\draw (0,-1.8)node[anchor=north]{$\mathbf{G}$};
\draw (2.4,0.25)node[anchor=north]{$\mathbf{G}$};
\draw (-1.8,-0.78)node[anchor=north]{$\mathbf{G}$};
\draw (1.2,0.25)node[anchor=north]{$\mathbf{F}$};
\draw (-1.2,0.25)node[anchor=north]{$\mathbf{F}$};
\draw (2.4,-0.95)node[anchor=north]{$\mathbf{F}$};
\draw (-0.6,-0.78)node[anchor=north]{$\mathbf{F}$};
\draw (-2.4,0.25)node[anchor=north]{$\mathbf{F}$};
\draw (0.6,-0.78)node[anchor=north]{$\mathbf{F}$};
\end{tikzpicture}
    \label{fig:composition}
\end{equation}

More precisely, the sum is over all the diagrams $\mathcalboondox{D}$ of type $\doubar{x}$ where the discs are filled with either $\mathbf{F}$ or $\mathbf{G}$ such that each outgoing arrow of a disc filled with $\mathbf{F}$ is connected to an incoming arrow of a disc filled with $\mathbf{G}$ and each incoming arrow of a disc filled with $\mathbf{G}$ is connected to an outgoing arrow of a disc filled with $\mathbf{F}$. In particular, the incoming (resp. outgoing) arrows of the diagram $\mathcalboondox{D}$ are incoming (resp. outgoing) arrows of a disc filled with $\mathbf{F}$ (resp. $\mathbf{G}$).
\end{definition}

\begin{proposition}
\label{prop:pCY-category}
The data of $d$-pre-Calabi-Yau categories together with $d$-pre-Calabi-Yau morphisms and their composition define a category, denoted by $\pCY$. Given a graded quiver $\MA$ with set of objects $\MO$, the identity morphism $(\MA,s_{d+1}M_{\MA})\rightarrow (\MA,s_{d+1}M_{\MA})$ is given by $\operatorname{
Id}^{x,y}=\id_{{}_{x}\MA_{y}}$ for $x,y\in\MO$ and $\operatorname{Id}^
{\bar{x}^1,\dots,\bar{x}^n}=0$ for $(\bar{x}^1,\dots,\bar{x}^n)\in\bar{\MO}^n$ such that $n\neq 1$ or $n=1$ and $\llg(\bar{x}^1)>2$.
\end{proposition}
\begin{proof}
    See \cite{moi} (cf. \cite{ktv,lv}).
\end{proof}

\subsection{\texorpdfstring{Pre-Calabi-Yau morphisms as Maurer-Cartan elements of an $L_{\infty}$-algebra}{Pre-Calabi-Yau morphisms as Maurer-Cartan elements of an L-infinity-algebra}}\label{section:brackets}
We now present the way of seeing $d$-pre-Calabi-Yau morphisms as Maurer-Cartan elements of an $L_{\infty}$-algebra.
In subsection \ref{subsection:non-fixed-pCY}, we define an $L_{\infty}$-algebra whose Maurer-Cartan elements are in bijection with triples $(s_{d+1}M_{\MB},s_{d+1}\mathbf{F},s_{d+1}M_{\MA})$ where $s_{d+1}M_{\MB}$ and $s_{d+1}M_{\MA}$ are $d$-pre-Calabi-Yau structures on graded quivers $\MB$ and $\MA$ respectively and $s_{d+1}\mathbf{F}$ is a $d$-pre-Calabi-Yau morphism $(\MA,s_{d+1}M_{\MA})\rightarrow (\MB,s_{d+1}M_{\MB})$. This is analogous to the construction for the case of $A_{\infty}$-algebras presented in \cite{borisov1}. 
In subsection \ref{subsection:fixed-pCY}, given $d$-pre-Calabi-Yau categories $(\MA,s_{d+1}M_{\MA})$ and $(\MB,s_{d+1}M_{\MB})$ we construct an $L_{\infty}$-algebra whose Maurer-Cartan elements coincide with $d$-pre-Calabi-Yau morphisms between them analogously to the case of $L_{\infty}$-morphisms developed in \cite{kraft-schnitzer}.

\subsubsection{The case of non-fixed pre-Calabi-Yau structures}
\label{subsection:non-fixed-pCY}
Given two graded quivers $\MA$ and $\MB$ and a map $\Phi_0 : \MO_{\MA}\rightarrow \MO_{\MB}$, we consider the graded Lie algebras $(\Multi_d^{\bullet}(\MA)^{C_{\llg(\bullet)}}[d+1],[-,-]_{\nec})$ and $(\Multi_d^{\bullet}(\MB)^{C_{\llg(\bullet)}}[d+1],[-,-]_{\nec})$ and see them as $L_{\infty}$-algebras (see Remark \ref{remark:lie-is-linfinity}) that we denote by $(\mathfrak{g}_{\MA},\ell_{\MA})$ and $(\mathfrak{g}_{\MB},\ell_{\MB})$ respectively. We also denote $\mathfrak{h}=\Multi_{d,\Phi_0}^{\bullet}(\MA,\MB)^{C_{\llg(\bullet)}}[d+1]$.

\begin{definition}\label{def:ell}
    We define an $\infty$-bracket $\ell$ on $\mathcal{L}_d^{\Phi_0}(\MA,\MB)=\mathfrak{g}_{\MB}\oplus\mathfrak{h}[-1]\oplus \mathfrak{g}_{\MA}$ as follows. Given $n\in\NN^*$, define $\ell^n_{|\mathfrak{g}_{\MB}^{\otimes n}}=\ell^n_{\MB}$, $\ell^n_{|\mathfrak{g}_{\MA}^{\otimes n}}=\ell^n_{\MA}$, \begin{equation}
    \begin{split}
    &\ell^{n}(s_{-1}f_{n-1},\dots,s_{-1}f_{i+1},g_{\MB},s_{-1}f_{i},\dots,s_{-1}f_{1})
    \\&\hspace{5cm}=(-1)^{\epsilon_i}\sum\limits_{\sigma\in\mathfrak{S}_{n-1}}(-1)^{\delta^{\sigma}(f_{n-1},\dots,f_{1})}s_{-1}\Phi_{g_{\MB}}^{\sigma}(f_{n-1},\dots,f_{1})
    \end{split}
    \end{equation} 
    for $g_{\MB}\in \mathfrak{g}_{\MB}$ and $f_j\in \mathfrak{h}$ for $j\in\llbracket 1,n-1\rrbracket$,
    \begin{equation}
    \begin{split}
    &\ell^n(s_{-1}f_{n-1},\dots,s_{-1}f_{i+1},g_{\MA},s_{-1}f_{i},\dots,s_{-1}f_{1})\\&\hspace{5cm}= (-1)^{\epsilon'_i}\sum\limits_{\sigma\in\mathfrak{S}_{n-1}}(-1)^{\delta^{\sigma}(f_{n-1},\dots,f_{1})}s_{-1}\Psi^{\sigma}_{g_{\MA}}(f_{n-1},\dots,f_{1})
     \end{split}
    \end{equation}
    for $g_{\MA}\in \mathfrak{g}_{\MA}$ and $f_j\in \mathfrak{h}$ for $j\in\llbracket 1,n-1\rrbracket$, $\ell^n(l_n,\dots,l_1)=0$ if \begin{itemize}
        \item  there exists $i,j\in\llbracket 1,n \rrbracket$, $i\neq j$ such that $l_i\in \mathfrak{g}_{\MA}$ and $l_j\in \mathfrak{g}_{\MB}$,
        \item $n>2$ and there exists $i,j\in\llbracket 1,n \rrbracket$, $i\neq j$ such that $l_i,l_j\in \mathfrak{g}_{\MA}$ or $l_i,l_j\in \mathfrak{g}_{\MB}$,
        \item $l_i\in\mathfrak{h}$ for every $i\in\llbracket 1,n\rrbracket$,
    \end{itemize} 
where 
\begin{small}
\begin{equation}
    \begin{split}
    \epsilon_i=i+1+i|g_{\MB}|+|g_{\MB}|\sum\limits_{j=i+1}^{n-1}|f_j|+\sum\limits_{j=1}^{n-1}(j-1)|f_j|
    \text{, }\epsilon'_i=i+i|g_{\MA}|+|g_{\MA}|\sum\limits_{j=1}^{i}|f_j|+\sum\limits_{j=1}^{n-1}(j-1)|f_j|
    \end{split}
\end{equation}
\end{small}
and where $\Phi^{\sigma}_{g_{\MB}}(f_{n-1},\dots,f_{1}), \Psi^{\sigma}_{g_{\MA}}(f_{n-1},\dots,f_{1})\in\mathfrak{h}$ are given by $\Phi^{\sigma}_{g_{\MB}}(f_{n-1},\dots,f_{1})=\sum\mathcal{E}(\mathcalboondox{D})$ and $\Psi^{\sigma}_{g_{\MA}}(f_{n-1},\dots,f_{1})=\sum\mathcal{E}(\mathcalboondox{D'})$ where the sums are over all the filled diagrams of the form

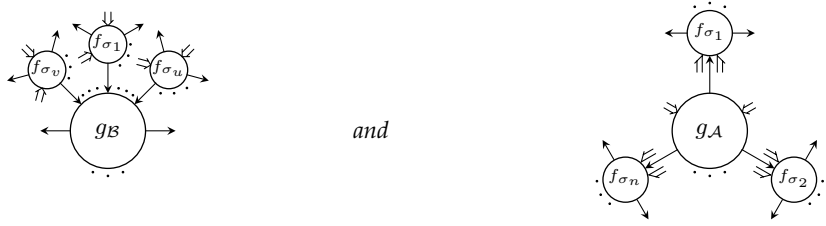
\begin{small}
\begin{equation}
    \label{eq:form-diag}
\begin{tikzpicture}[line cap=round,line join=round,x=1.0cm,y=1.0cm]
\clip(-3.5,-1.5) rectangle (4,1.8);
  \draw [line width=0.5pt] (0.,0.) circle (0.5cm);
     \draw [rotate=45] [<-,>=stealth,] (0.5,0)--(0.9,0);
     \draw [rotate=0] [->,>=stealth] (0.5,0)--(0.9,0);
     \draw [rotate=180] [->,>=stealth] (0.5,0)--(0.9,0);
     \draw [rotate=135] [<-,>=stealth] (0.5,0)--(0.9,0);
     \draw [rotate=90] [<-,>=stealth] (0.5,0)--(0.9,0);
     \draw [line width=0.5pt] (0.,1.15) circle (0.25cm);
     \draw [rotate around={150:(0,1.15)}] [->,>=stealth,] (0.25,1.15)--(0.55,1.15);
     \draw [rotate around={30:(0,1.15)}] [->,>=stealth,] (0.25,1.15)--(0.55,1.15);
     \draw [rotate around={90:(0,1.15)},shift={((0.25,1.15))}] \doublefleche;
     \draw [rotate around={210:(0,1.15)},shift={((0.25,1.15))}] \doublefleche;
     \draw [line width=0.5pt] (0.81,0.81) circle (0.25cm);
     \draw [rotate around={-15:(0.81,0.81)}] [->,>=stealth,] (1.06,0.81)--(1.36,0.81);
     \draw [rotate around={105:(0.81,0.81)}] [->,>=stealth,] (1.06,0.81)--(1.36,0.81);
     \draw [rotate around={165:(0.81,0.81)},shift={(1.06,0.81)}] \doublefleche;
     \draw [rotate around={45:(0.81,0.81)},shift={(1.06,0.81)}] \doublefleche;
     \draw [line width=0.5pt] (-0.81,0.81) circle (0.25cm);
     \draw [rotate around={75:(-0.81,0.81)}] [->,>=stealth,] (-0.56,0.81)--(-0.26,0.81);
     \draw [rotate around={195:(-0.81,0.81)}] [->,>=stealth,] (-0.56,0.81)--(-0.26,0.81);
     \draw [rotate around={135:(-0.81,0.81)},shift={(-0.56,0.81)}]\doublefleche;
     \draw [rotate around={255:(-0.81,0.81)},shift={(-0.56,0.81)}]\doublefleche;
\begin{scriptsize}
\draw [fill=black] (0.7,0.5) circle (0.3pt);
\draw [fill=black] (0.88,0.5) circle (0.3pt);
\draw [fill=black] (1.05,0.57) circle (0.3pt);
\draw [fill=black] (0.27,1) circle (0.3pt);
\draw [fill=black] (0.16,0.87) circle (0.3pt);
\draw [fill=black] (0.3,1.15) circle (0.3pt);
\draw [fill=black] (-0.55,1) circle (0.3pt);
\draw [fill=black] (-0.5,0.85) circle (0.3pt);
\draw [fill=black] (-0.52,0.7) circle (0.3pt);
\draw [fill=black] (0,-0.6) circle (0.3pt);
\draw [fill=black] (0.2,-0.55) circle (0.3pt);
\draw [fill=black] (-0.2,-0.55) circle (0.3pt);
\draw [rotate=-45][fill=black] (-0.5,0.3) circle (0.3pt);
\draw [rotate=-45][fill=black] (-0.55,0.2) circle (0.3pt);
\draw [rotate=-45][fill=black] (-0.57,0.1) circle (0.3pt);
\draw [rotate=45][fill=black] (0.5,0.3) circle (0.3pt);
\draw [rotate=45][fill=black] (0.55,0.2) circle (0.3pt);
\draw [rotate=45][fill=black] (0.57,0.1) circle (0.3pt);
\end{scriptsize}
\draw (0,0.2)node[anchor=north]{$g_{\MB}$};
\draw (-0.81,1.06)node[anchor=north]{$\scriptstyle{f_{\sigma_v}}$};
\draw (0.81,1.06)node[anchor=north]{$\scriptstyle{f_{\sigma_u}}$};
\draw (0,1.4)node[anchor=north]{$\scriptstyle{f_{\sigma_1}}$};
\draw (3.5,0.25) node[anchor=north]{and};
\end{tikzpicture}
    \begin{tikzpicture}[line cap=round,line join=round,x=1.0cm,y=1.0cm]
\clip(-4,-1.5) rectangle (5.249408935174429,1.8);
     \draw [line width=0.5pt] (0.,0.) circle (0.5cm);
     \shadedraw[rotate=30,shift={(0.5cm,0cm)}] \doublefleche;
     \shadedraw[rotate=150,shift={(0.5cm,0cm)}] \doublefleche;
     \draw [line width=0.5pt] (0,1.3) circle (0.3cm);
     \shadedraw [shift={(0cm,1cm)},rotate=-90] \doubleflechescindeeleft;
     \shadedraw [shift={(0cm,1cm)},rotate=-90] \doubleflechescindeeright;
     \shadedraw [shift={(0cm,1cm)},rotate=-90] \fleche;
     \draw [->,>=stealth] (0.3,1.3)--(0.6,1.3);
     \draw [->,>=stealth] (-0.3,1.3)--(-0.6,1.3);
     \draw [line width=0.5pt] (1.12,-0.65) circle (0.3cm);
     \shadedraw[shift={(0.86cm,-0.5cm)},rotate=150] \doubleflechescindeeleft;
     \shadedraw[shift={(0.86cm,-0.5cm)},rotate=150] \doubleflechescindeeright;
     \shadedraw[shift={(0.86cm,-0.5cm)},rotate=150] \fleche;
     \draw [rotate around ={60:(1.12,-0.65)}] [->,>=stealth] (1.43,-0.65)--(1.73,-0.65);
     \draw [rotate around ={-120:(1.12,-0.65)}] [->,>=stealth] (1.43,-0.65)--(1.73,-0.65);
     \draw [line width=0.5pt] (-1.12,-0.65) circle (0.3cm);
     \shadedraw[shift={(-0.86cm,-0.5cm)},rotate=30] \doubleflechescindeeleft;
      \shadedraw[shift={(-0.86cm,-0.5cm)},rotate=30] \doubleflechescindeeright;
       \shadedraw[shift={(-0.86cm,-0.5cm)},rotate=30] \fleche;
      \draw [rotate around ={-60:(-1.12,-0.65)}] [->,>=stealth] (-1.43,-0.65)--(-1.73,-0.65);
     \draw [rotate around ={120:(-1.12,-0.65)}] [->,>=stealth] (-1.43,-0.65)--(-1.73,-0.65);
\begin{scriptsize}
\draw [fill=black] (0,1.7) circle (0.3pt);
\draw [fill=black] (0.2,1.65) circle (0.3pt);
\draw [fill=black] (-0.2,1.65) circle (0.3pt);
\draw [fill=black] (0,-0.6) circle (0.3pt);
\draw [fill=black] (0.2,-0.55) circle (0.3pt);
\draw [fill=black] (-0.2,-0.55) circle (0.3pt);
\draw [fill=black] (1.45,-0.85) circle (0.3pt);
\draw [fill=black] (1.5,-0.67) circle (0.3pt);
\draw [fill=black] (1.33,-0.97) circle (0.3pt);
\draw [fill=black] (-1.45,-0.85) circle (0.3pt);
\draw [fill=black] (-1.5,-0.67) circle (0.3pt);
\draw [fill=black] (-1.33,-0.97) circle (0.3pt);
\end{scriptsize}
\draw (0,0.2)node[anchor=north]{$g_{\MA}$};
\draw (0,1.55)node[anchor=north]{$\scriptstyle{f_{\sigma_1}}$};
\draw (-1.12,-0.4)node[anchor=north]{$\scriptstyle{f_{\sigma_n}}$};
\draw (1.12,-0.4)node[anchor=north]{$\scriptstyle{f_{\sigma_2}}$};
\end{tikzpicture}
\end{equation}
\end{small}

\noindent respectively, where the bold arrow is placed such that $f_{\sigma_1}$ fills the first (resp. the second) disc in the application order of diagrams of the form $\mathcalboondox{D}$ (resp. $\mathcalboondox{D'}$).
\end{definition}

\begin{lemma}
\label{lemma:fin-diag}
    Given $(i,j)\in \NN^2$, there exists a finite number of diagrams of the forms \eqref{eq:form-diag} and of type $\doubar{x}$ where $\llg(\doubar{x})=j$ and $N(\doubar{x})=i$. In particular, for $F\in\mathcal{L}_d^{\Phi_0}(\MA,\MB)^1$, the projection $\sum_{n\geq 0}\ell^n(F,\dots,F)^{\doubar{x}}$ on 
    \[
    \Multi_d^{\text{\scriptsize{$\doubarl{\Phi_0(\doubar{x})}$}}}(\MB)^{\llg(\doubar{x})}[d+1]\oplus \Multi_{d,\Phi_0}^{\doubar{x}}(\MA,\MB)^{\llg(\doubar{x})}[d+1]\oplus \Multi_d^{\doubar{x}}(\MA)^{\llg(\doubar{x})}[d+1]
    \] is a finite sum for every $\doubar{x}\in\doubar{\MO}_{\MA}$.
\end{lemma}
\begin{proof}
   Consider $(i,j)\in\NN^2$. We first show that there exists a finite number of diagrams of the form \eqref{eq:form-diag} with a disc filled with $g_{\MB}$ carrying $j$ outputs and $i$ inputs. If we denote by $j'$ the number of outputs of the disc filled with $g_{\MB}$, we have that $1\leq j' \leq j$. Moreover, there are at most $j-j'$ discs filled with $f_r$ with more than $1$ output and at most $i$ discs filled with $f_r$ carrying one output.
   Therefore, we have a finite number of discs filled with $f_r$ and for each of those, a finite number of possible types. We thus have a finite number of diagrams of the form \eqref{eq:form-diag} with a disc filled with $g_{\MB}$ carrying $j$ outputs and $i$ inputs.

   Now, we show that there exists a finite number of diagrams of the form \eqref{eq:form-diag} with a disc filled with $g_{\MA}$ carrying $j$ outputs and $i$ inputs. Such a diagram is composed of at most $j$ discs filled with $f_r$. Moreover, each of those discs carry at most $j$ outputs and $i+1$ inputs and the disc filled with $g_{\MA}$ carries at most $j$ outputs and $i$ inputs. Thus, the diagram consists of a finite number of discs and we have for each of those a finite number of possible types.   
   \end{proof}

\begin{proposition}
    $(\mathcal{L}_d^{\Phi_0}(\MA,\MB),\ell)$ is a graded $L_{\infty}$-algebra whose Maurer-Cartan elements are in bijection with triples $(s_{d+1}M_{\MB},s_{-1}(s_{d+1}\mathbf{F}),s_{d+1}M_{\MA})$ where $s_{d+1}M_{\MB}$ and $s_{d+1}M_{\MA}$ are $d$-pre-Calabi-Yau structures on $\MB$ and $\MA$ respectively and $(\Phi_0,s_{d+1}\mathbf{F}) : (\MA,s_{d+1}M_{\MA})\rightarrow (\MB,s_{d+1}M_{\MB})$ is a $d$-pre-Calabi-Yau morphism.
\end{proposition}
\begin{proof}
    We first show that the $\infty$-bracket $\ell$ is anti-symmetric. Consider $n\geq 2$ and $l_i\in\mathcal{L}_d^{\Phi_0}(\MA,\MB)$ for $i\in\llbracket 1,n\rrbracket$.
    We have to show that the identity
    \begin{equation}
    \label{eq:l-bracket-antisymm}
        \ell^n(l_1,\dots,l_n)=-(-1)^{|l_i||l_{i+1}|}\ell^n(l_1,\dots,l_{i-1},l_{i+1},l_i,l_{i+2},\dots,l_n)
    \end{equation}
    holds for every $i\in\llbracket 1,n-1\rrbracket$.
It is clear by the definition that \eqref{eq:l-bracket-antisymm} is satisfied in the following cases:
\begin{itemize}
    \item $l_i\in\mathfrak{g}_{\MA}$ or $l_i\in\mathfrak{g}_{\MB}$ or $l_i\in\mathfrak{h}$ for every $i\in\llbracket 1,n\rrbracket$,
    \item there exists $i,j\in\llbracket 1,n \rrbracket$, $i\neq j$ such that $l_i\in \mathfrak{g}_{\MA}$ and $l_j\in \mathfrak{g}_{\MB}$,
    \item $n>2$ and there exists $i,j\in\llbracket 1,n \rrbracket$, $i\neq j$ such that $l_i,l_j\in \mathfrak{g}_{\MA}$ or $l_i,l_j\in \mathfrak{g}_{\MB}$.
\end{itemize}
Then, since $[-,-]_{\nec}$ is a Lie bracket, we only have to show that \eqref{eq:l-bracket-antisymm} holds when $l_i\in\mathfrak{g}_{\MA}$ or $l_i\in\mathfrak{g}_{\MB}$ for some $i\in\llbracket 1,n\rrbracket$ and $l_j\in\mathfrak{h}$ for $j\in\llbracket 1,n\rrbracket$ with $j\neq i$.

Consider elements $f_i\in\mathfrak{h}$ for $i\in\llbracket 1,n-1\rrbracket$ and $g_{\MB}\in\mathfrak{g}_{\MB}$. Take $j\in\llbracket 1,i-1\rrbracket$.
Then, we have that 
\begin{equation}
    \begin{split}
        &\ell^{n}(s_{-1}f_{n-1},\dots,s_{-1}f_{i+1},g_{\MB},s_{-1}f_i,\dots,s_{-1}f_{j},s_{-1}f_{j+1},s_{-1}f_{j-1},\dots,s_{-1}f_{1})
        \\&=(-1)^{\epsilon}\sum\limits_{\sigma\in\mathfrak{S}_{n-1}}(-1)^{\delta^{\sigma}(f_{n-1},\dots,f_j,f_{j+1},\dots,f_1)}s_{-1}\Phi_{g_{\MB}}^{\sigma}(f_{n-1},\dots,f_{j},f_{j+1},\dots,f_{1})
        \\&=(-1)^{\epsilon}\sum\limits_{\sigma\in\mathfrak{S}_{n-1}}(-1)^{\delta^{\sigma}(f_{n-1},\dots,f_1)+|f_j||f_{j+1}|}s_{-1}\Phi_{g_{\MB}}^{\sigma}(f_{n-1},\dots,f_{1})
        \\&=-(-1)^{|s_{-1}f_j||s_{-1}f_{j+1}|}\ell^{n}(s_{-1}f_{n-1},\dots,s_{-1}f_{i+1},g_{\MB},s_{-1}f_{i},\dots,s_{-1}f_{1})
    \end{split}
\end{equation}
where 
\begin{small}
\begin{equation}
    \epsilon=i+\sum\limits_{r=1}^{n-1}(r-1)|f_r|+|f_j|+|f_{j+1}|+i|g_{\MB}|+|g_{\MB}|\sum\limits_{r=i+1}^{n-1}|f_r|.
\end{equation}
\end{small}
The proof is completely similar for $j\in\llbracket i+2,n-1\rrbracket$. If $j=i$, we have 
\begin{equation}
    \begin{split}
        &\ell^{n}(s_{-1}f_{n-1},\dots,s_{-1}f_{i+1},s_{-1}f_i,g_{\MB},s_{-1}f_{i-1},\dots,s_{-1}f_{1})
        \\&=(-1)^{\epsilon'}\sum\limits_{\sigma\in\mathfrak{S}_{n-1}}(-1)^{\delta^{\sigma}(f_{n-1},\dots,f_1)}s_{-1}\Phi_{g_{\MB}}^{\sigma}(f_{n-1},\dots,f_{1})
        \\&=-(-1)^{|s_{-1}f_i||g_{\MB}|}\ell^{n}(s_{-1}f_{n-1   },\dots,s_{-1}f_{i+1},g_{\MB},s_{-1}f_{i},\dots,s_{-1}f_{1})
    \end{split}
\end{equation}
where
\begin{small}
\begin{equation}
    \epsilon'=i+\sum\limits_{r=1}^{n-1}(r-1)|f_r|+(i-1)|g_{\MB}|+|g_{\MB}|\sum\limits_{r=i}^{n-1}|f_r|.
\end{equation}
\end{small}
We can prove in the same way that 
\begin{equation}
    \begin{split}
        &\ell^{n}(s_{-1}f_{n-1},\dots,s_{-1}f_{i+2},g_{\MB},s_{-1}f_{i+1},s_{-1}f_{i},\dots,s_{-1}f_{1})
        \\&=-(-1)^{|s_{-1}f_{i+1}||g_{\MB}|}\ell^{n}(s_{-1}f_{n-1},\dots,s_{-1}f_{i+1},g_{\MB},s_{-1}f_{i},\dots,s_{-1}f_{1}).
    \end{split}
\end{equation}
Thus, \eqref{eq:l-bracket-antisymm} holds when $l_i\in\mathfrak{g}_{\MB}$ for some $i\in\llbracket 1,n\rrbracket$ and $l_j\in\mathfrak{h}$ for $j\in\llbracket 1,n\rrbracket$ with $j\neq i$. The fact that the latter identity holds when $l_i\in\mathfrak{g}_{\MA}$ for some $i\in\llbracket 1,n\rrbracket$ and $l_j\in\mathfrak{h}$ for $j\in\llbracket 1,n\rrbracket$ with $j\neq i$ can be proved in a similar manner.

We now show that the $\infty$-bracket $\ell$ satisfies the higher Jacobi identities \eqref{eq:Jac-l}.
We first note that by Remark \ref{remark:lie-is-linfinity} and by definition of $\ell$ the identity is trivially satisfied in the following cases:
\begin{itemize}
    \item $l_i\in\mathfrak{h}$ for every $i\in\llbracket 1,n\rrbracket$;
    \item $n=2$ and $l_1,l_2\in\mathfrak{g}_{\MA}\cup\mathfrak{g}_{\MB}$;
    \item there exists a unique $i\in\llbracket 1,n\rrbracket$ such that $l_i\in\mathfrak{g}_{\MA}\cup\mathfrak{g}_{\MB}$;
    \item there exists three different integers $i,j,k\in\llbracket 1,n\rrbracket$ such that $l_i,l_j,l_k\in\mathfrak{g}_{\MA}\cup\mathfrak{g}_{\MB}$.
\end{itemize}
Therefore, we only have to show that \eqref{eq:Jac-l} holds when $n\geq 3$ and there exists $i,j\in\llbracket 1,n\rrbracket$ with $i\neq j$ and $l_i,l_j\in\mathfrak{g}_{\MA}\cup\mathfrak{g}_{\MB}$ and for every $k\in\llbracket 1,n\rrbracket$ with $k\neq i,j$, $l_k\in \mathfrak{h}$.

We have to show that
\begin{equation}
    \begin{split}
    &\sum\limits_{k=3}^n(-1)^{k}\hskip-3mm\sum\limits_{\substack{(\bar{i},\bar{j})\in\mathcal{P}_k^n\\i=i_p, j=j_q}}\hskip-2mm(-1)^{\Delta'} \ell^{n-k+1}\big(\ell^k(s_{-1}l_{i_k},\dots,l_{i_p},\dots,s_{-1}l_{i_1}),s_{-1}l_{j_{n-k}},\dots,l_{j_q},\dots,s_{-1}l_{j_{1}}\big)
    \\&-\sum\limits_{k=3}^n(-1)^{k}\hskip-4mm\sum\limits_{\substack{(\bar{i},\bar{j})\in\mathcal{P}_k^n\\i=j_q, j=i_p}}\hskip-3mm(-1)^{\Delta'_1 } \ell^{n-k+1}\big(\ell^k(s_{-1}l_{i_k},\dots,l_{i_p},\dots,s_{-1}l_{i_1}),s_{-1}l_{j_{n-k}},\dots,l_{j_q},\dots,s_{-1}l_{j_{1}}\big)
    \\&+(-1)^{\Delta\dprime} \ell^{n-1}\big([l_i,l_j]_{\nec},s_{-1}l_{n},\dots,\Hat{l_i},\dots,\Hat{l_j},\dots,s_{-1}l_{1}\big)
    =0
    \end{split}
\end{equation}
with
\begin{small}
\allowdisplaybreaks
\begin{align*}
    \Delta'&=|l_{i_p}|\sum\limits_{\substack{t=i_p+1\\t\neq j_q}}^n(|l_t|+1)+|l_{i_p}|\sum\limits_{\substack{t=p+1}}^k(|l_{i_t}|+1)+\sum\limits_{\substack{t=1\\t\neq p}}^k(|l_{i_t}|+1)\sum\limits_{\substack{w=i_t+1\\w\neq i_p,j_q}}^n(|l_w|+1)+|l_{j_q}|\sum_{\substack{t=1\\t\neq p\\i_t<j_q}}^k(|l_{i_t}|+1)
    \\&+\sum\limits_{\substack{t=1\\t\neq p}}^k(|l_{i_t}|+1)\sum\limits_{\substack{w=t+1\\w\neq p}}^k(|l_{i_w}|+1)+n-i_p+k-p
    +\sum\limits_{\substack{t=1\\t\neq p}}^k\sum\limits_{\substack{w=i_t+1\\w\neq i_p}}^n1+\sum\limits_{\substack{t=1\\ t\neq p}}^k \sum\limits_{\substack{w=t+1\\w\neq p}}^k 1+j_q-q,
    \\ \Delta'_1&=\Delta'+|l_{i_p}||l_{j_q}|,
    \\\Delta\dprime&=|l_i|\sum\limits_{t=i+1}^{n}(|l_t|+1)+|l_j|\sum\limits_{\substack{t=j+1\\t\neq i}}^{n}(|l_t|+1)+i+j+1.
\end{align*}
\end{small}
Suppose that $l_i,l_j\in\mathfrak{g}_{\MA}$ for some $i,j\in\llbracket 1,n\rrbracket$ with $i>j$ and $l_k\in \mathfrak{h}$ for every $k\in\llbracket 1,n\rrbracket$ with $k\neq i,j$.

We have that 
\begin{equation}
    \begin{split}
        &\ell^{n-1}\big([l_i,l_j]_{\nec},s_{-1}l_{n},\dots,\Hat{l_i},\dots,\Hat{l_j},\dots,s_{-1}l_{1}\big)
        \\&=(-1)^{\epsilon} \sum\limits_{\sigma\in\mathfrak{S}_{n-2}}(-1)^{\delta^{\sigma}(l_{n},\dots,\Hat{l_i},\dots,\Hat{l_j},\dots,l_{1})}s_{-1}\Psi^{\sigma}_{[l_i,l_j]_{\nec}}(l_{n},\dots,\Hat{l_i},\dots,\Hat{l_j},\dots,l_{1})
        \\&=(-1)^{\epsilon} \sum\limits_{\sigma\in\mathfrak{S}_{n-2}}(-1)^{\delta^{\sigma}(l_{n},\dots,\Hat{l_i},\dots,\Hat{l_j},\dots,l_{1})}s_{-1}\Psi^{\sigma}_{l_i\underset{\nec}{\circ}l_j}(l_{n},\dots,\Hat{l_i},\dots,\Hat{l_j},\dots,l_{1})
        \\&-(-1)^{|l_i||l_j|+\epsilon}\sum\limits_{\sigma\in\mathfrak{S}_{n-2}}(-1)^{\delta^{\sigma}(l_{n},\dots,\Hat{l_i},\dots,\Hat{l_j},\dots,l_{1})}s_{-1}\Psi^{\sigma}_{l_j\underset{\nec}{\circ}l_i}(l_{n},\dots,\Hat{l_i},\dots,\Hat{l_j},\dots,l_{1})  
    \end{split}
\end{equation}
with 
\begin{small}
\begin{equation}
    \begin{split}
        \epsilon=n+(|l_i|+|l_j|)\sum\limits_{\substack{t=1\\t\neq i,j}}^n(|l_t|+1)+\sum\limits_{t=1}^{j-1}(t-1)|l_t|+\sum\limits_{t=j+1}^{i-1}t|l_t|+\sum\limits_{t=i+1}^{n}(t-1)|l_t|
    \end{split}
\end{equation}
\end{small}
and where $\Psi^{\sigma}_{l_i\underset{\nec}{\circ}l_j}(l_{n},\dots,\Hat{l_i},\dots,\Hat{l_j},\dots,l_{1})=\hskip-2mm\sum\limits_{1\leq u<v\leq n-2}\hskip-2mm\big((-1)^{|l_i|\sum\limits_{t=1}^{u}|l_{\sigma_t}|}\sum\mathcal{E}(\mathcalboondox{D_j})+\sum\mathcal{E}(\mathcalboondox{D_i})\big)$ where the last sum is over all the filled diagrams $\mathcalboondox{D_j}$ and $\mathcalboondox{D_i}$ of the form

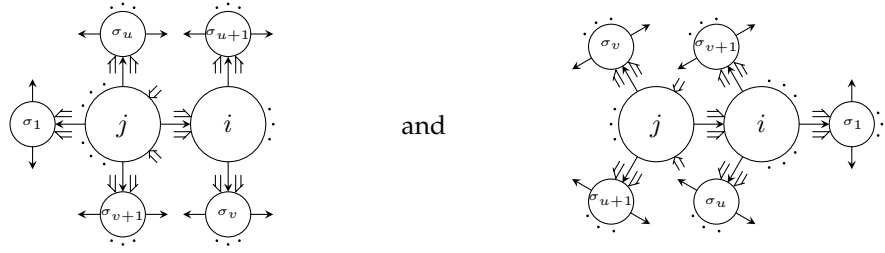
\begin{minipage}{21cm}
    \begin{tikzpicture}[line cap=round,line join=round,x=1.0cm,y=1.0cm]
\clip(-3,-2) rectangle (5,2);
  \draw (0,0) circle (0.5cm);
    \draw [rotate around={0:(0,0)}][->,>=stealth] (0.5,0)--(0.9,0);
    \draw [rotate around={-45:(0,0)},shift={(0.5,0)}]\doublefleche;
    \draw [rotate around={45:(0,0)},shift={(0.5,0)}]\doublefleche;
    \draw(1.4,0) circle (0.5cm);
    \draw [rotate around={180:(1.4,0)}, shift={(1.9,0)}] \doubleflechescindeeleft;
    \draw [rotate around={180:(1.4,0)}, shift={(1.9,0)}] \doubleflechescindeeright;
    \draw[rotate around={-90:(1.4,0)}][->,>=stealth](1.9,0)--(2.3,0);
    \draw (1.4,1.2) circle (0.3cm);
    \draw[rotate around={-90:(1.4,1.2)},shift={(1.7,1.2)}] \doubleflechescindeeleft;
    \draw[rotate around={-90:(1.4,1.2)},shift={(1.7,1.2)}] \doubleflechescindeeright;
    \draw[rotate around={0:(1.4,1.2)}][->,>=stealth](1.7,1.2)--(2,1.2);
    \draw[rotate around={180:(1.4,1.2)}][->,>=stealth](1.7,1.2)--(2,1.2);
    \draw[rotate around={90:(1.4,0)}][->,>=stealth](1.9,0)--(2.3,0);
    \draw (1.4,-1.2) circle (0.3cm);
    \draw[rotate around={90:(1.4,-1.2)},shift={(1.7,-1.2)}] \doubleflechescindeeleft;
    \draw[rotate around={90:(1.4,-1.2)},shift={(1.7,-1.2)}] \doubleflechescindeeright;
    \draw[rotate around={0:(1.4,-1.2)}][->,>=stealth](1.7,-1.2)--(2,-1.2);
    \draw[rotate around={180:(1.4,-1.2)}][->,>=stealth](1.7,-1.2)--(2,-1.2);
    \draw[rotate=-90][->,>=stealth](0.5,0)--(0.9,0);
     \draw (0,-1.2) circle (0.3cm);
    \draw[rotate around={90:(0,-1.2)},shift={(0.3,-1.2)}] \doubleflechescindeeleft;
    \draw[rotate around={90:(0,-1.2)},shift={(0.3,-1.2)}] \doubleflechescindeeright;
    \draw[rotate around={0:(0,-1.2)}][->,>=stealth](0.3,-1.2)--(0.6,-1.2);
    \draw[rotate around={180:(0,-1.2)}][->,>=stealth](0.3,-1.2)--(0.6,-1.2);
    \draw[rotate=90][->,>=stealth](0.5,0)--(0.9,0);
     \draw (0,1.2) circle (0.3cm);
    \draw[rotate around={-90:(0,1.2)},shift={(0.3,1.2)}] \doubleflechescindeeleft;
    \draw[rotate around={-90:(0,1.2)},shift={(0.3,1.2)}] \doubleflechescindeeright;
    \draw[rotate around={0:(0,1.2)}][->,>=stealth](0.3,1.2)--(0.6,1.2);
    \draw[rotate around={180:(0,1.2)}][->,>=stealth](0.3,1.2)--(0.6,1.2);
    \draw[rotate=180][->,>=stealth](0.5,0)--(0.9,0);
     \draw (-1.2,0) circle (0.3cm);
    \draw[rotate around={0:(-1.2,0)},shift={(-0.9,0)}] \doubleflechescindeeleft;
    \draw[rotate around={0:(-1.2,0)},shift={(-0.9,0)}] \doubleflechescindeeright;
    \draw[rotate around={90:(-1.2,0)}][->,>=stealth](-0.9,0)--(-0.6,0);
    \draw[rotate around={-90:(-1.2,0)}][->,>=stealth](-0.9,0)--(-0.6,0);
\draw[rotate=45][fill=black](-0.6,0) circle (0.3pt);
\draw[rotate=45][fill=black](-0.55,0.2) circle (0.3pt);
\draw[rotate=45][fill=black](-0.55,-0.2) circle (0.3pt);
\draw[rotate=-45][fill=black](-0.6,0) circle (0.3pt);
\draw[rotate=-45][fill=black](-0.55,0.2) circle (0.3pt);
\draw[rotate=-45][fill=black](-0.55,-0.2) circle (0.3pt);
\draw[fill=black](2,0) circle (0.3pt);
\draw[fill=black](1.95,0.2) circle (0.3pt);
\draw[fill=black](1.95,-0.2) circle (0.3pt);
\draw[fill=black](1.4,1.6) circle (0.3pt);
\draw[fill=black](1.25,1.55) circle (0.3pt);
\draw[fill=black](1.55,1.55) circle (0.3pt);
\draw [rotate around={180:(1.4,0)}][fill=black](1.4,1.6) circle (0.3pt);
\draw [rotate around={180:(1.4,0)}][fill=black](1.25,1.55) circle (0.3pt);
\draw [rotate around={180:(1.4,0)}][fill=black](1.55,1.55) circle (0.3pt);
\draw[fill=black](0,1.6) circle (0.3pt);
\draw[fill=black](0.15,1.55) circle (0.3pt);
\draw[fill=black](-0.15,1.55) circle (0.3pt);
\draw[rotate=180][fill=black](0,1.6) circle (0.3pt);
\draw[rotate=180][fill=black](0.15,1.55) circle (0.3pt);
\draw[rotate=180][fill=black](-0.15,1.55) circle (0.3pt);
\draw (0,0.25)node[anchor=north]{$j$};
\draw (1.4,0.25)node[anchor=north]{$i$};
\begin{scriptsize}
\draw (-1.2,0.15)node[anchor=north]{$\scriptstyle{\sigma_1}$};
\draw (0,1.35)node[anchor=north]{$\scriptstyle{\sigma_u}$};
\draw (0,-1.05)node[anchor=north]{$\scriptstyle{\sigma_{\scriptscriptstyle{v+1}}}$};
\draw (1.4,1.35)node[anchor=north]{$\scriptstyle{\sigma_{\scriptscriptstyle{u+1}}}$};
\draw (1.4,-1.05)node[anchor=north]{$\scriptstyle{\sigma_{v}}$};
\end{scriptsize}
\draw (4,0.25) node[anchor=north]{and};
\end{tikzpicture}
\begin{tikzpicture}[line cap=round,line join=round,x=1.0cm,y=1.0cm]
\clip(-2,-2) rectangle (7.242319607404366,2);
    \draw (0,0) circle (0.5cm);
    \draw [rotate around={0:(0,0)}][->,>=stealth] (0.5,0)--(0.9,0);
    \draw [rotate around={120:(0,0)}][->,>=stealth] (0.5,0)--(0.9,0);
    \draw [rotate around={-120:(0,0)}][->,>=stealth] (0.5,0)--(0.9,0);
    \draw [rotate around={-60:(0,0)},shift={(0.5,0)}]\doublefleche;
    \draw [rotate around={60:(0,0)},shift={(0.5,0)}]\doublefleche;
    \draw(1.4,0) circle (0.5cm);
    \draw [rotate around={180:(1.4,0)}, shift={(1.9,0)}] \doubleflechescindeeleft;
    \draw [rotate around={180:(1.4,0)}, shift={(1.9,0)}] \doubleflechescindeeright;
    \draw[rotate around={-120:(1.4,0)}][->,>=stealth](1.9,0)--(2.3,0);
    \draw (0.8,1.03) circle (0.3cm);
    \draw[rotate around={-60:(0.8,1.03)},shift={(1.1,1.03)}] \doubleflechescindeeleft;
    \draw[rotate around={-60:(0.8,1.03)},shift={(1.1,1.03)}] \doubleflechescindeeright;
    \draw[rotate around={30:(0.8,1.03)}][->,>=stealth](1.1,1.03)--(1.4,1.03);
    \draw[rotate around={210:(0.8,1.03)}][->,>=stealth](1.1,1.03)--(1.4,1.03);
     \draw[rotate around={120:(1.4,0)}][->,>=stealth](1.9,0)--(2.3,0);
    \draw (0.8,-1.03) circle (0.3cm);
    \draw[rotate around={60:(0.8,-1.03)},shift={(1.1,-1.03)}] \doubleflechescindeeleft;
    \draw[rotate around={60:(0.8,-1.03)},shift={(1.1,-1.03)}] \doubleflechescindeeright;
    \draw[rotate around={-30:(0.8,-1.03)}][->,>=stealth](1.1,-1.03)--(1.4,-1.03);
    \draw[rotate around={-210:(0.8,-1.03)}][->,>=stealth](1.1,-1.03)--(1.4,-1.03);
    \draw (-0.6,1.03) circle (0.3cm);
    \draw[rotate around={-60:(-0.6,1.03)},shift={(-0.3,1.03)}] \doubleflechescindeeleft;
    \draw[rotate around={-60:(-0.6,1.03)},shift={(-0.3,1.03)}] \doubleflechescindeeright;
    \draw[rotate around={30:(-0.6,1.03)}][->,>=stealth](-0.3,1.03)--(0,1.03);
    \draw[rotate around={210:(-0.6,1.03)}][->,>=stealth](-0.3,1.03)--(0,1.03);
    \draw (-0.6,-1.03) circle (0.3cm);
     \draw[rotate around={60:(-0.6,-1.03)},shift={(-0.3,-1.03)}] \doubleflechescindeeleft;
    \draw[rotate around={60:(-0.6,-1.03)},shift={(-0.3,-1.03)}] \doubleflechescindeeright;
    \draw[rotate around={-30:(-0.6,-1.03)}][->,>=stealth](-0.3,-1.03)--(0,-1.03);
    \draw[rotate around={-210:(-0.6,-1.03)}][->,>=stealth](-0.3,-1.03)--(0,-1.03);
    \draw[->,>=stealth](1.9,0)--(2.3,0);
    \draw(2.6,0) circle (0.3cm);
    \draw[rotate around={90:(2.6,0)}][->,>=stealth](2.9,0)--(3.2,0);
    \draw[rotate around={-90:(2.6,0)}][->,>=stealth](2.9,0)--(3.2,0);
    \draw[rotate around={180:(2.6,0)},shift={(2.9,0)}] \doubleflechescindeeleft;
    \draw[rotate around={180:(2.6,0)},shift={(2.9,0)}] \doubleflechescindeeright;
\draw[fill=black](-0.6,0) circle (0.3pt);
\draw[fill=black](-0.55,0.2) circle (0.3pt);
\draw[fill=black](-0.55,-0.2) circle (0.3pt);
\draw[fill=black](3,0) circle (0.3pt);
\draw[fill=black](2.95,0.15) circle (0.3pt);
\draw[fill=black](2.95,-0.15) circle (0.3pt);
\draw[rotate around={60:(1.4,0)}][fill=black](2,0) circle (0.3pt);
\draw[rotate around={60:(1.4,0)}][fill=black](1.95,0.2) circle (0.3pt);
\draw[rotate around={60:(1.4,0)}][fill=black](1.95,-0.2) circle (0.3pt);
\draw[rotate around={-60:(1.4,0)}][fill=black](2,0) circle (0.3pt);
\draw[rotate around={-60:(1.4,0)}][fill=black](1.95,0.2) circle (0.3pt);
\draw[rotate around={-60:(1.4,0)}][fill=black](1.95,-0.2) circle (0.3pt);
\draw[fill=black](-0.6,1.4) circle (0.3pt);
\draw[fill=black](-0.8,1.35) circle (0.3pt);
\draw[fill=black](-0.95,1.2) circle (0.3pt);
\draw[fill=black](0.45,1.2) circle (0.3pt);
\draw[fill=black](0.6,1.35) circle (0.3pt);
\draw[fill=black](0.8,1.4) circle (0.3pt);
\draw [rotate=120][fill=black](-0.6,1.4) circle (0.3pt);
\draw [rotate=120][fill=black](-0.8,1.35) circle (0.3pt);
\draw [rotate=120][fill=black](-0.95,1.2) circle (0.3pt);
\draw[rotate around={120:(1.4,0)}][fill=black](0.45,1.2) circle (0.3pt);
\draw[rotate around={120:(1.4,0)}][fill=black](0.6,1.35) circle (0.3pt);
\draw[rotate around={120:(1.4,0)}][fill=black](0.8,1.4) circle (0.3pt);
\draw (0,0.25)node[anchor=north]{$j$};
\draw (1.4,0.25)node[anchor=north]{$i$};
\begin{scriptsize}
\draw (2.62,0.15)node[anchor=north]{$\scriptstyle{\sigma_1}$};
\draw (-0.6,1.2)node[anchor=north]{$\scriptstyle{\sigma_v}$};
\draw (0.8,1.2)node[anchor=north]{$\scriptstyle{\sigma_{\scriptscriptstyle{v+1}}}$};
\draw (-0.6,-0.85)node[anchor=north]{$\scriptstyle{\sigma_{\scriptscriptstyle{u+1}}}$};
\draw (0.8,-0.9)node[anchor=north]{$\scriptstyle{\sigma_{u}}$};
\end{scriptsize}
\end{tikzpicture}
\end{minipage}

\noindent respectively, where the bold arrow is placed such that $l_{\sigma_1}$ is the first in the application order of the diagram and where discs filled with $l_{\sigma_k}$ by discs filled with $\sigma_k$ to simplify. Moreover, we have that 
\begin{equation}
    \begin{split}
   &\sum\limits_{k=3}^n(-1)^{k}\hskip-2mm\sum\limits_{\substack{(\bar{i},\bar{j})\in\mathcal{P}_k^n\\i=i_p, j=j_q}}(-1)^{\Delta'} \ell^{n-k+1}\big(\ell^k(s_{-1}l_{i_k},\dots,l_{i_p},\dots,s_{-1}l_{i_1}),s_{-1}l_{j_{n-k}},\dots,l_{j_q},\dots,s_{-1}l_{j_{1}}\big)
    \\&=\sum\limits_{k=3}^n(-1)^{k}\hskip-2mm\sum\limits_{\substack{(\bar{i},\bar{j})\in\mathcal{P}_k^n\\i=i_p, j=j_q}}\hskip-2mm
    (-1)^{\Delta'}\hskip-3mm\sum\limits_{\sigma'\in\mathfrak{S}_{k-1}} \hskip-3mm(-1)^{\epsilon_1+\delta^{\sigma'}(l_{i_k},\dots,\Hat{l_{i_p}},\dots,l_{i_1})} 
    \\&\hskip4cm\ell^{n-k+1}\big(s_{-1}\Psi^{\sigma'}_{l_{i_p}}(
    l_{i_k},\dots,\Hat{l_{i_p}},\dots,l_{i_1}),s_{-1}l_{j_{n-k}},\dots,l_{j_q},\dots,s_{-1}l_{j_{1}}\big)
    \end{split}
\end{equation}
with 
\begin{small}
\begin{equation}
    \begin{split}
        \epsilon_1&=p-1+|l_{i_p}|\sum\limits_{t=1}^{p-1}(|l_{i_t}|+1)+\sum\limits_{t=1}^{p-1} (t-1)|l_{i_t}|+\sum\limits_{t=p+1}^{k} t|l_{i_t}|.
    \end{split}
\end{equation}
\end{small}

Therefore, we have that 
\begin{equation}
    \begin{split}
   &\sum\limits_{k=3}^n(-1)^{k}\hskip-2mm\sum\limits_{\substack{(\bar{i},\bar{j})\in\mathcal{P}_k^n\\i=i_p, j=j_q}}(-1)^{\Delta'} \ell^{n-k+1}\big(\ell^k(s_{-1}l_{i_k},\dots,l_{i_p},\dots,s_{-1}l_{i_1}),s_{-1}l_{j_{n-k}},\dots,l_{j_q},\dots,s_{-1}l_{j_{1}}\big)
    \\&=\sum\limits_{k=3}^n(-1)^{k}\hskip-2mm\sum\limits_{\substack{(\bar{i},\bar{j})\in\mathcal{P}_k^n\\i=i_p, j=j_q}}\hskip-2mm(-1)^{\Delta'}\hskip-2mm\sum\limits_{\sigma'\in\mathfrak{S}_{k-1}} \hskip-2mm(-1)^{\epsilon_1+\delta^{\sigma'}(l_{i_k},\dots,\Hat{l_{i_p}},\dots,l_{i_1})} \hskip-3mm\sum\limits_{\sigma\dprime\in\mathfrak{S}_{n-k-1}}\hskip-4mm (-1)^{\epsilon_1'+\delta^{\sigma\dprime}(l_{j_{n-k}},\dots,\Hat{l_{j_q}},\dots,l_{j_{1}})}
    \\&\hskip6cms_{-1}\Psi_{l_{j_q}}^{\sigma\dprime}\big(\Psi_{l_{i_p}}^{\sigma'}(l_{i_k},\dots,\Hat{l_{i_p}},\dots,l_{i_1}),l_{j_{n-k}},\dots,\Hat{l_{j_q}},\dots,l_{j_{1}}\big)
    \end{split}
\end{equation}
with 
\begin{small}
\begin{equation}
    \begin{split}
    \epsilon_1'&=q-1+|l_{j_q}|\sum\limits_{t=1}^{q-1}(|l_{j_t}|+1)+\sum\limits_{t=1}^{q-1}(t-1)|l_{j_t}|+\sum\limits_{t=q+1}^{n-k}t|l_{j_t}|+(n-k-1)\sum\limits_{t=1}^k|l_{i_t}|
    \end{split}
\end{equation}
\end{small}
where
\allowdisplaybreaks
\begin{align*}
\Psi_{l_{j_q}}^{\sigma\dprime}\big(\Psi_{l_{i_p}}^{\sigma'}(l_{i_k},\dots,\Hat{l_{i_p}},\dots,l_{i_1}),l_{j_{n-k}},\dots,\Hat{l_{j_q}},\dots,l_{j_{1}}\big)&=\sum_{r=1}^{n-k-1}(-1)^{\epsilon_1\dprimeind}\sum\mathcal{E}(\mathcalboondox{D_j'})\\&\phantom{=}+\sum\limits_{s=1}^{k-1}(-1)^{\epsilon_1\tprime}\sum\mathcal{E}(\mathcalboondox{D_i'})
\end{align*}
with 
\begin{small}
\begin{equation}
    \begin{split}
        \epsilon_1\dprimeind=\sum\limits_{t=1}^{k}|l_{i_t}|\sum\limits_{t=r+1}^{n-k-1}|l_{j_{\sigma_t\dprimeind}}| \text{ , }\epsilon_1\tprime=(|l_{i_p}|+\sum\limits_{t=1}^{s-1}|l_{i_{\sigma'_t}}|)\sum\limits_{\substack{t=1\\t\neq q}}^{n-k}|l_{j_{t}}|
    \end{split}
\end{equation}
\end{small}
where the sums are over all the filled diagrams $\mathcalboondox{D_j'}$ and $\mathcalboondox{D_i'}$ of the form 

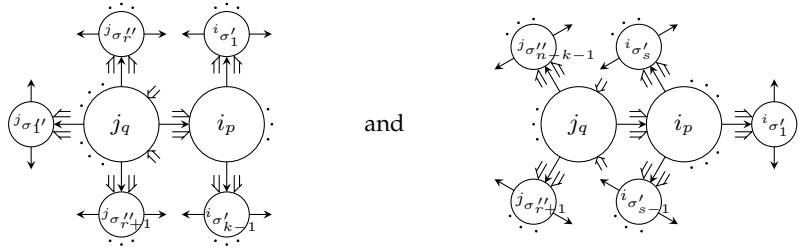
\begin{minipage}{21cm}
\begin{small}
\begin{tikzpicture}[line cap=round,line join=round,x=1.0cm,y=1.0cm]
\clip(-3.5,-2) rectangle (4,2);
 \draw (0,0) circle (0.5cm);
    \draw [rotate around={0:(0,0)}][->,>=stealth] (0.5,0)--(0.9,0);
    \draw [rotate around={-45:(0,0)},shift={(0.5,0)}]\doublefleche;
    \draw [rotate around={45:(0,0)},shift={(0.5,0)}]\doublefleche;
    \draw(1.4,0) circle (0.5cm);
    \draw [rotate around={180:(1.4,0)}, shift={(1.9,0)}] \doubleflechescindeeleft;
    \draw [rotate around={180:(1.4,0)}, shift={(1.9,0)}] \doubleflechescindeeright;
    \draw[rotate around={-90:(1.4,0)}][->,>=stealth](1.9,0)--(2.3,0);
    \draw (1.4,1.2) circle (0.3cm);
    \draw[rotate around={-90:(1.4,1.2)},shift={(1.7,1.2)}] \doubleflechescindeeleft;
    \draw[rotate around={-90:(1.4,1.2)},shift={(1.7,1.2)}] \doubleflechescindeeright;
    \draw[rotate around={0:(1.4,1.2)}][->,>=stealth](1.7,1.2)--(2,1.2);
    \draw[rotate around={180:(1.4,1.2)}][->,>=stealth](1.7,1.2)--(2,1.2);
    \draw[rotate around={90:(1.4,0)}][->,>=stealth](1.9,0)--(2.3,0);
    \draw (1.4,-1.2) circle (0.3cm);
    \draw[rotate around={90:(1.4,-1.2)},shift={(1.7,-1.2)}] \doubleflechescindeeleft;
    \draw[rotate around={90:(1.4,-1.2)},shift={(1.7,-1.2)}] \doubleflechescindeeright;
    \draw[rotate around={0:(1.4,-1.2)}][->,>=stealth](1.7,-1.2)--(2,-1.2);
    \draw[rotate around={180:(1.4,-1.2)}][->,>=stealth](1.7,-1.2)--(2,-1.2);
    \draw[rotate=-90][->,>=stealth](0.5,0)--(0.9,0);
     \draw (0,-1.2) circle (0.3cm);
    \draw[rotate around={90:(0,-1.2)},shift={(0.3,-1.2)}] \doubleflechescindeeleft;
    \draw[rotate around={90:(0,-1.2)},shift={(0.3,-1.2)}] \doubleflechescindeeright;
    \draw[rotate around={0:(0,-1.2)}][->,>=stealth](0.3,-1.2)--(0.6,-1.2);
    \draw[rotate around={180:(0,-1.2)}][->,>=stealth](0.3,-1.2)--(0.6,-1.2);
    \draw[rotate=90][->,>=stealth](0.5,0)--(0.9,0);
     \draw (0,1.2) circle (0.3cm);
    \draw[rotate around={-90:(0,1.2)},shift={(0.3,1.2)}] \doubleflechescindeeleft;
    \draw[rotate around={-90:(0,1.2)},shift={(0.3,1.2)}] \doubleflechescindeeright;
    \draw[rotate around={0:(0,1.2)}][->,>=stealth](0.3,1.2)--(0.6,1.2);
    \draw[rotate around={180:(0,1.2)}][->,>=stealth](0.3,1.2)--(0.6,1.2);
    \draw[rotate=180][->,>=stealth](0.5,0)--(0.9,0);
     \draw (-1.2,0) circle (0.3cm);
    \draw[rotate around={0:(-1.2,0)},shift={(-0.9,0)}] \doubleflechescindeeleft;
    \draw[rotate around={0:(-1.2,0)},shift={(-0.9,0)}] \doubleflechescindeeright;
    \draw[rotate around={90:(-1.2,0)}][->,>=stealth](-0.9,0)--(-0.6,0);
    \draw[rotate around={-90:(-1.2,0)}][->,>=stealth](-0.9,0)--(-0.6,0);
\draw[rotate=45][fill=black](-0.6,0) circle (0.3pt);
\draw[rotate=45][fill=black](-0.55,0.2) circle (0.3pt);
\draw[rotate=45][fill=black](-0.55,-0.2) circle (0.3pt);
\draw[rotate=-45][fill=black](-0.6,0) circle (0.3pt);
\draw[rotate=-45][fill=black](-0.55,0.2) circle (0.3pt);
\draw[rotate=-45][fill=black](-0.55,-0.2) circle (0.3pt);
\draw[fill=black](2,0) circle (0.3pt);
\draw[fill=black](1.95,0.2) circle (0.3pt);
\draw[fill=black](1.95,-0.2) circle (0.3pt);
\draw[fill=black](1.4,1.6) circle (0.3pt);
\draw[fill=black](1.25,1.55) circle (0.3pt);
\draw[fill=black](1.55,1.55) circle (0.3pt);
\draw [rotate around={180:(1.4,0)}][fill=black](1.4,1.6) circle (0.3pt);
\draw [rotate around={180:(1.4,0)}][fill=black](1.25,1.55) circle (0.3pt);
\draw [rotate around={180:(1.4,0)}][fill=black](1.55,1.55) circle (0.3pt);
\draw[fill=black](0,1.6) circle (0.3pt);
\draw[fill=black](0.15,1.55) circle (0.3pt);
\draw[fill=black](-0.15,1.55) circle (0.3pt);
\draw[rotate=180][fill=black](0,1.6) circle (0.3pt);
\draw[rotate=180][fill=black](0.15,1.55) circle (0.3pt);
\draw[rotate=180][fill=black](-0.15,1.55) circle (0.3pt);
\draw (0,0.25)node[anchor=north]{$j_q$};
\draw (1.4,0.25)node[anchor=north]{$i_p$};
\draw (-1.2,0.25)node[anchor=north]{$\scriptscriptstyle{j_{\sigma_1\dprimeind}}$};
\draw (0,1.45)node[anchor=north]{$\scriptscriptstyle{j_{\sigma_r\dprimeind}}$};
\draw (-0.03,-1)node[anchor=north]{$\scriptscriptstyle{j_{\sigma_{r+1}\dprimeindl}}$};
\draw (1.4,1.45)node[anchor=north]{$\scriptscriptstyle{i_{\sigma'_1}}$};
\draw (1.47,-1)node[anchor=north]{$\scriptscriptstyle{i_{\sigma'_{k-1}}}$};
\draw (3.5,0.25) node[anchor=north]{and};
\end{tikzpicture}
\begin{tikzpicture}[line cap=round,line join=round,x=1.0cm,y=1.0cm]
\clip(-2,-2) rectangle (7.242319607404366,2);
    \draw (0,0) circle (0.5cm);
    \draw [rotate around={0:(0,0)}][->,>=stealth] (0.5,0)--(0.9,0);
    \draw [rotate around={120:(0,0)}][->,>=stealth] (0.5,0)--(0.9,0);
    \draw [rotate around={-120:(0,0)}][->,>=stealth] (0.5,0)--(0.9,0);
    \draw [rotate around={-60:(0,0)},shift={(0.5,0)}]\doublefleche;
    \draw [rotate around={60:(0,0)},shift={(0.5,0)}]\doublefleche;
    \draw(1.4,0) circle (0.5cm);
    \draw [rotate around={180:(1.4,0)}, shift={(1.9,0)}] \doubleflechescindeeleft;
    \draw [rotate around={180:(1.4,0)}, shift={(1.9,0)}] \doubleflechescindeeright;
    \draw[rotate around={-120:(1.4,0)}][->,>=stealth](1.9,0)--(2.3,0);
    \draw (0.8,1.03) circle (0.3cm);
    \draw[rotate around={-60:(0.8,1.03)},shift={(1.1,1.03)}] \doubleflechescindeeleft;
    \draw[rotate around={-60:(0.8,1.03)},shift={(1.1,1.03)}] \doubleflechescindeeright;
    \draw[rotate around={30:(0.8,1.03)}][->,>=stealth](1.1,1.03)--(1.4,1.03);
    \draw[rotate around={210:(0.8,1.03)}][->,>=stealth](1.1,1.03)--(1.4,1.03);
     \draw[rotate around={120:(1.4,0)}][->,>=stealth](1.9,0)--(2.3,0);
    \draw (0.8,-1.03) circle (0.3cm);
    \draw[rotate around={60:(0.8,-1.03)},shift={(1.1,-1.03)}] \doubleflechescindeeleft;
    \draw[rotate around={60:(0.8,-1.03)},shift={(1.1,-1.03)}] \doubleflechescindeeright;
    \draw[rotate around={-30:(0.8,-1.03)}][->,>=stealth](1.1,-1.03)--(1.4,-1.03);
    \draw[rotate around={-210:(0.8,-1.03)}][->,>=stealth](1.1,-1.03)--(1.4,-1.03);
    \draw (-0.6,1.03) circle (0.3cm);
    \draw[rotate around={-60:(-0.6,1.03)},shift={(-0.3,1.03)}] \doubleflechescindeeleft;
    \draw[rotate around={-60:(-0.6,1.03)},shift={(-0.3,1.03)}] \doubleflechescindeeright;
    \draw[rotate around={30:(-0.6,1.03)}][->,>=stealth](-0.3,1.03)--(0,1.03);
    \draw[rotate around={210:(-0.6,1.03)}][->,>=stealth](-0.3,1.03)--(0,1.03);
    \draw (-0.6,-1.03) circle (0.3cm);
     \draw[rotate around={60:(-0.6,-1.03)},shift={(-0.3,-1.03)}] \doubleflechescindeeleft;
    \draw[rotate around={60:(-0.6,-1.03)},shift={(-0.3,-1.03)}] \doubleflechescindeeright;
    \draw[rotate around={-30:(-0.6,-1.03)}][->,>=stealth](-0.3,-1.03)--(0,-1.03);
    \draw[rotate around={-210:(-0.6,-1.03)}][->,>=stealth](-0.3,-1.03)--(0,-1.03);
    \draw[->,>=stealth](1.9,0)--(2.3,0);
    \draw(2.6,0) circle (0.3cm);
    \draw[rotate around={90:(2.6,0)}][->,>=stealth](2.9,0)--(3.2,0);
    \draw[rotate around={-90:(2.6,0)}][->,>=stealth](2.9,0)--(3.2,0);
    \draw[rotate around={180:(2.6,0)},shift={(2.9,0)}] \doubleflechescindeeleft;
    \draw[rotate around={180:(2.6,0)},shift={(2.9,0)}] \doubleflechescindeeright;
\draw[fill=black](-0.6,0) circle (0.3pt);
\draw[fill=black](-0.55,0.2) circle (0.3pt);
\draw[fill=black](-0.55,-0.2) circle (0.3pt);
\draw[fill=black](3,0) circle (0.3pt);
\draw[fill=black](2.95,0.15) circle (0.3pt);
\draw[fill=black](2.95,-0.15) circle (0.3pt);
\draw[rotate around={60:(1.4,0)}][fill=black](2,0) circle (0.3pt);
\draw[rotate around={60:(1.4,0)}][fill=black](1.95,0.2) circle (0.3pt);
\draw[rotate around={60:(1.4,0)}][fill=black](1.95,-0.2) circle (0.3pt);
\draw[rotate around={-60:(1.4,0)}][fill=black](2,0) circle (0.3pt);
\draw[rotate around={-60:(1.4,0)}][fill=black](1.95,0.2) circle (0.3pt);
\draw[rotate around={-60:(1.4,0)}][fill=black](1.95,-0.2) circle (0.3pt);
\draw[fill=black](-0.6,1.4) circle (0.3pt);
\draw[fill=black](-0.8,1.35) circle (0.3pt);
\draw[fill=black](-0.95,1.2) circle (0.3pt);
\draw[fill=black](0.45,1.2) circle (0.3pt);
\draw[fill=black](0.6,1.35) circle (0.3pt);
\draw[fill=black](0.8,1.4) circle (0.3pt);
\draw [rotate=120][fill=black](-0.6,1.4) circle (0.3pt);
\draw [rotate=120][fill=black](-0.8,1.35) circle (0.3pt);
\draw [rotate=120][fill=black](-0.95,1.2) circle (0.3pt);
\draw[rotate around={120:(1.4,0)}][fill=black](0.45,1.2) circle (0.3pt);
\draw[rotate around={120:(1.4,0)}][fill=black](0.6,1.35) circle (0.3pt);
\draw[rotate around={120:(1.4,0)}][fill=black](0.8,1.4) circle (0.3pt);
\draw (0,0.25)node[anchor=north]{$j_q$};
\draw (1.4,0.25)node[anchor=north]{$i_p$};
\draw (2.62,0.25)node[anchor=north]{$\scriptscriptstyle{i_{\sigma_1'}}$};
\draw (-0.6,1.25)node[anchor=north]{$\scriptscriptstyle{j_{\sigma_{n-k-1}\dprimeindtl}}$};
\draw (0.8,1.25)node[anchor=north]{$\scriptscriptstyle{i_{\sigma_{s}'}}$};
\draw (-0.6,-0.8)node[anchor=north]{$\scriptstyle{j_{\sigma_{r+1}\dprimeindl}}$};
\draw (0.9,-0.76)node[anchor=north]{$\scriptstyle{i_{\sigma_{s-1}'}}$};
\end{tikzpicture}
\end{small}
\end{minipage}

\noindent respectively, where the bold arrow on the filled diagrams $\mathcalboondox{D_j'}$ is placed such that the disc filled with $j_{\sigma_1\dprimeind}$ is the first after the disc filled with $j_q$ in the application order of diagrams $\mathcalboondox{D_j'}$ and the bold arrow on the filled diagrams $\mathcalboondox{D_i'}$ is placed such that the disc filled with $i_{\sigma'_1}$ is the first after discs filled with $i_p,j_q$ in the application order of diagrams $\mathcalboondox{D_i'}$. 

The term $\sum\mathcal{E}(\mathcalboondox{D_j'})$ in $\Psi_{l_{j_q}}^{\sigma\dprime}\big(\Psi_{l_{i_p}}^{\sigma'}(l_{i_1},\dots,\Hat{l_{i_p}},\dots,l_{i_k}),l_{j_1},\dots,\Hat{l_{j_q}},\dots,l_{j_{n-k}}\big)$ indexed by $r$ cancels with the evaluation $\mathcal{E}(\mathcalboondox{D_j})$ indexed by $u=r$, $v=u+k-1$ in $\Psi^{\sigma}_{l_i\underset{\nec}{\circ}l_j}(l_n,\dots,\Hat{l_i},\dots,\Hat{l_j},\dots,l_1)$ when $\sigma$ is the permutation given by 

\begin{equation}
    \begin{split}
\sigma_1=j_{\sigma_1\dprimeind},\dots,\sigma_{u}=j_{\sigma_r\dprimeind},\sigma_{u+1}=i_{\sigma'_1},\dots,\sigma_{v}=i_{\sigma'_{k-1}},\sigma_{v+1}=j_{\sigma_{r+1}\dprimeindl}\hskip2.2mm,\dots,\sigma_n=j_{\sigma_{n-k-1}\dprimeindtl}\hspace{0.7cm}.
    \end{split}
\end{equation}
Indeed, it is straightforward to check that 
\begin{small}
\begin{equation}
    \begin{split}   
    &\Delta\dprime+\epsilon+\delta^{\sigma}(l_{n},\dots,\Hat{l_i},\dots,\Hat{l_j},\dots,l_1)+|l_i|\sum\limits_{t=1}^u|l_{\sigma_t}|+k+\Delta'+\epsilon_1+\delta^{\sigma'}(l_{i_k},\dots,\Hat{l_{i_p}},\dots,l_{i_1})
    \\&+\epsilon'_1+\delta^{\sigma\dprime}(l_{j_{n-k}},\dots,\Hat{l_ {j_q}},\dots,l_{j_1})+\epsilon_1\dprimeind=1\mod 2.
    \end{split}
\end{equation}
\end{small}

On the other hand, the term $\sum\mathcal{E}(\mathcalboondox{D_i'})$ in \[\Psi_{l_{j_q}}^{\sigma\dprime}\big(\Psi_{l_{i_p}}^{\sigma'}(l_{i_1},\dots,\Hat{l_{i_p}},\dots,l_{i_k}),l_{j_1},\dots,\Hat{l_{j_q}},\dots,l_{j_{n-k}}\big)\] indexed by $s$ cancels with the evaluation $\mathcal{E}(\mathcalboondox{D_i})$ indexed by $u=s-1$, $v=n-k-1+u$ in $\Psi^{\sigma}_{l_i\underset{\nec}{\circ}l_j}(l_1,\dots,l_n)$ when $\sigma$ is the permutation given by 

\begin{equation}
    \begin{split}
    &\sigma_1=i_{\sigma'_1},\dots,\sigma_u=i_{\sigma'_{s-1}},\sigma_{u+1}=j_{\sigma_1\dprimeind},\dots,\sigma_v=j_{\sigma_{n-k-1}\dprimeindtl}\hskip6.2mm,\sigma_{v+1}=i_{\sigma'_s},\dots,\sigma_n=i_{\sigma'_{k-1}}.
    \end{split}
\end{equation}
Therefore, we have that
\begin{equation}
    \begin{split}
    &\sum\limits_{k=3}^n(-1)^{k}\hskip-3mm\sum\limits_{\substack{(\bar{i},\bar{j})\in\mathcal{P}_k^n\\i=i_p, j=j_q}}\hskip-2mm(-1)^{\Delta'} \ell^{n-k+1}\big(\ell^k(s_{-1}l_{i_k},\dots,l_{i_p},\dots,s_{-1}l_{i_1}),s_{-1}l_{j_{n-k}},\dots,l_{j_q},\dots,s_{-1}l_{j_{1}}\big)
    \\&+(-1)^{\Delta\dprime} \ell^{n-1}\big(l_i\underset{\nec}{\circ}l_j,s_{-1}l_{n},\dots,\Hat{l_i},\dots,\Hat{l_j},\dots,s_{-1}l_{1}\big)
    =0
    \end{split}
\end{equation}
and one can show in a similar manner that \vskip-1cm
\begin{equation}
    \begin{split}
\\&\sum\limits_{k=3}^n(-1)^{k}\hskip-3mm\sum\limits_{\substack{(\bar{i},\bar{j})\in\mathcal{P}_k^n\\i=j_q, j=i_p}}\hskip-2mm(-1)^{\Delta'} \ell^{n-k+1}\big(\ell^k(s_{-1}l_{i_k},\dots,l_{i_p},\dots,s_{-1}l_{i_1}),s_{-1}l_{j_{n-k}},\dots,l_{j_q},\dots,s_{-1}l_{j_{1}}\big)
    \\&-(-1)^{\Delta\dprime+|l_i||l_j|} \ell^{n-1}\big(l_j\underset{\nec}{\circ}l_i,s_{-1}l_{n},\dots,\Hat{l_i},\dots,\Hat{l_j},\dots,s_{-1}l_{1}\big)
    =0.
    \end{split}
\end{equation}


We now suppose that $l_i,l_j\in\mathfrak{g}_{\MB}$ for some $i,j\in\llbracket 1,n\rrbracket$ with $i>j$ and $l_k\in \mathfrak{h}$ for every $k\in\llbracket 1,n\rrbracket$ with $k\neq i,j$.

We have that 
\begin{equation}
    \begin{split}
        &\ell^{n-1}\big([l_i,l_j]_{\nec},s_{-1}l_{n},\dots,\Hat{l_i},\dots,\Hat{l_j},\dots,s_{-1}l_{1}\big)
        \\&=(-1)^{\zeta} \sum\limits_{\tau\in\mathfrak{S}_{n-2}}(-1)^{\delta^{\tau}(l_{n},\dots,\Hat{l_i},\dots,\Hat{l_j},\dots,l_{1})}s_{-1}\Phi^{\tau}_{[l_i,l_j]_{\nec}}(l_{n},\dots,\Hat{l_i},\dots,\Hat{l_j},\dots,l_{1})
        \\&=(-1)^{\zeta} \sum\limits_{\tau\in\mathfrak{S}_{n-2}}(-1)^{\delta^{\tau}(l_{n},\dots,\Hat{l_i},\dots,\Hat{l_j},\dots,l_{1})}s_{-1}\Phi^{\tau}_{l_i\underset{\nec}{\circ}l_j}(l_{n},\dots,\Hat{l_i},\dots,\Hat{l_j},\dots,l_{1})
        \\&-(-1)^{|l_i||l_j|+\zeta}\sum\limits_{\tau\in\mathfrak{S}_{n-2}}(-1)^{\delta^{\tau}(l_{n},\dots,\Hat{l_i},\dots,\Hat{l_j},\dots,l_{1})}s_{-1}\Phi^{\tau}_{l_j\underset{\nec}{\circ}l_i}(l_{n},\dots,\Hat{l_i},\dots,\Hat{l_j},\dots,l_{1})  
    \end{split}
\end{equation}
with 
\begin{small}
\begin{equation}
    \begin{split}
        \zeta=n+1+n(|l_i|+|l_j|)+\sum\limits_{t=1}^{j-1}(t-1)|l_t|+\sum\limits_{t=j+1}^{i-1}t|l_t|+\sum\limits_{t=i+1}^{n}(t-1)|l_t|
    \end{split}
\end{equation}
\end{small}
and where $\Phi^{\tau}_{l_j\underset{\nec}{\circ}l_i}(l_{n},\dots,\Hat{l_i},\dots,\Hat{l_j},\dots,l_{1})=\hskip-3mm\sum\limits_{1\leq u<v\leq n-2}\hskip-3mm\big((-1)^{|l_i|\sum\limits_{r=v+1}^{n-2}|l_{\tau_r}|}\sum\mathcal{E}(\mathcalboondox{D_j^1})+\sum\mathcal{E}(\mathcalboondox{D_i^1})\big)$
where the last sum is over all the filled diagrams $\mathcalboondox{D_j^1}$ and $\mathcalboondox{D_i^1}$ of the form

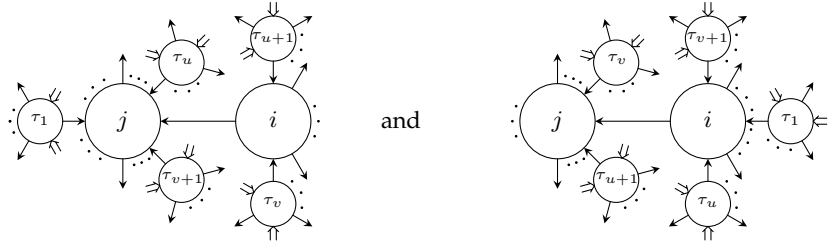
\begin{minipage}{21cm}
\begin{small}
    \begin{tikzpicture}[line cap=round,line join=round,x=1.0cm,y=1.0cm]
\clip(-3,-2) rectangle (4.2,2);
    \draw (0,0) circle (0.5cm);
    \draw [rotate=90][->,>=stealth] (0.5,0)--(0.9,0);
    \draw [rotate=-90][->,>=stealth] (0.5,0)--(0.9,0);
    \draw [rotate=45][<-,>=stealth] (0.5,0)--(0.8,0);
    \draw [rotate=-45][<-,>=stealth] (0.5,0)--(0.8,0);
    \draw (0.78,0.78) circle (0.3cm);
    \draw[rotate around={-15:(0.78,0.78)}][->,>=stealth] (1.08,0.78)--(1.38,0.78);
    \draw[rotate around={105:(0.78,0.78)}][->,>=stealth] (1.08,0.78)--(1.38,0.78);
    \draw[rotate around={45:(0.78,0.78)},shift={(1.08,0.78)}]\doublefleche;
    \draw[rotate around={165:(0.78,0.78)},shift={(1.08,0.78)}]\doublefleche;
     \draw (0.78,-0.78) circle (0.3cm);
    \draw[rotate around={15:(0.78,-0.78)}][->,>=stealth] (1.08,-0.78)--(1.38,-0.78);
    \draw[rotate around={-105:(0.78,-0.78)}][->,>=stealth] (1.08,-0.78)--(1.38,-0.78);
    \draw[rotate around={75:(0.78,-0.78)},shift={(1.08,-0.78)}]\doublefleche;
    \draw[rotate around={-165:(0.78,-0.78)},shift={(1.08,-0.78)}]\doublefleche;
    \draw (2,0) circle (0.5cm);
    \draw[rotate around={-60:(2,0)}][->,>=stealth](2.5,0)--(2.9,0);
    \draw[rotate around={60:(2,0)}][->,>=stealth](2.5,0)--(2.9,0);
    \draw[rotate around={180:(2,0)}][->,>=stealth](2.5,0)--(3.5,0);
    \draw[rotate around={-90:(2,0)}][<-,>=stealth](2.5,0)--(2.8,0);
    \draw (2,-1.1) circle (0.3cm);
    \draw[rotate around={-30:(2,-1.1)}][->,>=stealth](2.3,-1.1)--(2.6,-1.1);
    \draw[rotate around={-150:(2,-1.1)}][->,>=stealth](2.3,-1.1)--(2.6,-1.1);
    \draw[rotate around={-90:(2,-1.1)},shift={(2.3,-1.1)}]\doublefleche;
    \draw[rotate around={-210:(2,-1.1)},shift={(2.3,-1.1)}]\doublefleche;
    \draw[rotate around={90:(2,0)}][<-,>=stealth](2.5,0)--(2.8,0);
    \draw (2,1.1) circle (0.3cm);
    \draw[rotate around={30:(2,1.1)}][->,>=stealth](2.3,1.1)--(2.6,1.1);
    \draw[rotate around={150:(2,1.1)}][->,>=stealth](2.3,1.1)--(2.6,1.1);
    \draw[rotate around={90:(2,1.1)},shift={(2.3,1.1)}]\doublefleche;
    \draw[rotate around={210:(2,1.1)},shift={(2.3,1.1)}]\doublefleche;
    \draw[rotate=180][<-,>=stealth](0.5,0)--(0.8,0);
    \draw(-1.1,0) circle (0.3cm);
    \draw[rotate around={120:(-1.1,0)}][->,>=stealth](-0.8,0)--(-0.5,0);
    \draw[rotate around={-120:(-1.1,0)}][->,>=stealth](-0.8,0)--(-0.5,0);
    \draw[rotate around={60:(-1.1,0)},shift={(-0.8,0)}]\doublefleche;
    \draw[rotate around={-60:(-1.1,0)},shift={(-0.8,0)}]\doublefleche;
\draw (0,0.25)node[anchor=north]{$j$};
\draw (2,0.25)node[anchor=north]{$i$};
\draw (-1.1,0.2)node[anchor=north]{$\scriptstyle{\tau_1}$};
\draw (2,1.3)node[anchor=north]{$\scriptstyle{\tau_{u+1}}$};
\draw (0.8,1)node[anchor=north]{$\scriptstyle{\tau_u}$};
\draw (2,-0.9)node[anchor=north]{$\scriptstyle{\tau_{v}}$};
\draw (0.8,-0.57)node[anchor=north]{$\scriptstyle{\tau_{v+1}}$};
\draw (3.7,0.25) node[anchor=north]{and};
\draw[rotate=45][fill=black] (0.5,0.35) circle (0.3pt);
\draw[rotate=45][fill=black] (0.55,0.25) circle (0.3pt);
\draw[rotate=45][fill=black] (0.58,0.15) circle (0.3pt);
\draw[rotate=-90][fill=black] (0.5,0.35) circle (0.3pt);
\draw[rotate=-90][fill=black] (0.55,0.25) circle (0.3pt);
\draw[rotate=-90][fill=black] (0.58,0.15) circle (0.3pt);
\draw[fill=black] (-0.4,0.45) circle (0.3pt);
\draw[fill=black] (-0.5,0.3) circle (0.3pt);
\draw[fill=black] (-0.25,0.55) circle (0.3pt);
\draw[rotate=90][fill=black] (-0.4,0.45) circle (0.3pt);
\draw[rotate=90][fill=black] (-0.5,0.3) circle (0.3pt);
\draw[rotate=90][fill=black] (-0.25,0.55) circle (0.3pt);
\draw[fill=black] (-1.5,0) circle (0.3pt);
\draw[fill=black] (-1.45,0.15) circle (0.3pt);
\draw[fill=black] (-1.45,-0.15) circle (0.3pt);
\draw[fill=black] (-1.5,0) circle (0.3pt);
\draw[fill=black] (0.7,0.4) circle (0.3pt);
\draw[fill=black] (0.88,0.4) circle (0.3pt);
\draw[fill=black] (1,0.5) circle (0.3pt);
\draw[fill=black] (1.05,-1.05) circle (0.3pt);
\draw[fill=black] (1.15,-0.9) circle (0.3pt);
\draw[fill=black] (0.9,-1.15) circle (0.3pt);
\draw[fill=black] (2.55,0.2) circle (0.3pt);
\draw[fill=black] (2.6,0) circle (0.3pt);
\draw[fill=black] (2.55,-0.2) circle (0.3pt);
\draw[fill=black] (2.25,-0.8) circle (0.3pt);
\draw[fill=black] (2.37,-0.95) circle (0.3pt);
\draw[fill=black] (2.38,-1.15) circle (0.3pt);
\draw[fill=black] (2.25,0.8) circle (0.3pt);
\draw[fill=black] (2.37,0.95) circle (0.3pt);
\draw[fill=black] (2.38,1.15) circle (0.3pt);
\end{tikzpicture}
\begin{tikzpicture}[line cap=round,line join=round,x=1.0cm,y=1.0cm]
\clip(-1.5,-2) rectangle (7.880223717617432,2);
    \draw (0,0) circle (0.5cm);
    \draw [rotate=90][->,>=stealth] (0.5,0)--(0.9,0);
    \draw [rotate=-90][->,>=stealth] (0.5,0)--(0.9,0);
    \draw [rotate=45][<-,>=stealth] (0.5,0)--(0.8,0);
    \draw [rotate=-45][<-,>=stealth] (0.5,0)--(0.8,0);
    \draw (0.78,0.78) circle (0.3cm);
    \draw[rotate around={-15:(0.78,0.78)}][->,>=stealth] (1.08,0.78)--(1.38,0.78);
    \draw[rotate around={105:(0.78,0.78)}][->,>=stealth] (1.08,0.78)--(1.38,0.78);
    \draw[rotate around={45:(0.78,0.78)},shift={(1.08,0.78)}]\doublefleche;
    \draw[rotate around={165:(0.78,0.78)},shift={(1.08,0.78)}]\doublefleche;
     \draw (0.78,-0.78) circle (0.3cm);
    \draw[rotate around={15:(0.78,-0.78)}][->,>=stealth] (1.08,-0.78)--(1.38,-0.78);
    \draw[rotate around={-105:(0.78,-0.78)}][->,>=stealth] (1.08,-0.78)--(1.38,-0.78);
    \draw[rotate around={75:(0.78,-0.78)},shift={(1.08,-0.78)}]\doublefleche;
    \draw[rotate around={-165:(0.78,-0.78)},shift={(1.08,-0.78)}]\doublefleche;
    \draw (2,0) circle (0.5cm);
    \draw[rotate around={-60:(2,0)}][->,>=stealth](2.5,0)--(2.9,0);
    \draw[rotate around={60:(2,0)}][->,>=stealth](2.5,0)--(2.9,0);
    \draw[rotate around={180:(2,0)}][->,>=stealth](2.5,0)--(3.5,0);
    \draw[rotate around={-90:(2,0)}][<-,>=stealth](2.5,0)--(2.8,0);
    \draw (2,-1.1) circle (0.3cm);
    \draw[rotate around={-30:(2,-1.1)}][->,>=stealth](2.3,-1.1)--(2.6,-1.1);
    \draw[rotate around={-150:(2,-1.1)}][->,>=stealth](2.3,-1.1)--(2.6,-1.1);
    \draw[rotate around={-90:(2,-1.1)},shift={(2.3,-1.1)}]\doublefleche;
    \draw[rotate around={-210:(2,-1.1)},shift={(2.3,-1.1)}]\doublefleche;
    \draw[rotate around={90:(2,0)}][<-,>=stealth](2.5,0)--(2.8,0);
    \draw (2,1.1) circle (0.3cm);
    \draw[rotate around={30:(2,1.1)}][->,>=stealth](2.3,1.1)--(2.6,1.1);
    \draw[rotate around={150:(2,1.1)}][->,>=stealth](2.3,1.1)--(2.6,1.1);
    \draw[rotate around={90:(2,1.1)},shift={(2.3,1.1)}]\doublefleche;
    \draw[rotate around={210:(2,1.1)},shift={(2.3,1.1)}]\doublefleche;
    \draw[rotate around={0:(2,0)}][<-,>=stealth](2.5,0)--(2.8,0);
    \draw (3.1,0) circle (0.3cm);
    \draw[rotate around={60:(3.1,0)}][->,>=stealth](3.4,0)--(3.7,0);
    \draw[rotate around={-60:(3.1,0)}][->,>=stealth](3.4,0)--(3.7,0);
    \draw[rotate around={0:(3.1,0)},shift={(3.4,0)}]\doublefleche;
    \draw[rotate around={120:(3.1,0)},shift={(3.4,0)}]\doublefleche;
\draw (0,0.25)node[anchor=north]{$j$};
\draw (2,0.25)node[anchor=north]{$i$};
\draw (3.1,0.2)node[anchor=north]{$\scriptstyle{\tau_1}$};
\draw (2,1.3)node[anchor=north]{$\scriptstyle{\tau_{v+1}}$};
\draw (0.8,1)node[anchor=north]{$\scriptstyle{\tau_v}$};
\draw (2,-0.9)node[anchor=north]{$\scriptstyle{\tau_{u}}$};
\draw (0.8,-0.57)node[anchor=north]{$\scriptstyle{\tau_{u+1}}$};
\draw[rotate=45][fill=black] (0.5,0.35) circle (0.3pt);
\draw[rotate=45][fill=black] (0.55,0.25) circle (0.3pt);
\draw[rotate=45][fill=black] (0.58,0.15) circle (0.3pt);
\draw[rotate=-90][fill=black] (0.5,0.35) circle (0.3pt);
\draw[rotate=-90][fill=black] (0.55,0.25) circle (0.3pt);
\draw[rotate=-90][fill=black] (0.58,0.15) circle (0.3pt);
\draw[rotate around={-100:(2,0)}][fill=black] (1.6,0.45) circle (0.3pt);
\draw[rotate around={-100:(2,0)}][fill=black] (1.5,0.3) circle (0.3pt);
\draw[rotate around={-100:(2,0)}][fill=black] (1.75,0.55) circle (0.3pt);
\draw[rotate around={-160:(2,0)}][fill=black] (1.6,0.45) circle (0.3pt);
\draw[rotate around={-160:(2,0)}][fill=black] (1.5,0.3) circle (0.3pt);
\draw[rotate around={-160:(2,0)}][fill=black] (1.75,0.55) circle (0.3pt);
\draw[fill=black] (0.7,0.4) circle (0.3pt);
\draw[fill=black] (0.88,0.4) circle (0.3pt);
\draw[fill=black] (1,0.5) circle (0.3pt);
\draw[fill=black] (1.05,-1.05) circle (0.3pt);
\draw[fill=black] (1.15,-0.9) circle (0.3pt);
\draw[fill=black] (0.9,-1.15) circle (0.3pt);
\draw[fill=black] (-0.55,0.2) circle (0.3pt);
\draw[fill=black] (-0.6,0) circle (0.3pt);
\draw[fill=black] (-0.55,-0.2) circle (0.3pt);
\draw[fill=black] (2.25,-0.8) circle (0.3pt);
\draw[fill=black] (2.37,-0.95) circle (0.3pt);
\draw[fill=black] (2.38,-1.15) circle (0.3pt);
\draw[fill=black] (2.25,0.8) circle (0.3pt);
\draw[fill=black] (2.37,0.95) circle (0.3pt);
\draw[fill=black] (2.38,1.15) circle (0.3pt);
\draw[fill=black] (2.8,-0.25) circle (0.3pt);
\draw[fill=black] (2.93,-0.35) circle (0.3pt);
\draw[fill=black] (3.1,-0.38) circle (0.3pt);
\end{tikzpicture}
\end{small}
\end{minipage}

\noindent respectively, where the bold arrow is placed such that the disc filled with $l_{\tau_1}$ is the first in the application order. Moreover, we have that 
\begin{equation}
    \begin{split}
   &\sum\limits_{k=3}^n(-1)^{k}\hskip-3mm\sum\limits_{\substack{(\bar{i},\bar{j})\in\mathcal{P}_k^n\\i=i_p, j=j_q}}\hskip-2mm(-1)^{\Delta'} \ell^{n-k+1}\big(\ell^k(s_{-1}l_{i_k},\dots,l_{i_p},\dots,s_{-1}l_{i_1}),s_{-1}l_{j_{n-k}},\dots,l_{j_q},\dots,s_{-1}l_{j_{1}}\big)
    \\&=\sum\limits_{k=3}^n(-1)^{k}\hskip-3mm\sum\limits_{\substack{(\bar{i},\bar{j})\in\mathcal{P}_k^n\\i=i_p, j=j_q}}
    \hskip-2mm(-1)^{\Delta'}\sum\limits_{\tau'\in\mathfrak{S}_{k-1}} (-1)^{\zeta_1+\delta^{\tau'}(l_{i_k},\dots,\Hat{l_{i_p}},\dots,l_{i_1})} 
    \\&\hspace{5cm}\ell^{n-k+1}\big(s_{-1}\Phi^{\tau'}_{l_{i_p}}(
    l_{i_k},\dots,\Hat{l_{i_p}},\dots,l_{i_1}),s_{-1}l_{j_{n-k}},\dots,l_{j_q},\dots,s_{-1}l_{j_{1}}\big)
    \end{split}
\end{equation}
with 
\begin{small}
\begin{equation}
    \begin{split}
        \zeta_1&=p+(p-1)|l_{i_p}|+|l_{i_p}|\sum\limits_{t=p+1}^{k}|l_{i_t}|+\sum\limits_{t=1}^{p-1} (t-1)|l_{i_t}|+\sum\limits_{t=p+1}^{k} t|l_{i_t}|.
    \end{split}
\end{equation}
\end{small}
Therefore, we have that 
\begin{equation}
    \begin{split}
   &\sum\limits_{k=3}^n(-1)^{k}\hskip-3mm\sum\limits_{\substack{(\bar{i},\bar{j})\in\mathcal{P}_k^n\\i=i_p, j=j_q}}\hskip-2mm(-1)^{\Delta'} \ell^{n-k+1}\big(\ell^k(s_{-1}l_{i_k},\dots,l_{i_p},\dots,s_{-1}l_{i_1}),s_{-1}l_{j_{n-k}},\dots,l_{j_q},\dots,s_{-1}l_{j_{1}}\big)
    \\&=\sum\limits_{k=3}^n(-1)^{k}\hskip-3mm\sum\limits_{\substack{(\bar{i},\bar{j})\in\mathcal{P}_k^n\\i=i_p, j=j_q}}\hskip-2mm(-1)^{\Delta'}\sum\limits_{\tau'\in\mathfrak{S}_{k-1}} (-1)^{\zeta_1+\delta^{\tau'}(l_{i_k},\dots,\Hat{l_{i_p}},\dots,l_{i_1})} \hskip-2mm\sum\limits_{\tau\dprime\in\mathfrak{S}_{n-k-1}} (-1)^{\zeta_1'+\delta^{\tau\dprime}(l_{j_{n-k}},\dots,\Hat{l_{j_q}},\dots,l_{j_{1}})}
    \\&\hspace{7cm}s_{-1}\Phi_{l_{j_q}}^{\tau\dprime}\big(\Phi_{l_{i_p}}^{\tau'}(l_{i_k},\dots,\Hat{l_{i_p}},\dots,l_{i_1}),l_{j_{n-k}},\dots,\Hat{l_{j_q}},\dots,l_{j_{1}}\big)
    \end{split}
\end{equation}
with 
\begin{small}
\begin{equation}
    \begin{split}
    \zeta_1'&=q+(q-1)|l_{j_q}|+|l_{j_q}|\big(\sum\limits_{t=q+1}^{n-k}|l_{j_t}|+\sum\limits_{t=1}^k|l_{i_t}|\big)+\sum\limits_{t=1}^{q-1}(t-1)|l_{j_t}|
    +\sum\limits_{t=q+1}^{n-k}t|l_{j_t}|+(n-k-1)\sum\limits_{t=1}^k|l_{i_t}|
    \end{split}
\end{equation}
\end{small}
\noindent where
\begin{equation}
\begin{split}
\Phi_{l_{j_q}}^{\tau\dprime}\big(\Phi_{l_{i_p}}^{\tau'}(l_{i_k},\dots,\Hat{l_{i_p}},\dots,l_{i_1}),l_{j_{n-k}},\dots,\Hat{l_{j_q}},\dots,l_{j_{1}}\big)&=\sum_{r=1}^{n-k-2}(-1)^{\zeta_1\dprimeind}\sum\mathcal{E}(\mathcalboondox{D_j^2})\\&\phantom{=}+\sum\limits_{s=2}^n(-1)^{\zeta_1\tprime}\sum\mathcal{E}(\mathcalboondox{D_i^2})
\end{split}
\end{equation}
with 
\begin{small}
\begin{equation}
    \begin{split}
        \zeta_1\dprimeind=\sum\limits_{t=1}^{k}|l_{i_t}|\sum\limits_{t=r+1}^{n-k-1}|l_{j_{\tau_t\dprimeind}}| \text{ , }
        \zeta_1\tprime&=\sum\limits_{t=1}^{s-1}|l_{i_t}|\sum\limits_{t=1}^{n-k-1}|l_{j_{\sigma_t\dprimeind}}|
    \end{split}
\end{equation}
\end{small}
where the sums are over all the filled diagrams $\mathcalboondox{D_j^2}$ and $\mathcalboondox{D_i^2}$ of the form 

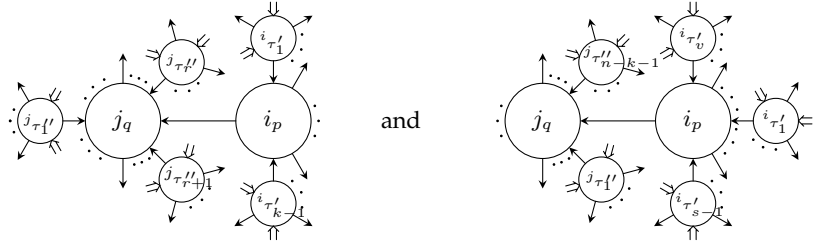
\begin{minipage}{21cm}
\begin{small}
\begin{tikzpicture}[line cap=round,line join=round,x=1.0cm,y=1.0cm]
\clip(-4,-2) rectangle (4.5,2);
    \draw (0,0) circle (0.5cm);
    \draw [rotate=90][->,>=stealth] (0.5,0)--(0.9,0);
    \draw [rotate=-90][->,>=stealth] (0.5,0)--(0.9,0);
    \draw [rotate=45][<-,>=stealth] (0.5,0)--(0.8,0);
    \draw [rotate=-45][<-,>=stealth] (0.5,0)--(0.8,0);
    \draw (0.78,0.78) circle (0.3cm);
    \draw[rotate around={-15:(0.78,0.78)}][->,>=stealth] (1.08,0.78)--(1.38,0.78);
    \draw[rotate around={105:(0.78,0.78)}][->,>=stealth] (1.08,0.78)--(1.38,0.78);
    \draw[rotate around={45:(0.78,0.78)},shift={(1.08,0.78)}]\doublefleche;
    \draw[rotate around={165:(0.78,0.78)},shift={(1.08,0.78)}]\doublefleche;
     \draw (0.78,-0.78) circle (0.3cm);
    \draw[rotate around={15:(0.78,-0.78)}][->,>=stealth] (1.08,-0.78)--(1.38,-0.78);
    \draw[rotate around={-105:(0.78,-0.78)}][->,>=stealth] (1.08,-0.78)--(1.38,-0.78);
    \draw[rotate around={75:(0.78,-0.78)},shift={(1.08,-0.78)}]\doublefleche;
    \draw[rotate around={-165:(0.78,-0.78)},shift={(1.08,-0.78)}]\doublefleche;
    \draw (2,0) circle (0.5cm);
    \draw[rotate around={-60:(2,0)}][->,>=stealth](2.5,0)--(2.9,0);
    \draw[rotate around={60:(2,0)}][->,>=stealth](2.5,0)--(2.9,0);
    \draw[rotate around={180:(2,0)}][->,>=stealth](2.5,0)--(3.5,0);
    \draw[rotate around={-90:(2,0)}][<-,>=stealth](2.5,0)--(2.8,0);
    \draw (2,-1.1) circle (0.3cm);
    \draw[rotate around={-30:(2,-1.1)}][->,>=stealth](2.3,-1.1)--(2.6,-1.1);
    \draw[rotate around={-150:(2,-1.1)}][->,>=stealth](2.3,-1.1)--(2.6,-1.1);
    \draw[rotate around={-90:(2,-1.1)},shift={(2.3,-1.1)}]\doublefleche;
    \draw[rotate around={-210:(2,-1.1)},shift={(2.3,-1.1)}]\doublefleche;
    \draw[rotate around={90:(2,0)}][<-,>=stealth](2.5,0)--(2.8,0);
    \draw (2,1.1) circle (0.3cm);
    \draw[rotate around={30:(2,1.1)}][->,>=stealth](2.3,1.1)--(2.6,1.1);
    \draw[rotate around={150:(2,1.1)}][->,>=stealth](2.3,1.1)--(2.6,1.1);
    \draw[rotate around={90:(2,1.1)},shift={(2.3,1.1)}]\doublefleche;
    \draw[rotate around={210:(2,1.1)},shift={(2.3,1.1)}]\doublefleche;
    \draw[rotate=180][<-,>=stealth](0.5,0)--(0.8,0);
    \draw(-1.1,0) circle (0.3cm);
    \draw[rotate around={120:(-1.1,0)}][->,>=stealth](-0.8,0)--(-0.5,0);
    \draw[rotate around={-120:(-1.1,0)}][->,>=stealth](-0.8,0)--(-0.5,0);
    \draw[rotate around={60:(-1.1,0)},shift={(-0.8,0)}]\doublefleche;
    \draw[rotate around={-60:(-1.1,0)},shift={(-0.8,0)}]\doublefleche;
\draw (0,0.25)node[anchor=north]{$j_q$};
\draw (2,0.25)node[anchor=north]{$i_p$};
\draw (-1.1,0.2)node[anchor=north]{$\scriptscriptstyle{j_{\tau_1\dprimeind}}$};
\draw (2,1.3)node[anchor=north]{$\scriptscriptstyle{i_{\tau_{1}'}}$};
\draw (0.8,1)node[anchor=north]{$\scriptscriptstyle{j_{\tau_{r}\dprimeind}}$};
\draw (2.12,-0.85)node[anchor=north]{$\scriptscriptstyle{i_{\tau_{k-1}'}}$};
\draw (0.78,-0.5)node[anchor=north]{$\scriptstyle{j_{\tau_{r+1}\dprimeindl}}$};
\draw[rotate=45][fill=black] (0.5,0.35) circle (0.3pt);
\draw[rotate=45][fill=black] (0.55,0.25) circle (0.3pt);
\draw[rotate=45][fill=black] (0.58,0.15) circle (0.3pt);
\draw[rotate=-90][fill=black] (0.5,0.35) circle (0.3pt);
\draw[rotate=-90][fill=black] (0.55,0.25) circle (0.3pt);
\draw[rotate=-90][fill=black] (0.58,0.15) circle (0.3pt);
\draw[fill=black] (-0.4,0.45) circle (0.3pt);
\draw[fill=black] (-0.5,0.3) circle (0.3pt);
\draw[fill=black] (-0.25,0.55) circle (0.3pt);
\draw[rotate=90][fill=black] (-0.4,0.45) circle (0.3pt);
\draw[rotate=90][fill=black] (-0.5,0.3) circle (0.3pt);
\draw[rotate=90][fill=black] (-0.25,0.55) circle (0.3pt);
\draw[fill=black] (-1.5,0) circle (0.3pt);
\draw[fill=black] (-1.45,0.15) circle (0.3pt);
\draw[fill=black] (-1.45,-0.15) circle (0.3pt);
\draw[fill=black] (-1.5,0) circle (0.3pt);
\draw[fill=black] (0.7,0.4) circle (0.3pt);
\draw[fill=black] (0.88,0.4) circle (0.3pt);
\draw[fill=black] (1,0.5) circle (0.3pt);
\draw[fill=black] (1.05,-1.05) circle (0.3pt);
\draw[fill=black] (1.15,-0.9) circle (0.3pt);
\draw[fill=black] (0.9,-1.15) circle (0.3pt);
\draw[fill=black] (2.55,0.2) circle (0.3pt);
\draw[fill=black] (2.6,0) circle (0.3pt);
\draw[fill=black] (2.55,-0.2) circle (0.3pt);
\draw[fill=black] (2.25,-0.8) circle (0.3pt);
\draw[fill=black] (2.37,-0.95) circle (0.3pt);
\draw[fill=black] (2.38,-1.15) circle (0.3pt);
\draw[fill=black] (2.25,0.8) circle (0.3pt);
\draw[fill=black] (2.37,0.95) circle (0.3pt);
\draw[fill=black] (2.38,1.15) circle (0.3pt);
\draw (3.7,0.25) node[anchor=north]{and};
\end{tikzpicture}
    \begin{tikzpicture}[line cap=round,line join=round,x=1.0cm,y=1.0cm]
\clip(-1,-2) rectangle (4.2,2);
    \draw (0,0) circle (0.5cm);
    \draw [rotate=90][->,>=stealth] (0.5,0)--(0.9,0);
    \draw [rotate=-90][->,>=stealth] (0.5,0)--(0.9,0);
    \draw [rotate=45][<-,>=stealth] (0.5,0)--(0.8,0);
    \draw [rotate=-45][<-,>=stealth] (0.5,0)--(0.8,0);
    \draw (0.78,0.78) circle (0.3cm);
    \draw[rotate around={-15:(0.78,0.78)}][->,>=stealth] (1.08,0.78)--(1.38,0.78);
    \draw[rotate around={105:(0.78,0.78)}][->,>=stealth] (1.08,0.78)--(1.38,0.78);
    \draw[rotate around={45:(0.78,0.78)},shift={(1.08,0.78)}]\doublefleche;
    \draw[rotate around={165:(0.78,0.78)},shift={(1.08,0.78)}]\doublefleche;
     \draw (0.78,-0.78) circle (0.3cm);
    \draw[rotate around={15:(0.78,-0.78)}][->,>=stealth] (1.08,-0.78)--(1.38,-0.78);
    \draw[rotate around={-105:(0.78,-0.78)}][->,>=stealth] (1.08,-0.78)--(1.38,-0.78);
    \draw[rotate around={75:(0.78,-0.78)},shift={(1.08,-0.78)}]\doublefleche;
    \draw[rotate around={-165:(0.78,-0.78)},shift={(1.08,-0.78)}]\doublefleche;
    \draw (2,0) circle (0.5cm);
    \draw[rotate around={-60:(2,0)}][->,>=stealth](2.5,0)--(2.9,0);
    \draw[rotate around={60:(2,0)}][->,>=stealth](2.5,0)--(2.9,0);
    \draw[rotate around={180:(2,0)}][->,>=stealth](2.5,0)--(3.5,0);
    \draw[rotate around={-90:(2,0)}][<-,>=stealth](2.5,0)--(2.8,0);
    \draw (2,-1.1) circle (0.3cm);
    \draw[rotate around={-30:(2,-1.1)}][->,>=stealth](2.3,-1.1)--(2.6,-1.1);
    \draw[rotate around={-150:(2,-1.1)}][->,>=stealth](2.3,-1.1)--(2.6,-1.1);
    \draw[rotate around={-90:(2,-1.1)},shift={(2.3,-1.1)}]\doublefleche;
    \draw[rotate around={-210:(2,-1.1)},shift={(2.3,-1.1)}]\doublefleche;
    \draw[rotate around={90:(2,0)}][<-,>=stealth](2.5,0)--(2.8,0);
    \draw (2,1.1) circle (0.3cm);
    \draw[rotate around={30:(2,1.1)}][->,>=stealth](2.3,1.1)--(2.6,1.1);
    \draw[rotate around={150:(2,1.1)}][->,>=stealth](2.3,1.1)--(2.6,1.1);
    \draw[rotate around={90:(2,1.1)},shift={(2.3,1.1)}]\doublefleche;
    \draw[rotate around={210:(2,1.1)},shift={(2.3,1.1)}]\doublefleche;
    \draw[rotate around={0:(2,0)}][<-,>=stealth](2.5,0)--(2.8,0);
    \draw (3.1,0) circle (0.3cm);
    \draw[rotate around={60:(3.1,0)}][->,>=stealth](3.4,0)--(3.7,0);
    \draw[rotate around={-60:(3.1,0)}][->,>=stealth](3.4,0)--(3.7,0);
    \draw[rotate around={0:(3.1,0)},shift={(3.4,0)}]\doublefleche;
    \draw[rotate around={120:(3.1,0)},shift={(3.4,0)}]\doublefleche;
\draw (0,0.25)node[anchor=north]{$j_q$};
\draw (2,0.25)node[anchor=north]{$i_p$};
\draw (3.1,0.25)node[anchor=north]{$\scriptscriptstyle{i_{\tau'_1}}$};
\draw (2,1.35)node[anchor=north]{$\scriptscriptstyle{i_{\tau'_{v}}}$};
\draw (0.75,1.1)node[anchor=north]{$\scriptscriptstyle{j_{\tau_{n-k-1}\dprimeindtl}}$};
\draw (2.1,-0.85)node[anchor=north]{$\scriptscriptstyle{i_{\tau'_{s-1}}}$};
\draw (0.8,-0.55)node[anchor=north]{$\scriptscriptstyle{j_{\tau_1\dprimeind}}$};
\draw[rotate=45][fill=black] (0.5,0.35) circle (0.3pt);
\draw[rotate=45][fill=black] (0.55,0.25) circle (0.3pt);
\draw[rotate=45][fill=black] (0.58,0.15) circle (0.3pt);
\draw[rotate=-90][fill=black] (0.5,0.35) circle (0.3pt);
\draw[rotate=-90][fill=black] (0.55,0.25) circle (0.3pt);
\draw[rotate=-90][fill=black] (0.58,0.15) circle (0.3pt);
\draw[rotate around={-100:(2,0)}][fill=black] (1.6,0.45) circle (0.3pt);
\draw[rotate around={-100:(2,0)}][fill=black] (1.5,0.3) circle (0.3pt);
\draw[rotate around={-100:(2,0)}][fill=black] (1.75,0.55) circle (0.3pt);
\draw[rotate around={-160:(2,0)}][fill=black] (1.6,0.45) circle (0.3pt);
\draw[rotate around={-160:(2,0)}][fill=black] (1.5,0.3) circle (0.3pt);
\draw[rotate around={-160:(2,0)}][fill=black] (1.75,0.55) circle (0.3pt);
\draw[fill=black] (0.7,0.4) circle (0.3pt);
\draw[fill=black] (0.88,0.4) circle (0.3pt);
\draw[fill=black] (1,0.5) circle (0.3pt);
\draw[fill=black] (1.05,-1.05) circle (0.3pt);
\draw[fill=black] (1.15,-0.9) circle (0.3pt);
\draw[fill=black] (0.9,-1.15) circle (0.3pt);
\draw[fill=black] (-0.55,0.2) circle (0.3pt);
\draw[fill=black] (-0.6,0) circle (0.3pt);
\draw[fill=black] (-0.55,-0.2) circle (0.3pt);
\draw[fill=black] (2.25,-0.8) circle (0.3pt);
\draw[fill=black] (2.37,-0.95) circle (0.3pt);
\draw[fill=black] (2.38,-1.15) circle (0.3pt);
\draw[fill=black] (2.25,0.8) circle (0.3pt);
\draw[fill=black] (2.37,0.95) circle (0.3pt);
\draw[fill=black] (2.38,1.15) circle (0.3pt);
\draw[fill=black] (2.8,-0.25) circle (0.3pt);
\draw[fill=black] (2.93,-0.35) circle (0.3pt);
\draw[fill=black] (3.1,-0.38) circle (0.3pt);
\end{tikzpicture}
\end{small}
\end{minipage}

\noindent respectively.

Moreover, the term $\sum\mathcal{E}(\mathcalboondox{D^2_j})$ in $\Phi_{l_{j_q}}^{\tau\dprime}\big(\Phi_{l_{i_p}}^{\tau'}(l_{i_1},\dots,\Hat{l_{i_p}},\dots,l_{i_k}),l_{j_1},\dots,\Hat{l_{j_q}},\dots,l_{j_{n-k}}\big)$ indexed by $r$ cancels with the evaluation $\mathcal{E}(\mathcalboondox{D_j^1})$ indexed by $u=r$, $v=u+k-1$ in $\Phi^{\sigma}_{l_j\underset{\nec}{\circ}l_i}(l_n,\dots,\Hat{l_i},\dots,\Hat{l_j},\dots,l_1)$ when $\tau$ is the permutation given by 

\begin{equation}
    \begin{split}
\tau_1=j_{\tau_1\dprimeind},\dots,\tau_{u}=j_{\tau_r\dprimeind},\tau_{u+1}=i_{\tau'_1},\dots,\tau_{v}=i_{\tau'_{k-1}},\tau_{v+1}=j_{\tau_{r+1}\dprimeindl}\hskip2.2mm,\dots,\tau_n=j_{\tau_{n-k-1}\dprimeindtl}.
    \end{split}
\end{equation}
Indeed, it is straightforward to check that 
\begin{small}
\begin{equation}
    \begin{split}   
    &\Delta\dprime+|l_i||l_j|+1+\zeta+\delta^{\tau}(l_{n},\dots,\Hat{l_i},\dots,\Hat{l_j},\dots,l_1)+|l_i|\sum\limits_{t=v+1}^{n-2}|l_{\tau_t}|+k+\Delta'+\zeta_1
    \\&+\delta^{\tau'}(l_{i_k},\dots,\Hat{l_{i_p}},\dots,l_{i_1})
    +\zeta'_1+\delta^{\tau\dprime}(l_{j_{n-k}},\dots,\Hat{l_ {j_q}},\dots,l_{j_1})+\zeta_1\dprimeind=1\mod 2.
    \end{split}
\end{equation}
\end{small}
On the other hand, the term $\sum\mathcal{E}(\mathcalboondox{D_i^1})$ in $\Psi_{l_{j_q}}^{\sigma\dprime}\big(\Psi_{l_{i_p}}^{\sigma'}(l_{i_1},\dots,\Hat{l_{i_p}},\dots,l_{i_k}),l_{j_1},\dots,\Hat{l_{j_q}},\dots,l_{j_{n-k}}\big)$ indexed by $s$ cancels with the evaluation $\mathcal{E}(\mathcalboondox{D_i})$ indexed by $u=s-1$, $v=n-k-1+u$ in $\Psi^{\tau}_{l_j\underset{\nec}{\circ}l_i}(l_1,\dots,l_n)$ when $\sigma$ is the permutation given by 

\begin{equation}
    \begin{split}
    &\tau_1=i_{\tau'_1},\dots,\tau_u=i_{\tau'_{s-1}},\tau_{u+1}=j_{\tau_1\dprimeind},\dots,\tau_v=j_{\tau_{n-k-1}\dprimeindtl}\hskip6.2mm,\tau_{v+1}=i_{\tau'_s},\dots,\tau_n=i_{\tau'_{k-1}}.
    \end{split}
\end{equation}
Therefore, we have that
\begin{equation}
    \begin{split}
    &\sum\limits_{k=3}^n(-1)^{k}\hskip-3mm\sum\limits_{\substack{(\bar{i},\bar{j})\in\mathcal{P}_k^n\\i=i_p, j=j_q}}\hskip-2mm
    (-1)^{\Delta'} \ell^{n-k+1}\big(\ell^k(s_{-1}l_{i_k},\dots,l_{i_p},\dots,s_{-1}l_{i_1}),s_{-1}l_{j_{n-k}},\dots,l_{j_q},\dots,s_{-1}l_{j_{1}}\big)
    \\&-(-1)^{\Delta\dprime+|l_i||l_j|} \ell^{n-1}\big(l_j\underset{\nec}{\circ}l_i,s_{-1}l_{n},\dots,\Hat{l_i},\dots,\Hat{l_j},\dots,s_{-1}l_{1}\big)
    =0
    \end{split}
\end{equation}
and one can show in a similar manner that \vskip-1cm
\begin{equation}
    \begin{split}
\\&\sum\limits_{k=3}^n(-1)^{k}\hskip-3mm\sum\limits_{\substack{(\bar{i},\bar{j})\in\mathcal{P}_k^n\\i=j_q, j=i_p}}\hskip-2mm(-1)^{\Delta'} \ell^{n-k+1}\big(\ell^k(s_{-1}l_{i_k},\dots,l_{i_p},\dots,s_{-1}l_{i_1}),s_{-1}l_{j_{n-k}},\dots,l_{j_q},\dots,s_{-1}l_{j_{1}}\big)
    \\&+(-1)^{\Delta\dprime} \ell^{n-1}\big(l_i\underset{\nec}{\circ}l_j,s_{-1}l_{n},\dots,\Hat{l_i},\dots,\Hat{l_j},\dots,s_{-1}l_{1}\big)
    =0.
    \end{split}
\end{equation}


Finally, suppose that $l_i\in \mathfrak{g}_{\MA},l_j\in\mathfrak{g}_{\MB}$ for some $i,j\in\llbracket 1,n\rrbracket$ with $i>j$ and $l_k\in \mathfrak{h}$ for every $k\in\llbracket 1,n\rrbracket$ with $k\neq i,j$.

Since in this case $[l_i,l_j]_{\nec}=0$, we have to show that
\begin{equation}
    \begin{split}
    &\sum\limits_{k=3}^n(-1)^{k}\sum\limits_{\substack{(\bar{i},\bar{j})\in\mathcal{P}_k^n\\i=i_p, j=j_q}}(-1)^{\Delta'} \ell^{n-k+1}\big(\ell^k(s_{-1}l_{i_k},\dots,l_{i_p},\dots,s_{-1}l_{i_1}),s_{-1}l_{j_{n-k}},\dots,l_{j_q},\dots,s_{-1}l_{j_{1}}\big)
    \\&-\hskip-1mm\sum\limits_{k'=3}^n(-1)^{k'}\hskip-3mm\sum\limits_{\substack{(\bar{i'},\bar{j'})\in\mathcal{P}_k^n\\i=j'_q, j=i'_p}}\hskip-2mm(-1)^{\Delta'_2} \ell^{n-k'+1}\big(\ell^{k'}(s_{-1}l_{i'_{k'}},\dots,l_{i'_p},\dots,s_{-1}l_{i'_1}),s_{-1}l_{j'_{n-k'}},\dots,l_{j'_q},\dots,s_{-1}l_{j'_{1}}\big)
    \\&=0
    \end{split}
\end{equation}

where 
\begin{small}
\begin{equation}
    \begin{split}
        \Delta'_2&=|l_{i'_{p'}}|\sum\limits_{\substack{t=i'_{p'}+1\\t\neq j'_{q'}}}^n(|l_t|+1)+|l_{i'_{p'}}|\sum\limits_{\substack{t=p'+1}}^{k'}(|l_{i'_t}|+1)+\sum\limits_{\substack{t=1\\t\neq p'}}^{k'}(|l_{i'_t}|+1)\hskip-2mm\sum\limits_{\substack{w=i'_t+1\\w\neq i'_{p'},j'_{q'}}}^n\hskip-2mm(|l_w|+1)+|l_{j'_{q'}}|\sum_{\substack{t=1\\t\neq p'\\i'_t<j'_{q'}}}^{k'}(|l_{i'_t}|+1)
    \\&\phantom{=}+\sum\limits_{\substack{t=1\\t\neq p'}}^{k'}(|l_{i'_t}|+1)\hskip-2mm\sum\limits_{\substack{w=t+1\\w\neq p'}}^{k'}(|l_{i'_w}|+1)+n-i'_{p'}+k'-p'
    +\sum\limits_{\substack{t=1\\t\neq p'}}^{k'}\sum\limits_{\substack{w=i'_t+1\\w\neq i'_{p'}}}^n\hskip-2mm1+\sum\limits_{\substack{t=1\\ t\neq p'}}^{k'} \sum\limits_{\substack{w=t+1\\w\neq p'}}^{k'}\hskip-2mm1+j'_{q'}-q'+|l_{i'_p}||l_{j'_q}|.
    \end{split}
\end{equation}
\end{small}

We have that 
\begin{equation}
    \begin{split}
   &\sum\limits_{k=3}^n(-1)^{k}\hskip-3mm\sum\limits_{\substack{(\bar{i},\bar{j})\in\mathcal{P}_k^n\\i=i_p, j=j_q}}\hskip-2mm(-1)^{\Delta'} \ell^{n-k+1}\big(\ell^k(s_{-1}l_{i_k},\dots,l_{i_p},\dots,s_{-1}l_{i_1}),s_{-1}l_{j_{n-k}},\dots,l_{j_q},\dots,s_{-1}l_{j_{1}}\big)
    \\&=\sum\limits_{k=3}^n(-1)^{k}\hskip-3mm\sum\limits_{\substack{(\bar{i},\bar{j})\in\mathcal{P}_k^n\\i=i_p, j=j_q}}
    (-1)^{\Delta'}\sum\limits_{\upsilon'\in\mathfrak{S}_{k-1}} (-1)^{\eta_1+\delta^{\upsilon'}(l_{i_k},\dots,\Hat{l_{i_p}},\dots,l_{i_1})} 
    \\&\hskip4cm\ell^{n-k+1}\big(s_{-1}\Psi^{\upsilon'}_{l_{i_p}}(
    l_{i_k},\dots,\Hat{l_{i_p}},\dots,l_{i_1}),s_{-1}l_{j_{n-k}},\dots,l_{j_q},\dots,s_{-1}l_{j_{1}}\big)
    \end{split}
\end{equation}
with 
\begin{small}
\begin{equation}
    \begin{split}
        \eta_1&=p-1+|l_{i_p}|\sum\limits_{t=1}^{p-1}(|l_{i_t}|+1)+\sum\limits_{t=1}^{p-1} (t-1)|l_{i_t}|+\sum\limits_{t=p+1}^{k} t|l_{i_t}|.
    \end{split}
\end{equation}
\end{small}

Therefore, we have that 
\begin{equation}
    \begin{split}
   &\sum\limits_{k=3}^n(-1)^{k}\hskip-3mm\sum\limits_{\substack{(\bar{i},\bar{j})\in\mathcal{P}_k^n\\i=i_p, j=j_q}}\hskip-2mm(-1)^{\Delta'} \ell^{n-k+1}\big(\ell^k(s_{-1}l_{i_k},\dots,l_{i_p},\dots,s_{-1}l_{i_1}),s_{-1}l_{j_{n-k}},\dots,l_{j_q},\dots,s_{-1}l_{j_{1}}\big)
    \\&=\sum\limits_{k=3}^n(-1)^{k}\hskip-3mm\sum\limits_{\substack{(\bar{i},\bar{j})\in\mathcal{P}_k^n\\i=i_p, j=j_q}}\hskip-3mm(-1)^{\Delta'}\hskip-3mm\sum\limits_{\upsilon'\in\mathfrak{S}_{k-1}} \hskip-3mm(-1)^{\eta_1+\delta^{\upsilon'}(l_{i_k},\dots,\Hat{l_{i_p}},\dots,l_{i_1})} \hskip-2mm\sum\limits_{\upsilon\dprime\in\mathfrak{S}_{n-k-1}}\hskip-3mm (-1)^{\eta_1'+\delta^{\upsilon\dprime}(l_{j_{n-k}},\dots,\Hat{l_{j_q}},\dots,l_{j_{1}})}
    \\&\hskip6cms_{-1}\Phi^{\upsilon\dprime}_{l_{j_q}}(\big(\Psi^{\upsilon'}_{l_{i_p}}(l_{i_k},\dots,\Hat{l_{i_p}},\dots,l_{i_1}),l_{j_{n-k}},\dots,\Hat{l_{j_q}},\dots,l_{j_{1}}\big)
    \end{split}
\end{equation}
with 
\begin{small}
\begin{equation}
    \begin{split}
    \eta_1'&=q+(q-1)|l_{j_q}|+|l_{j_q}|\sum\limits_{t=q+1}^{n-k}|l_{j_t}|+|l_{j_q}|\sum\limits_{t=1}^k|l_{i_t}|+\sum\limits_{t=1}^{q-1}(t-1)|l_{j_t}|
    +\sum\limits_{t=q+1}^{n-k}t|l_{j_t}|
    \\&\phantom{=}+(n-k-1)\sum\limits_{t=1}^k|l_{i_t}|
    \end{split}
\end{equation}
\end{small}
where
\begin{equation}
\begin{split}
 \Phi^{\upsilon\dprime}_{l_{j_q}}
 \big(\Psi^{\upsilon'}_{l_{i_p}}(l_{i_k},\dots,\Hat{l_{i_p}},\dots,l_{i_1}),l_{j_{n-k}},\dots,\Hat{l_{j_q}},\dots,l_{j_{1}}\big)&=\sum\limits_{r=1}^{n-k-1}(-1)^{\eta_1\dprimeind}\sum\mathcal{E}(\mathcalboondox{D_1})
 \\&\phantom{=}+\sum\limits_{s=2}^k(-1)^{\eta_1\tprime}\sum\mathcal{E}(\mathcalboondox{D_1'})
 \end{split}
\end{equation}
with 
\begin{small}
\begin{equation}
    \begin{split}
        \eta_1\dprimeind&=\sum\limits_{t=1}^k|l_{i_t}|\sum\limits_{t=r+1}^{n-k-1}|l_{j_{\upsilon_t\dprimeind}}|
        \text{ , }\eta_1\tprime=(|l_{i_p}|+\sum\limits_{t=1}^{s-1}|l_{i_{\upsilon'_t}}|)\sum\limits_{t=1}^{n-k-1}|l_{j_{\upsilon_t\dprimeind}}|
    \end{split}
\end{equation}
\end{small}
where the sums are over all the filled diagrams $\mathcalboondox{D_1}$ and $\mathcalboondox{D_1'}$ of the form

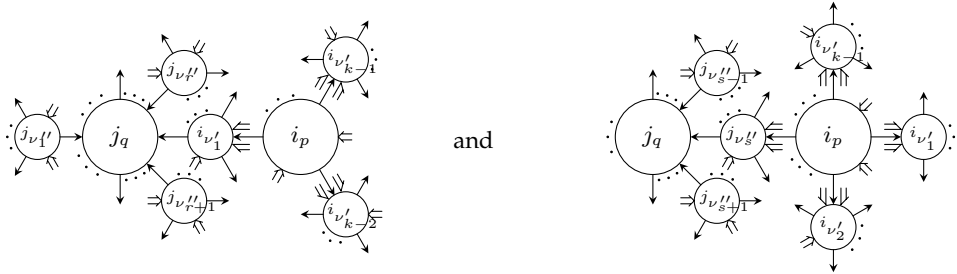
\begin{minipage}{21cm}
\begin{small}
\begin{tikzpicture}[line cap=round,line join=round,x=1.0cm,y=1.0cm]
\clip(-3.5,-2) rectangle (5,2);
    \draw (0,0) circle (0.3cm);
    \draw[rotate=60][->,>=stealth] (0.3,0)--(0.7,0);
    \draw[rotate=180][->,>=stealth] (0.3,0)--(0.7,0);
    \draw[rotate=-60][->,>=stealth] (0.3,0)--(0.7,0);
    \draw[rotate=-120,shift={(0.3,0)}] \doublefleche;
    \draw[rotate=0,shift={(0.3,0)}] \doubleflechescindeeleft;
    \draw[rotate=0,shift={(0.3,0)}] \doubleflechescindeeright;
    \draw[rotate=0][<-,>=stealth] (0.3,0)--(0.7,0);
    \draw(1.2,0)circle (0.5cm);
    \draw[rotate around={0:(1.2,0)},shift={(1.7,0)}] \doublefleche;
    \draw[rotate around={-120:(1.2,0)},shift={(1.7,0)}] \doublefleche;
    \draw[rotate around={-60:(1.2,0)}][->,>=stealth] (1.7,0)--(2.1,0);
    \draw(1.8,1.03)circle (0.3cm);
    \draw[rotate around={120:(1.8,1.03)},shift={(2.1,1.03)}] \doublefleche;
    \draw[rotate around={-120:(1.8,1.03)},shift={(2.1,1.03)}] \doubleflechescindeeleft;
    \draw[rotate around={-120:(1.8,1.03)},shift={(2.1,1.03)}] \doubleflechescindeeright;
    \draw[rotate around={-60:(1.8,1.03)}][->,>=stealth] (2.1,1.03)--(2.4,1.03);
    \draw[rotate around={180:(1.8,1.03)}][->,>=stealth] (2.1,1.03)--(2.4,1.03);
    \draw[rotate around={60:(1.8,1.03)}][->,>=stealth] (2.1,1.03)--(2.4,1.03);
    \draw[rotate around={60:(1.2,0)}][->,>=stealth] (1.7,0)--(2.1,0);
    \draw(1.8,-1.03)circle (0.3cm);
    \draw[rotate around={0:(1.8,-1.03)},shift={(2.1,-1.03)}] \doublefleche;
    \draw[rotate around={120:(1.8,-1.03)},shift={(2.1,-1.03)}] \doubleflechescindeeright;
    \draw[rotate around={120:(1.8,-1.03)},shift={(2.1,-1.03)}] \doubleflechescindeeleft;
    \draw[rotate around={60:(1.8,-1.03)}][->,>=stealth] (2.1,-1.03)--(2.4,-1.03);
    \draw[rotate around={-60:(1.8,-1.03)}][->,>=stealth] (2.1,-1.03)--(2.4,-1.03);
    \draw[rotate around={180:(1.8,-1.03)}][->,>=stealth] (2.1,-1.03)--(2.4,-1.03);
    \draw(-1.2,0)circle (0.5cm);
    \draw [rotate around={-45:(-1.2,0)}][<-,>=stealth] (-0.7,0)--(-0.3,0);
    \draw [rotate around={45:(-1.2,0)}][<-,>=stealth] (-0.7,0)--(-0.3,0);
    \draw [rotate around={90:(-1.2,0)}][->,>=stealth] (-0.7,0)--(-0.3,0);
    \draw [rotate around={-90:(-1.2,0)}][->,>=stealth] (-0.7,0)--(-0.3,0);
    \draw(-0.35,-0.85)circle (0.3cm);
    \draw [rotate around={-120:(-0.35,-0.85)}][->,>=stealth] (-0.05,-0.85)--(0.25,-0.85);
    \draw [rotate around={0:(-0.35,-0.85)}][->,>=stealth] (-0.05,-0.85)--(0.25,-0.85);
    \draw [rotate around={180:(-0.35,-0.85)},shift={(-0.05,-0.85)}]\doublefleche;
     \draw [rotate around={-60:(-0.35,-0.85)},shift={(-0.05,-0.85)}]\doublefleche;
     \draw(-0.35,0.85)circle (0.3cm);
    \draw [rotate around={120:(-0.35,0.85)}][->,>=stealth] (-0.05,0.85)--(0.25,0.85);
    \draw [rotate around={0:(-0.35,0.85)}][->,>=stealth] (-0.05,0.85)--(0.25,0.85);
    \draw [rotate around={180:(-0.35,0.85)},shift={(-0.05,0.85)}]\doublefleche;
     \draw [rotate around={60:(-0.35,0.85)},shift={(-0.05,0.85)}]\doublefleche;
     \draw[rotate around={180:(-1.2,0)}][<-,>=stealth] (-0.7,0)--(-0.4,0);
     \draw (-2.3,0) circle (0.3cm);
     \draw[rotate around={120:(-2.3,0)}][->,>=stealth] (-2,0)--(-1.7,0);
     \draw[rotate around={-120:(-2.3,0)}][->,>=stealth] (-2,0)--(-1.7,0);
     \draw[rotate around={60:(-2.3,0)},shift={(-2,0)}]\doublefleche;
     \draw[rotate around={-60:(-2.3,0)},shift={(-2,0)}]\doublefleche;
\draw [fill=black] (-2.65,0.15) circle (0.3pt);
\draw [fill=black] (-2.7,0) circle (0.3pt);
\draw [fill=black] (-2.65,-0.15) circle (0.3pt);
\draw [fill=black] (-0.3,0.2) circle (0.3pt);
\draw [fill=black] (-0.2,0.3) circle (0.3pt);
\draw [fill=black] (-0.05,0.35) circle (0.3pt);
\draw [rotate around={-45:(-1.2,0)}][fill=black] (-0.3,0.2) circle (0.3pt);
\draw [rotate around={-45:(-1.2,0)}][fill=black] (-0.2,0.3) circle (0.3pt);
\draw [rotate around={-45:(-1.2,0)}][fill=black] (-0.05,0.35) circle (0.3pt);
\draw [fill=black] (-0.3,0.5) circle (0.3pt);
\draw [fill=black] (-0.15,0.54) circle (0.3pt);
\draw [fill=black] (-0.05,0.65) circle (0.3pt);
\draw [fill=black] (2.2,1.03) circle (0.3pt);
\draw [fill=black] (2.15,0.9) circle (0.3pt);
\draw [fill=black] (2.15,1.16) circle (0.3pt);
\draw [rotate around={-120:(1.2,0)}][fill=black] (2.2,1.03) circle (0.3pt);
\draw [rotate around={-120:(1.2,0)}][fill=black] (2.15,0.9) circle (0.3pt);
\draw [rotate around={-120:(1.2,0)}][fill=black] (2.15,1.16) circle (0.3pt);
\draw [fill=black] (0.75,0.35) circle (0.3pt);
\draw [fill=black] (0.9,0.5) circle (0.3pt);
\draw [fill=black] (1.1,0.55) circle (0.3pt);
\draw [rotate around={40:(-1.2,0)}][fill=black] (-0.78,0.4) circle (0.3pt);
\draw [rotate around={40:(-1.2,0)}][fill=black] (-0.7,0.3) circle (0.3pt);
\draw [rotate around={40:(-1.2,0)}][fill=black] (-0.65,0.18) circle (0.3pt);
\draw [rotate around={-100:(-1.2,0)}][fill=black] (-0.78,0.4) circle (0.3pt);
\draw [rotate around={-100:(-1.2,0)}][fill=black] (-0.7,0.3) circle (0.3pt);
\draw [rotate around={-100:(-1.2,0)}][fill=black] (-0.65,0.18) circle (0.3pt);
\draw [fill=black] (-1.7,0.3) circle (0.3pt);
\draw [fill=black] (-1.6,0.45) circle (0.3pt);
\draw [fill=black] (-1.45,0.55) circle (0.3pt);
\draw [rotate around={90:(-1.2,0)}][fill=black] (-1.7,0.3) circle (0.3pt);
\draw [rotate around={90:(-1.2,0)}][fill=black] (-1.6,0.45) circle (0.3pt);
\draw [rotate around={90:(-1.2,0)}][fill=black] (-1.45,0.55) circle (0.3pt);
\draw (3.5,0.2) node[anchor=north] {and};
\draw (1.2,0.25) node[anchor=north] {$i_p$};
\draw (-1.2,0.25) node[anchor=north] {$j_q$};
\draw (0,0.25) node[anchor=north] {$\scriptstyle{i_{\nu'_{1}}}$};
\draw (1.9,-0.78) node[anchor=north] {$\scriptstyle{i_{\nu'_{k-2}}}$};
\draw (1.9,1.28) node[anchor=north] {$\scriptstyle{i_{\nu'_{k-1}}}$};
\draw (-2.3,0.25) node[anchor=north] {$\scriptstyle{j_{\nu_{1}\dprimeind}}$};
\draw (-0.35,-0.6) node[anchor=north] {$\scriptstyle{j_{\nu_{r+1}\dprimeindl}}$};
\draw (-0.35,1.1) node[anchor=north] {$\scriptstyle{j_{\nu_{r}\dprimeind}}$};
\end{tikzpicture}
\begin{tikzpicture}[line cap=round,line join=round,x=1.0cm,y=1.0cm]
\clip(-2,-2) rectangle (8.32706640373885,2);
  \draw (0,0) circle (0.3cm);
    \draw[rotate=60][->,>=stealth] (0.3,0)--(0.7,0);
    \draw[rotate=180][->,>=stealth] (0.3,0)--(0.7,0);
    \draw[rotate=-60][->,>=stealth] (0.3,0)--(0.7,0);
    \draw[rotate=-120,shift={(0.3,0)}] \doublefleche;
    \draw[rotate=0,shift={(0.3,0)}] \doubleflechescindeeleft;
    \draw[rotate=0,shift={(0.3,0)}] \doubleflechescindeeright;
    \draw[rotate=0][<-,>=stealth] (0.3,0)--(0.7,0);
    \draw(1.2,0)circle (0.5cm);
    \draw[->,>=stealth] (1.7,0)--(2.1,0);
    \draw(2.4,0) circle (0.3cm);
    \draw[rotate around={180:(2.4,0)},shift={(2.7,0)}]\doubleflechescindeeleft;
    \draw[rotate around={180:(2.4,0)},shift={(2.7,0)}]\doubleflechescindeeright;
    \draw[rotate around={90:(2.4,0)}][->,>=stealth](2.7,0)--(3,0);
    \draw[rotate around={-90:(2.4,0)}][->,>=stealth](2.7,0)--(3,0);
    \draw[rotate around={45:(1.2,0)},shift={(1.7,0)}] \doublefleche;
    \draw[rotate around={-45:(1.2,0)},shift={(1.7,0)}] \doublefleche;
    \draw[rotate around={-90:(1.2,0)}][->,>=stealth] (1.7,0)--(2.1,0);
    \draw(1.2,-1.2)circle (0.3cm);
    \draw[rotate around={-150:(1.2,-1.2)},shift={(1.5,-1.2)}] \doublefleche;
    \draw[rotate around={90:(1.2,-1.2)},shift={(1.5,-1.2)}] \doubleflechescindeeleft;
    \draw[rotate around={90:(1.2,-1.2)},shift={(1.5,-1.2)}] \doubleflechescindeeright;
    \draw[rotate around={30:(1.2,-1.2)}][->,>=stealth] (1.5,-1.2)--(1.8,-1.2);
    \draw[rotate around={150:(1.2,-1.2)}][->,>=stealth] (1.5,-1.2)--(1.8,-1.2);
    \draw[rotate around={-90:(1.2,-1.2)}][->,>=stealth] (1.5,-1.2)--(1.8,-1.2);
    \draw[rotate around={90:(1.2,0)}][->,>=stealth] (1.7,0)--(2.1,0);
    \draw(1.2,1.2)circle (0.3cm);
    \draw[rotate around={150:(1.2,1.2)},shift={(1.5,1.2)}] \doublefleche;
    \draw[rotate around={-90:(1.2,1.2)},shift={(1.5,1.2)}] \doubleflechescindeeleft;
    \draw[rotate around={-90:(1.2,1.2)},shift={(1.5,1.2)}] \doubleflechescindeeright;
    \draw[rotate around={-30:(1.2,1.2)}][->,>=stealth] (1.5,1.2)--(1.8,1.2);
    \draw[rotate around={-150:(1.2,1.2)}][->,>=stealth] (1.5,1.2)--(1.8,1.2);
    \draw[rotate around={90:(1.2,1.2)}][->,>=stealth] (1.5,1.2)--(1.8,1.2);
    \draw(-1.2,0)circle (0.5cm);
    \draw [rotate around={-45:(-1.2,0)}][<-,>=stealth] (-0.7,0)--(-0.3,0);
    \draw [rotate around={45:(-1.2,0)}][<-,>=stealth] (-0.7,0)--(-0.3,0);
    \draw [rotate around={90:(-1.2,0)}][->,>=stealth] (-0.7,0)--(-0.3,0);
    \draw [rotate around={-90:(-1.2,0)}][->,>=stealth] (-0.7,0)--(-0.3,0);
    \draw(-0.35,-0.85)circle (0.3cm);
    \draw [rotate around={-120:(-0.35,-0.85)}][->,>=stealth] (-0.05,-0.85)--(0.25,-0.85);
    \draw [rotate around={0:(-0.35,-0.85)}][->,>=stealth] (-0.05,-0.85)--(0.25,-0.85);
    \draw [rotate around={180:(-0.35,-0.85)},shift={(-0.05,-0.85)}]\doublefleche;
     \draw [rotate around={-60:(-0.35,-0.85)},shift={(-0.05,-0.85)}]\doublefleche;
     \draw(-0.35,0.85)circle (0.3cm);
    \draw [rotate around={120:(-0.35,0.85)}][->,>=stealth] (-0.05,0.85)--(0.25,0.85);
    \draw [rotate around={0:(-0.35,0.85)}][->,>=stealth] (-0.05,0.85)--(0.25,0.85);
    \draw [rotate around={180:(-0.35,0.85)},shift={(-0.05,0.85)}]\doublefleche;
     \draw [rotate around={60:(-0.35,0.85)},shift={(-0.05,0.85)}]\doublefleche;
\draw [fill=black] (-1.75,0.2) circle (0.3pt);
\draw [fill=black] (-1.8,0) circle (0.3pt);
\draw [fill=black] (-1.75,-0.2) circle (0.3pt);
\draw [fill=black] (-0.3,0.2) circle (0.3pt);
\draw [fill=black] (-0.2,0.3) circle (0.3pt);
\draw [fill=black] (-0.05,0.35) circle (0.3pt);
\draw [rotate around={-45:(-1.2,0)}][fill=black] (-0.3,0.2) circle (0.3pt);
\draw [rotate around={-45:(-1.2,0)}][fill=black] (-0.2,0.3) circle (0.3pt);
\draw [rotate around={-45:(-1.2,0)}][fill=black] (-0.05,0.35) circle (0.3pt);
\draw [fill=black] (-0.3,0.5) circle (0.3pt);
\draw [fill=black] (-0.15,0.54) circle (0.3pt);
\draw [fill=black] (-0.05,0.65) circle (0.3pt);
\draw [rotate around={40:(-1.2,0)}][fill=black] (-0.78,0.4) circle (0.3pt);
\draw [rotate around={40:(-1.2,0)}][fill=black] (-0.7,0.3) circle (0.3pt);
\draw [rotate around={40:(-1.2,0)}][fill=black] (-0.65,0.18) circle (0.3pt);
\draw [rotate around={-100:(-1.2,0)}][fill=black] (-0.78,0.4) circle (0.3pt);
\draw [rotate around={-100:(-1.2,0)}][fill=black] (-0.7,0.3) circle (0.3pt);
\draw [rotate around={-100:(-1.2,0)}][fill=black] (-0.65,0.18) circle (0.3pt);
\draw [fill=black] (2.75,0.15) circle (0.3pt);
\draw [fill=black] (2.8,0) circle (0.3pt);
\draw [fill=black] (2.75,-0.15) circle (0.3pt);
\draw [fill=black] (1.58,1.18) circle (0.3pt);
\draw [fill=black] (1.55,1.35) circle (0.3pt);
\draw [fill=black] (1.45,1.5) circle (0.3pt);
\draw [fill=black] (1.58,-1.18) circle (0.3pt);
\draw [fill=black] (1.55,-1.35) circle (0.3pt);
\draw [fill=black] (1.45,-1.5) circle (0.3pt);
\draw [fill=black] (0.7,0.3) circle (0.3pt);
\draw [fill=black] (0.8,0.45) circle (0.3pt);
\draw [fill=black] (0.95,0.55) circle (0.3pt);
\draw [rotate around={90:(1.2,0)}][fill=black] (0.7,0.3) circle (0.3pt);
\draw [rotate around={90:(1.2,0)}][fill=black] (0.8,0.45) circle (0.3pt);
\draw [rotate around={90:(1.2,0)}][fill=black] (0.9,0.55) circle (0.3pt);
\draw (1.2,0.25) node[anchor=north] {$i_p$};
\draw (-1.2,0.25) node[anchor=north] {$j_q$};
\draw (2.4,0.25) node[anchor=north] {$\scriptstyle{i_{\nu'_{1}}}$};
\draw (1.2,-0.95) node[anchor=north] {$\scriptstyle{i_{\nu'_{2}}}$};
\draw (1.3,1.45) node[anchor=north] {$\scriptstyle{i_{\nu'_{k-1}}}$};
\draw (0,0.25) node[anchor=north] {$\scriptstyle{j_{\nu_{s}\dprimeind}}$};
\draw (-0.35,-0.6) node[anchor=north] {$\scriptstyle{j_{\nu_{s+1}\dprimeindl}}$};
\draw (-0.35,1.1) node[anchor=north] {$\scriptstyle{j_{\nu_{s-1}\dprimeindl}}$};
\end{tikzpicture}
\end{small}
\end{minipage}

\noindent respectively, where the bold arrow is placed such that $j_{\nu_1\dprimeind}$ (resp. $i_{\nu'_1}$) is the first (resp. after $i_p$) in the application order of diagrams of the form $\mathcalboondox{D_1}$ (resp. $\mathcalboondox{D_1'}$).

On the other hand, we have that 
\begin{small}
\begin{equation}
    \begin{split}
   &-\hskip-1mm\sum\limits_{k'=3}^n(-1)^{k'}\hskip-5mm\sum\limits_{\substack{(\bar{i'},\bar{j'})\in\mathcal{P}_{k'}^n\\i=j'_{q'}, j=i'_{p'}}}\hskip-4mm(-1)^{\Delta'_2} \ell^{n-k'+1}\big(\ell^{k'}(s_{-1}l_{i'_{k'}},\dots,l_{i'_{p'}},\dots,s_{-1}l_{i'_1}),s_{-1}l_{j'_{n-k'}},\dots,l_{j'_{q'}},\dots,s_{-1}l_{j'_{1}}\big)
    \\&=-\sum\limits_{k'=3}^n(-1)^{k'}\hskip-4mm\sum\limits_{\substack{(\bar{i'},\bar{j'})\in\mathcal{P}_{k'}^n\\i=j'_{q'}, j=i'_{p'}}}\hskip-3mm(-1)^{\Delta'_2}\hskip-3mm
    \sum\limits_{\rho'\in\mathfrak{S}_{k'-1}} (-1)^{\eta_2+\delta^{\rho'}(l_{i'_k},\dots,\Hat{l_{i'_{p'}}},\dots,l_{i'_1})} 
    \\&\hskip5cm\ell^{n-k'+1}\big(\Phi^{\rho'}_{l_{i'_{p'}}}( l_{i'_{k'}},\dots,\Hat{l_{i'_{p'}}},\dots,l_{i'_1}),s_{-1}l_{j'_{n-k'}},\dots,l_{j'_{q'}},\dots,s_{-1}l_{j'_{1}}\big)
    \end{split}
\end{equation}
\end{small}
with
\begin{small}
\begin{equation}
    \begin{split}
        \eta_2&=p'+(p'-1)|l_{i'_{p'}}|+|l_{i'_ {p'}}|\sum\limits_{t=p'+1}^{k'}|l_{i'_t}|+\sum\limits_{t=1}^{p'-1} (t-1)|l_{i'_t}|+\sum\limits_{t=p'+1}^{k'} t|l_{i'_t}|.
    \end{split}
\end{equation}
\end{small}
Therefore, we have that 
\begin{small}
\begin{equation}
    \begin{split}
   &-\hskip-1mm\sum\limits_{k'=3}^n(-1)^{k'}\hskip-5mm\sum\limits_{\substack{(\bar{i'},\bar{j'})\in\mathcal{P}_{k'}^n\\i=j'_{q'}, j=i'_{p'}}}\hskip-4mm(-1)^{\Delta'_2} \ell^{n-k'+1}\big(\ell^{k'}(s_{-1}l_{i'_{k'}},\dots,l_{i'_{p'}},\dots,s_{-1}l_{i'_1}),s_{-1}l_{j'_{n-k'}},\dots,l_{j'_{q'}},\dots,s_{-1}l_{j'_{1}}\big)
    \\&=-\sum\limits_{k'=3}^n(-1)^{k'}\hskip-5mm\sum\limits_{\substack{(\bar{i'},\bar{j'})\in\mathcal{P}_{k'}^n\\i=j'_{q'}, j=i'_{p'}}}\hskip-4mm(-1)^{\Delta'_2}\hskip-4mm
    \sum\limits_{\rho'\in\mathfrak{S}_{k'-1}} \hskip-3mm(-1)^{\eta_2+\delta^{\rho'}(l_{i'_{k'}},\dots,\Hat{l_{i'_{p'}}},\dots,l_{i'_1})} \sum\limits_{\rho\dprime\in\mathfrak{S}_{n-k'-1}} (-1)^{\eta_2'+\delta^{\rho\dprime}(l_{j'_{n-k'}},\dots,\Hat{l_{j'_{q'}}},\dots,l_{j'_{1}})}
    \\&\hspace{4cm}s_{-1}\Psi^{\rho\dprime}_{l_{j'_{q'}}}(\big(\Phi^{\rho'}_{l_{i'_{p'}}}(l_{i'_{k'}},\dots,\Hat{l_{i'_{p'}}},\dots,l_{i'_1}),l_{j'_{n-k'}},\dots,\Hat{l_{j'_{q'}}},\dots,l_{j'_{1}}\big)
    \end{split}
\end{equation}
\end{small}
with 
\begin{small}
\begin{equation}
    \begin{split}
    \eta_2'&=q'-1+|l_{j'_{q'}}|\sum\limits_{t=1}^{q'-1}(|l_{j'_t}|+1)+\sum\limits_{t=1}^{q'-1}(t-1)|l_{j'_t}|
    +\sum\limits_{t=q'+1}^{n-k'}t|l_{j'_t}|
    +(n-k'-1)\sum\limits_{t=1}^{k'}|l_{i'_t}|
    \end{split}
\end{equation}
\end{small}
where
\begin{equation}
\begin{split}
 &\Psi^{\rho\dprime}_{l_{j'_q}}
 \big(\Phi^{\rho'}_{l_{i'_p}}(l_{i'_{k'}},\dots,\Hat{l_{i'_p}},\dots,l_{i'_1}),l_{j'_{n-k'}},\dots,\Hat{l_{j'_q}},\dots,l_{j'_{1}}\big)
 \\&\hskip3cm=\sum\limits_{u=1}^{n-k-1}\big((-1)^{\eta_2\dprimeind}\sum\mathcal{E}(\mathcalboondox{D_2})+\sum\limits_{v=1}^{k'}(-1)^{\eta_2\tprime}\sum\mathcal{E}(\mathcalboondox{D_2'})\big)
 \end{split}
\end{equation}
with
\begin{small}
\begin{equation}
    \begin{split}
        \eta_2\dprimeind=\sum\limits_{t=1}^{n-k'}|l_{j'_{t}}|\sum\limits_{t=1}^{u}|l_{i'_{\rho'_t}}|
        \text{ , }\eta_2\tprime=\sum\limits_{t=v}^{n-k'}|l_{j'_{\rho_t\dprimeind}}|\sum\limits_{t=1}^{k'}|l_{i_{\rho'_t}}|
    \end{split}
\end{equation}
\end{small}
where the sums are over all the filled diagrams $\mathcalboondox{D_2}$ and $\mathcalboondox{D_2'}$ of the form 

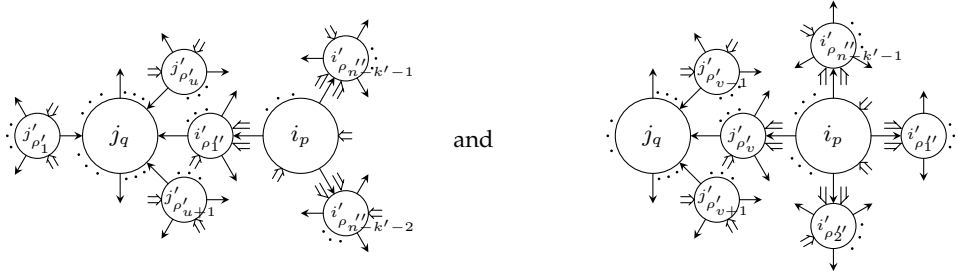
\begin{minipage}{21cm}
\begin{small}
    \begin{tikzpicture}[line cap=round,line join=round,x=1.0cm,y=1.0cm]
\clip(-3.5,-2) rectangle (5,2);
    \draw (0,0) circle (0.3cm);
    \draw[rotate=60][->,>=stealth] (0.3,0)--(0.7,0);
    \draw[rotate=180][->,>=stealth] (0.3,0)--(0.7,0);
    \draw[rotate=-60][->,>=stealth] (0.3,0)--(0.7,0);
    \draw[rotate=-120,shift={(0.3,0)}] \doublefleche;
    \draw[rotate=0,shift={(0.3,0)}] \doubleflechescindeeleft;
    \draw[rotate=0,shift={(0.3,0)}] \doubleflechescindeeright;
    \draw[rotate=0][<-,>=stealth] (0.3,0)--(0.7,0);
    \draw(1.2,0)circle (0.5cm);
    \draw[rotate around={0:(1.2,0)},shift={(1.7,0)}] \doublefleche;
    \draw[rotate around={-120:(1.2,0)},shift={(1.7,0)}] \doublefleche;
    \draw[rotate around={-60:(1.2,0)}][->,>=stealth] (1.7,0)--(2.1,0);
    \draw(1.8,1.03)circle (0.3cm);
    \draw[rotate around={120:(1.8,1.03)},shift={(2.1,1.03)}] \doublefleche;
    \draw[rotate around={-120:(1.8,1.03)},shift={(2.1,1.03)}] \doubleflechescindeeleft;
    \draw[rotate around={-120:(1.8,1.03)},shift={(2.1,1.03)}] \doubleflechescindeeright;
    \draw[rotate around={-60:(1.8,1.03)}][->,>=stealth] (2.1,1.03)--(2.4,1.03);
    \draw[rotate around={180:(1.8,1.03)}][->,>=stealth] (2.1,1.03)--(2.4,1.03);
    \draw[rotate around={60:(1.8,1.03)}][->,>=stealth] (2.1,1.03)--(2.4,1.03);
    \draw[rotate around={60:(1.2,0)}][->,>=stealth] (1.7,0)--(2.1,0);
    \draw(1.8,-1.03)circle (0.3cm);
    \draw[rotate around={0:(1.8,-1.03)},shift={(2.1,-1.03)}] \doublefleche;
    \draw[rotate around={120:(1.8,-1.03)},shift={(2.1,-1.03)}] \doubleflechescindeeright;
    \draw[rotate around={120:(1.8,-1.03)},shift={(2.1,-1.03)}] \doubleflechescindeeleft;
    \draw[rotate around={60:(1.8,-1.03)}][->,>=stealth] (2.1,-1.03)--(2.4,-1.03);
    \draw[rotate around={-60:(1.8,-1.03)}][->,>=stealth] (2.1,-1.03)--(2.4,-1.03);
    \draw[rotate around={180:(1.8,-1.03)}][->,>=stealth] (2.1,-1.03)--(2.4,-1.03);
    \draw(-1.2,0)circle (0.5cm);
    \draw [rotate around={-45:(-1.2,0)}][<-,>=stealth] (-0.7,0)--(-0.3,0);
    \draw [rotate around={45:(-1.2,0)}][<-,>=stealth] (-0.7,0)--(-0.3,0);
    \draw [rotate around={90:(-1.2,0)}][->,>=stealth] (-0.7,0)--(-0.3,0);
    \draw [rotate around={-90:(-1.2,0)}][->,>=stealth] (-0.7,0)--(-0.3,0);
    \draw(-0.35,-0.85)circle (0.3cm);
    \draw [rotate around={-120:(-0.35,-0.85)}][->,>=stealth] (-0.05,-0.85)--(0.25,-0.85);
    \draw [rotate around={0:(-0.35,-0.85)}][->,>=stealth] (-0.05,-0.85)--(0.25,-0.85);
    \draw [rotate around={180:(-0.35,-0.85)},shift={(-0.05,-0.85)}]\doublefleche;
     \draw [rotate around={-60:(-0.35,-0.85)},shift={(-0.05,-0.85)}]\doublefleche;
     \draw(-0.35,0.85)circle (0.3cm);
    \draw [rotate around={120:(-0.35,0.85)}][->,>=stealth] (-0.05,0.85)--(0.25,0.85);
    \draw [rotate around={0:(-0.35,0.85)}][->,>=stealth] (-0.05,0.85)--(0.25,0.85);
    \draw [rotate around={180:(-0.35,0.85)},shift={(-0.05,0.85)}]\doublefleche;
     \draw [rotate around={60:(-0.35,0.85)},shift={(-0.05,0.85)}]\doublefleche;
     \draw[rotate around={180:(-1.2,0)}][<-,>=stealth] (-0.7,0)--(-0.4,0);
     \draw (-2.3,0) circle (0.3cm);
     \draw[rotate around={120:(-2.3,0)}][->,>=stealth] (-2,0)--(-1.7,0);
     \draw[rotate around={-120:(-2.3,0)}][->,>=stealth] (-2,0)--(-1.7,0);
     \draw[rotate around={60:(-2.3,0)},shift={(-2,0)}]\doublefleche;
     \draw[rotate around={-60:(-2.3,0)},shift={(-2,0)}]\doublefleche;
\draw [fill=black] (-2.65,0.15) circle (0.3pt);
\draw [fill=black] (-2.7,0) circle (0.3pt);
\draw [fill=black] (-2.65,-0.15) circle (0.3pt);
\draw [fill=black] (-0.3,0.2) circle (0.3pt);
\draw [fill=black] (-0.2,0.3) circle (0.3pt);
\draw [fill=black] (-0.05,0.35) circle (0.3pt);
\draw [rotate around={-45:(-1.2,0)}][fill=black] (-0.3,0.2) circle (0.3pt);
\draw [rotate around={-45:(-1.2,0)}][fill=black] (-0.2,0.3) circle (0.3pt);
\draw [rotate around={-45:(-1.2,0)}][fill=black] (-0.05,0.35) circle (0.3pt);
\draw [fill=black] (-0.3,0.5) circle (0.3pt);
\draw [fill=black] (-0.15,0.54) circle (0.3pt);
\draw [fill=black] (-0.05,0.65) circle (0.3pt);
\draw [fill=black] (2.2,1.03) circle (0.3pt);
\draw [fill=black] (2.15,0.9) circle (0.3pt);
\draw [fill=black] (2.15,1.16) circle (0.3pt);
\draw [rotate around={-120:(1.2,0)}][fill=black] (2.2,1.03) circle (0.3pt);
\draw [rotate around={-120:(1.2,0)}][fill=black] (2.15,0.9) circle (0.3pt);
\draw [rotate around={-120:(1.2,0)}][fill=black] (2.15,1.16) circle (0.3pt);
\draw [fill=black] (0.75,0.35) circle (0.3pt);
\draw [fill=black] (0.9,0.5) circle (0.3pt);
\draw [fill=black] (1.1,0.55) circle (0.3pt);
\draw [rotate around={40:(-1.2,0)}][fill=black] (-0.78,0.4) circle (0.3pt);
\draw [rotate around={40:(-1.2,0)}][fill=black] (-0.7,0.3) circle (0.3pt);
\draw [rotate around={40:(-1.2,0)}][fill=black] (-0.65,0.18) circle (0.3pt);
\draw [rotate around={-100:(-1.2,0)}][fill=black] (-0.78,0.4) circle (0.3pt);
\draw [rotate around={-100:(-1.2,0)}][fill=black] (-0.7,0.3) circle (0.3pt);
\draw [rotate around={-100:(-1.2,0)}][fill=black] (-0.65,0.18) circle (0.3pt);
\draw [fill=black] (-1.7,0.3) circle (0.3pt);
\draw [fill=black] (-1.6,0.45) circle (0.3pt);
\draw [fill=black] (-1.45,0.55) circle (0.3pt);
\draw [rotate around={90:(-1.2,0)}][fill=black] (-1.7,0.3) circle (0.3pt);
\draw [rotate around={90:(-1.2,0)}][fill=black] (-1.6,0.45) circle (0.3pt);
\draw [rotate around={90:(-1.2,0)}][fill=black] (-1.45,0.55) circle (0.3pt);
\draw (3.5,0.2) node[anchor=north] {and};
\draw (1.2,0.25) node[anchor=north] {$i_p$};
\draw (-1.2,0.25) node[anchor=north] {$j_q$};
\draw (0,0.25) node[anchor=north] {$\scriptstyle{i'_{\rho_{1}\dprimeind}}$};
\draw (1.85,-0.78) node[anchor=north] {$\scriptstyle{i'_{\rho_{n-k'-2}\dprimeindtl}}$};
\draw (1.85,1.28) node[anchor=north] {$\scriptstyle{i'_{\rho_{n-k'-1}\dprimeindtl}}$};
\draw (-2.3,0.27) node[anchor=north] {$\scriptstyle{j'_{\rho'_{1}}}$};
\draw (-0.25,-0.57) node[anchor=north] {$\scriptstyle{j'_{\rho'_{u+1}}}$};
\draw (-0.32,1.15) node[anchor=north] {$\scriptstyle{j'_{\rho'_{u}}}$};
\end{tikzpicture}
\begin{tikzpicture}[line cap=round,line join=round,x=1.0cm,y=1.0cm]
\clip(-2,-2) rectangle (8.32706640373885,2);
  \draw (0,0) circle (0.3cm);
    \draw[rotate=60][->,>=stealth] (0.3,0)--(0.7,0);
    \draw[rotate=180][->,>=stealth] (0.3,0)--(0.7,0);
    \draw[rotate=-60][->,>=stealth] (0.3,0)--(0.7,0);
    \draw[rotate=-120,shift={(0.3,0)}] \doublefleche;
    \draw[rotate=0,shift={(0.3,0)}] \doubleflechescindeeleft;
    \draw[rotate=0,shift={(0.3,0)}] \doubleflechescindeeright;
    \draw[rotate=0][<-,>=stealth] (0.3,0)--(0.7,0);
    \draw(1.2,0)circle (0.5cm);
    \draw[->,>=stealth] (1.7,0)--(2.1,0);
    \draw(2.4,0) circle (0.3cm);
    \draw[rotate around={180:(2.4,0)},shift={(2.7,0)}]\doubleflechescindeeleft;
    \draw[rotate around={180:(2.4,0)},shift={(2.7,0)}]\doubleflechescindeeright;
    \draw[rotate around={90:(2.4,0)}][->,>=stealth](2.7,0)--(3,0);
    \draw[rotate around={-90:(2.4,0)}][->,>=stealth](2.7,0)--(3,0);
    \draw[rotate around={45:(1.2,0)},shift={(1.7,0)}] \doublefleche;
    \draw[rotate around={-45:(1.2,0)},shift={(1.7,0)}] \doublefleche;
    \draw[rotate around={-90:(1.2,0)}][->,>=stealth] (1.7,0)--(2.1,0);
    \draw(1.2,-1.2)circle (0.3cm);
    \draw[rotate around={-150:(1.2,-1.2)},shift={(1.5,-1.2)}] \doublefleche;
    \draw[rotate around={90:(1.2,-1.2)},shift={(1.5,-1.2)}] \doubleflechescindeeleft;
    \draw[rotate around={90:(1.2,-1.2)},shift={(1.5,-1.2)}] \doubleflechescindeeright;
    \draw[rotate around={30:(1.2,-1.2)}][->,>=stealth] (1.5,-1.2)--(1.8,-1.2);
    \draw[rotate around={150:(1.2,-1.2)}][->,>=stealth] (1.5,-1.2)--(1.8,-1.2);
    \draw[rotate around={-90:(1.2,-1.2)}][->,>=stealth] (1.5,-1.2)--(1.8,-1.2);
    \draw[rotate around={90:(1.2,0)}][->,>=stealth] (1.7,0)--(2.1,0);
    \draw(1.2,1.2)circle (0.3cm);
    \draw[rotate around={150:(1.2,1.2)},shift={(1.5,1.2)}] \doublefleche;
    \draw[rotate around={-90:(1.2,1.2)},shift={(1.5,1.2)}] \doubleflechescindeeleft;
    \draw[rotate around={-90:(1.2,1.2)},shift={(1.5,1.2)}] \doubleflechescindeeright;
    \draw[rotate around={-30:(1.2,1.2)}][->,>=stealth] (1.5,1.2)--(1.8,1.2);
    \draw[rotate around={-150:(1.2,1.2)}][->,>=stealth] (1.5,1.2)--(1.8,1.2);
    \draw[rotate around={90:(1.2,1.2)}][->,>=stealth] (1.5,1.2)--(1.8,1.2);
    \draw(-1.2,0)circle (0.5cm);
    \draw [rotate around={-45:(-1.2,0)}][<-,>=stealth] (-0.7,0)--(-0.3,0);
    \draw [rotate around={45:(-1.2,0)}][<-,>=stealth] (-0.7,0)--(-0.3,0);
    \draw [rotate around={90:(-1.2,0)}][->,>=stealth] (-0.7,0)--(-0.3,0);
    \draw [rotate around={-90:(-1.2,0)}][->,>=stealth] (-0.7,0)--(-0.3,0);
    \draw(-0.35,-0.85)circle (0.3cm);
    \draw [rotate around={-120:(-0.35,-0.85)}][->,>=stealth] (-0.05,-0.85)--(0.25,-0.85);
    \draw [rotate around={0:(-0.35,-0.85)}][->,>=stealth] (-0.05,-0.85)--(0.25,-0.85);
    \draw [rotate around={180:(-0.35,-0.85)},shift={(-0.05,-0.85)}]\doublefleche;
     \draw [rotate around={-60:(-0.35,-0.85)},shift={(-0.05,-0.85)}]\doublefleche;
     \draw(-0.35,0.85)circle (0.3cm);
    \draw [rotate around={120:(-0.35,0.85)}][->,>=stealth] (-0.05,0.85)--(0.25,0.85);
    \draw [rotate around={0:(-0.35,0.85)}][->,>=stealth] (-0.05,0.85)--(0.25,0.85);
    \draw [rotate around={180:(-0.35,0.85)},shift={(-0.05,0.85)}]\doublefleche;
     \draw [rotate around={60:(-0.35,0.85)},shift={(-0.05,0.85)}]\doublefleche;
\draw [fill=black] (-1.75,0.2) circle (0.3pt);
\draw [fill=black] (-1.8,0) circle (0.3pt);
\draw [fill=black] (-1.75,-0.2) circle (0.3pt);
\draw [fill=black] (-0.3,0.2) circle (0.3pt);
\draw [fill=black] (-0.2,0.3) circle (0.3pt);
\draw [fill=black] (-0.05,0.35) circle (0.3pt);
\draw [rotate around={-45:(-1.2,0)}][fill=black] (-0.3,0.2) circle (0.3pt);
\draw [rotate around={-45:(-1.2,0)}][fill=black] (-0.2,0.3) circle (0.3pt);
\draw [rotate around={-45:(-1.2,0)}][fill=black] (-0.05,0.35) circle (0.3pt);
\draw [fill=black] (-0.3,0.5) circle (0.3pt);
\draw [fill=black] (-0.15,0.54) circle (0.3pt);
\draw [fill=black] (-0.05,0.65) circle (0.3pt);
\draw [rotate around={40:(-1.2,0)}][fill=black] (-0.78,0.4) circle (0.3pt);
\draw [rotate around={40:(-1.2,0)}][fill=black] (-0.7,0.3) circle (0.3pt);
\draw [rotate around={40:(-1.2,0)}][fill=black] (-0.65,0.18) circle (0.3pt);
\draw [rotate around={-100:(-1.2,0)}][fill=black] (-0.78,0.4) circle (0.3pt);
\draw [rotate around={-100:(-1.2,0)}][fill=black] (-0.7,0.3) circle (0.3pt);
\draw [rotate around={-100:(-1.2,0)}][fill=black] (-0.65,0.18) circle (0.3pt);
\draw [fill=black] (2.75,0.15) circle (0.3pt);
\draw [fill=black] (2.8,0) circle (0.3pt);
\draw [fill=black] (2.75,-0.15) circle (0.3pt);
\draw [fill=black] (1.58,1.18) circle (0.3pt);
\draw [fill=black] (1.55,1.35) circle (0.3pt);
\draw [fill=black] (1.45,1.5) circle (0.3pt);
\draw [fill=black] (1.58,-1.18) circle (0.3pt);
\draw [fill=black] (1.55,-1.35) circle (0.3pt);
\draw [fill=black] (1.45,-1.5) circle (0.3pt);
\draw [fill=black] (0.7,0.3) circle (0.3pt);
\draw [fill=black] (0.8,0.45) circle (0.3pt);
\draw [fill=black] (0.95,0.55) circle (0.3pt);
\draw [rotate around={90:(1.2,0)}][fill=black] (0.7,0.3) circle (0.3pt);
\draw [rotate around={90:(1.2,0)}][fill=black] (0.8,0.45) circle (0.3pt);
\draw [rotate around={90:(1.2,0)}][fill=black] (0.9,0.55) circle (0.3pt);
\draw (1.2,0.25) node[anchor=north] {$i_p$};
\draw (-1.2,0.25) node[anchor=north] {$j_q$};
\draw (2.4,0.25) node[anchor=north] {$\scriptstyle{i'_{\rho_{1}\dprimeind}}$};
\draw (1.2,-0.95) node[anchor=north] {$\scriptstyle{i'_{\rho_{2}\dprimeind}}$};
\draw (1.27,1.5) node[anchor=north] {$\scriptstyle{i'_{\rho_{n-k'-1}\dprimeindtl}}$};
\draw (0,0.25) node[anchor=north] {$\scriptstyle{j'_{\rho'_{v}}}$};
\draw (-0.25,-0.55) node[anchor=north] {$\scriptstyle{j'_{\rho'_{v+1}}}$};
\draw (-0.25,1.15) node[anchor=north] {$\scriptstyle{j'_{\rho'_{v-1}}}$};
\end{tikzpicture}
\end{small}
\end{minipage}
\noindent respectively, where the bold arrow is placed such that $j'_{\rho_1\dprimeind}$ (resp. $i'_{\rho'_1}$) is the first (resp. after $i_p$) in the application order of diagrams of the form $\mathcalboondox{D_2}$ (resp. $\mathcalboondox{D_2'}$).

Thus, the term $\sum\mathcal{E}(\mathcalboondox{D_1})$ in $\Phi^{\upsilon\dprime}_{l_{j_q}}
 \big(\Psi^{\upsilon'}_{l_{i_p}}(l_{i_k},\dots,\Hat{l_{i_p}},\dots,l_{i_1}),l_{j_{n-k}},\dots,\Hat{l_{j_q}},\dots,l_{j_{1}}\big)$ indexed by $r$ cancels with the term $\sum\mathcal{E}(\mathcalboondox{D_2})$ in $\Psi^{\rho\dprime}_{l_{j'_{q'}}}
 \big(\Phi^{\rho'}_{l_{i'_{p'}}}(l_{i'_{k'}},\dots,\Hat{l_{i'_{p'}}},\dots,l_{i'_1}),l_{j'_{n-k'}},\dots,\Hat{l_{j'_{q'}}},\dots,l_{j'_{1}}\big)$ indexed by $u$ when $k'=n-k$, $p'=q$, $q'=p$, $\upsilon'=\rho\dprime$, $\upsilon\dprime=\rho'$ and for every $t\in\llbracket 1,k'\rrbracket$, $t'\in \llbracket 1,n-k'\rrbracket$, $i'_t=j_t$ and $j'_{t'}=i_{t'}$.

Indeed, it is straightforward to check that

\begin{small}
\begin{equation}
    \begin{split}
        &k+\Delta'+\eta_1+\delta^{\upsilon'}(l_{i_k},\dots,\Hat{l_{i_p}},\dots,l_{i_1})+\eta_1'+\delta^{\upsilon\dprime}(l_{j_{n-k}},\dots,\Hat{l_{j_q}},\dots,l_{j_{1}})+\eta_1\dprimeind
        +k'\\&+\Delta'_2+\eta_2+\delta^{\rho'}(l_{i'_k},\dots,\Hat{l_{i'_{p'}}},\dots,l_{i'_1})+\eta'_2+\delta^{\rho\dprime}(l_{j'_{n-k}},\dots,\Hat{l_{j'_{q'}}},\dots,l_{j'_{1}})+\eta_2\dprimeind=0\mod 2.
    \end{split}
\end{equation}
\end{small}
Similarly, the term $\sum\mathcal{E}(\mathcalboondox{D'_1})$ in $\Phi^{\upsilon\dprime}_{l_{j_q}}
 \big(\Psi^{\upsilon'}_{l_{i_p}}(l_{i_k},\dots,\Hat{l_{i_p}},\dots,l_{i_1}),l_{j_{n-k}},\dots,\Hat{l_{j_q}},\dots,l_{j_{1}}\big)$ indexed by $s$ cancels with the term $\sum\mathcal{E}(\mathcalboondox{D'_2})$ in $\Psi^{\rho\dprime}_{l_{j'_{q'}}}
 \big(\Phi^{\rho'}_{l_{i'_{p'}}}(l_{i'_k},\dots,\Hat{l_{i'_{p'}}},\dots,l_{i'_1}),l_{j'_{n-k}},\dots,\Hat{l_{j'_{q'}}},\dots,l_{j'_{1}}\big)$ indexed by $v$ under the same conditions as before.

The case where $l_i\in\mathfrak{g}_{\MB}$ and $l_j\in\mathfrak{g}_{\MA}$ is completely similar. 
Therefore, $(\mathcal{L}_d^{\Phi_0}(\MA,\MB),\ell)$ is a graded $L_{\infty}$-algebra.

We now show that the Maurer-Cartan elements of this graded $L_{\infty}$-algebra are equivalent to $d$-pre-Calabi-Yau morphisms. By Lemma \ref{lemma:fin-diag}, a Maurer-Cartan element of $(\mathcal{L}_d^{\Phi_0}(\MA,\MB),\ell)$ is simply a degree $1$ element $L\in\mathcal{L}_d^{\Phi_0}(\MA,\MB)$ satisfying the Maurer-Cartan equation $\sum_{n\geq 0}\frac{1}{n!}\ell^n(L,\dots,L)^{\doubar{x}}=0$ for every $\doubar{x}\in\doubar{\MO}_{\MA}$.

We thus consider a Maurer-Cartan element of $(\mathcal{L}_d^{\Phi_0}(\MA,\MB),\ell)$, \textit{i.e.} a triple $(g_{\MB},s_{-1}l,g_{\MA})$ of degree $1$ such that 
\begin{itemize}
    \item $[g_{\MB},g_{\MB}]_{\nec}=0$,
    \item $\sum\limits_{n=2}^{\infty}\big(\ell^{n}(s_{-1}l,\dots,s_{-1}l,g_{\MB},\underbrace{s_{-1}l,\dots,s_{-1}l}_{i \text{ times }})+\ell^{n}(s_{-1}l,\dots,s_{-1}l,g_{\MA},\underbrace{s_{-1}l,\dots,s_{-1}l}_{i \text{ times }})\big)^{\doubar{x}}=0$ for every $\doubar{x}\in\doubar{\MO}_{\MA}$,
    \item $[g_{\MA},g_{\MA}]_{\nec}=0$.
\end{itemize}
Therefore, we conclude that $g_{\MB}$ and $g_{\MA}$ are Maurer-Cartan elements of the graded Lie algebras $(\Multi_d^{\bullet}(\MB)^{C_{\llg(\bullet)}}[d+1],[-,-]_{\nec})$ and $(\Multi_d^{\bullet}(\MA)^{C_{\llg(\bullet)}}[d+1],[-,-]_{\nec})$ respectively and thus define $d$-pre-Calabi-Yau structures on $\MB$ and $\MA$.
Moreover, we have that 
\begin{equation}
    \begin{split}
        &\sum\limits_{n=2}^{\infty}\sum\limits_{i=1}^n\frac{1}{n!}\ell^{n}(s_{-1}l,\dots,s_{-1}l,g_{\MB},\underbrace{s_{-1}l,\dots,s_{-1}l}_{i \text{ times }})^{\doubar{x}}
        =\sum\limits_{n=2}^{\infty}\frac{1}{(n-1)!}\ell^{n}(g_{\MB},s_{-1}l,\dots,s_{-1}l)^{\doubar{x}}
        \\&=\sum\limits_{n=2}^{\infty}\frac{1}{(n-1)!}
        \sum\limits_{\sigma\in\mathfrak{S}_{n-1}}(-1)^{\delta^{\sigma}(l,\dots,l)}s_{-1}\Phi_{g_{\MB}}^{\sigma}(l,\dots,l)^{\doubar{x}}\\&=\sum\limits_{n=2}^{\infty}s_{-1}\Phi_{g_{\MB}}^{\id}(l,\dots,l)^{\doubar{x}}
    \end{split}
\end{equation}
and similarly, 
\begin{equation}
    \begin{split}
        &\sum\limits_{n=2}^{\infty}\sum\limits_{i=1}^n\frac{1}{n!}\ell^{n}(s_{-1}l,\dots,s_{-1}l,g_{\MA},\underbrace{s_{-1}l,\dots,s_{-1}l}_{i \text{ times }})^{\doubar{x}}=-\sum\limits_{n=2}^{\infty}s_{-1}\Psi_{g_{\MA}}^{\id}(l,\dots,l)^{\doubar{x}}
    \end{split}
\end{equation}
for every $\doubar{x}\in\doubar{\MO}_{\MA}$. Therefore,
\[
\sum\limits_{n=2}^{\infty}\sum\limits_{i=1}^n\frac{1}{n!}\big(\ell^{n}(s_{-1}l,\dots,s_{-1}l,g_{\MB},\underbrace{s_{-1}l,\dots,s_{-1}l}_{i \text{ times }})+\ell^{n}(s_{-1}l,\dots,s_{-1}l,g_{\MA},\underbrace{s_{-1}l,\dots,s_{-1}l}_{i \text{ times }})\big)^{\doubar{x}}=0
\]
for every $\doubar{x}\in\doubar{\MO}_{\MA}$
is tantamount to 
\begin{equation}
    \begin{split}
        \sum\limits_{n=2}^{\infty}\Phi_{g_{\MB}}^{\id}(l,\dots,l)^{\doubar{x}}=\sum\limits_{n=2}^{\infty}\Psi_{g_{\MA}}^{\id}(l,\dots,l)^{\doubar{x}}
    \end{split}
\end{equation}
for every $\doubar{x}\in\doubar{\MO}_{\MA}$, which is tantamount to say that $l$ is a $d$-pre-Calabi-Yau morphism between the $d$-pre-Calabi-Yau categories $(\MA,g_{\MA})$ and $(\MB,g_{\MB})$.
On the other hand, $d$-pre-Calabi-Yau structures $s_{d+1}M_{\MB}$ and $s_{d+1}M_{\MA}$ together with a $d$-pre-Calabi-Yau morphism $(\Phi_0,s_{d+1}\mathbf{F})$ between them give rise to a Maurer-Cartan element $(s_{d+1}M_{\MB},s_{-1}(s_{d+1}\mathbf{F}),s_{d+1}M_{\MA})\in\mathcal{L}_d^{\Phi_0}(\MA,\MB)$.
\end{proof}

\begin{remark}
\label{remark:A-inf-MC}
    A particular case of this result is the case of $A_{\infty}$-structures recalled from \cite{borisov1} in Proposition \ref{prop:A-inf-L-inf-borisov}.
\end{remark}

\subsubsection{The case of fixed pre-Calabi-Yau structures}
\label{subsection:fixed-pCY}
As in Section \ref{section:A-inf-htpy} we now define a graded $L_{\infty}$-algebra whose Maurer-Cartan elements are in correspondence with $d$-pre-Calabi-Yau morphisms between fixed $d$-pre-Calabi-Yau categories. Throughout this subsection, we thus consider graded quivers $\MA$ and $\MB$ with respective sets of objects $\MO_{\MA}$ and $\MO_{\MB}$ together with a map $\Phi_0 : \MO_{\MA}\rightarrow\MO_{\MB}$ and $d$-pre-Calabi-Yau structures denoted by $s_{d+1}M_{\MA}$ and $s_{d+1}M_{\MB}$ respectively.

\begin{definition}
\label{def:bar-ell}
    Consider the graded vector space $\mathfrak{h}=\Multi_{d,\Phi_0}^{\bullet}(\MA,\MB)^{C_{\llg(\bullet)}}[d+1]$.
   We define an $\infty$-bracket $\Bar{\ell}$ on it by 
\begin{equation}
\begin{split}
    \Bar{\ell}^{n}(s_{-1}f_{n},\dots,s_{-1}f_{1})&=(-1)^{\epsilon}\sum\limits_{\sigma\in\mathfrak{S}_n}(-1)^{\delta^{\sigma}(f_{n},\dots,f_{1})}s_{-1}\Phi_{s_{d+1}M_{\MB}}^{\sigma}(f_{n},\dots,f_{1})
    \\&\phantom{=}-(-1)^{\epsilon'}\sum\limits_{\sigma\in\mathfrak{S}_n}(-1)^{\delta^{\sigma}(f_{n},\dots,f_{1})}s_{-1}\Psi_{s_{d+1}M_{\MA}}^{\sigma}(f_{n},\dots,f_{1})
\end{split}
\end{equation}
for $n\in\NN^*$ where 
\begin{small}
\begin{equation}
    \begin{split}
    \epsilon=\sum\limits_{j=1}^{n}(j-1)|f_j|
    \text{, }\epsilon'&=\sum\limits_{j=1}^n|f_j|+\sum\limits_{j=1}^{n}(j-1)|f_j|
    \end{split}
\end{equation}
\end{small}
and where $\Phi_{s_{d+1}M_{\MB}}^{\sigma}(f_{n},\dots,f_{1}), \Psi_{s_{d+1}M_{\MA}}^{\sigma}(f_{n},\dots,f_{1})\in\mathfrak{h}$ are given by $\Phi_{s_{d+1}M_{\MB}}^{\sigma}(f_{n},\dots,f_{1})^{\doubar{x}}=\sum\mathcal{E}(\mathcalboondox{D})$ and $\Psi_{s_{d+1}M_{\MA}}^{\sigma}(f_{n},\dots,f_{1})^{\doubar{x}}=\sum\mathcal{E}(\mathcalboondox{D'})$ where the sums are over all the filled diagrams $\mathcalboondox{D}$ and $\mathcalboondox{D'}$ of type $\doubar{x}$ and of the form

\begin{equation}
\label{eq:form-diag-2}
\begin{tikzpicture}[line cap=round,line join=round,x=1.0cm,y=1.0cm]
\clip(-3.5,-1.5) rectangle (4,1.8);
  \draw [line width=0.5pt] (0.,0.) circle (0.5cm);
     \draw [rotate=45] [<-,>=stealth,] (0.5,0)--(0.9,0);
     \draw [rotate=0] [->,>=stealth] (0.5,0)--(0.9,0);
     \draw [rotate=180] [->,>=stealth] (0.5,0)--(0.9,0);
     \draw [rotate=135] [<-,>=stealth] (0.5,0)--(0.9,0);
     \draw [rotate=90] [<-,>=stealth] (0.5,0)--(0.9,0);
     \draw [line width=0.5pt] (0.,1.15) circle (0.25cm);
     \draw [rotate around={150:(0,1.15)}] [->,>=stealth,] (0.25,1.15)--(0.55,1.15);
     \draw [rotate around={30:(0,1.15)}] [->,>=stealth,] (0.25,1.15)--(0.55,1.15);
     \draw [rotate around={90:(0,1.15)},shift={((0.25,1.15))}] \doublefleche;
     \draw [rotate around={210:(0,1.15)},shift={((0.25,1.15))}] \doublefleche;
     \draw [line width=0.5pt] (0.81,0.81) circle (0.25cm);
     \draw [rotate around={-15:(0.81,0.81)}] [->,>=stealth,] (1.06,0.81)--(1.36,0.81);
     \draw [rotate around={105:(0.81,0.81)}] [->,>=stealth,] (1.06,0.81)--(1.36,0.81);
     \draw [rotate around={165:(0.81,0.81)},shift={(1.06,0.81)}] \doublefleche;
     \draw [rotate around={45:(0.81,0.81)},shift={(1.06,0.81)}] \doublefleche;
     \draw [line width=0.5pt] (-0.81,0.81) circle (0.25cm);
     \draw [rotate around={75:(-0.81,0.81)}] [->,>=stealth,] (-0.56,0.81)--(-0.26,0.81);
     \draw [rotate around={195:(-0.81,0.81)}] [->,>=stealth,] (-0.56,0.81)--(-0.26,0.81);
     \draw [rotate around={135:(-0.81,0.81)},shift={(-0.56,0.81)}]\doublefleche;
     \draw [rotate around={255:(-0.81,0.81)},shift={(-0.56,0.81)}]\doublefleche;
\begin{scriptsize}
\draw [fill=black] (0.7,0.5) circle (0.3pt);
\draw [fill=black] (0.88,0.5) circle (0.3pt);
\draw [fill=black] (1.05,0.57) circle (0.3pt);
\draw [fill=black] (0.27,1) circle (0.3pt);
\draw [fill=black] (0.16,0.87) circle (0.3pt);
\draw [fill=black] (0.3,1.15) circle (0.3pt);
\draw [fill=black] (-0.55,1) circle (0.3pt);
\draw [fill=black] (-0.5,0.85) circle (0.3pt);
\draw [fill=black] (-0.52,0.7) circle (0.3pt);
\draw [fill=black] (0,-0.6) circle (0.3pt);
\draw [fill=black] (0.2,-0.55) circle (0.3pt);
\draw [fill=black] (-0.2,-0.55) circle (0.3pt);
\draw [rotate=-45][fill=black] (-0.5,0.3) circle (0.3pt);
\draw [rotate=-45][fill=black] (-0.55,0.2) circle (0.3pt);
\draw [rotate=-45][fill=black] (-0.57,0.1) circle (0.3pt);
\draw [rotate=45][fill=black] (0.5,0.3) circle (0.3pt);
\draw [rotate=45][fill=black] (0.55,0.2) circle (0.3pt);
\draw [rotate=45][fill=black] (0.57,0.1) circle (0.3pt);
\end{scriptsize}
\draw (0,0.25)node[anchor=north]{$M_{\MB}$};
\draw (-0.81,1.06)node[anchor=north]{$\scriptstyle{f_{\sigma_v}}$};
\draw (0.81,1.06)node[anchor=north]{$\scriptstyle{f_{\sigma_u}}$};
\draw (0,1.4)node[anchor=north]{$\scriptstyle{f_{\sigma_1}}$};
\draw (3.5,0.25) node[anchor=north]{and};
\end{tikzpicture}
    \begin{tikzpicture}[line cap=round,line join=round,x=1.0cm,y=1.0cm]
\clip(-4,-1.5) rectangle (5.249408935174429,1.8);
     \draw [line width=0.5pt] (0.,0.) circle (0.5cm);
     \shadedraw[rotate=30,shift={(0.5cm,0cm)}] \doublefleche;
     \shadedraw[rotate=150,shift={(0.5cm,0cm)}] \doublefleche;
     \draw [line width=0.5pt] (0,1.3) circle (0.3cm);
     \shadedraw [shift={(0cm,1cm)},rotate=-90] \doubleflechescindeeleft;
     \shadedraw [shift={(0cm,1cm)},rotate=-90] \doubleflechescindeeright;
     \shadedraw [shift={(0cm,1cm)},rotate=-90] \fleche;
     \draw [->,>=stealth] (0.3,1.3)--(0.6,1.3);
     \draw [->,>=stealth] (-0.3,1.3)--(-0.6,1.3);
     \draw [line width=0.5pt] (1.12,-0.65) circle (0.3cm);
     \shadedraw[shift={(0.86cm,-0.5cm)},rotate=150] \doubleflechescindeeleft;
     \shadedraw[shift={(0.86cm,-0.5cm)},rotate=150] \doubleflechescindeeright;
     \shadedraw[shift={(0.86cm,-0.5cm)},rotate=150] \fleche;
     \draw [rotate around ={60:(1.12,-0.65)}] [->,>=stealth] (1.43,-0.65)--(1.73,-0.65);
     \draw [rotate around ={-120:(1.12,-0.65)}] [->,>=stealth] (1.43,-0.65)--(1.73,-0.65);
     \draw [line width=0.5pt] (-1.12,-0.65) circle (0.3cm);
     \shadedraw[shift={(-0.86cm,-0.5cm)},rotate=30] \doubleflechescindeeleft;
      \shadedraw[shift={(-0.86cm,-0.5cm)},rotate=30] \doubleflechescindeeright;
       \shadedraw[shift={(-0.86cm,-0.5cm)},rotate=30] \fleche;
      \draw [rotate around ={-60:(-1.12,-0.65)}] [->,>=stealth] (-1.43,-0.65)--(-1.73,-0.65);
     \draw [rotate around ={120:(-1.12,-0.65)}] [->,>=stealth] (-1.43,-0.65)--(-1.73,-0.65);
\begin{scriptsize}
\draw [fill=black] (0,1.7) circle (0.3pt);
\draw [fill=black] (0.2,1.65) circle (0.3pt);
\draw [fill=black] (-0.2,1.65) circle (0.3pt);
\draw [fill=black] (0,-0.6) circle (0.3pt);
\draw [fill=black] (0.2,-0.55) circle (0.3pt);
\draw [fill=black] (-0.2,-0.55) circle (0.3pt);

\draw [fill=black] (1.45,-0.85) circle (0.3pt);
\draw [fill=black] (1.5,-0.67) circle (0.3pt);
\draw [fill=black] (1.33,-0.97) circle (0.3pt);
\draw [fill=black] (-1.45,-0.85) circle (0.3pt);
\draw [fill=black] (-1.5,-0.67) circle (0.3pt);
\draw [fill=black] (-1.33,-0.97) circle (0.3pt);
\end{scriptsize}
\draw (0,0.25)node[anchor=north]{$M_{\MA}$};
\draw (0,1.55)node[anchor=north]{$\scriptstyle{f_{\sigma_1}}$};
\draw (-1.12,-0.4)node[anchor=north]{$\scriptstyle{f_{\sigma_n}}$};
\draw (1.12,-0.4)node[anchor=north]{$\scriptstyle{f_{\sigma_2}}$};
\end{tikzpicture}
\end{equation}
\noindent respectively.
\end{definition}

\begin{lemma}
    Given $(i,j)\in \NN^2$, there exists a finite number of diagrams of the forms \eqref{eq:form-diag-2} and of type $\doubar{x}$ where $\llg(\doubar{x})=j$ and $N(\doubar{x})=i$. In particular, given $F\in\mathcal{L}_d^{\Phi_0}(\MA,\MB)^1$, the projection $\pi_{\doubar{x}}\circ \sum_{n\geq 0}\bar{\ell}^n(F,\dots,F)$ on 
    $\Multi_{d,\Phi_0}^{\doubar{x}}(\MA,\MB)^{\llg(\doubar{x})}[d+1]$ is a finite sum for every $\doubar{x}\in\doubar{\MO}_{\MA}$.
\end{lemma}
\begin{proof}
    The proof is completely similar as the proof of Lemma \ref{lemma:fin-diag}.
\end{proof}

\begin{proposition}
    $(\mathcal{L}_d^{\Phi_0}(\MA,\MB),\bar{\ell})$ is a graded $L_{\infty}$-algebra whose Maurer-Cartan elements are in correspondence with $d$-pre-Calabi-Yau morphisms $(\Phi_0,s_{d+1}\mathbf{F}) : (\MA,s_{d+1}M_{\MA})\rightarrow (\MB,s_{d+1}M_{\MB})$.
\end{proposition}
\begin{proof}
    We first show that $\Bar{\ell}$ is anti-symmetric. Consider $n\geq 2$, $i\in\llbracket 1,n\rrbracket$ and $l_1,\dots,l_n\in\mathfrak{h}$. Then, we have that 
    \begin{equation}
        \begin{split}
            &\Bar{\ell}^n(s_{-1}l_n,\dots,s_{-1}l_i,s_{-1}l_{i+1},\dots,s_{-1}l_1)
            \\&=(-1)^{\eta}\sum\limits_{\sigma\in\mathfrak{S}_n}(-1)^{\delta^{\sigma}(l_{n},\dots,l_i,l_{i+1},\dots,l_{1})}s_{-1}\Phi_{s_{d+1}M_{\MB}}^{\sigma}(l_{n},\dots,l_i,l_{i+1},\dots,l_{1})
    \\&\phantom{=}-(-1)^{\eta'}\sum\limits_{\sigma\in\mathfrak{S}_n}(-1)^{\delta^{\sigma}(l_{n},\dots,l_i,l_{i+1},\dots,l_{1})}s_{-1}\Psi_{s_{d+1}M_{\MA}}^{\sigma}(l_{n},\dots,l_i,l_{i+1},\dots,l_{1})
        \end{split}
    \end{equation}
    where 
    \begin{small}
    \begin{equation}
        \begin{split}
            \eta=\sum\limits_{\substack{j=1\\j\neq i,i+1}}^{n}(j-1)|l_j|+i|l_i|+(i-1)|l_{i+1}|\text{, } \eta'=\sum\limits_{j=1}^n|l_j|+\sum\limits_{\substack{j=1\\j\neq i,i+1}}^{n}(j-1)|l_j|+i|l_i|+(i-1)|l_{i+1}|
        \end{split}
    \end{equation}
    \end{small}
    which gives that
    \begin{equation}
        \begin{split}
            &\Bar{\ell}^n(s_{-1}l_n,\dots,s_{-1}l_i,s_{-1}l_{i+1},\dots,s_{-1}l_1)
            \\&=(-1)^{\epsilon+|l_i|+|l_{i+1}|}\sum\limits_{\sigma\in\mathfrak{S}_n}(-1)^{\delta^{\sigma}(l_{n},\dots,l_{1})+|l_i||l_{i+1}|}s_{-1}\Phi_{s_{d+1}M_{\MB}}^{\sigma}(l_{n},\dots,l_{1})
    \\&\phantom{=}-(-1)^{\epsilon'+|l_i|+|l_{i+1}|}\sum\limits_{\sigma\in\mathfrak{S}_n}(-1)^{\delta^{\sigma}(l_{n},\dots,l_{1})+|l_i||l_{i+1}|}s_{-1}\Psi_{s_{d+1}M_{\MA}}^{\sigma}(l_{n},\dots,l_{1})
    \\&=-(-1)^{|s_{-1}l_i||s_{-1}l_{i+1}|}\Bar{\ell}^n(s_{-1}l_n,\dots,s_{-1}l_1)
        \end{split}
    \end{equation}
    which means that $\Bar{\ell}$ is anti-symmetric.
    
    We now prove that $\Bar{\ell}$ satisfies the higher Jacobi identities. 
    We have that 
    \begin{equation}
        \begin{split}
            &\sum\limits_{k=1}^n(-1)^{k}\sum\limits_{(\bar{i},\bar{j})\in\mathcal{P}_k^n}(-1)^{\Delta} \Bar{\ell}^{n-k+1}\big(\Bar{\ell}^k(s_{-1}l_{i_k},\dots,s_{-1}l_{i_1}),s_{-1}l_{j_{n-k}},\dots,s_{-1}l_{j_{1}}\big)
            \\&=\sum\limits_{k=1}^n(-1)^{k}\sum\limits_{(\bar{i},\bar{j})\in\mathcal{P}_k^n}(-1)^{\Delta+\epsilon_1}\sum\limits_{\sigma\in\mathfrak{S}_k}(-1)^{\delta^{\sigma}(l_{i_k},\dots,l_{i_1})} 
            \\&\hskip5cm\Bar{\ell}^{n-k+1}\big(s_{-1}\Phi^{\sigma}_{s_{d+1}M_{\MB}}(l_{i_k},\dots,l_{i_1}),s_{-1}l_{j_{n-k}},\dots,s_{-1}l_{j_{1}}\big)
            \\&\phantom{=}-\sum\limits_{k'=1}^n(-1)^{k'}\hskip-5mm\sum\limits_{(\bar{i}',\bar{j}')\in\mathcal{P}_{k'}^n}\hskip-3mm(-1)^{\Delta'+\eta_1}\hskip-2mm\sum\limits_{\tau\in\mathfrak{S}_{k'}}(-1)^{\delta^{\tau}(l_{i'_{k'}},\dots,l_{i'_1})} \\&\hskip5cm\Bar{\ell}^{n-k'+1}\big(s_{-1}\Psi^{\tau}_{s_{d+1}M_{\MA}}(l_{i'_{k'}},\dots,l_{i'_1}),s_{-1}l_{j'_{n-{k'}}},\dots,s_{-1}l_{j'_{1}}\big)
        \end{split}
    \end{equation}
    where
    \begin{small}
    \begin{equation}
        \begin{split}
            \epsilon_1&=\sum\limits_{r=1}^k(r-1)|l_{i_r}| \text{, } \eta_1=\sum\limits_{r=1}^{k'}(r-1)|l_{i'_r}|+\sum\limits_{r=1}^{k'}|l_{i'_r}|,
            \\ \Delta&=\sum\limits_{r=1}^k(|l_{i_r}|+1)\sum\limits_{s=i_r+1}^n(|l_s|+1)+\sum\limits_{r=1}^k(|l_{i_r}|+1)\sum\limits_{s=r+1}^k(|l_{i_s}|+1)+\sum\limits_{r=1}^k\sum\limits_{s=i_r+1}^n1+\sum\limits_{r=1}^k \sum\limits_{s=r+1}^k1,
            \\\Delta'&=\sum\limits_{r=1}^{k'}(|l_{i'_r}|+1)\sum\limits_{s=i'_r+1}^n(|l_s|+1)+\sum\limits_{r=1}^{k'}(|l_{i'_r}|+1)\sum\limits_{s=r+1}^{k'}(|l_{i'_s}|+1)+\sum\limits_{r=1}^{k'}\sum\limits_{s=i'_r+1}^n1+\sum\limits_{r=1}^{k'} \sum\limits_{s=r+1}^{k'}1.
        \end{split}
    \end{equation}
    \end{small}
    Moreover, we have that 
    \begin{small}
    \begin{equation}
    \label{eq:higher-Jac-bar-ell-1}
        \begin{split}
            &\sum\limits_{k=1}^n(-1)^{k}\hskip-3mm\sum\limits_{(\bar{i},\bar{j})\in\mathcal{P}_k^n}\hskip-2mm(-1)^{\Delta+\epsilon_1}\hskip-2mm\sum\limits_{\sigma\in\mathfrak{S}_k}(-1)^{\delta^{\sigma}(l_{i_k},\dots,l_{i_1})} \Bar{\ell}^{n-k+1}\big(s_{-1}\Phi^{\sigma}_{s_{d+1}M_{\MB}}(l_{i_k},\dots,l_{i_1}),s_{-1}l_{j_{n-k}},\dots,s_{-1}l_{j_{1}}\big)
            \\&=\sum\limits_{k=1}^n(-1)^{k}\sum\limits_{(\bar{i},\bar{j})\in\mathcal{P}_k^n}(-1)^{\Delta+\epsilon_1}\sum\limits_{\sigma\in\mathfrak{S}_k}(-1)^{\delta^{\sigma}(l_{i_k},\dots,l_{i_1})+\epsilon_2}\sum\limits_{\sigma'\in\mathfrak{S}_{n-k}}(-1)^{\delta^{\sigma'}(l_{j_{n-k}},\dots,l_{j_1})} \\&\hskip7cms_{-1}\Phi^{\sigma'}_{s_{d+1}M_{\MB}}\big(\Phi^{\sigma}_{s_{d+1}M_{\MB}}(l_{i_k},\dots,l_{i_1}),l_{j_{n-k}},\dots,l_{j_{1}}\big)
            \\&\phantom{=}-\sum\limits_{k=1}^n(-1)^{k}\sum\limits_{(\bar{i},\bar{j})\in\mathcal{P}_k^n}(-1)^{\Delta+\epsilon_1}\sum\limits_{\sigma\in\mathfrak{S}_k}(-1)^{\delta^{\sigma}(l_{i_k},\dots,l_{i_1})+\eta_2}\sum\limits_{\tau'\in\mathfrak{S}_{n-k}}(-1)^{\delta^{\tau'}(l_{j_{n-k}},\dots,l_{j_1})} \\&\hskip7cms_{-1}\Psi^{\tau'}_{s_{d+1}M_{\MA}}\big(\Phi^{\sigma}_{s_{d+1}M_{\MB}}(l_{i_k},\dots,l_{i_1}),l_{j_{n-k}},\dots,l_{j_{1}}\big)
        \end{split}
    \end{equation}
    \end{small}
    where 
    \begin{small}
    \begin{equation}
        \begin{split}
             \epsilon_2&=\sum\limits_{s=1}^{n-k}(s-1)|l_{j_s}|+(n-k)(\sum\limits_{r=1}^k|l_{i_r}|+1),
            \\ \eta_2&=\sum\limits_{s=1}^{n-k}(s-1)|l_{j_s}|+(n-k)(\sum\limits_{r=1}^k|l_{i_r}|+1)+\sum\limits_{r=1}^k|l_{i_r}|+\sum\limits_{s=1}^{n-k}|l_{j_s}|+1
        \end{split}
    \end{equation}
    \end{small}
    and where 
    \begin{equation}
        \begin{split}
            \Phi^{\sigma'}_{s_{d+1}M_{\MB}}\big(\Phi^{\sigma}_{s_{d+1}M_{\MB}}(l_{i_k},\dots,l_{i_1}),l_{j_{n-k}},\dots,l_{j_{1}}\big)&=\sum\limits_{u=1}^{n-k}(-1)^{(\sum\limits_{r=1}^k|l_{i_r}|+1)\sum\limits_{s=u+1}^{n-k}|l_{j_{\sigma'_s}}|}\sum\mathcal{E}(\mathcalboondox{D_1})
            \\&\phantom{=}+\sum\limits_{v=1}^k(-1)^{\sum\limits_{r=1}^v|l_{i_{\sigma_r}}|\sum\limits_{s=1}^{n-k}|l_{j_s}|}\sum\mathcal{E}(\mathcalboondox{D_1'})
            \\\Psi^{\tau'}_{s_{d+1}M_{\MA}}\big(\Phi^{\sigma}_{s_{d+1}M_{\MB}}(l_{i_k},\dots,l_{i_1}),l_{j_{n-k}},\dots,l_{j_{1}}\big)&=\sum\limits_{u=1}^{n-k}(-1)^{(\sum\limits_{r=1}^k|l_{i_r}|+1)\sum\limits_{s=u+1}^{n-k}|l_{j_{\tau'_s}}|}\sum\mathcal{E}(\mathcalboondox{D_2})
            \\&\phantom{=}+\sum\limits_{v=1}^k(-1)^{\sum\limits_{r=1}^v|l_{i_{\sigma_r}}|(\sum\limits_{s=1}^{n-k}|l_{j_s}|+1)}\sum\mathcal{E}(\mathcalboondox{D_2'})
        \end{split}
    \end{equation}
    where the sums are over all the filled diagrams $\mathcalboondox{D_1}$, $\mathcalboondox{D_1'}$, $\mathcalboondox{D_2}$ and $\mathcalboondox{D_2'}$ of the form

    \begin{minipage}{21cm}
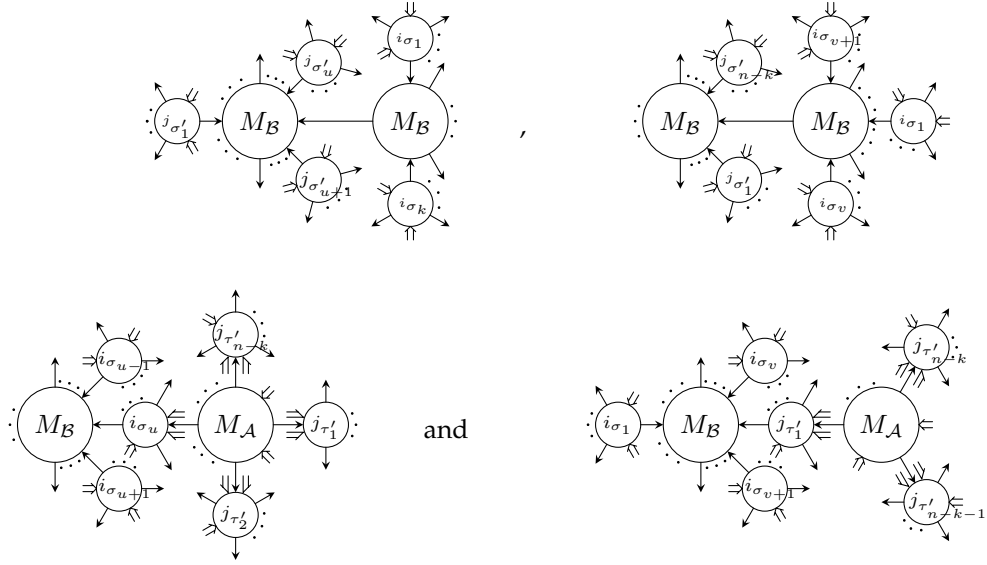
\begin{tikzpicture}[line cap=round,line join=round,x=1.0cm,y=1.0cm]
\clip(-4,-2) rectangle (4.5,2);
    \draw (0,0) circle (0.5cm);
    \draw [rotate=90][->,>=stealth] (0.5,0)--(0.9,0);
    \draw [rotate=-90][->,>=stealth] (0.5,0)--(0.9,0);
    \draw [rotate=45][<-,>=stealth] (0.5,0)--(0.8,0);
    \draw [rotate=-45][<-,>=stealth] (0.5,0)--(0.8,0);
    \draw (0.78,0.78) circle (0.3cm);
    \draw[rotate around={-15:(0.78,0.78)}][->,>=stealth] (1.08,0.78)--(1.38,0.78);
    \draw[rotate around={105:(0.78,0.78)}][->,>=stealth] (1.08,0.78)--(1.38,0.78);
    \draw[rotate around={45:(0.78,0.78)},shift={(1.08,0.78)}]\doublefleche;
    \draw[rotate around={165:(0.78,0.78)},shift={(1.08,0.78)}]\doublefleche;
     \draw (0.78,-0.78) circle (0.3cm);
    \draw[rotate around={15:(0.78,-0.78)}][->,>=stealth] (1.08,-0.78)--(1.38,-0.78);
    \draw[rotate around={-105:(0.78,-0.78)}][->,>=stealth] (1.08,-0.78)--(1.38,-0.78);
    \draw[rotate around={75:(0.78,-0.78)},shift={(1.08,-0.78)}]\doublefleche;
    \draw[rotate around={-165:(0.78,-0.78)},shift={(1.08,-0.78)}]\doublefleche;
    \draw (2,0) circle (0.5cm);
    \draw[rotate around={-60:(2,0)}][->,>=stealth](2.5,0)--(2.9,0);
    \draw[rotate around={60:(2,0)}][->,>=stealth](2.5,0)--(2.9,0);
    \draw[rotate around={180:(2,0)}][->,>=stealth](2.5,0)--(3.5,0);
    \draw[rotate around={-90:(2,0)}][<-,>=stealth](2.5,0)--(2.8,0);
    \draw (2,-1.1) circle (0.3cm);
    \draw[rotate around={-30:(2,-1.1)}][->,>=stealth](2.3,-1.1)--(2.6,-1.1);
    \draw[rotate around={-150:(2,-1.1)}][->,>=stealth](2.3,-1.1)--(2.6,-1.1);
    \draw[rotate around={-90:(2,-1.1)},shift={(2.3,-1.1)}]\doublefleche;
    \draw[rotate around={-210:(2,-1.1)},shift={(2.3,-1.1)}]\doublefleche;
    \draw[rotate around={90:(2,0)}][<-,>=stealth](2.5,0)--(2.8,0);
    \draw (2,1.1) circle (0.3cm);
    \draw[rotate around={30:(2,1.1)}][->,>=stealth](2.3,1.1)--(2.6,1.1);
    \draw[rotate around={150:(2,1.1)}][->,>=stealth](2.3,1.1)--(2.6,1.1);
    \draw[rotate around={90:(2,1.1)},shift={(2.3,1.1)}]\doublefleche;
    \draw[rotate around={210:(2,1.1)},shift={(2.3,1.1)}]\doublefleche;
    \draw[rotate=180][<-,>=stealth](0.5,0)--(0.8,0);
    \draw(-1.1,0) circle (0.3cm);
    \draw[rotate around={120:(-1.1,0)}][->,>=stealth](-0.8,0)--(-0.5,0);
    \draw[rotate around={-120:(-1.1,0)}][->,>=stealth](-0.8,0)--(-0.5,0);
    \draw[rotate around={60:(-1.1,0)},shift={(-0.8,0)}]\doublefleche;
    \draw[rotate around={-60:(-1.1,0)},shift={(-0.8,0)}]\doublefleche;
\draw (0,0.25)node[anchor=north]{$M_{\MB}$};
\draw (2,0.25)node[anchor=north]{$M_{\MB}$};
\draw (-1.1,0.2)node[anchor=north]{$\scriptscriptstyle{j_{\sigma'_1}}$};
\draw (2,1.3)node[anchor=north]{$\scriptscriptstyle{i_{\sigma_{1}}}$};
\draw (0.8,1)node[anchor=north]{$\scriptscriptstyle{j_{\sigma'_{u}}}$};
\draw (2.05,-0.9)node[anchor=north]{$\scriptscriptstyle{i_{\sigma_{k}}}$};
\draw (0.9,-0.55)node[anchor=north]{$\scriptstyle{j_{\sigma'_{u+1}}}$};
\draw[rotate=45][fill=black] (0.5,0.35) circle (0.3pt);
\draw[rotate=45][fill=black] (0.55,0.25) circle (0.3pt);
\draw[rotate=45][fill=black] (0.58,0.15) circle (0.3pt);
\draw[rotate=-90][fill=black] (0.5,0.35) circle (0.3pt);
\draw[rotate=-90][fill=black] (0.55,0.25) circle (0.3pt);
\draw[rotate=-90][fill=black] (0.58,0.15) circle (0.3pt);
\draw[fill=black] (-0.4,0.45) circle (0.3pt);
\draw[fill=black] (-0.5,0.3) circle (0.3pt);
\draw[fill=black] (-0.25,0.55) circle (0.3pt);
\draw[rotate=90][fill=black] (-0.4,0.45) circle (0.3pt);
\draw[rotate=90][fill=black] (-0.5,0.3) circle (0.3pt);
\draw[rotate=90][fill=black] (-0.25,0.55) circle (0.3pt);
\draw[fill=black] (-1.5,0) circle (0.3pt);
\draw[fill=black] (-1.45,0.15) circle (0.3pt);
\draw[fill=black] (-1.45,-0.15) circle (0.3pt);
\draw[fill=black] (-1.5,0) circle (0.3pt);
\draw[fill=black] (0.7,0.4) circle (0.3pt);
\draw[fill=black] (0.88,0.4) circle (0.3pt);
\draw[fill=black] (1,0.5) circle (0.3pt);
\draw[fill=black] (1.05,-1.05) circle (0.3pt);
\draw[fill=black] (1.15,-0.9) circle (0.3pt);
\draw[fill=black] (0.9,-1.15) circle (0.3pt);
\draw[fill=black] (2.55,0.2) circle (0.3pt);
\draw[fill=black] (2.6,0) circle (0.3pt);
\draw[fill=black] (2.55,-0.2) circle (0.3pt);
\draw[fill=black] (2.25,-0.8) circle (0.3pt);
\draw[fill=black] (2.37,-0.95) circle (0.3pt);
\draw[fill=black] (2.38,-1.15) circle (0.3pt);
\draw[fill=black] (2.25,0.8) circle (0.3pt);
\draw[fill=black] (2.37,0.95) circle (0.3pt);
\draw[fill=black] (2.38,1.15) circle (0.3pt);
\draw (3.5,0) node[anchor=north]{,};
\end{tikzpicture}
    \begin{tikzpicture}[line cap=round,line join=round,x=1.0cm,y=1.0cm]
\clip(-1,-2) rectangle (4.2,2);
    \draw (0,0) circle (0.5cm);
    \draw [rotate=90][->,>=stealth] (0.5,0)--(0.9,0);
    \draw [rotate=-90][->,>=stealth] (0.5,0)--(0.9,0);
    \draw [rotate=45][<-,>=stealth] (0.5,0)--(0.8,0);
    \draw [rotate=-45][<-,>=stealth] (0.5,0)--(0.8,0);
    \draw (0.78,0.78) circle (0.3cm);
    \draw[rotate around={-15:(0.78,0.78)}][->,>=stealth] (1.08,0.78)--(1.38,0.78);
    \draw[rotate around={105:(0.78,0.78)}][->,>=stealth] (1.08,0.78)--(1.38,0.78);
    \draw[rotate around={45:(0.78,0.78)},shift={(1.08,0.78)}]\doublefleche;
    \draw[rotate around={165:(0.78,0.78)},shift={(1.08,0.78)}]\doublefleche;
     \draw (0.78,-0.78) circle (0.3cm);
    \draw[rotate around={15:(0.78,-0.78)}][->,>=stealth] (1.08,-0.78)--(1.38,-0.78);
    \draw[rotate around={-105:(0.78,-0.78)}][->,>=stealth] (1.08,-0.78)--(1.38,-0.78);
    \draw[rotate around={75:(0.78,-0.78)},shift={(1.08,-0.78)}]\doublefleche;
    \draw[rotate around={-165:(0.78,-0.78)},shift={(1.08,-0.78)}]\doublefleche;
    \draw (2,0) circle (0.5cm);
    \draw[rotate around={-60:(2,0)}][->,>=stealth](2.5,0)--(2.9,0);
    \draw[rotate around={60:(2,0)}][->,>=stealth](2.5,0)--(2.9,0);
    \draw[rotate around={180:(2,0)}][->,>=stealth](2.5,0)--(3.5,0);
    \draw[rotate around={-90:(2,0)}][<-,>=stealth](2.5,0)--(2.8,0);
    \draw (2,-1.1) circle (0.3cm);
    \draw[rotate around={-30:(2,-1.1)}][->,>=stealth](2.3,-1.1)--(2.6,-1.1);
    \draw[rotate around={-150:(2,-1.1)}][->,>=stealth](2.3,-1.1)--(2.6,-1.1);
    \draw[rotate around={-90:(2,-1.1)},shift={(2.3,-1.1)}]\doublefleche;
    \draw[rotate around={-210:(2,-1.1)},shift={(2.3,-1.1)}]\doublefleche;
    \draw[rotate around={90:(2,0)}][<-,>=stealth](2.5,0)--(2.8,0);
    \draw (2,1.1) circle (0.3cm);
    \draw[rotate around={30:(2,1.1)}][->,>=stealth](2.3,1.1)--(2.6,1.1);
    \draw[rotate around={150:(2,1.1)}][->,>=stealth](2.3,1.1)--(2.6,1.1);
    \draw[rotate around={90:(2,1.1)},shift={(2.3,1.1)}]\doublefleche;
    \draw[rotate around={210:(2,1.1)},shift={(2.3,1.1)}]\doublefleche;
    \draw[rotate around={0:(2,0)}][<-,>=stealth](2.5,0)--(2.8,0);
    \draw (3.1,0) circle (0.3cm);
    \draw[rotate around={60:(3.1,0)}][->,>=stealth](3.4,0)--(3.7,0);
    \draw[rotate around={-60:(3.1,0)}][->,>=stealth](3.4,0)--(3.7,0);
    \draw[rotate around={0:(3.1,0)},shift={(3.4,0)}]\doublefleche;
    \draw[rotate around={120:(3.1,0)},shift={(3.4,0)}]\doublefleche;
\draw (0,0.25)node[anchor=north]{$M_{\MB}$};
\draw (2,0.25)node[anchor=north]{$M_{\MB}$};
\draw (3.1,0.2)node[anchor=north]{$\scriptscriptstyle{i_{\sigma_1}}$};
\draw (2.1,1.35)node[anchor=north]{$\scriptscriptstyle{i_{\sigma_{v+1}}}$};
\draw (0.9,1)node[anchor=north]{$\scriptscriptstyle{j_{\sigma'_{n-k}}}$};
\draw (2.05,-0.9)node[anchor=north]{$\scriptscriptstyle{i_{\sigma_{v}}}$};
\draw (0.8,-0.55)node[anchor=north]{$\scriptscriptstyle{j_{\sigma'_1}}$};
\draw[rotate=45][fill=black] (0.5,0.35) circle (0.3pt);
\draw[rotate=45][fill=black] (0.55,0.25) circle (0.3pt);
\draw[rotate=45][fill=black] (0.58,0.15) circle (0.3pt);
\draw[rotate=-90][fill=black] (0.5,0.35) circle (0.3pt);
\draw[rotate=-90][fill=black] (0.55,0.25) circle (0.3pt);
\draw[rotate=-90][fill=black] (0.58,0.15) circle (0.3pt);
\draw[rotate around={-100:(2,0)}][fill=black] (1.6,0.45) circle (0.3pt);
\draw[rotate around={-100:(2,0)}][fill=black] (1.5,0.3) circle (0.3pt);
\draw[rotate around={-100:(2,0)}][fill=black] (1.75,0.55) circle (0.3pt);
\draw[rotate around={-160:(2,0)}][fill=black] (1.6,0.45) circle (0.3pt);
\draw[rotate around={-160:(2,0)}][fill=black] (1.5,0.3) circle (0.3pt);
\draw[rotate around={-160:(2,0)}][fill=black] (1.75,0.55) circle (0.3pt);
\draw[fill=black] (0.7,0.4) circle (0.3pt);
\draw[fill=black] (0.88,0.4) circle (0.3pt);
\draw[fill=black] (1,0.5) circle (0.3pt);
\draw[fill=black] (1.05,-1.05) circle (0.3pt);
\draw[fill=black] (1.15,-0.9) circle (0.3pt);
\draw[fill=black] (0.9,-1.15) circle (0.3pt);
\draw[fill=black] (-0.55,0.2) circle (0.3pt);
\draw[fill=black] (-0.6,0) circle (0.3pt);
\draw[fill=black] (-0.55,-0.2) circle (0.3pt);
\draw[fill=black] (2.25,-0.8) circle (0.3pt);
\draw[fill=black] (2.37,-0.95) circle (0.3pt);
\draw[fill=black] (2.38,-1.15) circle (0.3pt);
\draw[fill=black] (2.25,0.8) circle (0.3pt);
\draw[fill=black] (2.37,0.95) circle (0.3pt);
\draw[fill=black] (2.38,1.15) circle (0.3pt);
\draw[fill=black] (2.8,-0.25) circle (0.3pt);
\draw[fill=black] (2.93,-0.35) circle (0.3pt);
\draw[fill=black] (3.1,-0.38) circle (0.3pt);
\end{tikzpicture}
\end{minipage}
\begin{minipage}{21cm}
\begin{tikzpicture}[line cap=round,line join=round,x=1.0cm,y=1.0cm]
\clip(-3,-2) rectangle (5,2);
  \draw (0,0) circle (0.3cm);
    \draw[rotate=60][->,>=stealth] (0.3,0)--(0.7,0);
    \draw[rotate=180][->,>=stealth] (0.3,0)--(0.7,0);
    \draw[rotate=-60][->,>=stealth] (0.3,0)--(0.7,0);
    \draw[rotate=-120,shift={(0.3,0)}] \doublefleche;
    \draw[rotate=0,shift={(0.3,0)}] \doubleflechescindeeleft;
    \draw[rotate=0,shift={(0.3,0)}] \doubleflechescindeeright;
    \draw[rotate=0][<-,>=stealth] (0.3,0)--(0.7,0);
    \draw(1.2,0)circle (0.5cm);
    \draw[->,>=stealth] (1.7,0)--(2.1,0);
    \draw(2.4,0) circle (0.3cm);
    \draw[rotate around={180:(2.4,0)},shift={(2.7,0)}]\doubleflechescindeeleft;
    \draw[rotate around={180:(2.4,0)},shift={(2.7,0)}]\doubleflechescindeeright;
    \draw[rotate around={90:(2.4,0)}][->,>=stealth](2.7,0)--(3,0);
    \draw[rotate around={-90:(2.4,0)}][->,>=stealth](2.7,0)--(3,0);
    \draw[rotate around={45:(1.2,0)},shift={(1.7,0)}] \doublefleche;
    \draw[rotate around={-45:(1.2,0)},shift={(1.7,0)}] \doublefleche;
    \draw[rotate around={-90:(1.2,0)}][->,>=stealth] (1.7,0)--(2.1,0);
    \draw(1.2,-1.2)circle (0.3cm);
    \draw[rotate around={-150:(1.2,-1.2)},shift={(1.5,-1.2)}] \doublefleche;
    \draw[rotate around={90:(1.2,-1.2)},shift={(1.5,-1.2)}] \doubleflechescindeeleft;
    \draw[rotate around={90:(1.2,-1.2)},shift={(1.5,-1.2)}] \doubleflechescindeeright;
    \draw[rotate around={30:(1.2,-1.2)}][->,>=stealth] (1.5,-1.2)--(1.8,-1.2);
    \draw[rotate around={150:(1.2,-1.2)}][->,>=stealth] (1.5,-1.2)--(1.8,-1.2);
    \draw[rotate around={-90:(1.2,-1.2)}][->,>=stealth] (1.5,-1.2)--(1.8,-1.2);
    \draw[rotate around={90:(1.2,0)}][->,>=stealth] (1.7,0)--(2.1,0);
    \draw(1.2,1.2)circle (0.3cm);
    \draw[rotate around={150:(1.2,1.2)},shift={(1.5,1.2)}] \doublefleche;
    \draw[rotate around={-90:(1.2,1.2)},shift={(1.5,1.2)}] \doubleflechescindeeleft;
    \draw[rotate around={-90:(1.2,1.2)},shift={(1.5,1.2)}] \doubleflechescindeeright;
    \draw[rotate around={-30:(1.2,1.2)}][->,>=stealth] (1.5,1.2)--(1.8,1.2);
    \draw[rotate around={-150:(1.2,1.2)}][->,>=stealth] (1.5,1.2)--(1.8,1.2);
    \draw[rotate around={90:(1.2,1.2)}][->,>=stealth] (1.5,1.2)--(1.8,1.2);
    \draw(-1.2,0)circle (0.5cm);
    \draw [rotate around={-45:(-1.2,0)}][<-,>=stealth] (-0.7,0)--(-0.3,0);
    \draw [rotate around={45:(-1.2,0)}][<-,>=stealth] (-0.7,0)--(-0.3,0);
    \draw [rotate around={90:(-1.2,0)}][->,>=stealth] (-0.7,0)--(-0.3,0);
    \draw [rotate around={-90:(-1.2,0)}][->,>=stealth] (-0.7,0)--(-0.3,0);
    \draw(-0.35,-0.85)circle (0.3cm);
    \draw [rotate around={-120:(-0.35,-0.85)}][->,>=stealth] (-0.05,-0.85)--(0.25,-0.85);
    \draw [rotate around={0:(-0.35,-0.85)}][->,>=stealth] (-0.05,-0.85)--(0.25,-0.85);
    \draw [rotate around={180:(-0.35,-0.85)},shift={(-0.05,-0.85)}]\doublefleche;
     \draw [rotate around={-60:(-0.35,-0.85)},shift={(-0.05,-0.85)}]\doublefleche;
     \draw(-0.35,0.85)circle (0.3cm);
    \draw [rotate around={120:(-0.35,0.85)}][->,>=stealth] (-0.05,0.85)--(0.25,0.85);
    \draw [rotate around={0:(-0.35,0.85)}][->,>=stealth] (-0.05,0.85)--(0.25,0.85);
    \draw [rotate around={180:(-0.35,0.85)},shift={(-0.05,0.85)}]\doublefleche;
     \draw [rotate around={60:(-0.35,0.85)},shift={(-0.05,0.85)}]\doublefleche;
\draw [fill=black] (-1.75,0.2) circle (0.3pt);
\draw [fill=black] (-1.8,0) circle (0.3pt);
\draw [fill=black] (-1.75,-0.2) circle (0.3pt);
\draw [fill=black] (-0.3,0.2) circle (0.3pt);
\draw [fill=black] (-0.2,0.3) circle (0.3pt);
\draw [fill=black] (-0.05,0.35) circle (0.3pt);
\draw [rotate around={-45:(-1.2,0)}][fill=black] (-0.3,0.2) circle (0.3pt);
\draw [rotate around={-45:(-1.2,0)}][fill=black] (-0.2,0.3) circle (0.3pt);
\draw [rotate around={-45:(-1.2,0)}][fill=black] (-0.05,0.35) circle (0.3pt);
\draw [fill=black] (-0.3,0.5) circle (0.3pt);
\draw [fill=black] (-0.15,0.54) circle (0.3pt);
\draw [fill=black] (-0.05,0.65) circle (0.3pt);
\draw [rotate around={40:(-1.2,0)}][fill=black] (-0.78,0.4) circle (0.3pt);
\draw [rotate around={40:(-1.2,0)}][fill=black] (-0.7,0.3) circle (0.3pt);
\draw [rotate around={40:(-1.2,0)}][fill=black] (-0.65,0.18) circle (0.3pt);
\draw [rotate around={-100:(-1.2,0)}][fill=black] (-0.78,0.4) circle (0.3pt);
\draw [rotate around={-100:(-1.2,0)}][fill=black] (-0.7,0.3) circle (0.3pt);
\draw [rotate around={-100:(-1.2,0)}][fill=black] (-0.65,0.18) circle (0.3pt);
\draw [fill=black] (2.75,0.15) circle (0.3pt);
\draw [fill=black] (2.8,0) circle (0.3pt);
\draw [fill=black] (2.75,-0.15) circle (0.3pt);
\draw [fill=black] (1.58,1.18) circle (0.3pt);
\draw [fill=black] (1.55,1.35) circle (0.3pt);
\draw [fill=black] (1.45,1.5) circle (0.3pt);
\draw [fill=black] (1.58,-1.18) circle (0.3pt);
\draw [fill=black] (1.55,-1.35) circle (0.3pt);
\draw [fill=black] (1.45,-1.5) circle (0.3pt);
\draw [fill=black] (0.7,0.3) circle (0.3pt);
\draw [fill=black] (0.8,0.45) circle (0.3pt);
\draw [fill=black] (0.95,0.55) circle (0.3pt);
\draw [rotate around={90:(1.2,0)}][fill=black] (0.7,0.3) circle (0.3pt);
\draw [rotate around={90:(1.2,0)}][fill=black] (0.8,0.45) circle (0.3pt);
\draw [rotate around={90:(1.2,0)}][fill=black] (0.9,0.55) circle (0.3pt);
\draw (1.2,0.25) node[anchor=north] {$M_{\MA}$};
\draw (-1.2,0.25) node[anchor=north] {$M_{\MB}$};
\draw (2.4,0.25) node[anchor=north] {$\scriptstyle{j_{\tau'_{1}}}$};
\draw (1.2,-0.95) node[anchor=north] {$\scriptstyle{j_{\tau'_{2}}}$};
\draw (1.3,1.45) node[anchor=north] {$\scriptstyle{j_{\tau'_{n-k}}}$};
\draw (0,0.25) node[anchor=north] {$\scriptstyle{i_{\sigma_{u}}}$};
\draw (-0.26,-0.6) node[anchor=north] {$\scriptstyle{i_{\sigma_{u+1}}}$};
\draw (-0.26,1.1) node[anchor=north] {$\scriptstyle{i_{\sigma_{u-1}}}$};
\draw (4,0.2) node[anchor=north] {and};
\end{tikzpicture}
    \begin{tikzpicture}[line cap=round,line join=round,x=1.0cm,y=1.0cm]
\clip(-3.5,-2) rectangle (5,2);
    \draw (0,0) circle (0.3cm);
    \draw[rotate=60][->,>=stealth] (0.3,0)--(0.7,0);
    \draw[rotate=180][->,>=stealth] (0.3,0)--(0.7,0);
    \draw[rotate=-60][->,>=stealth] (0.3,0)--(0.7,0);
    \draw[rotate=-120,shift={(0.3,0)}] \doublefleche;
    \draw[rotate=0,shift={(0.3,0)}] \doubleflechescindeeleft;
    \draw[rotate=0,shift={(0.3,0)}] \doubleflechescindeeright;
    \draw[rotate=0][<-,>=stealth] (0.3,0)--(0.7,0);
    \draw(1.2,0)circle (0.5cm);
    \draw[rotate around={0:(1.2,0)},shift={(1.7,0)}] \doublefleche;
    \draw[rotate around={-120:(1.2,0)},shift={(1.7,0)}] \doublefleche;
    \draw[rotate around={-60:(1.2,0)}][->,>=stealth] (1.7,0)--(2.1,0);
    \draw(1.8,1.03)circle (0.3cm);
    \draw[rotate around={120:(1.8,1.03)},shift={(2.1,1.03)}] \doublefleche;
    \draw[rotate around={-120:(1.8,1.03)},shift={(2.1,1.03)}] \doubleflechescindeeleft;
    \draw[rotate around={-120:(1.8,1.03)},shift={(2.1,1.03)}] \doubleflechescindeeright;
    \draw[rotate around={-60:(1.8,1.03)}][->,>=stealth] (2.1,1.03)--(2.4,1.03);
    \draw[rotate around={180:(1.8,1.03)}][->,>=stealth] (2.1,1.03)--(2.4,1.03);
    \draw[rotate around={60:(1.8,1.03)}][->,>=stealth] (2.1,1.03)--(2.4,1.03);
    \draw[rotate around={60:(1.2,0)}][->,>=stealth] (1.7,0)--(2.1,0);
    \draw(1.8,-1.03)circle (0.3cm);
    \draw[rotate around={0:(1.8,-1.03)},shift={(2.1,-1.03)}] \doublefleche;
    \draw[rotate around={120:(1.8,-1.03)},shift={(2.1,-1.03)}] \doubleflechescindeeright;
    \draw[rotate around={120:(1.8,-1.03)},shift={(2.1,-1.03)}] \doubleflechescindeeleft;
    \draw[rotate around={60:(1.8,-1.03)}][->,>=stealth] (2.1,-1.03)--(2.4,-1.03);
    \draw[rotate around={-60:(1.8,-1.03)}][->,>=stealth] (2.1,-1.03)--(2.4,-1.03);
    \draw[rotate around={180:(1.8,-1.03)}][->,>=stealth] (2.1,-1.03)--(2.4,-1.03);
    \draw(-1.2,0)circle (0.5cm);
    \draw [rotate around={-45:(-1.2,0)}][<-,>=stealth] (-0.7,0)--(-0.3,0);
    \draw [rotate around={45:(-1.2,0)}][<-,>=stealth] (-0.7,0)--(-0.3,0);
    \draw [rotate around={90:(-1.2,0)}][->,>=stealth] (-0.7,0)--(-0.3,0);
    \draw [rotate around={-90:(-1.2,0)}][->,>=stealth] (-0.7,0)--(-0.3,0);
    \draw(-0.35,-0.85)circle (0.3cm);
    \draw [rotate around={-120:(-0.35,-0.85)}][->,>=stealth] (-0.05,-0.85)--(0.25,-0.85);
    \draw [rotate around={0:(-0.35,-0.85)}][->,>=stealth] (-0.05,-0.85)--(0.25,-0.85);
    \draw [rotate around={180:(-0.35,-0.85)},shift={(-0.05,-0.85)}]\doublefleche;
     \draw [rotate around={-60:(-0.35,-0.85)},shift={(-0.05,-0.85)}]\doublefleche;
     \draw(-0.35,0.85)circle (0.3cm);
    \draw [rotate around={120:(-0.35,0.85)}][->,>=stealth] (-0.05,0.85)--(0.25,0.85);
    \draw [rotate around={0:(-0.35,0.85)}][->,>=stealth] (-0.05,0.85)--(0.25,0.85);
    \draw [rotate around={180:(-0.35,0.85)},shift={(-0.05,0.85)}]\doublefleche;
     \draw [rotate around={60:(-0.35,0.85)},shift={(-0.05,0.85)}]\doublefleche;
     \draw[rotate around={180:(-1.2,0)}][<-,>=stealth] (-0.7,0)--(-0.4,0);
     \draw (-2.3,0) circle (0.3cm);
     \draw[rotate around={120:(-2.3,0)}][->,>=stealth] (-2,0)--(-1.7,0);
     \draw[rotate around={-120:(-2.3,0)}][->,>=stealth] (-2,0)--(-1.7,0);
     \draw[rotate around={60:(-2.3,0)},shift={(-2,0)}]\doublefleche;
     \draw[rotate around={-60:(-2.3,0)},shift={(-2,0)}]\doublefleche;
\draw [fill=black] (-2.65,0.15) circle (0.3pt);
\draw [fill=black] (-2.7,0) circle (0.3pt);
\draw [fill=black] (-2.65,-0.15) circle (0.3pt);
\draw [fill=black] (-0.3,0.2) circle (0.3pt);
\draw [fill=black] (-0.2,0.3) circle (0.3pt);
\draw [fill=black] (-0.05,0.35) circle (0.3pt);
\draw [rotate around={-45:(-1.2,0)}][fill=black] (-0.3,0.2) circle (0.3pt);
\draw [rotate around={-45:(-1.2,0)}][fill=black] (-0.2,0.3) circle (0.3pt);
\draw [rotate around={-45:(-1.2,0)}][fill=black] (-0.05,0.35) circle (0.3pt);
\draw [fill=black] (-0.3,0.5) circle (0.3pt);
\draw [fill=black] (-0.15,0.54) circle (0.3pt);
\draw [fill=black] (-0.05,0.65) circle (0.3pt);
\draw [fill=black] (2.2,1.03) circle (0.3pt);
\draw [fill=black] (2.15,0.9) circle (0.3pt);
\draw [fill=black] (2.15,1.16) circle (0.3pt);
\draw [rotate around={-120:(1.2,0)}][fill=black] (2.2,1.03) circle (0.3pt);
\draw [rotate around={-120:(1.2,0)}][fill=black] (2.15,0.9) circle (0.3pt);
\draw [rotate around={-120:(1.2,0)}][fill=black] (2.15,1.16) circle (0.3pt);
\draw [fill=black] (0.75,0.35) circle (0.3pt);
\draw [fill=black] (0.9,0.5) circle (0.3pt);
\draw [fill=black] (1.1,0.55) circle (0.3pt);
\draw [rotate around={40:(-1.2,0)}][fill=black] (-0.78,0.4) circle (0.3pt);
\draw [rotate around={40:(-1.2,0)}][fill=black] (-0.7,0.3) circle (0.3pt);
\draw [rotate around={40:(-1.2,0)}][fill=black] (-0.65,0.18) circle (0.3pt);
\draw [rotate around={-100:(-1.2,0)}][fill=black] (-0.78,0.4) circle (0.3pt);
\draw [rotate around={-100:(-1.2,0)}][fill=black] (-0.7,0.3) circle (0.3pt);
\draw [rotate around={-100:(-1.2,0)}][fill=black] (-0.65,0.18) circle (0.3pt);
\draw [fill=black] (-1.7,0.3) circle (0.3pt);
\draw [fill=black] (-1.6,0.45) circle (0.3pt);
\draw [fill=black] (-1.45,0.55) circle (0.3pt);
\draw [rotate around={90:(-1.2,0)}][fill=black] (-1.7,0.3) circle (0.3pt);
\draw [rotate around={90:(-1.2,0)}][fill=black] (-1.6,0.45) circle (0.3pt);
\draw [rotate around={90:(-1.2,0)}][fill=black] (-1.45,0.55) circle (0.3pt);
\draw (1.2,0.25) node[anchor=north] {$M_{\MA}$};
\draw (-1.2,0.25) node[anchor=north] {$M_{\MB}$};
\draw (0,0.25) node[anchor=north] {$\scriptstyle{j_{\tau'_{1}}}$};
\draw (2.1,-0.78) node[anchor=north] {$\scriptstyle{j_{\tau'_{n-k-1}}}$};
\draw (2,1.28) node[anchor=north] {$\scriptstyle{j_{\tau'_{n-k}}}$};
\draw (-2.3,0.25) node[anchor=north] {$\scriptstyle{i_{\sigma_{1}}}$};
\draw (-0.26,-0.6) node[anchor=north] {$\scriptstyle{i_{\sigma_{v+1}}}$};
\draw (-0.35,1.1) node[anchor=north] {$\scriptstyle{i_{\sigma_{v}}}$};
\end{tikzpicture}
\end{minipage}

    \noindent respectively. Since $s_{d+1}M_{\MB}$ is a $d$-pre-Calabi-Yau structure on $\MB$, we have that 
    \[
    s_{-1}\Phi^{\sigma'}_{s_{d+1}M_{\MB}}\big(\Phi^{\sigma}_{s_{d+1}M_{\MB}}(l_{i_k},\dots,l_{i_1}),l_{j_{n-k}},\dots,l_{j_{1}}\big)=0.
    \] 
    On the other hand, we have that 
    \begin{small}
    \begin{equation}
    \label{eq:higher-Jac-bar-ell-2}
        \begin{split}
            &\sum\limits_{k'=1}^n(-1)^{k'}\hskip-4mm\sum\limits_{(\bar{i}',\bar{j}')\in\mathcal{P}_{k'}^n}\hskip-4mm(-1)^{\Delta'+\eta_1}\hskip-2mm\sum\limits_{\tau\in\mathfrak{S}_{k'}}\hskip-2mm(-1)^{\delta^{\tau}(l_{i'_{k'}},\dots,l_{i'_1})} \Bar{\ell}^{n-k'+1}\big(s_{-1}\Psi^{\tau}_{s_{d+1}M_{\MA}}(l_{i'_{k'}},\dots,l_{i'_1}),s_{-1}l_{j'_{n-k'}},\dots,s_{-1}l_{j'_{1}}\big)
            \\&=\sum\limits_{k'=1}^n(-1)^{k'}\sum\limits_{(\bar{i}',\bar{j}')\in\mathcal{P}_{k'}^n}\hskip-3mm(-1)^{\Delta'+\eta_1}\sum\limits_{\tau\in\mathfrak{S}_{k'}}\hskip-2mm(-1)^{\delta^{\tau}(l_{i'_{k'}},\dots,l_{i'_1})+\epsilon_2'} \sum\limits_{\sigma\dprime\in\mathfrak{S}_{n-k'}}(-1)^{\delta^{\sigma\dprime}(l_{j'_{n-k'}},\dots,l_{j'_1})}\\&\hskip6cms_{-1}\Phi^{\sigma\dprime}_{s_{d+1}M_{\MB}}\big(\Psi^{\tau}_{s_{d+1}M_{\MA}}(l_{i'_{k'}},\dots,l_{i'_1}),l_{j'_{n-k'}},\dots,l_{j'_{1}}\big)
            \\&\phantom{=}-\sum\limits_{k'=1}^n(-1)^{k'}\sum\limits_{(\bar{i}',\bar{j}')\in\mathcal{P}_{k'}^n}(-1)^{\Delta'+\eta_1}\sum\limits_{\tau\in\mathfrak{S}_{k'}}(-1)^{\delta^{\tau}(l_{i'_{k'}},\dots,l_{i'_1})+\eta_2'} \sum\limits_{\tau\dprime\in\mathfrak{S}_{n-k'}}(-1)^{\delta^{\tau\dprime}(l_{j'_{n-k'}},\dots,l_{j'_1})}\\&\hskip6cms_{-1}\Psi^{\tau\dprime}_{s_{d+1}M_{\MA}}\big(\Psi^{\tau}_{s_{d+1}M_{\MA}}(l_{i'_{k'}},\dots,l_{i'_1}),l_{j'_{n-k'}},\dots,l_{j'_{1}}\big)
        \end{split}
    \end{equation}
    \end{small}
    where
    \begin{small}
    \begin{equation}
        \begin{split}
            \epsilon_2'&=\sum\limits_{s=1}^{n-k'}(s-1)|l_{j'_s}|+(n-k')(\sum\limits_{r=1}^{k'}|l_{i'_r}|+1),
            \\ \eta_2'&=\sum\limits_{s=1}^{n-k'}(s-1)|l_{j'_s}|+(n-k')(\sum\limits_{r=1}^{k'}|l_{i'_r}|+1)+\sum\limits_{r=1}^{k'}|l_{i'_r}|+\sum\limits_{s=1}^{n-k'}|l_{j'_s}|+1
        \end{split}
    \end{equation}
    \end{small}
   and where 
   \allowdisplaybreaks
    \begin{align*}            &\Phi^{\sigma\dprime}_{s_{d+1}M_{\MB}}\big(\Psi^{\tau}_{s_{d+1}M_{\MA}}(l_{i'_{k'}},\dots,l_{i'_1}),l_{j'_{n-k'}},\dots,l_{j'_{1}}\big)\\&=\sum\limits_{u'=1}^{n-k'}(-1)^{(\sum\limits_{r=1}^{k'}|l_{i'_r}|+1)\sum\limits_{s=u'+1}^{n-k'}|l_{j'_{\sigma_s\dprimeind}}|}\sum\mathcal{E}(\mathcalboondox{D_3})
            +\sum\limits_{v'=1}^{k'}(-1)^{(1+\sum\limits_{r=1}^{v'}|l_{i'_{\tau_r}}|)\sum\limits_{s=1}^{n-k'}|l_{j'_s}|+\sum\limits_{r=v'+1}^{k'}|l_{i'_{\tau_r}}|}\sum\mathcal{E}(\mathcalboondox{D_3'})
            \\&\Psi^{\tau\dprime}_{s_{d+1}M_{\MA}}\big(\Psi^{\tau}_{s_{d+1}M_{\MA}}(l_{i'_{k'}},\dots,l_{i'_1}),l_{j'_{n-k'}},\dots,l_{j'_{1}}\big)\\&=\sum\limits_{u'=1}^{n-k'}(-1)^{(\sum\limits_{r=1}^{k'}|l_{i'_r}|+1)\sum\limits_{s=u'+1}^{n-k'}|l_{j'_{\sigma'_s}}|}\sum\mathcal{E}(\mathcalboondox{D_4})
            +\sum\limits_{v'=1}^{k'}(-1)^{\sum\limits_{r=1}^{v'}|l_{i'_{\tau_r}}|\sum\limits_{s=1}^{n-k'}|l_{j'_s}|}\sum\mathcal{E}(\mathcalboondox{D_4'})
        \end{align*}
    where the sums are over all the filled diagrams $\mathcalboondox{D_3}$, $\mathcalboondox{D_3'}$, $\mathcalboondox{D_4}$ and $\mathcalboondox{D_4'}$ of the form

    \begin{minipage}{21cm}
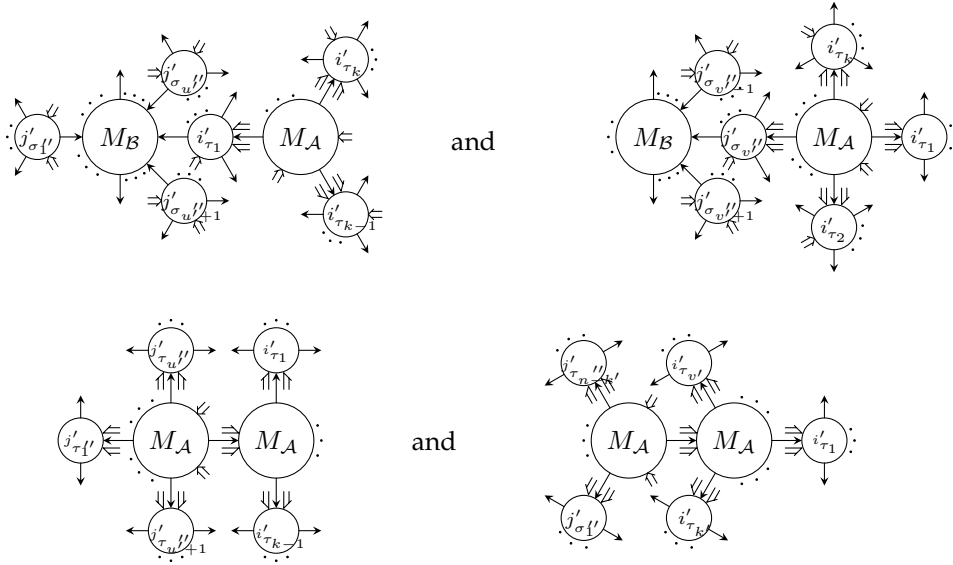
    \begin{tikzpicture}[line cap=round,line join=round,x=1.0cm,y=1.0cm]
\clip(-3.5,-2) rectangle (5,2);
    \draw (0,0) circle (0.3cm);
    \draw[rotate=60][->,>=stealth] (0.3,0)--(0.7,0);
    \draw[rotate=180][->,>=stealth] (0.3,0)--(0.7,0);
    \draw[rotate=-60][->,>=stealth] (0.3,0)--(0.7,0);
    \draw[rotate=-120,shift={(0.3,0)}] \doublefleche;
    \draw[rotate=0,shift={(0.3,0)}] \doubleflechescindeeleft;
    \draw[rotate=0,shift={(0.3,0)}] \doubleflechescindeeright;
    \draw[rotate=0][<-,>=stealth] (0.3,0)--(0.7,0);
    \draw(1.2,0)circle (0.5cm);
    \draw[rotate around={0:(1.2,0)},shift={(1.7,0)}] \doublefleche;
    \draw[rotate around={-120:(1.2,0)},shift={(1.7,0)}] \doublefleche;
    \draw[rotate around={-60:(1.2,0)}][->,>=stealth] (1.7,0)--(2.1,0);
    \draw(1.8,1.03)circle (0.3cm);
    \draw[rotate around={120:(1.8,1.03)},shift={(2.1,1.03)}] \doublefleche;
    \draw[rotate around={-120:(1.8,1.03)},shift={(2.1,1.03)}] \doubleflechescindeeleft;
    \draw[rotate around={-120:(1.8,1.03)},shift={(2.1,1.03)}] \doubleflechescindeeright;
    \draw[rotate around={-60:(1.8,1.03)}][->,>=stealth] (2.1,1.03)--(2.4,1.03);
    \draw[rotate around={180:(1.8,1.03)}][->,>=stealth] (2.1,1.03)--(2.4,1.03);
    \draw[rotate around={60:(1.8,1.03)}][->,>=stealth] (2.1,1.03)--(2.4,1.03);
    \draw[rotate around={60:(1.2,0)}][->,>=stealth] (1.7,0)--(2.1,0);
    \draw(1.8,-1.03)circle (0.3cm);
    \draw[rotate around={0:(1.8,-1.03)},shift={(2.1,-1.03)}] \doublefleche;
    \draw[rotate around={120:(1.8,-1.03)},shift={(2.1,-1.03)}] \doubleflechescindeeright;
    \draw[rotate around={120:(1.8,-1.03)},shift={(2.1,-1.03)}] \doubleflechescindeeleft;
    \draw[rotate around={60:(1.8,-1.03)}][->,>=stealth] (2.1,-1.03)--(2.4,-1.03);
    \draw[rotate around={-60:(1.8,-1.03)}][->,>=stealth] (2.1,-1.03)--(2.4,-1.03);
    \draw[rotate around={180:(1.8,-1.03)}][->,>=stealth] (2.1,-1.03)--(2.4,-1.03);
    \draw(-1.2,0)circle (0.5cm);
    \draw [rotate around={-45:(-1.2,0)}][<-,>=stealth] (-0.7,0)--(-0.3,0);
    \draw [rotate around={45:(-1.2,0)}][<-,>=stealth] (-0.7,0)--(-0.3,0);
    \draw [rotate around={90:(-1.2,0)}][->,>=stealth] (-0.7,0)--(-0.3,0);
    \draw [rotate around={-90:(-1.2,0)}][->,>=stealth] (-0.7,0)--(-0.3,0);
    \draw(-0.35,-0.85)circle (0.3cm);
    \draw [rotate around={-120:(-0.35,-0.85)}][->,>=stealth] (-0.05,-0.85)--(0.25,-0.85);
    \draw [rotate around={0:(-0.35,-0.85)}][->,>=stealth] (-0.05,-0.85)--(0.25,-0.85);
    \draw [rotate around={180:(-0.35,-0.85)},shift={(-0.05,-0.85)}]\doublefleche;
     \draw [rotate around={-60:(-0.35,-0.85)},shift={(-0.05,-0.85)}]\doublefleche;
     \draw(-0.35,0.85)circle (0.3cm);
    \draw [rotate around={120:(-0.35,0.85)}][->,>=stealth] (-0.05,0.85)--(0.25,0.85);
    \draw [rotate around={0:(-0.35,0.85)}][->,>=stealth] (-0.05,0.85)--(0.25,0.85);
    \draw [rotate around={180:(-0.35,0.85)},shift={(-0.05,0.85)}]\doublefleche;
     \draw [rotate around={60:(-0.35,0.85)},shift={(-0.05,0.85)}]\doublefleche;
     \draw[rotate around={180:(-1.2,0)}][<-,>=stealth] (-0.7,0)--(-0.4,0);
     \draw (-2.3,0) circle (0.3cm);
     \draw[rotate around={120:(-2.3,0)}][->,>=stealth] (-2,0)--(-1.7,0);
     \draw[rotate around={-120:(-2.3,0)}][->,>=stealth] (-2,0)--(-1.7,0);
     \draw[rotate around={60:(-2.3,0)},shift={(-2,0)}]\doublefleche;
     \draw[rotate around={-60:(-2.3,0)},shift={(-2,0)}]\doublefleche;
\draw [fill=black] (-2.65,0.15) circle (0.3pt);
\draw [fill=black] (-2.7,0) circle (0.3pt);
\draw [fill=black] (-2.65,-0.15) circle (0.3pt);
\draw [fill=black] (-0.3,0.2) circle (0.3pt);
\draw [fill=black] (-0.2,0.3) circle (0.3pt);
\draw [fill=black] (-0.05,0.35) circle (0.3pt);
\draw [rotate around={-45:(-1.2,0)}][fill=black] (-0.3,0.2) circle (0.3pt);
\draw [rotate around={-45:(-1.2,0)}][fill=black] (-0.2,0.3) circle (0.3pt);
\draw [rotate around={-45:(-1.2,0)}][fill=black] (-0.05,0.35) circle (0.3pt);
\draw [fill=black] (-0.3,0.5) circle (0.3pt);
\draw [fill=black] (-0.15,0.54) circle (0.3pt);
\draw [fill=black] (-0.05,0.65) circle (0.3pt);
\draw [fill=black] (2.2,1.03) circle (0.3pt);
\draw [fill=black] (2.15,0.9) circle (0.3pt);
\draw [fill=black] (2.15,1.16) circle (0.3pt);
\draw [rotate around={-120:(1.2,0)}][fill=black] (2.2,1.03) circle (0.3pt);
\draw [rotate around={-120:(1.2,0)}][fill=black] (2.15,0.9) circle (0.3pt);
\draw [rotate around={-120:(1.2,0)}][fill=black] (2.15,1.16) circle (0.3pt);
\draw [fill=black] (0.75,0.35) circle (0.3pt);
\draw [fill=black] (0.9,0.5) circle (0.3pt);
\draw [fill=black] (1.1,0.55) circle (0.3pt);
\draw [rotate around={40:(-1.2,0)}][fill=black] (-0.78,0.4) circle (0.3pt);
\draw [rotate around={40:(-1.2,0)}][fill=black] (-0.7,0.3) circle (0.3pt);
\draw [rotate around={40:(-1.2,0)}][fill=black] (-0.65,0.18) circle (0.3pt);
\draw [rotate around={-100:(-1.2,0)}][fill=black] (-0.78,0.4) circle (0.3pt);
\draw [rotate around={-100:(-1.2,0)}][fill=black] (-0.7,0.3) circle (0.3pt);
\draw [rotate around={-100:(-1.2,0)}][fill=black] (-0.65,0.18) circle (0.3pt);
\draw [fill=black] (-1.7,0.3) circle (0.3pt);
\draw [fill=black] (-1.6,0.45) circle (0.3pt);
\draw [fill=black] (-1.45,0.55) circle (0.3pt);
\draw [rotate around={90:(-1.2,0)}][fill=black] (-1.7,0.3) circle (0.3pt);
\draw [rotate around={90:(-1.2,0)}][fill=black] (-1.6,0.45) circle (0.3pt);
\draw [rotate around={90:(-1.2,0)}][fill=black] (-1.45,0.55) circle (0.3pt);
\draw (3.5,0.2) node[anchor=north] {and};
\draw (1.2,0.25) node[anchor=north] {$M_{\MA}$};
\draw (-1.2,0.25) node[anchor=north] {$M_{\MB}$};
\draw (0,0.25) node[anchor=north] {$\scriptstyle{i'_{\tau_{1}}}$};
\draw (1.9,-0.78) node[anchor=north] {$\scriptstyle{i'_{\tau_{k-1}}}$};
\draw (1.85,1.28) node[anchor=north] {$\scriptstyle{i'_{\tau_{k}}}$};
\draw (-2.3,0.25) node[anchor=north] {$\scriptstyle{j'_{\sigma_{1}\dprimeind}}$};
\draw (-0.35,-0.6) node[anchor=north] {$\scriptstyle{j'_{\sigma_{u'+1}\dprimeindl}}$};
\draw (-0.35,1.1) node[anchor=north] {$\scriptstyle{j'_{\sigma_{u'}\dprimeind}}$};
\end{tikzpicture}
\begin{tikzpicture}[line cap=round,line join=round,x=1.0cm,y=1.0cm]
\clip(-2,-2) rectangle (8.32706640373885,2);
  \draw (0,0) circle (0.3cm);
    \draw[rotate=60][->,>=stealth] (0.3,0)--(0.7,0);
    \draw[rotate=180][->,>=stealth] (0.3,0)--(0.7,0);
    \draw[rotate=-60][->,>=stealth] (0.3,0)--(0.7,0);
    \draw[rotate=-120,shift={(0.3,0)}] \doublefleche;
    \draw[rotate=0,shift={(0.3,0)}] \doubleflechescindeeleft;
    \draw[rotate=0,shift={(0.3,0)}] \doubleflechescindeeright;
    \draw[rotate=0][<-,>=stealth] (0.3,0)--(0.7,0);
    \draw(1.2,0)circle (0.5cm);
    \draw[->,>=stealth] (1.7,0)--(2.1,0);
    \draw(2.4,0) circle (0.3cm);
    \draw[rotate around={180:(2.4,0)},shift={(2.7,0)}]\doubleflechescindeeleft;
    \draw[rotate around={180:(2.4,0)},shift={(2.7,0)}]\doubleflechescindeeright;
    \draw[rotate around={90:(2.4,0)}][->,>=stealth](2.7,0)--(3,0);
    \draw[rotate around={-90:(2.4,0)}][->,>=stealth](2.7,0)--(3,0);
    \draw[rotate around={45:(1.2,0)},shift={(1.7,0)}] \doublefleche;
    \draw[rotate around={-45:(1.2,0)},shift={(1.7,0)}] \doublefleche;
    \draw[rotate around={-90:(1.2,0)}][->,>=stealth] (1.7,0)--(2.1,0);
    \draw(1.2,-1.2)circle (0.3cm);
    \draw[rotate around={-150:(1.2,-1.2)},shift={(1.5,-1.2)}] \doublefleche;
    \draw[rotate around={90:(1.2,-1.2)},shift={(1.5,-1.2)}] \doubleflechescindeeleft;
    \draw[rotate around={90:(1.2,-1.2)},shift={(1.5,-1.2)}] \doubleflechescindeeright;
    \draw[rotate around={30:(1.2,-1.2)}][->,>=stealth] (1.5,-1.2)--(1.8,-1.2);
    \draw[rotate around={150:(1.2,-1.2)}][->,>=stealth] (1.5,-1.2)--(1.8,-1.2);
    \draw[rotate around={-90:(1.2,-1.2)}][->,>=stealth] (1.5,-1.2)--(1.8,-1.2);
    \draw[rotate around={90:(1.2,0)}][->,>=stealth] (1.7,0)--(2.1,0);
    \draw(1.2,1.2)circle (0.3cm);
    \draw[rotate around={150:(1.2,1.2)},shift={(1.5,1.2)}] \doublefleche;
    \draw[rotate around={-90:(1.2,1.2)},shift={(1.5,1.2)}] \doubleflechescindeeleft;
    \draw[rotate around={-90:(1.2,1.2)},shift={(1.5,1.2)}] \doubleflechescindeeright;
    \draw[rotate around={-30:(1.2,1.2)}][->,>=stealth] (1.5,1.2)--(1.8,1.2);
    \draw[rotate around={-150:(1.2,1.2)}][->,>=stealth] (1.5,1.2)--(1.8,1.2);
    \draw[rotate around={90:(1.2,1.2)}][->,>=stealth] (1.5,1.2)--(1.8,1.2);
    \draw(-1.2,0)circle (0.5cm);
    \draw [rotate around={-45:(-1.2,0)}][<-,>=stealth] (-0.7,0)--(-0.3,0);
    \draw [rotate around={45:(-1.2,0)}][<-,>=stealth] (-0.7,0)--(-0.3,0);
    \draw [rotate around={90:(-1.2,0)}][->,>=stealth] (-0.7,0)--(-0.3,0);
    \draw [rotate around={-90:(-1.2,0)}][->,>=stealth] (-0.7,0)--(-0.3,0);
    \draw(-0.35,-0.85)circle (0.3cm);
    \draw [rotate around={-120:(-0.35,-0.85)}][->,>=stealth] (-0.05,-0.85)--(0.25,-0.85);
    \draw [rotate around={0:(-0.35,-0.85)}][->,>=stealth] (-0.05,-0.85)--(0.25,-0.85);
    \draw [rotate around={180:(-0.35,-0.85)},shift={(-0.05,-0.85)}]\doublefleche;
     \draw [rotate around={-60:(-0.35,-0.85)},shift={(-0.05,-0.85)}]\doublefleche;
     \draw(-0.35,0.85)circle (0.3cm);
    \draw [rotate around={120:(-0.35,0.85)}][->,>=stealth] (-0.05,0.85)--(0.25,0.85);
    \draw [rotate around={0:(-0.35,0.85)}][->,>=stealth] (-0.05,0.85)--(0.25,0.85);
    \draw [rotate around={180:(-0.35,0.85)},shift={(-0.05,0.85)}]\doublefleche;
     \draw [rotate around={60:(-0.35,0.85)},shift={(-0.05,0.85)}]\doublefleche;
\draw [fill=black] (-1.75,0.2) circle (0.3pt);
\draw [fill=black] (-1.8,0) circle (0.3pt);
\draw [fill=black] (-1.75,-0.2) circle (0.3pt);
\draw [fill=black] (-0.3,0.2) circle (0.3pt);
\draw [fill=black] (-0.2,0.3) circle (0.3pt);
\draw [fill=black] (-0.05,0.35) circle (0.3pt);
\draw [rotate around={-45:(-1.2,0)}][fill=black] (-0.3,0.2) circle (0.3pt);
\draw [rotate around={-45:(-1.2,0)}][fill=black] (-0.2,0.3) circle (0.3pt);
\draw [rotate around={-45:(-1.2,0)}][fill=black] (-0.05,0.35) circle (0.3pt);
\draw [fill=black] (-0.3,0.5) circle (0.3pt);
\draw [fill=black] (-0.15,0.54) circle (0.3pt);
\draw [fill=black] (-0.05,0.65) circle (0.3pt);
\draw [rotate around={40:(-1.2,0)}][fill=black] (-0.78,0.4) circle (0.3pt);
\draw [rotate around={40:(-1.2,0)}][fill=black] (-0.7,0.3) circle (0.3pt);
\draw [rotate around={40:(-1.2,0)}][fill=black] (-0.65,0.18) circle (0.3pt);
\draw [rotate around={-100:(-1.2,0)}][fill=black] (-0.78,0.4) circle (0.3pt);
\draw [rotate around={-100:(-1.2,0)}][fill=black] (-0.7,0.3) circle (0.3pt);
\draw [rotate around={-100:(-1.2,0)}][fill=black] (-0.65,0.18) circle (0.3pt);
\draw [fill=black] (2.75,0.15) circle (0.3pt);
\draw [fill=black] (2.8,0) circle (0.3pt);
\draw [fill=black] (2.75,-0.15) circle (0.3pt);
\draw [fill=black] (1.58,1.18) circle (0.3pt);
\draw [fill=black] (1.55,1.35) circle (0.3pt);
\draw [fill=black] (1.45,1.5) circle (0.3pt);
\draw [fill=black] (1.58,-1.18) circle (0.3pt);
\draw [fill=black] (1.55,-1.35) circle (0.3pt);
\draw [fill=black] (1.45,-1.5) circle (0.3pt);
\draw [fill=black] (0.7,0.3) circle (0.3pt);
\draw [fill=black] (0.8,0.45) circle (0.3pt);
\draw [fill=black] (0.95,0.55) circle (0.3pt);
\draw [rotate around={90:(1.2,0)}][fill=black] (0.7,0.3) circle (0.3pt);
\draw [rotate around={90:(1.2,0)}][fill=black] (0.8,0.45) circle (0.3pt);
\draw [rotate around={90:(1.2,0)}][fill=black] (0.9,0.55) circle (0.3pt);
\draw (1.2,0.25) node[anchor=north] {$M_{\MA}$};
\draw (-1.2,0.25) node[anchor=north] {$M_{\MB}$};
\draw (2.4,0.25) node[anchor=north] {$\scriptstyle{i'_{\tau_{1}}}$};
\draw (1.2,-0.95) node[anchor=north] {$\scriptstyle{i'_{\tau_{2}}}$};
\draw (1.3,1.45) node[anchor=north] {$\scriptstyle{i'_{\tau_{k}}}$};
\draw (0,0.25) node[anchor=north] {$\scriptstyle{j'_{\sigma_{v'}\dprimeind}}$};
\draw (-0.35,-0.6) node[anchor=north] {$\scriptstyle{j'_{\sigma_{v'+1}\dprimeindl}}$};
\draw (-0.35,1.1) node[anchor=north] {$\scriptstyle{j'_{\sigma_{v'-1}\dprimeindl}}$};
\end{tikzpicture}
\end{minipage}
\begin{minipage}{21cm}
\begin{tikzpicture}[line cap=round,line join=round,x=1.0cm,y=1.0cm]
\clip(-3.5,-2) rectangle (4,2);
 \draw (0,0) circle (0.5cm);
    \draw [rotate around={0:(0,0)}][->,>=stealth] (0.5,0)--(0.9,0);
    \draw [rotate around={-45:(0,0)},shift={(0.5,0)}]\doublefleche;
    \draw [rotate around={45:(0,0)},shift={(0.5,0)}]\doublefleche;
    \draw(1.4,0) circle (0.5cm);
    \draw [rotate around={180:(1.4,0)}, shift={(1.9,0)}] \doubleflechescindeeleft;
    \draw [rotate around={180:(1.4,0)}, shift={(1.9,0)}] \doubleflechescindeeright;
    \draw[rotate around={-90:(1.4,0)}][->,>=stealth](1.9,0)--(2.3,0);
    \draw (1.4,1.2) circle (0.3cm);
    \draw[rotate around={-90:(1.4,1.2)},shift={(1.7,1.2)}] \doubleflechescindeeleft;
    \draw[rotate around={-90:(1.4,1.2)},shift={(1.7,1.2)}] \doubleflechescindeeright;
    \draw[rotate around={0:(1.4,1.2)}][->,>=stealth](1.7,1.2)--(2,1.2);
    \draw[rotate around={180:(1.4,1.2)}][->,>=stealth](1.7,1.2)--(2,1.2);
    \draw[rotate around={90:(1.4,0)}][->,>=stealth](1.9,0)--(2.3,0);
    \draw (1.4,-1.2) circle (0.3cm);
    \draw[rotate around={90:(1.4,-1.2)},shift={(1.7,-1.2)}] \doubleflechescindeeleft;
    \draw[rotate around={90:(1.4,-1.2)},shift={(1.7,-1.2)}] \doubleflechescindeeright;
    \draw[rotate around={0:(1.4,-1.2)}][->,>=stealth](1.7,-1.2)--(2,-1.2);
    \draw[rotate around={180:(1.4,-1.2)}][->,>=stealth](1.7,-1.2)--(2,-1.2);
    \draw[rotate=-90][->,>=stealth](0.5,0)--(0.9,0);
     \draw (0,-1.2) circle (0.3cm);
    \draw[rotate around={90:(0,-1.2)},shift={(0.3,-1.2)}] \doubleflechescindeeleft;
    \draw[rotate around={90:(0,-1.2)},shift={(0.3,-1.2)}] \doubleflechescindeeright;
    \draw[rotate around={0:(0,-1.2)}][->,>=stealth](0.3,-1.2)--(0.6,-1.2);
    \draw[rotate around={180:(0,-1.2)}][->,>=stealth](0.3,-1.2)--(0.6,-1.2);
    \draw[rotate=90][->,>=stealth](0.5,0)--(0.9,0);
     \draw (0,1.2) circle (0.3cm);
    \draw[rotate around={-90:(0,1.2)},shift={(0.3,1.2)}] \doubleflechescindeeleft;
    \draw[rotate around={-90:(0,1.2)},shift={(0.3,1.2)}] \doubleflechescindeeright;
    \draw[rotate around={0:(0,1.2)}][->,>=stealth](0.3,1.2)--(0.6,1.2);
    \draw[rotate around={180:(0,1.2)}][->,>=stealth](0.3,1.2)--(0.6,1.2);
    \draw[rotate=180][->,>=stealth](0.5,0)--(0.9,0);
     \draw (-1.2,0) circle (0.3cm);
    \draw[rotate around={0:(-1.2,0)},shift={(-0.9,0)}] \doubleflechescindeeleft;
    \draw[rotate around={0:(-1.2,0)},shift={(-0.9,0)}] \doubleflechescindeeright;
    \draw[rotate around={90:(-1.2,0)}][->,>=stealth](-0.9,0)--(-0.6,0);
    \draw[rotate around={-90:(-1.2,0)}][->,>=stealth](-0.9,0)--(-0.6,0);
\draw[rotate=45][fill=black](-0.6,0) circle (0.3pt);
\draw[rotate=45][fill=black](-0.55,0.2) circle (0.3pt);
\draw[rotate=45][fill=black](-0.55,-0.2) circle (0.3pt);
\draw[rotate=-45][fill=black](-0.6,0) circle (0.3pt);
\draw[rotate=-45][fill=black](-0.55,0.2) circle (0.3pt);
\draw[rotate=-45][fill=black](-0.55,-0.2) circle (0.3pt);
\draw[fill=black](2,0) circle (0.3pt);
\draw[fill=black](1.95,0.2) circle (0.3pt);
\draw[fill=black](1.95,-0.2) circle (0.3pt);
\draw[fill=black](1.4,1.6) circle (0.3pt);
\draw[fill=black](1.25,1.55) circle (0.3pt);
\draw[fill=black](1.55,1.55) circle (0.3pt);
\draw [rotate around={180:(1.4,0)}][fill=black](1.4,1.6) circle (0.3pt);
\draw [rotate around={180:(1.4,0)}][fill=black](1.25,1.55) circle (0.3pt);
\draw [rotate around={180:(1.4,0)}][fill=black](1.55,1.55) circle (0.3pt);
\draw[fill=black](0,1.6) circle (0.3pt);
\draw[fill=black](0.15,1.55) circle (0.3pt);
\draw[fill=black](-0.15,1.55) circle (0.3pt);
\draw[rotate=180][fill=black](0,1.6) circle (0.3pt);
\draw[rotate=180][fill=black](0.15,1.55) circle (0.3pt);
\draw[rotate=180][fill=black](-0.15,1.55) circle (0.3pt);
\draw (0,0.25)node[anchor=north]{$M_{\MA}$};
\draw (1.4,0.25)node[anchor=north]{$M_{\MA}$};
\draw (-1.2,0.25)node[anchor=north]{$\scriptscriptstyle{j'_{\tau_1\dprimeind}}$};
\draw (0,1.45)node[anchor=north]{$\scriptscriptstyle{j'_{\tau_{u'}\dprimeind}}$};
\draw (0,-1)node[anchor=north]{$\scriptscriptstyle{j'_{\tau_{u'+1}\dprimeindl}}$};
\draw (1.4,1.45)node[anchor=north]{$\scriptscriptstyle{i'_{\tau_1}}$};
\draw (1.47,-1)node[anchor=north]{$\scriptscriptstyle{i'_{\tau_{k-1}}}$};
\draw (3.5,0.25) node[anchor=north]{and};
\end{tikzpicture}
\begin{tikzpicture}[line cap=round,line join=round,x=1.0cm,y=1.0cm]
\clip(-2,-2) rectangle (7.242319607404366,2);
    \draw (0,0) circle (0.5cm);
    \draw [rotate around={0:(0,0)}][->,>=stealth] (0.5,0)--(0.9,0);
    \draw [rotate around={120:(0,0)}][->,>=stealth] (0.5,0)--(0.9,0);
    \draw [rotate around={-120:(0,0)}][->,>=stealth] (0.5,0)--(0.9,0);
    \draw [rotate around={-60:(0,0)},shift={(0.5,0)}]\doublefleche;
    \draw [rotate around={60:(0,0)},shift={(0.5,0)}]\doublefleche;
    \draw(1.4,0) circle (0.5cm);
    \draw [rotate around={180:(1.4,0)}, shift={(1.9,0)}] \doubleflechescindeeleft;
    \draw [rotate around={180:(1.4,0)}, shift={(1.9,0)}] \doubleflechescindeeright;
    \draw[rotate around={-120:(1.4,0)}][->,>=stealth](1.9,0)--(2.3,0);
    \draw (0.8,1.03) circle (0.3cm);
    \draw[rotate around={-60:(0.8,1.03)},shift={(1.1,1.03)}] \doubleflechescindeeleft;
    \draw[rotate around={-60:(0.8,1.03)},shift={(1.1,1.03)}] \doubleflechescindeeright;
    \draw[rotate around={30:(0.8,1.03)}][->,>=stealth](1.1,1.03)--(1.4,1.03);
    \draw[rotate around={210:(0.8,1.03)}][->,>=stealth](1.1,1.03)--(1.4,1.03);
     \draw[rotate around={120:(1.4,0)}][->,>=stealth](1.9,0)--(2.3,0);
    \draw (0.8,-1.03) circle (0.3cm);
    \draw[rotate around={60:(0.8,-1.03)},shift={(1.1,-1.03)}] \doubleflechescindeeleft;
    \draw[rotate around={60:(0.8,-1.03)},shift={(1.1,-1.03)}] \doubleflechescindeeright;
    \draw[rotate around={-30:(0.8,-1.03)}][->,>=stealth](1.1,-1.03)--(1.4,-1.03);
    \draw[rotate around={-210:(0.8,-1.03)}][->,>=stealth](1.1,-1.03)--(1.4,-1.03);
    \draw (-0.6,1.03) circle (0.3cm);
    \draw[rotate around={-60:(-0.6,1.03)},shift={(-0.3,1.03)}] \doubleflechescindeeleft;
    \draw[rotate around={-60:(-0.6,1.03)},shift={(-0.3,1.03)}] \doubleflechescindeeright;
    \draw[rotate around={30:(-0.6,1.03)}][->,>=stealth](-0.3,1.03)--(0,1.03);
    \draw[rotate around={210:(-0.6,1.03)}][->,>=stealth](-0.3,1.03)--(0,1.03);
    \draw (-0.6,-1.03) circle (0.3cm);
     \draw[rotate around={60:(-0.6,-1.03)},shift={(-0.3,-1.03)}] \doubleflechescindeeleft;
    \draw[rotate around={60:(-0.6,-1.03)},shift={(-0.3,-1.03)}] \doubleflechescindeeright;
    \draw[rotate around={-30:(-0.6,-1.03)}][->,>=stealth](-0.3,-1.03)--(0,-1.03);
    \draw[rotate around={-210:(-0.6,-1.03)}][->,>=stealth](-0.3,-1.03)--(0,-1.03);
    \draw[->,>=stealth](1.9,0)--(2.3,0);
    \draw(2.6,0) circle (0.3cm);
    \draw[rotate around={90:(2.6,0)}][->,>=stealth](2.9,0)--(3.2,0);
    \draw[rotate around={-90:(2.6,0)}][->,>=stealth](2.9,0)--(3.2,0);
    \draw[rotate around={180:(2.6,0)},shift={(2.9,0)}] \doubleflechescindeeleft;
    \draw[rotate around={180:(2.6,0)},shift={(2.9,0)}] \doubleflechescindeeright;
\draw[fill=black](-0.6,0) circle (0.3pt);
\draw[fill=black](-0.55,0.2) circle (0.3pt);
\draw[fill=black](-0.55,-0.2) circle (0.3pt);
\draw[fill=black](3,0) circle (0.3pt);
\draw[fill=black](2.95,0.15) circle (0.3pt);
\draw[fill=black](2.95,-0.15) circle (0.3pt);
\draw[rotate around={60:(1.4,0)}][fill=black](2,0) circle (0.3pt);
\draw[rotate around={60:(1.4,0)}][fill=black](1.95,0.2) circle (0.3pt);
\draw[rotate around={60:(1.4,0)}][fill=black](1.95,-0.2) circle (0.3pt);
\draw[rotate around={-60:(1.4,0)}][fill=black](2,0) circle (0.3pt);
\draw[rotate around={-60:(1.4,0)}][fill=black](1.95,0.2) circle (0.3pt);
\draw[rotate around={-60:(1.4,0)}][fill=black](1.95,-0.2) circle (0.3pt);
\draw[fill=black](-0.6,1.4) circle (0.3pt);
\draw[fill=black](-0.8,1.35) circle (0.3pt);
\draw[fill=black](-0.95,1.2) circle (0.3pt);
\draw[fill=black](0.45,1.2) circle (0.3pt);
\draw[fill=black](0.6,1.35) circle (0.3pt);
\draw[fill=black](0.8,1.4) circle (0.3pt);
\draw [rotate=120][fill=black](-0.6,1.4) circle (0.3pt);
\draw [rotate=120][fill=black](-0.8,1.35) circle (0.3pt);
\draw [rotate=120][fill=black](-0.95,1.2) circle (0.3pt);
\draw[rotate around={120:(1.4,0)}][fill=black](0.45,1.2) circle (0.3pt);
\draw[rotate around={120:(1.4,0)}][fill=black](0.6,1.35) circle (0.3pt);
\draw[rotate around={120:(1.4,0)}][fill=black](0.8,1.4) circle (0.3pt);
\draw (0,0.25)node[anchor=north]{$M_{\MA}$};
\draw (1.4,0.25)node[anchor=north]{$M_{\MA}$};
\draw (2.62,0.25)node[anchor=north]{$\scriptscriptstyle{i'_{\tau_1}}$};
\draw (-0.6,1.25)node[anchor=north]{$\scriptscriptstyle{j'_{\tau_{n-k'}\dprimeindl}}$};
\draw (0.8,1.25)node[anchor=north]{$\scriptscriptstyle{i'_{\tau_{v'}}}$};
\draw (-0.6,-0.8)node[anchor=north]{$\scriptstyle{j'_{\sigma_{1}\dprimeind}}$};
\draw (0.9,-0.78)node[anchor=north]{$\scriptstyle{i'_{\tau_{k'}}}$};
\end{tikzpicture}
\end{minipage}

    \noindent respectively. Since $s_{d+1}M_{\MA}$ is a $d$-pre-Calabi-Yau structure, we have that 
    \[
    \Psi^{\tau\dprime}_{s_{d+1}M_{\MA}}\big(\Psi^{\tau}_{s_{d+1}M_{\MA}}(l_{i_k},\dots,l_{i_1}),l_{j_{n-k}},\dots,l_{j_{1}}\big)=0.
    \]
    Moreover, the term of \eqref{eq:higher-Jac-bar-ell-1} indexed by $k$, $\Bar{i},\Bar{j}$, $\sigma$, $\tau'$ and $u$ cancels with the term of \eqref{eq:higher-Jac-bar-ell-2} indexed by $k'=n-k$, $i_r=j'_r$ pour $r\in \llbracket 1,k\rrbracket$, $j_s=i'_s$ pour $s\in\llbracket 1,n-k\rrbracket$, $\tau=\tau'$, $\sigma\dprime=\sigma$ and $v'=u$
    and similarly the term of \eqref{eq:higher-Jac-bar-ell-1} indexed by $k$, $\Bar{i},\Bar{j}$, $\sigma$, $\tau'$ and $v$ cancels with the term of \eqref{eq:higher-Jac-bar-ell-2} indexed by $k'=n-k$, $i_r=j'_r$ pour $r\in \llbracket 1,k\rrbracket$, $j_s=i'_s$ pour $s\in\llbracket 1,n-k\rrbracket$, $\tau=\tau'$, $\sigma\dprime=\sigma$ and $u'=v$. Therefore, $(\mathfrak{h},\Bar{\ell})$ is a graded $L_{\infty}$-algebra. 

    Finally, a Maurer-Cartan element of $(\mathfrak{h},\Bar{\ell})$ is a degree $1$ element $s_{-1}(s_{d+1}\mathbf{F})$ of the graded vector space $\Multi_{d,\Phi_0}^{\bullet}(\MA,\MB)[d+1][-1]$, \textit{i.e.} $s_{d+1}\mathbf{F}\in\Multi_{d,\Phi_0}^{\bullet}(\MA,\MB)[d+1]$ of degree $0$, such that 
\begin{equation}
     \sum\limits_{n=1}^{\infty}\frac{1}{n!}\Bar{\ell}^n(s_{-1}s_{d+1}\mathbf{F},\dots,s_{-1}s_{d+1}\mathbf{F})=0
\end{equation}
which gives that $s_{d+1}\mathbf{F}$ is a $d$-pre-Calabi-Yau morphism. Reciprocally, given a $d$-pre-Calabi-Yau morphism $(\Phi_0,s_{d+1}\mathbf{F}):(\MA,s_{d+1}M_{\MA})\rightarrow (\MB,s_{d+1}M_{\MB}$), $s_{-1}s_{d+1}\mathbf{F}$ is clearly a Maurer-Cartan element of $(\mathfrak{h},\Bar{\ell})$.
\end{proof}

\begin{remark}
\label{remark:A-inf-MC-bar}
   A particular case of this result is the case of $A_{\infty}$-structures presented in Proposition \ref{prop:A-inf-L-inf}.
\end{remark}


\subsection{Homotopy and weak homotopy for pre-Calabi-Yau morphisms}\label{section:htpy}

In this section, we present two different notions of homotopy for pre-Calabi-Yau morphisms and their first properties.
Throughout this section, we denote 
\[
\mathcal{L}_d^{\Phi_0}(\MA,\MB)=\Multi_d^{\bullet}(\MB)^{C_{\llg(\bullet)}}[d+1]\oplus \Multi_{d,\Phi_0}^{\bullet}(\MA,\MB)^{C_{\llg(\bullet)}}[d+1][-1]\oplus \Multi_d^{\bullet}(\MA)^{C_{\llg(\bullet)}}[d+1]
\]
and given $d$-pre-Calabi-Yau categories $(\MA,s_{d+1}M_{\MA})$, $(\MB,s_{d+1}M_{\MB})$ together with $\Phi_0 : \MO_{\MA}\rightarrow\MO_{\MB}$ of the respective underlying graded quivers, we denote $(\Bar{\mathcal{L}}_d^{\Phi_0}(\MA,\MB),\Bar{\ell}_{M_{\MA},M_{\MB}})$ the graded $L_{\infty}$-algebra defined in Definition \ref{def:bar-ell} whose underlying graded vector space is $\Multi_{d,\Phi_0}^{\bullet}(\MA,\MB)^{C_{\llg(\bullet)}}[d+1][-1]$.

\begin{definition}\label{def:weak-htpy}
    Given graded quivers $\MA$ and $\MB$ with respective sets of objects $\MO_{\MA}$ and $\MO_{\MB}$ and a map $\Phi_0 : \MO_{\MA}\rightarrow\MO_{\MB}$, consider $d$-pre-Calabi-Yau structures $s_{d+1}M_{\MA}$, $s_{d+1}M'_{\MA}$ on $\MA$ and $s_{d+1}M_{\MB}$, $s_{d+1}M'_{\MB}$ on $\MB$.
    We say that two $d$-pre-Calabi-Yau morphisms 
    \begin{equation}
        (\Phi_0,s_{d+1}\mathbf{F}) : (\MA,s_{d+1}M_{\MA})\rightarrow (\MB,s_{d+1}M_{\MB}) \text{ and } (\Phi_0,s_{d+1}\mathbf{G}) : (\MA,s_{d+1}M'_{\MA})\rightarrow (\MB,s_{d+1}M'_{\MB})
    \end{equation} are \textbf{\textcolor{ultramarine}{weakly homotopic}} if the corresponding Maurer-Cartan elements $(s_{d+1}M_{\MB},s_{-1}(s_{d+1}\mathbf{F}),s_{d+1}M_{\MA})$ and $(s_{d+1}M'_{\MA},s_{-1}(s_{d+1}\mathbf{G}),s_{d+1}M'_{\MB})$ of $(\mathcal{L}_d^{\Phi_0}(\MA,\MB),\ell)$ are homotopic in the sense of Definition \ref{def:htpic-MC-elts}.
\end{definition}

\begin{definition}\label{def:htpy}
     Given $d$-pre-Calabi-Yau categories $(\MA,s_{d+1}M_{\MA})$ and $(\MB,s_{d+1}M_{\MB})$ and denote $\MA$ and $\MB$ the sets of objects $\MO_{\MA}$ and $\MO_{\MB}$ of the repsective underlying graded quivers.
    We say that two $d$-pre-Calabi-Yau morphisms $(\Phi_0,s_{d+1}\mathbf{F}), (\Phi_0,s_{d+1}\mathbf{G}) : (\MA,s_{d+1}M_{\MA})\rightarrow (\MB,s_{d+1}M_{\MB})$ are \textbf{\textcolor{ultramarine}{homotopic}} if the corresponding Maurer-Cartan elements $s_{-1}(s_{d+1}\mathbf{F})$ and $s_{-1}(s_{d+1}\mathbf{G})$ of $(\Bar{\mathcal{L}}_d^{\Phi_0}(\MA,\MB),\Bar{\ell}_{M_{\MA},M_{\MB}})$ are homotopic in the sense of \ref{def:htpic-MC-elts}. 
    
    This is equivalent by Remark \ref{remark:htpic-MC-elts-equiv} to the existence of an element $a(t)+b(t)dt\in \Bar{\mathcal{L}}_d^{\Phi_0}(\MA,\MB)[t,dt]$ such that $sa(0)=s_{d+1}\mathbf{F}$, $sa(1)=s_{d+1}\mathbf{G}$, $a(t)$ is a Maurer-Cartan element of $(\Bar{\mathcal{L}}_d^{\Phi_0}(\MA,\MB),\Bar{\ell}_{M_{\MA},M_{\MB}})$ and \[
    \frac{\partial a}{\partial t}(t)^{\doubar{x}}=\sum\limits_{n=2}^{\infty}\frac{1}{(n-1)!}\ell^n\big(a(t),\dots,a(t),b(t)\big)^{\doubar{x}}
    \]
    for every $t\in\Bbbk$, $\doubar{x}\in\doubar{\MO}_{\MA}$.
\end{definition}

We now relate those two notions of homotopy together and the notion homotopy with chain homotopy of maps of dg quivers.

\begin{proposition}
     Consider two $d$-pre-Calabi-Yau categories $(\MA,s_{d+1}M_{\MA})$ and $(\MB,s_{d+1}M_{\MB})$ and denote by $\MO_{\MA}$, $\MO_{\MB}$ the sets of objects of the respective underlying graded quivers. Then, two homotopic $d$-pre-Calabi-Yau morphisms $(\Phi_0,s_{d+1}\mathbf{F}), (\Phi_0,s_{d+1}\mathbf{G}) : (\MA,s_{d+1}M_{\MA})\rightarrow (\MB,s_{d+1}M_{\MB})$ are weakly homotopic.
\end{proposition}
\begin{proof}
    Consider a homotopy $a(t)+b(t)dt\in \Bar{\mathcal{L}}_d^{\Phi_0}(\MA,\MB)[t,dt]$ between the Maurer-Cartan elements corresponding to the $d$-pre-Calabi-Yau morphisms $(\Phi_0,s_{d+1}\mathbf{F})$ and $(\Phi_0,s_{d+1}\mathbf{G})$. Then, we have that $A(t)+B(t)dt\in \mathcal{L}_d^{\Phi_0}(\MA,\MB)[t,dt]$ where $A(t)=(s_{d+1}M_{\MB},a(t),s_{d+1}M_{\MA})$ and $B(t)=(0,b(t),0)$ is a weak homotopy between $(\Phi_0,s_{d+1}\mathbf{F})$ and $(\Phi_0,s_{d+1}\mathbf{G})$. Indeed, by definition $a(t)$ is a Maurer-Cartan element of $(\Bar{\mathcal{L}}_d^{\Phi_0}(\MA,\MB),\Bar{\ell}_{M_{\MA},M_{\MB}})$ for any $t\in\Bbbk$, \textit{i.e.} $sa(t)$ is a $d$-pre-Calabi-Yau morphism $(\MA,s_{d+1}M_{\MA})\rightarrow (\MB,s_{d+1}M_{\MB})$, meaning that $A(t)$ is a Maurer-Cartan element of $(\mathcal{L}_d^{\Phi_0}(\MA,\MB),\ell)$. Moreover, $A(0)=(s_{d+1}M_{\MB},s_{-1}s_{d+1}\mathbf{F},s_{d+1}M_{\MA})$ and $A(1)=(s_{d+1}M_{\MB},s_{-1}s_{d+1}\mathbf{G},s_{d+1}M_{\MA})$. Finally, we have that $\frac{\partial A}{\partial t}(t)=(0,\frac{\partial a}{\partial t}(t),0)$ and 
    \begin{equation}
        \begin{split}
            \sum\limits_{n=2}^{\infty}\frac{1}{(n-1)!}& \ell^n\big(A(t),\dots,A(t),B(t)\big)^{\doubar{x}}
            \\&=(0,\sum\limits_{n=2}^{\infty}\frac{1}{(n-1)!}  \sum\limits_{i=0}^{n-2}\ell^n\big(\underbrace{a(t),\dots,a(t)}_{i\text{ times}},s_{d+1}M_{\MA},a(t),\dots,a(t),b(t)\big)^{\doubar{x}},0)
            \\&\phantom{=}(0,\sum\limits_{n=2}^{\infty}\frac{1}{(n-1)!}  \sum\limits_{i=0}^{n-2}\ell^n\big(\underbrace{a(t),\dots,a(t)}_{i\text{ times}},s_{d+1}M_{\MB},a(t),\dots,a(t),b(t)\big)^{\doubar{x}},0)
            \\&=(0,\sum\limits_{n=2}^{\infty} \frac{1}{(n-1)!} \Bar{\ell}_{M_{\MA},M_{\MB}}^n\big(a(t),\dots,a(t),b(t)\big)^{\doubar{x}},0)
            \\&=(0,\frac{\partial a}{\partial t}(t)^{\doubar{x}},0)
        \end{split}
    \end{equation}
    for every $\doubar{x}\in\doubar{\MO}_{\MA}$ since $a(t)+b(t)dt$ is a homotopy between $(\Phi_0,s_{d+1}\mathbf{F})$ and $(\Phi_0,s_{d+1}\mathbf{G})$.
\end{proof}

\begin{proposition}
\label{prop:chain-htpy-strong-htpy}
Let $(\MA,s_{d+1}M_{\MA})$, $(\MB,s_{d+1}M_{\MB})$ be $d$-pre-Calabi-Yau categories and denote by $\MO_{\MA}$, $\MO_{\MB}$ the sets of objects of the respective underlying graded quivers. Consider two homotopic morphisms $(\Phi_0,s_{d+1}\mathbf{F}), (\Phi_0,s_{d+1}\mathbf{G}) : (\MA,s_{d+1}M_{\MA})\rightarrow (\MB,s_{d+1}M_{\MB})$. Then, the following holds:
\begin{enumerate}[label=(\roman*)]
    \item\label{item:chain-htpy} The chain maps $s_{d+1}F^{x,y},s_{d+1}G^{x,y}$ are chain homotopic for every $x,y\in\MO_{\MA}$.
    \item\label{item:htpy-q-iso} If $(\Phi_0,s_{d+1}\mathbf{F})$ is a quasi-isomorphism of $d$-pre-Calabi-Yau categories then so is $(\Phi_0,s_{d+1}\mathbf{G})$.
\end{enumerate}
\end{proposition}
\begin{proof}
    We first prove \ref{item:chain-htpy}. We thus consider an element $a(t)+b(t)dt\in \Bar{\mathcal{L}}_d^{\Phi_0}(\MA,\MB)[t,dt]$ satisfying that $sa(0)=s_{d+1}\mathbf{F}$, $sa(1)=s_{d+1}\mathbf{G}$, $a(t)$ is a Maurer-Cartan element of $(\Bar{\mathcal{L}}_d^{\Phi_0}(\MA,\MB),\Bar{\ell}_{M_{\MA},M_{\MB}})$ and

    \[
    \frac{\partial a}{\partial t}(t)^{\doubar{x}}=\sum\limits_{n=2}^{\infty}\frac{1}{(n-1)!} \ell^n\big(a(t),\dots,a(t),b(t)\big)^{\doubar{x}}
    \]
    for every $t\in\Bbbk$, $\doubar{x}\in\doubar{\MO}_{\MA}$.
   
   Therefore, we have that
    \begin{equation}
        \begin{split}
            \big(a(1)-a(0)\big)^{x,y}&=\int_{0}^{1} \big(s_{d+1}M_{\MB}^{\Phi_0(x),\Phi_0(y)}\circ b(t)^{x,y}+b(t)^{x,y}\circ s_{d+1}M_{\MA}^{x,y}\big) \,dt
            \\&=s_{d+1}M_{\MB}^{\Phi_0(x),\Phi_0(y)}\circ \int_{0}^{1} b(t)^{x,y}dt+\int_{0}^{1} b(t)^{x,y}dt \circ s_{d+1}M_{\MA}^{x,y}
        \end{split}
    \end{equation}
    for every $x,y\in\MO_{\MA}$ which means that $\int_0^1 b(t)^{x,y}dt$ is a chain homotopy between $s_{d+1}F^{x,y}$ and $s_{d+1}G^{x,y}$.
    
    We now prove \ref{item:htpy-q-iso}. Suppose that $(\Phi_0,s_{d+1}\mathbf{F})$ is a quasi-isomorphism of $d$-pre-Calabi-Yau categories. Then, $s_{d+1}F^{x,y}$ is a quasi-isomorphism of complexes. By \ref{item:chain-htpy}, $s_{d+1}F^{x,y}$ and $s_{d+1}G^{x,y}$ are chain homotopic so $s_{d+1}G^{x,y}$ is a quasi-isomorphism of complexes, meaning that $(\Phi_0,s_{d+1}\mathbf{G})$ is a quasi-isomorphism of $d$-pre-Calabi-Yau categories.
\end{proof}

\begin{proposition}
\label{prop:htpy-cp-right}
    Let $(\MA,s_{d+1}M_{\MA})$, $(\MB,s_{d+1}M_{\MB})$ and $(\mathcal{C},s_{d+1}M_{\mathcal{C}})$ be $d$-pre-Calabi-Yau categories and denote by $\MO_{\MA}$, $\MO_{\MB}$ and $
    \MO_{\mathcal{C}}$ the sets of objects of the respective underlying graded quivers.
    Consider two homotopic morphisms $(\Phi_0,s_{d+1}\mathbf{F}),(\Phi_0,s_{d+1}\mathbf{G}) : (\MA,s_{d+1}M_{\MA})\rightarrow (\MB,s_{d+1}M_{\MB})$ as well as a $d$-pre-Calabi-Yau morphism $(\Psi_0,s_{d+1}\mathbf{H}) : (\mathcal{C},s_{d+1}M_{\mathcal{C}})\rightarrow (\MA,s_{d+1}M_{\MA})$. 
    Then, $(\Phi_0\circ \Psi_0,s_{d+1}\mathbf{F}\circ s_{d+1}\mathbf{H})$ and $(\Phi_0\circ \Psi_0,s_{d+1}\mathbf{G}\circ s_{d+1}\mathbf{H})$ are homotopic $d$-pre-Calabi-Yau morphisms.
\end{proposition}
\begin{proof}
    Consider $a(t)+b(t)dt\in \Bar{\mathcal{L}}_d^{\Phi_0}(\MA,\MB)[t,dt]$ be such that $sa(0)=s_{d+1}\mathbf{F}$, $sa(1)=s_{d+1}\mathbf{G}$, $a(t)$ is a Maurer-Cartan element of $(\Bar{\mathcal{L}}_d^{\Phi_0}(\MA,\MB),\Bar{\ell}_{M_{\MA},M_{\MB}})$ and
    \[
    \frac{\partial a}{\partial t}(t)^{\doubar{x}}=\sum\limits_{n=2}^{\infty}\frac{1}{(n-1)!}\ell^n\big(a(t),\dots,a(t),b(t)\big)^{\doubar{x}}
    \] 
    for every $t\in\Bbbk$, $\doubar{x}\in\doubar{\MO}_{\MA}$.
    Then, by Proposition \ref{prop:pCY-category}, $sa'(t)=sa(t)\circ s_{d+1}\mathbf{H}$ is a Maurer-Cartan element of the $L_{\infty}$-algebra of $(\Bar{\mathcal{L}}_d^{\Phi_0}(\mathcal{C},\MB),\Bar{\ell}_{M_{\mathcal{C}},M_{\MB}})$ for every $t\in\Bbbk$. Moreover, it is clear that $sa'(0)=s_{d+1}\mathbf{F}\circ s_{d+1}\mathbf{H}$ and $sa'(1)=s_{d+1}\mathbf{G}\circ s_{d+1}\mathbf{H}$. We define $sb'(t)^{\doubar{x}}=\sum\mathcal{E}(\mathcalboondox{D})$ where the sum is over all the filled diagrams $\mathcalboondox{D}$ of type $\doubar{x}$ and of the form

   
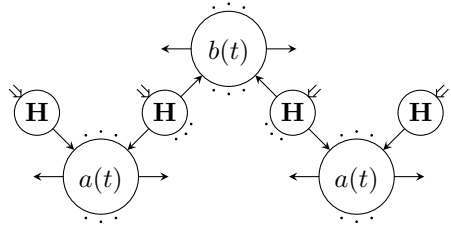
\begin{tikzpicture}[line cap=round,line join=round,x=1.0cm,y=1.0cm]
\clip(-7,-2.7) rectangle (5,1);
       \draw (0,0) circle (0.5cm);
       \draw[->,>=stealth] (0.5,0)--(0.9,0);
       \draw[rotate=180][->,>=stealth] (0.5,0)--(0.9,0);
       \draw[rotate=-45][<-,>=stealth] (0.5,0)--(0.9,0);
       \draw[rotate=-135][<-,>=stealth] (0.5,0)--(0.9,0);
       \draw (-0.85,-0.85) circle (0.3cm);
       \draw[rotate around={135:(-0.85,-0.85)},shift={(-0.55,-0.85)}]\doublefleche;
       \draw (0.85,-0.85) circle (0.3cm);
       \draw[rotate around={45:(0.85,-0.85)},shift={(1.15,-0.85)}]\doublefleche;
       \draw[rotate around={-45:(0.85,-0.85)}][->,>=stealth](1.15,-0.85)--(1.55,-0.85);
       \draw (1.7,-1.7) circle (0.5cm);
       \draw[rotate around={0:(1.7,-1.7)}][->,>=stealth](2.2,-1.7)--(2.6,-1.7);
       \draw[rotate around={180:(1.7,-1.7)}][->,>=stealth](2.2,-1.7)--(2.6,-1.7);
       \draw[rotate around={45:(1.7,-1.7)}][<-,>=stealth](2.2,-1.7)--(2.6,-1.7);
       \draw (2.55,-0.85) circle (0.3cm);
       \draw[rotate around={45:(2.55,-0.85)},shift={(2.85,-0.85)}]\doublefleche;
        \draw[rotate around={-45:(-2.55,-0.85)}][->,>=stealth](-2.25,-0.85)--(-1.85,-0.85);
       \draw (-1.7,-1.7) circle (0.5cm);
       \draw[rotate around={0:(-1.7,-1.7)}][->,>=stealth](-1.2,-1.7)--(-0.8,-1.7);
       \draw[rotate around={180:(-1.7,-1.7)}][->,>=stealth](-1.2,-1.7)--(-0.8,-1.7);
       \draw[rotate around={45:(-1.7,-1.7)}][<-,>=stealth](-1.2,-1.7)--(-0.8,-1.7);
       \draw (-2.55,-0.85) circle (0.3cm);
       \draw[rotate around={135:(-2.55,-0.85)},shift={(-2.25,-0.85)}]\doublefleche;
    \draw (0,0.25) node[anchor=north]{$b(t)$};
    \draw (-1.7,-1.45) node[anchor=north]{$a(t)$};
    \draw (1.7,-1.45) node[anchor=north]{$a(t)$};
    \draw (-2.55,-0.6) node[anchor=north]{$\mathbf{H}$};
    \draw (2.55,-0.6) node[anchor=north]{$\mathbf{H}$};
    \draw (0.85,-0.6) node[anchor=north]{$\mathbf{H}$};
    \draw (-0.85,-0.6) node[anchor=north]{$\mathbf{H}$};
    \draw[fill=black] (0,0.6) circle (0.3pt);
    \draw[fill=black] (0.2,0.55) circle (0.3pt);
    \draw[fill=black] (-0.2,0.55) circle (0.3pt);
    \draw[fill=black] (0,-0.6) circle (0.3pt);
    \draw[fill=black] (0.2,-0.55) circle (0.3pt);
    \draw[fill=black] (-0.2,-0.55) circle (0.3pt);
    \draw[fill=black] (0.5,-1) circle (0.3pt);
    \draw[fill=black] (0.57,-1.13) circle (0.3pt);
    \draw[fill=black] (0.7,-1.2) circle (0.3pt);
    \draw[fill=black] (-0.5,-1) circle (0.3pt);
    \draw[fill=black] (-0.57,-1.13) circle (0.3pt);
    \draw[fill=black] (-0.7,-1.2) circle (0.3pt);
    \draw[fill=black] (1.7,0.6-1.7) circle (0.3pt);
    \draw[fill=black] (1.9,0.55-1.7) circle (0.3pt);
    \draw[fill=black] (1.5,0.55-1.7) circle (0.3pt);
   \draw[fill=black] (1.7,-0.6-1.7) circle (0.3pt);
    \draw[fill=black] (1.9,-0.55-1.7) circle (0.3pt);
    \draw[fill=black] (1.5,-0.55-1.7) circle (0.3pt);
    \draw[fill=black] (-1.7,0.6-1.7) circle (0.3pt);
    \draw[fill=black] (-1.9,0.55-1.7) circle (0.3pt);
    \draw[fill=black] (-1.5,0.55-1.7) circle (0.3pt);
   \draw[fill=black] (-1.7,-0.6-1.7) circle (0.3pt);
    \draw[fill=black] (-1.9,-0.55-1.7) circle (0.3pt);
    \draw[fill=black] (-1.5,-0.55-1.7) circle (0.3pt);
   \end{tikzpicture}
   
    \noindent for every $\doubar{x}\in\doubar{\MO}_{\mathcal{C}}$. 
    We have to show that 
    \[\frac{\partial a'}{\partial t}(t)^{\doubar{x}}=\sum\limits_{n=2}^{\infty}\frac{1}{(n-1)!}\ell^n\big(a'(t),\dots,a'(t),b'(t)\big)^{\doubar{x}}
    \] 
    for every $\doubar{x}\in\doubar{\MO}_{\mathcal{C}}$. We have that 
    \[
    \frac{\partial a'}{\partial t}(t)^{\doubar{x}}=\sum\mathcal{E}(\mathcalboondox{D_1})+\sum\mathcal{E}(\mathcalboondox{D_2})
    \]where the sums are over all the filled diagrams $\mathcalboondox{D_1}$ and $\mathcalboondox{D_2}$ of type $\doubar{x}$ and of the form

       \begin{minipage}{21cm}

    \end{minipage}
    
    \noindent respectively.
    On the other hand, 
    \[\sum\limits_{n=2}^{\infty}\frac{1}{(n-1)!}\ell^n\big(a'(t),\dots,a'(t),b'(t)\big)^{\doubar{x}}=\sum\mathcal{E}(\mathcalboondox{D'_1})+\sum\mathcal{E}(\mathcalboondox{D'_2})+\sum\mathcal{E}(\mathcalboondox{D'_3})+\sum\mathcal{E}(\mathcalboondox{D'_4})
    \]
    where the sums are over all the filled diagrams $\mathcalboondox{D'_i}$ for $i\in\llbracket 1,4\rrbracket$ of type $\doubar{x}$ and of the form

    \begin{minipage}{21cm}
                %
    \end{minipage}
    
    \noindent respectively. More precisely, diagrams of the form $\mathcalboondox{D'_3}$ are diagrams composed of one disc filled with $M_{\mathcal{C}}$ sharing each of its outgoing arrows with discs filled with $\mathbf{H}$ which do not share any outgoing arrow with the disc filled with $b(t)$ in the diagram. We will call the number of discs filled with $\mathbf{H}$ between the disc filled with $M_{\mathcal{C}}$ and the disc filled with $b(t)$ the distance from $M_{\mathcal{C}}$ to $b(t)$. Note that diagrams of the form $\mathcalboondox{D'_1}$ are the only diagrams where this distance is $1$. Moreover, the distance from $M_{\mathcal{C}}$ to $b(t)$ for diagrams of the form $\mathcalboondox{D'_3}$ is at least $2$ and at most $\llg(\doubar{x})$.
    Similarly, diagrams of the form $\mathcalboondox{D'_4}$ are diagrams composed of one disc filled with $M_{\mathcal{B}}$ sharing each of its incoming arrows witch discs filled with $a(t)$. We will call the number of discs filled with $\mathbf{H}$ between the disc filled with $M_{\mathcal{B}}$ and the disc filled with $b(t)$ the distance from $M_{\mathcal{B}}$ to $b(t)$. Note that diagrams of the form $\mathcalboondox{D'_2}$ are the only diagrams where this distance is $0$. The distance from $M_{\mathcal{B}}$ to $b(t)$ for diagrams of the form $\mathcalboondox{D'_4}$ is at least $1$ and at most $\llg(\doubar{x})-1$. Moreover, diagrams of the form $\mathcalboondox{D'_3}$ (resp. $\mathcalboondox{D'_4}$) carry a minus sign if $b(t)$ is before $M_{\mathcal{C}}$ (resp. after $M_{\mathcal{B}}$) in the application order since $|b(t)|=-1$. 
    We have that $\sum\mathcal{E}(\mathcalboondox{D'_2})=\sum\mathcal{E}(\mathcalboondox{D_2})$. Moreover, since $(\Psi_0,s_{d+1}\mathbf{H})$ is a $d$-pre-Calabi-Yau morphism, we have that $\sum\mathcal{E}(\mathcalboondox{D'_1})=\sum\mathcal{E}(\mathcalboondox{D_1})+\sum\mathcal{E}(\mathcalboondox{D_1\dprimeind})$ where the last sum is over all the filled diagrams of type $\doubar{x}$ and of the form 

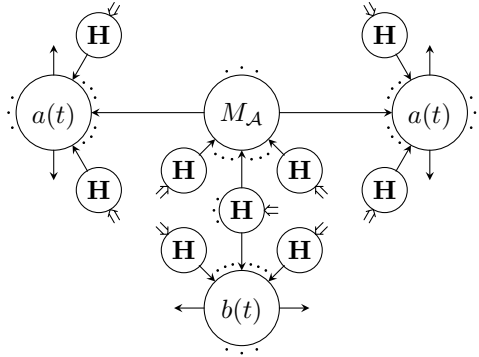
\begin{tikzpicture}[line cap=round,line join=round,x=1.0cm,y=1.0cm]
\clip(-7,-3.3) rectangle (5,1.5);
        \draw (0,0) circle (0.5cm);
        \draw[->,>=stealth] (0.5,0)--(2,0);
    \draw (2.5,0) circle (0.5cm);
    \draw[rotate around={90:(2.5,0)}] [->,>=stealth] (3,0)--(3.4,0);
     \draw[rotate around={-90:(2.5,0)}] [->,>=stealth](3,0)--(3.4,0);
     \draw[rotate around={120:(2.5,0)}][<-,>=stealth] (3,0)--(3.4,0);
     \draw (1.9,1.03) circle (0.3cm);
     \draw[rotate around={120:(1.9,1.03)},shift={(2.2,1.03)}]\doublefleche;
     \draw[rotate around={-120:(2.5,0)}][<-,>=stealth] (3,0)--(3.4,0);
     \draw (1.9,-1.03) circle (0.3cm);
     \draw[rotate around={-120:(1.9,-1.03)},shift={(2.2,-1.03)}]\doublefleche;
     \draw (-2.5,0) circle (0.5cm);
    \draw[rotate around={90:(-2.5,0)}] [->,>=stealth] (-2,0)--(-1.6,0);
     \draw[rotate around={-90:(-2.5,0)}] [->,>=stealth] (-2,0)--(-1.6,0);
     \draw[rotate around={60:(-2.5,0)}][<-,>=stealth] (-2,0)--(-1.6,0);
     \draw (-1.9,-1.03) circle (0.3cm);
     \draw[rotate around={60:(-1.9,1.03)},shift={(-1.6,1.03)}]\doublefleche;
     \draw[rotate around={-60:(-2.5,0)}][<-,>=stealth] (-2,0)--(-1.6,0);
     \draw (-1.9,1.03) circle (0.3cm);
     \draw[rotate around={-60:(-1.9,-1.03)},shift={(-1.6,-1.03)}]\doublefleche;
     \draw[rotate=180][->,>=stealth] (0.5,0)--(2,0);
        \draw[rotate=-135][<-,>=stealth](0.5,0)--(0.8,0);
        \draw (-0.77,-0.77) circle (0.3cm);
        \draw[rotate around={-135:(-0.77,-0.77)},shift={(-0.47,-0.77)}]\doublefleche;
        \draw[rotate=-45][<-,>=stealth](0.5,0)--(0.8,0);
        \draw (0.77,-0.77) circle (0.3cm);
        \draw[rotate around={-45:(0.77,-0.77)},shift={(1.07,-0.77)}]\doublefleche;
        \draw[rotate=-90][<-,>=stealth](0.5,0)--(1,0);
        \draw (0,-1.3) circle (0.3cm);
        \draw[rotate around={0:(0,-1.3)},shift={(0.3,-1.3))}] \doublefleche;
        \draw[rotate around={-90:(0,-1.3)}][->,>=stealth] (0.3,-1.3)--(0.8,-1.3);
        \draw (0,-2.6) circle (0.5cm);
        \draw[rotate around={0:(0,-2.6)}][->,>=stealth] (0.5,-2.6)--(0.9,-2.6);
        \draw[rotate around={180:(0,-2.6)}][->,>=stealth] (0.5,-2.6)--(0.9,-2.6);
        \draw[rotate around={45:(0,-2.6)}][<-,>=stealth](0.5,-2.6)--(0.8,-2.6);
        \draw (0.77,-1.83) circle (0.3cm);
        \draw [rotate around={45:(0.77,-1.83)},shift={(1.07,-1.83)}] \doublefleche;
        \draw[rotate around={135:(0,-2.6)}][<-,>=stealth](0.5,-2.6)--(0.8,-2.6);
        \draw (-0.77,-1.83) circle (0.3cm);
        \draw [rotate around={135:(-0.77,-1.83)},shift={(-0.47,-1.83)}] \doublefleche;
\draw (0,0.25) node[anchor=north]{$M_{\MA}$};
\draw (-2.5,0.25) node[anchor=north]{$a(t)$};
\draw (2.5,0.25) node[anchor=north]{$a(t)$};
\draw (0,-2.35) node[anchor=north]{$b(t)$};
\draw (0,-1.05) node[anchor=north]{$\mathbf{H}$};
\draw (-0.77,-1.58) node[anchor=north]{$\mathbf{H}$};
\draw (0.77,-1.58) node[anchor=north]{$\mathbf{H}$};
\draw (0.77,-0.52) node[anchor=north]{$\mathbf{H}$};
\draw (-0.77,-0.52) node[anchor=north]{$\mathbf{H}$};
\draw (1.9,1.28) node[anchor=north]{$\mathbf{H}$};
\draw (-1.9,1.28) node[anchor=north]{$\mathbf{H}$};
\draw (1.9,-0.78) node[anchor=north]{$\mathbf{H}$};
\draw (-1.9,-0.78) node[anchor=north]{$\mathbf{H}$};
    \draw[fill=black] (0,0.6) circle (0.3pt);
    \draw[fill=black] (0.2,0.55) circle (0.3pt);
    \draw[fill=black] (-0.2,0.55) circle (0.3pt);
    \draw[fill=black] (-0.25,-0.6) circle (0.3pt);
    \draw[fill=black] (-0.35,-0.55) circle (0.3pt);
    \draw[fill=black] (-0.13,-0.62) circle (0.3pt);
    \draw[fill=black] (0.25,-0.6) circle (0.3pt);
    \draw[fill=black] (0.35,-0.55) circle (0.3pt);
    \draw[fill=black] (0.13,-0.62) circle (0.3pt);
    \draw[fill=black] (-3.1,0) circle (0.3pt);
    \draw[fill=black] (-3.05,0.2) circle (0.3pt);
    \draw[fill=black] (-3.05,-0.2) circle (0.3pt);
     \draw[fill=black] (3.1,0) circle (0.3pt);
    \draw[fill=black] (3.05,0.2) circle (0.3pt);
    \draw[fill=black] (3.05,-0.2) circle (0.3pt);
    \draw[fill=black] (-0.37,-1.3) circle (0.3pt);
    \draw[fill=black] (-0.32,-1.15) circle (0.3pt);
    \draw[fill=black] (-0.32,-1.45) circle (0.3pt);
    \draw[fill=black] (-1.95,0.2) circle (0.3pt);
    \draw[fill=black] (-2,0.32) circle (0.3pt);
    \draw[fill=black] (-2.1,0.4) circle (0.3pt);
    \draw[fill=black] (-1.95,-0.2) circle (0.3pt);
    \draw[fill=black] (-2,-0.32) circle (0.3pt);
    \draw[fill=black] (-2.1,-0.4) circle (0.3pt);
    \draw[fill=black] (1.95,0.2) circle (0.3pt);
    \draw[fill=black] (2,0.32) circle (0.3pt);
    \draw[fill=black] (2.1,0.4) circle (0.3pt);
    \draw[fill=black] (1.95,-0.2) circle (0.3pt);
    \draw[fill=black] (2,-0.32) circle (0.3pt);
    \draw[fill=black] (2.1,-0.4) circle (0.3pt);
    \draw[fill=black] (0.2,-3.15) circle (0.3pt);
    \draw[fill=black] (-0.2,-3.15) circle (0.3pt);
    \draw[fill=black] (0,-3.2) circle (0.3pt);
    \draw[fill=black] (-0.12,-2.02) circle (0.3pt);
    \draw[fill=black] (-0.22,-2.05) circle (0.3pt);
    \draw[fill=black] (-0.32,-2.1) circle (0.3pt);
    \draw[fill=black] (0.12,-2.02) circle (0.3pt);
    \draw[fill=black] (0.22,-2.05) circle (0.3pt);
    \draw[fill=black] (0.32,-2.1) circle (0.3pt);
    \end{tikzpicture}  

    \noindent carrying a minus sign if the disc filled with $b(t)$ is after the one filled with $M_{\MA}$ in the application order. Note that the distance between $M_{\MA}$ and $b(t)$ in diagrams of this form is $1$. Using that $(\Psi_0,s_{d+1}\mathbf{H})$ is a $d$-pre-Calabi-Yau morphism, we get $\sum\mathcal{E}(\mathcalboondox{D'_3})=\sum\mathcal{E}(\mathcalboondox{D_3\dprimeind})+\sum\mathcal{E}(\mathcalboondox{D_3\tprime})$ where the sums are over all the filled diagrams of type $\doubar{x}$ and of the form 

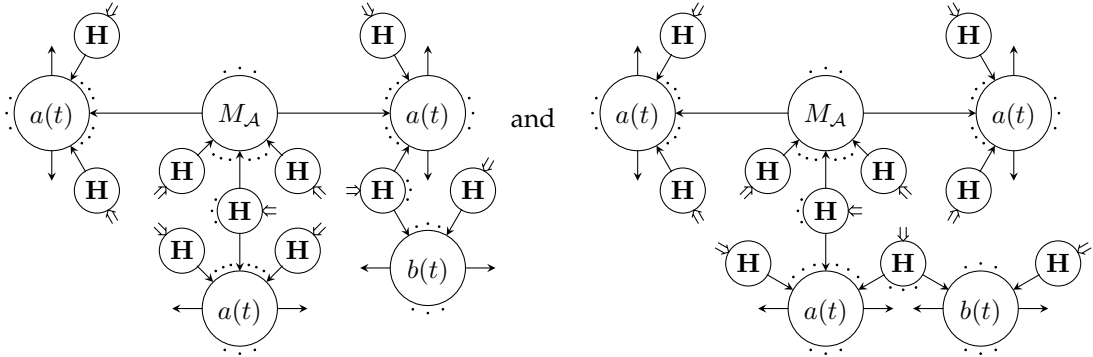
\begin{minipage}{21cm}
    \begin{tikzpicture}[line cap=round,line join=round,x=1.0cm,y=1.0cm]
\clip(-3.2,-3.5) rectangle (4.5,2);
        \draw (0,0) circle (0.5cm);
        \draw[->,>=stealth] (0.5,0)--(2,0);
    \draw (2.5,0) circle (0.5cm);
    \draw[rotate around={90:(2.5,0)}] [->,>=stealth] (3,0)--(3.4,0);
     \draw[rotate around={-90:(2.5,0)}] [->,>=stealth](3,0)--(3.4,0);
     \draw[rotate around={120:(2.5,0)}][<-,>=stealth] (3,0)--(3.4,0);
     \draw (1.9,1.03) circle (0.3cm);
     \draw[rotate around={120:(1.9,1.03)},shift={(2.2,1.03)}]\doublefleche;
     \draw[rotate around={-120:(2.5,0)}][<-,>=stealth] (3,0)--(3.4,0);
     \draw (1.9,-1.03) circle (0.3cm);
     \draw[rotate around={180:(1.9,-1.03)},shift={(2.2,-1.03)}]\doublefleche;
     \draw[rotate around={-60:(1.9,-1.03)}][->,>=stealth] (2.2,-1.03)--(2.6,-1.03);
     \draw (2.5,-2.06) circle (0.5cm);
     \draw[rotate around={0:(2.5,-2.06)}][->,>=stealth](3,-2.06)--(3.4,-2.06);
     \draw[rotate around={180:(2.5,-2.06)}][->,>=stealth](3,-2.06)--(3.4,-2.06);
     \draw[rotate around={60:(2.5,-2.06)}][<-,>=stealth](3,-2.06)--(3.4,-2.06);
     \draw (3.1,-1.03) circle (0.3cm);
    \draw[rotate around={60:(3.1,-1.03)},shift={(3.4,-1.03)}]\doublefleche;
     \draw (-2.5,0) circle (0.5cm);
    \draw[rotate around={90:(-2.5,0)}] [->,>=stealth] (-2,0)--(-1.6,0);
     \draw[rotate around={-90:(-2.5,0)}] [->,>=stealth] (-2,0)--(-1.6,0);
     \draw[rotate around={60:(-2.5,0)}][<-,>=stealth] (-2,0)--(-1.6,0);
     \draw (-1.9,-1.03) circle (0.3cm);
     \draw[rotate around={60:(-1.9,1.03)},shift={(-1.6,1.03)}]\doublefleche;
     \draw[rotate around={-60:(-2.5,0)}][<-,>=stealth] (-2,0)--(-1.6,0);
     \draw (-1.9,1.03) circle (0.3cm);
     \draw[rotate around={-60:(-1.9,-1.03)},shift={(-1.6,-1.03)}]\doublefleche;
     \draw[rotate=180][->,>=stealth] (0.5,0)--(2,0);
        \draw[rotate=-135][<-,>=stealth](0.5,0)--(0.8,0);
        \draw (-0.77,-0.77) circle (0.3cm);
        \draw[rotate around={-135:(-0.77,-0.77)},shift={(-0.47,-0.77)}]\doublefleche;
        \draw[rotate=-45][<-,>=stealth](0.5,0)--(0.8,0);
        \draw (0.77,-0.77) circle (0.3cm);
        \draw[rotate around={-45:(0.77,-0.77)},shift={(1.07,-0.77)}]\doublefleche;
        \draw[rotate=-90][<-,>=stealth](0.5,0)--(1,0);
        \draw (0,-1.3) circle (0.3cm);
        \draw[rotate around={0:(0,-1.3)},shift={(0.3,-1.3))}] \doublefleche;
        \draw[rotate around={-90:(0,-1.3)}][->,>=stealth] (0.3,-1.3)--(0.8,-1.3);
        \draw (0,-2.6) circle (0.5cm);
        \draw[rotate around={0:(0,-2.6)}][->,>=stealth] (0.5,-2.6)--(0.9,-2.6);
        \draw[rotate around={180:(0,-2.6)}][->,>=stealth] (0.5,-2.6)--(0.9,-2.6);
        \draw[rotate around={45:(0,-2.6)}][<-,>=stealth](0.5,-2.6)--(0.8,-2.6);
        \draw (0.77,-1.83) circle (0.3cm);
        \draw [rotate around={45:(0.77,-1.83)},shift={(1.07,-1.83)}] \doublefleche;
        \draw[rotate around={135:(0,-2.6)}][<-,>=stealth](0.5,-2.6)--(0.8,-2.6);
        \draw (-0.77,-1.83) circle (0.3cm);
        \draw [rotate around={135:(-0.77,-1.83)},shift={(-0.47,-1.83)}] \doublefleche;
\draw (0,0.25) node[anchor=north]{$M_{\MA}$};
\draw (-2.5,0.25) node[anchor=north]{$a(t)$};
\draw (2.5,0.25) node[anchor=north]{$a(t)$};
\draw (0,-2.35) node[anchor=north]{$a(t)$};
\draw (2.5,-1.8) node[anchor=north]{$b(t)$};
\draw (0,-1.05) node[anchor=north]{$\mathbf{H}$};
\draw (-0.77,-1.58) node[anchor=north]{$\mathbf{H}$};
\draw (0.77,-1.58) node[anchor=north]{$\mathbf{H}$};
\draw (0.77,-0.52) node[anchor=north]{$\mathbf{H}$};
\draw (-0.77,-0.52) node[anchor=north]{$\mathbf{H}$};
\draw (1.9,1.28) node[anchor=north]{$\mathbf{H}$};
\draw (-1.9,1.28) node[anchor=north]{$\mathbf{H}$};
\draw (1.9,-0.78) node[anchor=north]{$\mathbf{H}$};
\draw (-1.9,-0.78) node[anchor=north]{$\mathbf{H}$};
\draw (3.1,-0.78) node[anchor=north]{$\mathbf{H}$};
\draw (3.9,0.2) node[anchor=north] {and};
    \draw[fill=black] (0,0.6) circle (0.3pt);
    \draw[fill=black] (0.2,0.55) circle (0.3pt);
    \draw[fill=black] (-0.2,0.55) circle (0.3pt);
    \draw[fill=black] (-0.25,-0.6) circle (0.3pt);
    \draw[fill=black] (-0.35,-0.55) circle (0.3pt);
    \draw[fill=black] (-0.13,-0.62) circle (0.3pt);
    \draw[fill=black] (0.25,-0.6) circle (0.3pt);
    \draw[fill=black] (0.35,-0.55) circle (0.3pt);
    \draw[fill=black] (0.13,-0.62) circle (0.3pt);
    \draw[fill=black] (-3.1,0) circle (0.3pt);
    \draw[fill=black] (-3.05,0.2) circle (0.3pt);
    \draw[fill=black] (-3.05,-0.2) circle (0.3pt);
     \draw[fill=black] (3.1,0) circle (0.3pt);
    \draw[fill=black] (3.05,0.2) circle (0.3pt);
    \draw[fill=black] (3.05,-0.2) circle (0.3pt);
    \draw[fill=black] (-0.37,-1.3) circle (0.3pt);
    \draw[fill=black] (-0.32,-1.15) circle (0.3pt);
    \draw[fill=black] (-0.32,-1.45) circle (0.3pt);
    \draw[fill=black] (-1.95,0.2) circle (0.3pt);
    \draw[fill=black] (-2,0.32) circle (0.3pt);
    \draw[fill=black] (-2.1,0.4) circle (0.3pt);
    \draw[fill=black] (-1.95,-0.2) circle (0.3pt);
    \draw[fill=black] (-2,-0.32) circle (0.3pt);
    \draw[fill=black] (-2.1,-0.4) circle (0.3pt);
    \draw[fill=black] (1.95,0.2) circle (0.3pt);
    \draw[fill=black] (2,0.32) circle (0.3pt);
    \draw[fill=black] (2.1,0.4) circle (0.3pt);
    \draw[fill=black] (1.95,-0.2) circle (0.3pt);
    \draw[fill=black] (2,-0.32) circle (0.3pt);
    \draw[fill=black] (2.1,-0.4) circle (0.3pt);
    \draw[fill=black] (0.2,-3.15) circle (0.3pt);
    \draw[fill=black] (-0.2,-3.15) circle (0.3pt);
    \draw[fill=black] (0,-3.2) circle (0.3pt);
    \draw[fill=black] (-0.12,-2.02) circle (0.3pt);
    \draw[fill=black] (-0.22,-2.05) circle (0.3pt);
    \draw[fill=black] (-0.32,-2.1) circle (0.3pt);
    \draw[fill=black] (0.12,-2.02) circle (0.3pt);
    \draw[fill=black] (0.22,-2.05) circle (0.3pt);
    \draw[fill=black] (0.32,-2.1) circle (0.3pt);
    \draw[fill=black] (2.5,-2.66) circle (0.3pt);
    \draw[fill=black] (2.3,-2.61) circle (0.3pt);
    \draw[fill=black] (2.7,-2.61) circle (0.3pt);
    \draw[fill=black] (2.5,-1.47) circle (0.3pt);
    \draw[fill=black] (2.35,-1.51) circle (0.3pt);
    \draw[fill=black] (2.65,-1.51) circle (0.3pt);
    \draw[fill=black] (2.23,-0.9) circle (0.3pt);
    \draw[fill=black] (2.23,-1.15) circle (0.3pt);
    \draw[fill=black] (2.27,-1.02) circle (0.3pt);
 \end{tikzpicture}
    \begin{tikzpicture}[line cap=round,line join=round,x=1.0cm,y=1.0cm]
\clip(-3.2,-3.5) rectangle (5,2);
        \draw (0,0) circle (0.5cm);
        \draw[->,>=stealth] (0.5,0)--(2,0);
    \draw (2.5,0) circle (0.5cm);
    \draw[rotate around={90:(2.5,0)}] [->,>=stealth] (3,0)--(3.4,0);
     \draw[rotate around={-90:(2.5,0)}] [->,>=stealth](3,0)--(3.4,0);
     \draw[rotate around={120:(2.5,0)}][<-,>=stealth] (3,0)--(3.4,0);
     \draw (1.9,1.03) circle (0.3cm);
     \draw[rotate around={120:(1.9,1.03)},shift={(2.2,1.03)}]\doublefleche;
     \draw[rotate around={-120:(2.5,0)}][<-,>=stealth] (3,0)--(3.4,0);
     \draw (1.9,-1.03) circle (0.3cm);
     \draw[rotate around={-120:(1.9,-1.03)},shift={(2.2,-1.03)}]\doublefleche;
     \draw (-2.5,0) circle (0.5cm);
    \draw[rotate around={90:(-2.5,0)}] [->,>=stealth] (-2,0)--(-1.6,0);
     \draw[rotate around={-90:(-2.5,0)}] [->,>=stealth] (-2,0)--(-1.6,0);
     \draw[rotate around={60:(-2.5,0)}][<-,>=stealth] (-2,0)--(-1.6,0);
     \draw (-1.9,-1.03) circle (0.3cm);
     \draw[rotate around={60:(-1.9,1.03)},shift={(-1.6,1.03)}]\doublefleche;
     \draw[rotate around={-60:(-2.5,0)}][<-,>=stealth] (-2,0)--(-1.6,0);
     \draw (-1.9,1.03) circle (0.3cm);
     \draw[rotate around={-60:(-1.9,-1.03)},shift={(-1.6,-1.03)}]\doublefleche;
     \draw[rotate=180][->,>=stealth] (0.5,0)--(2,0);
        \draw[rotate=-135][<-,>=stealth](0.5,0)--(0.8,0);
        \draw (-0.77,-0.77) circle (0.3cm);
        \draw[rotate around={-135:(-0.77,-0.77)},shift={(-0.47,-0.77)}]\doublefleche;
        \draw[rotate=-45][<-,>=stealth](0.5,0)--(0.8,0);
        \draw (0.77,-0.77) circle (0.3cm);
        \draw[rotate around={-45:(0.77,-0.77)},shift={(1.07,-0.77)}]\doublefleche;
        \draw[rotate=-90][<-,>=stealth](0.5,0)--(1,0);
        \draw (0,-1.3) circle (0.3cm);
        \draw[rotate around={0:(0,-1.3)},shift={(0.3,-1.3))}] \doublefleche;
        \draw[rotate around={-90:(0,-1.3)}][->,>=stealth] (0.3,-1.3)--(0.8,-1.3);
        \draw (0,-2.6) circle (0.5cm);
        \draw[rotate around={0:(0,-2.6)}][->,>=stealth] (0.5,-2.6)--(0.9,-2.6);
        \draw[rotate around={180:(0,-2.6)}][->,>=stealth] (0.5,-2.6)--(0.9,-2.6);
        \draw[rotate around={30:(0,-2.6)}][<-,>=stealth](0.5,-2.6)--(0.9,-2.6);
        \draw (1.03,-2) circle (0.3cm);
        \draw [rotate around={90:(1.03,-2)},shift={(1.33,-2)}] \doublefleche;
        \draw[rotate around={-30:(1.03,-2)}][->,>=stealth](1.33,-2)--(1.73,-2);
        \draw (2.06,-2.6) circle (0.5cm);
         \draw[rotate around={0:(2.06,-2.6)}][->,>=stealth] (2.56,-2.6)--(2.96,-2.6);
        \draw[rotate around={180:(2.06,-2.6)}][->,>=stealth] (2.56,-2.6)--(2.96,-2.6);
        \draw[rotate around={30:(2.06,-2.6)}][<-,>=stealth] (2.56,-2.6)--(2.96,-2.6);
        \draw (3.09,-2) circle (0.3cm);
        \draw[rotate around={30:(3.09,-2)},shift={(3.39,-2)}]\doublefleche;
        \draw[rotate around={150:(0,-2.6)}][<-,>=stealth](0.5,-2.6)--(0.9,-2.6);
        \draw (-1.03,-2) circle (0.3cm);
        \draw [rotate around={150:(-1.03,-2)},shift={(-0.73,-2)}] \doublefleche;
\draw (0,0.25) node[anchor=north]{$M_{\MA}$};
\draw (-2.5,0.25) node[anchor=north]{$a(t)$};
\draw (2.5,0.25) node[anchor=north]{$a(t)$};
\draw(2.06,-2.35) node[anchor=north]{$b(t)$};
\draw (0,-2.35) node[anchor=north]{$a(t)$};
\draw (0,-1.05) node[anchor=north]{$\mathbf{H}$};
\draw (3.09,-1.75) node[anchor=north]{$\mathbf{H}$};
\draw (-1.03,-1.75) node[anchor=north]{$\mathbf{H}$};
\draw (1.03,-1.75) node[anchor=north]{$\mathbf{H}$};
\draw (0.77,-0.52) node[anchor=north]{$\mathbf{H}$};
\draw (-0.77,-0.52) node[anchor=north]{$\mathbf{H}$};
\draw (1.9,1.28) node[anchor=north]{$\mathbf{H}$};
\draw (-1.9,1.28) node[anchor=north]{$\mathbf{H}$};
\draw (1.9,-0.78) node[anchor=north]{$\mathbf{H}$};
\draw (-1.9,-0.78) node[anchor=north]{$\mathbf{H}$};
    \draw[fill=black] (0,0.6) circle (0.3pt);
    \draw[fill=black] (0.2,0.55) circle (0.3pt);
    \draw[fill=black] (-0.2,0.55) circle (0.3pt);
    \draw[fill=black] (-0.25,-0.6) circle (0.3pt);
    \draw[fill=black] (-0.35,-0.55) circle (0.3pt);
    \draw[fill=black] (-0.13,-0.62) circle (0.3pt);
    \draw[fill=black] (0.25,-0.6) circle (0.3pt);
    \draw[fill=black] (0.35,-0.55) circle (0.3pt);
    \draw[fill=black] (0.13,-0.62) circle (0.3pt);
    \draw[fill=black] (-3.1,0) circle (0.3pt);
    \draw[fill=black] (-3.05,0.2) circle (0.3pt);
    \draw[fill=black] (-3.05,-0.2) circle (0.3pt);
     \draw[fill=black] (3.1,0) circle (0.3pt);
    \draw[fill=black] (3.05,0.2) circle (0.3pt);
    \draw[fill=black] (3.05,-0.2) circle (0.3pt);
    \draw[fill=black] (-0.37,-1.3) circle (0.3pt);
    \draw[fill=black] (-0.32,-1.15) circle (0.3pt);
    \draw[fill=black] (-0.32,-1.45) circle (0.3pt);
    \draw[fill=black] (-1.95,0.2) circle (0.3pt);
    \draw[fill=black] (-2,0.32) circle (0.3pt);
    \draw[fill=black] (-2.1,0.4) circle (0.3pt);
    \draw[fill=black] (-1.95,-0.2) circle (0.3pt);
    \draw[fill=black] (-2,-0.32) circle (0.3pt);
    \draw[fill=black] (-2.1,-0.4) circle (0.3pt);
    \draw[fill=black] (1.95,0.2) circle (0.3pt);
    \draw[fill=black] (2,0.32) circle (0.3pt);
    \draw[fill=black] (2.1,0.4) circle (0.3pt);
    \draw[fill=black] (1.95,-0.2) circle (0.3pt);
    \draw[fill=black] (2,-0.32) circle (0.3pt);
    \draw[fill=black] (2.1,-0.4) circle (0.3pt);
    \draw[fill=black] (0.2,-3.15) circle (0.3pt);
    \draw[fill=black] (-0.2,-3.15) circle (0.3pt);
    \draw[fill=black] (0,-3.2) circle (0.3pt);
    \draw[fill=black] (-0.15,-2.02) circle (0.3pt);
    \draw[fill=black] (-0.3,-2.08) circle (0.3pt);
    \draw[fill=black] (-0.43,-2.18) circle (0.3pt);
    \draw[fill=black] (0.15,-2.02) circle (0.3pt);
    \draw[fill=black] (0.3,-2.08) circle (0.3pt);
    \draw[fill=black] (0.43,-2.18) circle (0.3pt);
    \draw[fill=black] (1.03,-2.35) circle (0.3pt);
    \draw[fill=black] (0.86,-2.32) circle (0.3pt);
    \draw[fill=black] (1.2,-2.32) circle (0.3pt);
    \draw[fill=black] (2.06,-3.2) circle (0.3pt);
    \draw[fill=black] (2.26,-3.15) circle (0.3pt);
    \draw[fill=black] (1.86,-3.15) circle (0.3pt);
    \draw[fill=black] (2.06,-2) circle (0.3pt);
    \draw[fill=black] (2.26,-2.05) circle (0.3pt);
    \draw[fill=black] (1.86,-2.05) circle (0.3pt);
    \end{tikzpicture}
    \end{minipage}

    \noindent respectively, where the distance between $M_{\MA}$ and $b(t)$ is at least $1$ in diagrams of the form $\mathcalboondox{D_3\dprimeind}$ and at least $2$ in diagrams of the form $\mathcalboondox{D_3\tprime}$ and where the diagrams carry a minus sign if the disc filled with $b(t)$ is after the one filled with $M_{\MA}$ in the application order.
    Finally, using that $a(t)$ is a Maurer-Cartan element of $(\bar {\mathcal{L}}^{\Phi_0}_d(\MA,\MB),\Bar{\ell}_{M_{\MA},M_{\MB}})$, we obtain the equality $\sum\mathcal{E}(\mathcalboondox{D_4'})=\sum\mathcal{E}(\mathcalboondox{D_4\dprimeind})+\sum\mathcal{E}(\mathcalboondox{D_4\tprime})$ where the sums are over all the filled diagrams of type $\doubar{x}$ and of the form

        \begin{minipage}{21cm}
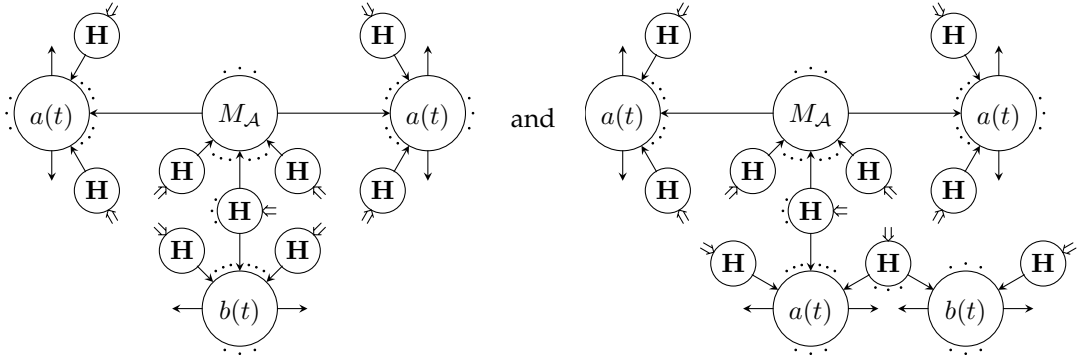
             \begin{tikzpicture}[line cap=round,line join=round,x=1.0cm,y=1.0cm]
\clip(-3.2,-3.5) rectangle (4.5,2);
        \draw (0,0) circle (0.5cm);
        \draw[->,>=stealth] (0.5,0)--(2,0);
    \draw (2.5,0) circle (0.5cm);
    \draw[rotate around={90:(2.5,0)}] [->,>=stealth] (3,0)--(3.4,0);
     \draw[rotate around={-90:(2.5,0)}] [->,>=stealth](3,0)--(3.4,0);
     \draw[rotate around={120:(2.5,0)}][<-,>=stealth] (3,0)--(3.4,0);
     \draw (1.9,1.03) circle (0.3cm);
     \draw[rotate around={120:(1.9,1.03)},shift={(2.2,1.03)}]\doublefleche;
     \draw[rotate around={-120:(2.5,0)}][<-,>=stealth] (3,0)--(3.4,0);
     \draw (1.9,-1.03) circle (0.3cm);
     \draw[rotate around={-120:(1.9,-1.03)},shift={(2.2,-1.03)}]\doublefleche;
     \draw (-2.5,0) circle (0.5cm);
    \draw[rotate around={90:(-2.5,0)}] [->,>=stealth] (-2,0)--(-1.6,0);
     \draw[rotate around={-90:(-2.5,0)}] [->,>=stealth] (-2,0)--(-1.6,0);
     \draw[rotate around={60:(-2.5,0)}][<-,>=stealth] (-2,0)--(-1.6,0);
     \draw (-1.9,-1.03) circle (0.3cm);
     \draw[rotate around={60:(-1.9,1.03)},shift={(-1.6,1.03)}]\doublefleche;
     \draw[rotate around={-60:(-2.5,0)}][<-,>=stealth] (-2,0)--(-1.6,0);
     \draw (-1.9,1.03) circle (0.3cm);
     \draw[rotate around={-60:(-1.9,-1.03)},shift={(-1.6,-1.03)}]\doublefleche;
     \draw[rotate=180][->,>=stealth] (0.5,0)--(2,0);
        \draw[rotate=-135][<-,>=stealth](0.5,0)--(0.8,0);
        \draw (-0.77,-0.77) circle (0.3cm);
        \draw[rotate around={-135:(-0.77,-0.77)},shift={(-0.47,-0.77)}]\doublefleche;
        \draw[rotate=-45][<-,>=stealth](0.5,0)--(0.8,0);
        \draw (0.77,-0.77) circle (0.3cm);
        \draw[rotate around={-45:(0.77,-0.77)},shift={(1.07,-0.77)}]\doublefleche;
        \draw[rotate=-90][<-,>=stealth](0.5,0)--(1,0);
        \draw (0,-1.3) circle (0.3cm);
        \draw[rotate around={0:(0,-1.3)},shift={(0.3,-1.3))}] \doublefleche;
        \draw[rotate around={-90:(0,-1.3)}][->,>=stealth] (0.3,-1.3)--(0.8,-1.3);
        \draw (0,-2.6) circle (0.5cm);
        \draw[rotate around={0:(0,-2.6)}][->,>=stealth] (0.5,-2.6)--(0.9,-2.6);
        \draw[rotate around={180:(0,-2.6)}][->,>=stealth] (0.5,-2.6)--(0.9,-2.6);
        \draw[rotate around={45:(0,-2.6)}][<-,>=stealth](0.5,-2.6)--(0.8,-2.6);
        \draw (0.77,-1.83) circle (0.3cm);
        \draw [rotate around={45:(0.77,-1.83)},shift={(1.07,-1.83)}] \doublefleche;
        \draw[rotate around={135:(0,-2.6)}][<-,>=stealth](0.5,-2.6)--(0.8,-2.6);
        \draw (-0.77,-1.83) circle (0.3cm);
        \draw [rotate around={135:(-0.77,-1.83)},shift={(-0.47,-1.83)}] \doublefleche;
\draw (0,0.25) node[anchor=north]{$M_{\MA}$};
\draw (-2.5,0.25) node[anchor=north]{$a(t)$};
\draw (2.5,0.25) node[anchor=north]{$a(t)$};
\draw (0,-2.35) node[anchor=north]{$b(t)$};
\draw (0,-1.05) node[anchor=north]{$\mathbf{H}$};
\draw (-0.77,-1.58) node[anchor=north]{$\mathbf{H}$};
\draw (0.77,-1.58) node[anchor=north]{$\mathbf{H}$};
\draw (0.77,-0.52) node[anchor=north]{$\mathbf{H}$};
\draw (-0.77,-0.52) node[anchor=north]{$\mathbf{H}$};
\draw (1.9,1.28) node[anchor=north]{$\mathbf{H}$};
\draw (-1.9,1.28) node[anchor=north]{$\mathbf{H}$};
\draw (1.9,-0.78) node[anchor=north]{$\mathbf{H}$};
\draw (-1.9,-0.78) node[anchor=north]{$\mathbf{H}$};
\draw (3.9,0.2) node[anchor=north] {and};
    \draw[fill=black] (0,0.6) circle (0.3pt);
    \draw[fill=black] (0.2,0.55) circle (0.3pt);
    \draw[fill=black] (-0.2,0.55) circle (0.3pt);
    \draw[fill=black] (-0.25,-0.6) circle (0.3pt);
    \draw[fill=black] (-0.35,-0.55) circle (0.3pt);
    \draw[fill=black] (-0.13,-0.62) circle (0.3pt);
    \draw[fill=black] (0.25,-0.6) circle (0.3pt);
    \draw[fill=black] (0.35,-0.55) circle (0.3pt);
    \draw[fill=black] (0.13,-0.62) circle (0.3pt);
    \draw[fill=black] (-3.1,0) circle (0.3pt);
    \draw[fill=black] (-3.05,0.2) circle (0.3pt);
    \draw[fill=black] (-3.05,-0.2) circle (0.3pt);
     \draw[fill=black] (3.1,0) circle (0.3pt);
    \draw[fill=black] (3.05,0.2) circle (0.3pt);
    \draw[fill=black] (3.05,-0.2) circle (0.3pt);
    \draw[fill=black] (-0.37,-1.3) circle (0.3pt);
    \draw[fill=black] (-0.32,-1.15) circle (0.3pt);
    \draw[fill=black] (-0.32,-1.45) circle (0.3pt);
    \draw[fill=black] (-1.95,0.2) circle (0.3pt);
    \draw[fill=black] (-2,0.32) circle (0.3pt);
    \draw[fill=black] (-2.1,0.4) circle (0.3pt);
    \draw[fill=black] (-1.95,-0.2) circle (0.3pt);
    \draw[fill=black] (-2,-0.32) circle (0.3pt);
    \draw[fill=black] (-2.1,-0.4) circle (0.3pt);
    \draw[fill=black] (1.95,0.2) circle (0.3pt);
    \draw[fill=black] (2,0.32) circle (0.3pt);
    \draw[fill=black] (2.1,0.4) circle (0.3pt);
    \draw[fill=black] (1.95,-0.2) circle (0.3pt);
    \draw[fill=black] (2,-0.32) circle (0.3pt);
    \draw[fill=black] (2.1,-0.4) circle (0.3pt);
    \draw[fill=black] (0.2,-3.15) circle (0.3pt);
    \draw[fill=black] (-0.2,-3.15) circle (0.3pt);
    \draw[fill=black] (0,-3.2) circle (0.3pt);
    \draw[fill=black] (-0.12,-2.02) circle (0.3pt);
    \draw[fill=black] (-0.22,-2.05) circle (0.3pt);
    \draw[fill=black] (-0.32,-2.1) circle (0.3pt);
    \draw[fill=black] (0.12,-2.02) circle (0.3pt);
    \draw[fill=black] (0.22,-2.05) circle (0.3pt);
    \draw[fill=black] (0.32,-2.1) circle (0.3pt);
    \end{tikzpicture}
    \begin{tikzpicture}[line cap=round,line join=round,x=1.0cm,y=1.0cm]
\clip(-3,-3.5) rectangle (5,2);
        \draw (0,0) circle (0.5cm);
        \draw[->,>=stealth] (0.5,0)--(2,0);
    \draw (2.5,0) circle (0.5cm);
    \draw[rotate around={90:(2.5,0)}] [->,>=stealth] (3,0)--(3.4,0);
     \draw[rotate around={-90:(2.5,0)}] [->,>=stealth](3,0)--(3.4,0);
     \draw[rotate around={120:(2.5,0)}][<-,>=stealth] (3,0)--(3.4,0);
     \draw (1.9,1.03) circle (0.3cm);
     \draw[rotate around={120:(1.9,1.03)},shift={(2.2,1.03)}]\doublefleche;
     \draw[rotate around={-120:(2.5,0)}][<-,>=stealth] (3,0)--(3.4,0);
     \draw (1.9,-1.03) circle (0.3cm);
     \draw[rotate around={-120:(1.9,-1.03)},shift={(2.2,-1.03)}]\doublefleche;
     \draw (-2.5,0) circle (0.5cm);
    \draw[rotate around={90:(-2.5,0)}] [->,>=stealth] (-2,0)--(-1.6,0);
     \draw[rotate around={-90:(-2.5,0)}] [->,>=stealth] (-2,0)--(-1.6,0);
     \draw[rotate around={60:(-2.5,0)}][<-,>=stealth] (-2,0)--(-1.6,0);
     \draw (-1.9,-1.03) circle (0.3cm);
     \draw[rotate around={60:(-1.9,1.03)},shift={(-1.6,1.03)}]\doublefleche;
     \draw[rotate around={-60:(-2.5,0)}][<-,>=stealth] (-2,0)--(-1.6,0);
     \draw (-1.9,1.03) circle (0.3cm);
     \draw[rotate around={-60:(-1.9,-1.03)},shift={(-1.6,-1.03)}]\doublefleche;
     \draw[rotate=180][->,>=stealth] (0.5,0)--(2,0);
        \draw[rotate=-135][<-,>=stealth](0.5,0)--(0.8,0);
        \draw (-0.77,-0.77) circle (0.3cm);
        \draw[rotate around={-135:(-0.77,-0.77)},shift={(-0.47,-0.77)}]\doublefleche;
        \draw[rotate=-45][<-,>=stealth](0.5,0)--(0.8,0);
        \draw (0.77,-0.77) circle (0.3cm);
        \draw[rotate around={-45:(0.77,-0.77)},shift={(1.07,-0.77)}]\doublefleche;
        \draw[rotate=-90][<-,>=stealth](0.5,0)--(1,0);
        \draw (0,-1.3) circle (0.3cm);
        \draw[rotate around={0:(0,-1.3)},shift={(0.3,-1.3))}] \doublefleche;
        \draw[rotate around={-90:(0,-1.3)}][->,>=stealth] (0.3,-1.3)--(0.8,-1.3);
        \draw (0,-2.6) circle (0.5cm);
        \draw[rotate around={0:(0,-2.6)}][->,>=stealth] (0.5,-2.6)--(0.9,-2.6);
        \draw[rotate around={180:(0,-2.6)}][->,>=stealth] (0.5,-2.6)--(0.9,-2.6);
        \draw[rotate around={30:(0,-2.6)}][<-,>=stealth](0.5,-2.6)--(0.9,-2.6);
        \draw (1.03,-2) circle (0.3cm);
        \draw [rotate around={90:(1.03,-2)},shift={(1.33,-2)}] \doublefleche;
        \draw[rotate around={-30:(1.03,-2)}][->,>=stealth](1.33,-2)--(1.73,-2);
        \draw (2.06,-2.6) circle (0.5cm);
         \draw[rotate around={0:(2.06,-2.6)}][->,>=stealth] (2.56,-2.6)--(2.96,-2.6);
        \draw[rotate around={180:(2.06,-2.6)}][->,>=stealth] (2.56,-2.6)--(2.96,-2.6);
        \draw[rotate around={30:(2.06,-2.6)}][<-,>=stealth] (2.56,-2.6)--(2.96,-2.6);
        \draw (3.09,-2) circle (0.3cm);
        \draw[rotate around={30:(3.09,-2)},shift={(3.39,-2)}]\doublefleche;
        \draw[rotate around={150:(0,-2.6)}][<-,>=stealth](0.5,-2.6)--(0.9,-2.6);
        \draw (-1.03,-2) circle (0.3cm);
        \draw [rotate around={150:(-1.03,-2)},shift={(-0.73,-2)}] \doublefleche;
\draw (0,0.25) node[anchor=north]{$M_{\MA}$};
\draw (-2.5,0.25) node[anchor=north]{$a(t)$};
\draw (2.5,0.25) node[anchor=north]{$a(t)$};
\draw(2.06,-2.35) node[anchor=north]{$b(t)$};
\draw (0,-2.35) node[anchor=north]{$a(t)$};
\draw (0,-1.05) node[anchor=north]{$\mathbf{H}$};
\draw (3.09,-1.75) node[anchor=north]{$\mathbf{H}$};
\draw (-1.03,-1.75) node[anchor=north]{$\mathbf{H}$};
\draw (1.03,-1.75) node[anchor=north]{$\mathbf{H}$};
\draw (0.77,-0.52) node[anchor=north]{$\mathbf{H}$};
\draw (-0.77,-0.52) node[anchor=north]{$\mathbf{H}$};
\draw (1.9,1.28) node[anchor=north]{$\mathbf{H}$};
\draw (-1.9,1.28) node[anchor=north]{$\mathbf{H}$};
\draw (1.9,-0.78) node[anchor=north]{$\mathbf{H}$};
\draw (-1.9,-0.78) node[anchor=north]{$\mathbf{H}$};
    \draw[fill=black] (0,0.6) circle (0.3pt);
    \draw[fill=black] (0.2,0.55) circle (0.3pt);
    \draw[fill=black] (-0.2,0.55) circle (0.3pt);
    \draw[fill=black] (-0.25,-0.6) circle (0.3pt);
    \draw[fill=black] (-0.35,-0.55) circle (0.3pt);
    \draw[fill=black] (-0.13,-0.62) circle (0.3pt);
    \draw[fill=black] (0.25,-0.6) circle (0.3pt);
    \draw[fill=black] (0.35,-0.55) circle (0.3pt);
    \draw[fill=black] (0.13,-0.62) circle (0.3pt);
    \draw[fill=black] (-3.1,0) circle (0.3pt);
    \draw[fill=black] (-3.05,0.2) circle (0.3pt);
    \draw[fill=black] (-3.05,-0.2) circle (0.3pt);
     \draw[fill=black] (3.1,0) circle (0.3pt);
    \draw[fill=black] (3.05,0.2) circle (0.3pt);
    \draw[fill=black] (3.05,-0.2) circle (0.3pt);
    \draw[fill=black] (-0.37,-1.3) circle (0.3pt);
    \draw[fill=black] (-0.32,-1.15) circle (0.3pt);
    \draw[fill=black] (-0.32,-1.45) circle (0.3pt);
    \draw[fill=black] (-1.95,0.2) circle (0.3pt);
    \draw[fill=black] (-2,0.32) circle (0.3pt);
    \draw[fill=black] (-2.1,0.4) circle (0.3pt);
    \draw[fill=black] (-1.95,-0.2) circle (0.3pt);
    \draw[fill=black] (-2,-0.32) circle (0.3pt);
    \draw[fill=black] (-2.1,-0.4) circle (0.3pt);
    \draw[fill=black] (1.95,0.2) circle (0.3pt);
    \draw[fill=black] (2,0.32) circle (0.3pt);
    \draw[fill=black] (2.1,0.4) circle (0.3pt);
    \draw[fill=black] (1.95,-0.2) circle (0.3pt);
    \draw[fill=black] (2,-0.32) circle (0.3pt);
    \draw[fill=black] (2.1,-0.4) circle (0.3pt);
    \draw[fill=black] (0.2,-3.15) circle (0.3pt);
    \draw[fill=black] (-0.2,-3.15) circle (0.3pt);
    \draw[fill=black] (0,-3.2) circle (0.3pt);
    \draw[fill=black] (-0.12,-2.02) circle (0.3pt);
    \draw[fill=black] (-0.22,-2.05) circle (0.3pt);
    \draw[fill=black] (-0.32,-2.1) circle (0.3pt);
    \draw[fill=black] (0.12,-2.02) circle (0.3pt);
    \draw[fill=black] (0.22,-2.05) circle (0.3pt);
    \draw[fill=black] (0.32,-2.1) circle (0.3pt);
    \draw[fill=black] (1.03,-2.35) circle (0.3pt);
    \draw[fill=black] (0.86,-2.32) circle (0.3pt);
    \draw[fill=black] (1.2,-2.32) circle (0.3pt);
    \draw[fill=black] (2.06,-3.2) circle (0.3pt);
    \draw[fill=black] (2.26,-3.15) circle (0.3pt);
    \draw[fill=black] (1.86,-3.15) circle (0.3pt);
    \draw[fill=black] (2.06,-2) circle (0.3pt);
    \draw[fill=black] (2.26,-2.05) circle (0.3pt);
    \draw[fill=black] (1.86,-2.05) circle (0.3pt);
    \end{tikzpicture}   
    \end{minipage}

    \noindent respectively where the distance between $M_{\MA}$ and $b(t)$ is at least $1$ in both diagrams and where the diagrams carry a minus sign if the disc filled with $b(t)$ is before the one filled with $M_{\MA}$ in the application order. 
    
    Therefore, $\sum\mathcal{E}(\mathcalboondox{D_1}\dprimeind)+\sum\mathcal{E}(\mathcalboondox{D_3}\dprimeind)+\sum\mathcal{E}(\mathcalboondox{D_3}\tprime)=-\big(\sum\mathcal{E}(\mathcalboondox{D_4\dprimeind})+\sum\mathcal{E}(\mathcalboondox{D_4\tprime})\big)$ which concludes the proof.
\end{proof}

\begin{proposition}
\label{prop:htpy-cp-left}
    Let $(\MA,s_{d+1}M_{\MA})$, $(\MB,s_{d+1}M_{\MB})$ and $(\mathcal{C},s_{d+1}M_{\mathcal{C}})$ be $d$-pre-Calabi-Yau categories and denote by $\MO_{\MA}$, $\MO_{\MB}$ and $
    \MO_{\mathcal{C}}$ the sets of objects of the respective underlying graded quivers. Consider two homotopic morphisms $(\Phi_0,s_{d+1}\mathbf{F}),(\Phi_0,s_{d+1}\mathbf{G}) : (\MA,s_{d+1}M_{\MA})\rightarrow (\MB,s_{d+1}M_{\MB})$ as well as a $d$-pre-Calabi-Yau morphism $(\Phi_0,s_{d+1}\mathbf{H}) : (\MB,s_{d+1}M_{\MB})\rightarrow (\mathcal{C},s_{d+1}M_{\mathcal{C}})$. 
    Then, $(\Phi_0\circ \Psi_0,s_{d+1}\mathbf{H}\circ s_{d+1}\mathbf{F})$ and $(\Phi_0\circ \Psi_0,s_{d+1}\mathbf{H}\circ s_{d+1}\mathbf{G})$ are homotopic $d$-pre-Calabi-Yau morphisms.
\end{proposition}
\begin{proof}
    Consider $a(t)+b(t)dt\in \Bar{\mathcal{L}}_d^{\Phi_0}(\MA,\MB)[t,dt]$ such that $a(0)=s_{d+1}\mathbf{F}$, $a(1)=s_{d+1}\mathbf{G}$, $a(t)$ is a Maurer-Cartan element of $(\Bar{\mathcal{L}}_d^{\Phi_0}(\MA,\MB),\Bar{\ell}_{M_{\MA},M_{\MB}})$ and
    \[
    \frac{\partial a}{\partial t}(t)^{\doubar{x}}=\sum\limits_{n=2}^{\infty}\frac{1}{(n-1)!}\ell^n\big(a(t),\dots,a(t),b(t)\big)^{\doubar{x}}
    \] 
    for every $t\in\Bbbk$, $\doubar{x}\in\doubar{\MO}_{\MA}$. Then, by Proposition \ref{prop:pCY-category}, $s_{-1}(sa'(t))=s_{-1}(s_{d+1}\mathbf{H}\circ sa(t))$ is a Maurer-Cartan element of the $L_{\infty}$-algebra of $(\Bar{\mathcal{L}}_d^{\Phi_0}(\mathcal{C},\MB),\Bar{\ell}_{M_{\mathcal{C}},M_{\MB}})$ for every $t\in\Bbbk$. Moreover, it is clear that we have $sa'(0)=s_{d+1}\mathbf{H}\circ s_{d+1}\mathbf{F}$ and $sa'(1)=s_{d+1}\mathbf{H}\circ s_{d+1}\mathbf{G}$.
    We define $sb'(t)^{\doubar{x}}=\sum\mathcal{E}(\mathcalboondox{D})$ where the sum is over all the filled diagrams $\mathcalboondox{D}$ of type $\doubar{x}$ and of the form

    
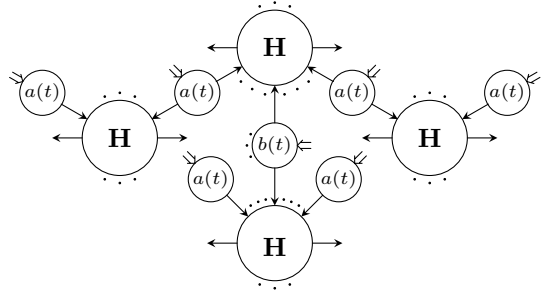
\begin{tikzpicture}[line cap=round,line join=round,x=1.0cm,y=1.0cm]
\clip(-7,-3.5) rectangle (5,1);
        \draw (0,0) circle (0.5cm);
        \draw[->,>=stealth] (0.5,0)--(0.9,0);
        \draw[rotate=180][->,>=stealth] (0.5,0)--(0.9,0);
        \draw[rotate=-150][<-,>=stealth](0.5,0)--(0.9,0);
        \draw (-1.03,-0.6) circle (0.3cm);
        \draw[rotate around={135:(-1.03,-0.6)},shift={(-0.73,-0.6))}] \doublefleche;
        \draw[rotate around={-150:(-1.03,-0.6)}][->,>=stealth] (-0.73,-0.6)--(-0.33,-0.6);
        \draw (-2.06,-1.2) circle (0.5cm);
        \draw[rotate around={180:(-2.06,-1.2)}][->,>=stealth](-1.56,-1.2)--(-1.16,-1.2);
        \draw[rotate around={0:(-2.06,-1.2)}][->,>=stealth](-1.56,-1.2)--(-1.16,-1.2);
        \draw[rotate around={150:(-2.06,-1.2)}][<-,>=stealth](-1.56,-1.2)--(-1.16,-1.2);
        \draw (-3.09,-0.6) circle (0.3cm);
        \draw[rotate around={150:(-3.09,-0.6)},shift={(-2.79,-0.6))}] \doublefleche;
        \draw[rotate=-30][<-,>=stealth](0.5,0)--(0.9,0);
        \draw (1.03,-0.6) circle (0.3cm);
        \draw[rotate around={45:(1.03,-0.6)},shift={(1.33,-0.6)}] \doublefleche;
        \draw[rotate around={-30:(1.03,-0.6)}][->,>=stealth] (1.33,-0.6)--(1.73,-0.6);
        \draw (2.06,-1.2) circle (0.5cm);
        \draw[rotate around={180:(2.06,-1.2)}][->,>=stealth](2.56,-1.2)--(2.96,-1.2);
        \draw[rotate around={0:(2.06,-1.2)}][->,>=stealth](2.56,-1.2)--(2.96,-1.2);
        \draw[rotate around={30:(2.06,-1.2)}][<-,>=stealth](2.56,-1.2)--(2.96,-1.2);
        \draw (3.09,-0.6) circle (0.3cm);
        \draw[rotate around={30:(3.09,-0.6)},shift={(3.39,-0.6))}] \doublefleche;
        \draw[rotate=-90][<-,>=stealth](0.5,0)--(1,0);
        \draw (0,-1.3) circle (0.3cm);
        \draw[rotate around={0:(0,-1.3)},shift={(0.3,-1.3))}] \doublefleche;
        \draw[rotate around={-90:(0,-1.3)}][->,>=stealth] (0.3,-1.3)--(0.8,-1.3);
        \draw (0,-2.6) circle (0.5cm);
        \draw[rotate around={0:(0,-2.6)}][->,>=stealth] (0.5,-2.6)--(0.9,-2.6);
        \draw[rotate around={180:(0,-2.6)}][->,>=stealth] (0.5,-2.6)--(0.9,-2.6);
        \draw[rotate around={45:(0,-2.6)}][<-,>=stealth](0.5,-2.6)--(0.9,-2.6);
        \draw (0.85,-1.75) circle (0.3cm);
        \draw [rotate around={45:(0.85,-1.75)},shift={(1.15,-1.75)}] \doublefleche;
        \draw[rotate around={135:(0,-2.6)}][<-,>=stealth](0.5,-2.6)--(0.9,-2.6);
        \draw (-0.85,-1.75) circle (0.3cm);
        \draw [rotate around={135:(-0.85,-1.75)},shift={(-0.55,-1.75)}] \doublefleche;
\draw (0,0.25) node[anchor=north] {$\mathbf{H}$};
\draw (2.06,-0.95) node[anchor=north] {$\mathbf{H}$};
\draw (0,-2.4) node[anchor=north] {$\mathbf{H}$};
\draw (-2.06,-0.95) node[anchor=north] {$\mathbf{H}$};
\draw (0,-1.05) node[anchor=north] {$\scriptstyle{b(t)}$};
\draw (0.85,-1.5) node[anchor=north] {$\scriptstyle{a(t)}$};
\draw (-0.85,-1.5) node[anchor=north] {$\scriptstyle{a(t)}$};
\draw (-3.09,-0.35) node[anchor=north] {$\scriptstyle{a(t)}$};
\draw (3.09,-0.35) node[anchor=north] {$\scriptstyle{a(t)}$};
\draw (-1.03,-0.35) node[anchor=north] {$\scriptstyle{a(t)}$};
\draw (1.03,-0.35) node[anchor=north] {$\scriptstyle{a(t)}$};
    \draw[fill=black] (0,0.6) circle (0.3pt);
    \draw[fill=black] (0.2,0.55) circle (0.3pt);
    \draw[fill=black] (-0.2,0.55) circle (0.3pt);
    \draw[fill=black] (-0.3,-0.55) circle (0.3pt);
    \draw[fill=black] (-0.45,-0.45) circle (0.3pt);
    \draw[fill=black] (-0.13,-0.62) circle (0.3pt);
    \draw[fill=black] (0.3,-0.55) circle (0.3pt);
    \draw[fill=black] (0.45,-0.45) circle (0.3pt);
    \draw[fill=black] (0.13,-0.62) circle (0.3pt);
    \draw[fill=black] (2.06,-0.6) circle (0.3pt);
    \draw[fill=black] (2.26,-0.65) circle (0.3pt);
    \draw[fill=black] (1.86,-0.65) circle (0.3pt);
    \draw[fill=black] (2.06,-1.8) circle (0.3pt);
    \draw[fill=black] (2.26,-1.75) circle (0.3pt);
    \draw[fill=black] (1.86,-1.75) circle (0.3pt);
    \draw[fill=black] (-2.06,-0.6) circle (0.3pt);
    \draw[fill=black] (-2.26,-0.65) circle (0.3pt);
    \draw[fill=black] (-1.86,-0.65) circle (0.3pt);
    \draw[fill=black] (-2.06,-1.8) circle (0.3pt);
    \draw[fill=black] (-2.26,-1.75) circle (0.3pt);
    \draw[fill=black] (-1.86,-1.75) circle (0.3pt);
    \draw[fill=black] (-0.37,-1.3) circle (0.3pt);
    \draw[fill=black] (-0.32,-1.15) circle (0.3pt);
    \draw[fill=black] (-0.32,-1.45) circle (0.3pt);
    \draw[fill=black] (0.2,-3.15) circle (0.3pt);
    \draw[fill=black] (-0.2,-3.15) circle (0.3pt);
    \draw[fill=black] (0,-3.2) circle (0.3pt);
    \draw[fill=black] (-0.12,-2.02) circle (0.3pt);
    \draw[fill=black] (-0.22,-2.05) circle (0.3pt);
    \draw[fill=black] (-0.32,-2.1) circle (0.3pt);
    \draw[fill=black] (0.12,-2.02) circle (0.3pt);
    \draw[fill=black] (0.22,-2.05) circle (0.3pt);
    \draw[fill=black] (0.32,-2.1) circle (0.3pt);
    \end{tikzpicture}

\noindent for every $\doubar{x}\in\doubar{\MO}_{\MA}$. We have to show that 
    \[
    \frac{\partial a'}{\partial t}(t)^{\doubar{x}}=\sum\limits_{n=2}^{\infty}\frac{1}{(n-1)!}\ell^n\big(a'(t),\dots,a'(t),b'(t)\big)^{\doubar{x}}
    \]
    for every $\doubar{x}\in\doubar{\MO}_{\MA}$. We have that 
    \[
    \frac{\partial a'}{\partial t}(t)^{\doubar{x}}=\sum\mathcal{E}(\mathcalboondox{D_1})+\sum\mathcal{E}(\mathcalboondox{D_2})
    \]
    where the sums are over all the filled diagrams $\mathcalboondox{D_1}$ and $\mathcalboondox{D_2}$ of type $\doubar{x}$ and of the form

    \begin{minipage}{21cm}

    \end{minipage}
    
    \noindent respectively.
    On the other hand, \[\sum\limits_{n=2}^{\infty}\frac{1}{(n-1)!}\ell^n\big(a'(t),\dots,a'(t),b'(t)\big)^{\doubar{x}}=\sum\mathcal{E}(\mathcalboondox{D'_1})+\sum\mathcal{E}(\mathcalboondox{D'_2})+\sum\mathcal{E}(\mathcalboondox{D'_3})+\sum\mathcal{E}(\mathcalboondox{D'_4})
    \]
    where the sums are over all the filled diagrams $\mathcalboondox{D'_i}$ for $i\in\llbracket 1,4\rrbracket$ of type $\doubar{x}$ and of the form

        \begin{minipage}{21cm}
        %
    
    \end{minipage}
    
    \noindent respectively. More precisely, diagrams of the form $\mathcalboondox{D'_3}$ are diagrams composed of one disc filled with $M_{\mathcal{A}}$ sharing each of its outgoing arrows witch discs filled with $a(t)$. We will call the number of discs filled with $\mathbf{H}$ between the disc filled with $M_{\mathcal{A}}$ and the disc filled with $b(t)$ the distance from $M_{\mathcal{A}}$ to $b(t)$. Note that diagrams of the form $\mathcalboondox{D'_1}$ are the only diagrams where this distance is $0$. Moreover, the distance from $M_{\mathcal{A}}$ to $b(t)$ for diagrams of the form $\mathcalboondox{D'_3}$ is at least $1$ and at most $\llg(\doubar{x})-1$.
    Similarly, diagrams of the form $\mathcalboondox{D'_4}$ are diagrams composed of one disc filled with $M_{\mathcal{C}}$ sharing each of its incoming arrows witch discs filled with $\mathbf{H}$ which do not share any of their incoming arrows with a disc filled with $b(t)$. We will call the number of discs filled with $\mathbf{H}$ between the disc filled with $M_{\mathcal{C}}$ and the disc filled with $b(t)$ the distance from $M_{\mathcal{C}}$ to $b(t)$. Note that diagrams of the form $\mathcalboondox{D'_2}$ are the only diagrams where this distance is $1$. The distance from $M_{\mathcal{C}}$ to $b(t)$ for diagrams of the form $\mathcalboondox{D'_4}$ is at least $2$ and at most $\llg(\doubar{x})$. Moreover, diagrams of the form $\mathcalboondox{D'_3}$ (resp. $\mathcalboondox{D'_4}$) carry a minus sign if $b(t)$ is before $M_{\mathcal{A}}$ (resp. after $M_{\mathcal{C}}$) in the application order since $|b(t)|=-1$. 
    We have that $\sum\mathcal{E}(\mathcalboondox{D'_1})=\sum\mathcal{E}(\mathcalboondox{D_1})$ and since $(\Phi_0,sa(t))$ is a $d$-pre-Calabi-Yau morphism, we have that $\sum\mathcal{E}(\mathcalboondox{D'_3})=\sum\mathcal{E}(\mathcalboondox{D_3\dprimeind})+\sum\mathcal{E}(\mathcalboondox{D_3\tprime})$ where the sums are over all the filled diagrams of type $\doubar{x}$ and of the form 

   \begin{minipage}{21cm}
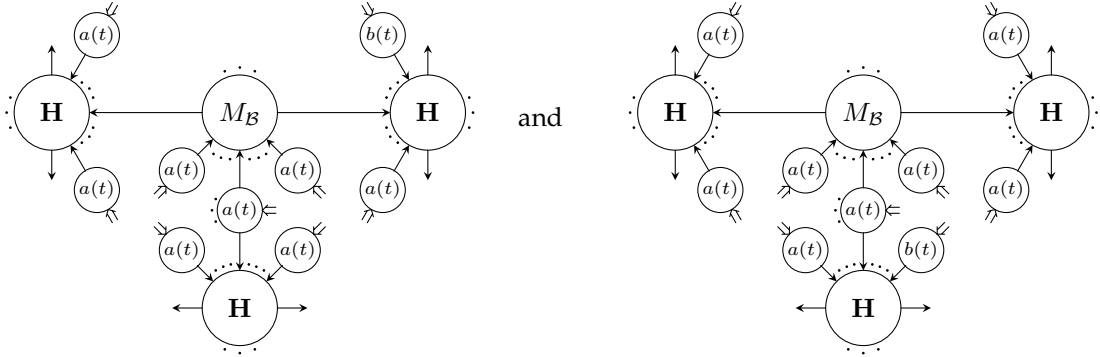
           \begin{tikzpicture}[line cap=round,line join=round,x=1.0cm,y=1.0cm]
\clip(-3.2,-3.5) rectangle (5,2);
        \draw (0,0) circle (0.5cm);
        \draw[->,>=stealth] (0.5,0)--(2,0);
    \draw (2.5,0) circle (0.5cm);
    \draw[rotate around={90:(2.5,0)}] [->,>=stealth] (3,0)--(3.4,0);
     \draw[rotate around={-90:(2.5,0)}] [->,>=stealth](3,0)--(3.4,0);
     \draw[rotate around={120:(2.5,0)}][<-,>=stealth] (3,0)--(3.4,0);
     \draw (1.9,1.03) circle (0.3cm);
     \draw[rotate around={120:(1.9,1.03)},shift={(2.2,1.03)}]\doublefleche;
     \draw[rotate around={-120:(2.5,0)}][<-,>=stealth] (3,0)--(3.4,0);
     \draw (1.9,-1.03) circle (0.3cm);
     \draw[rotate around={-120:(1.9,-1.03)},shift={(2.2,-1.03)}]\doublefleche;
     \draw (-2.5,0) circle (0.5cm);
    \draw[rotate around={90:(-2.5,0)}] [->,>=stealth] (-2,0)--(-1.6,0);
     \draw[rotate around={-90:(-2.5,0)}] [->,>=stealth] (-2,0)--(-1.6,0);
     \draw[rotate around={60:(-2.5,0)}][<-,>=stealth] (-2,0)--(-1.6,0);
     \draw (-1.9,-1.03) circle (0.3cm);
     \draw[rotate around={60:(-1.9,1.03)},shift={(-1.6,1.03)}]\doublefleche;
     \draw[rotate around={-60:(-2.5,0)}][<-,>=stealth] (-2,0)--(-1.6,0);
     \draw (-1.9,1.03) circle (0.3cm);
     \draw[rotate around={-60:(-1.9,-1.03)},shift={(-1.6,-1.03)}]\doublefleche;
     \draw[rotate=180][->,>=stealth] (0.5,0)--(2,0);
        \draw[rotate=-135][<-,>=stealth](0.5,0)--(0.8,0);
        \draw (-0.77,-0.77) circle (0.3cm);
        \draw[rotate around={-135:(-0.77,-0.77)},shift={(-0.47,-0.77)}]\doublefleche;
        \draw[rotate=-45][<-,>=stealth](0.5,0)--(0.8,0);
        \draw (0.77,-0.77) circle (0.3cm);
        \draw[rotate around={-45:(0.77,-0.77)},shift={(1.07,-0.77)}]\doublefleche;
        \draw[rotate=-90][<-,>=stealth](0.5,0)--(1,0);
        \draw (0,-1.3) circle (0.3cm);
        \draw[rotate around={0:(0,-1.3)},shift={(0.3,-1.3))}] \doublefleche;
        \draw[rotate around={-90:(0,-1.3)}][->,>=stealth] (0.3,-1.3)--(0.8,-1.3);
        \draw (0,-2.6) circle (0.5cm);
        \draw[rotate around={0:(0,-2.6)}][->,>=stealth] (0.5,-2.6)--(0.9,-2.6);
        \draw[rotate around={180:(0,-2.6)}][->,>=stealth] (0.5,-2.6)--(0.9,-2.6);
        \draw[rotate around={45:(0,-2.6)}][<-,>=stealth](0.5,-2.6)--(0.8,-2.6);
        \draw (0.77,-1.83) circle (0.3cm);
        \draw [rotate around={45:(0.77,-1.83)},shift={(1.07,-1.83)}] \doublefleche;
        \draw[rotate around={135:(0,-2.6)}][<-,>=stealth](0.5,-2.6)--(0.8,-2.6);
        \draw (-0.77,-1.83) circle (0.3cm);
        \draw [rotate around={135:(-0.77,-1.83)},shift={(-0.47,-1.83)}] \doublefleche;
\draw (0,0.25) node[anchor=north]{$M_{\MB}$};
\draw (-2.5,0.25) node[anchor=north]{$\mathbf{H}$};
\draw (2.5,0.25) node[anchor=north]{$\mathbf{H}$};
\draw (0,-2.35) node[anchor=north]{$\mathbf{H}$};
\draw (0,-1.05) node[anchor=north]{$\scriptstyle{a(t)}$};
\draw (-0.77,-1.58) node[anchor=north]{$\scriptstyle{a(t)}$};
\draw (0.77,-1.58) node[anchor=north]{$\scriptstyle{a(t)}$};
\draw (0.77,-0.52) node[anchor=north]{$\scriptstyle{a(t)}$};
\draw (-0.77,-0.52) node[anchor=north]{$\scriptstyle{a(t)}$};
\draw (1.9,1.28) node[anchor=north]{$\scriptstyle{b(t)}$};
\draw (-1.9,1.28) node[anchor=north]{$\scriptstyle{a(t)}$};
\draw (1.9,-0.78) node[anchor=north]{$\scriptstyle{a(t)}$};
\draw (-1.9,-0.78) node[anchor=north]{$\scriptstyle{a(t)}$};
\draw (4,0.2) node[anchor=north]{and};
    \draw[fill=black] (0,0.6) circle (0.3pt);
    \draw[fill=black] (0.2,0.55) circle (0.3pt);
    \draw[fill=black] (-0.2,0.55) circle (0.3pt);
    \draw[fill=black] (-0.25,-0.6) circle (0.3pt);
    \draw[fill=black] (-0.35,-0.55) circle (0.3pt);
    \draw[fill=black] (-0.13,-0.62) circle (0.3pt);
    \draw[fill=black] (0.25,-0.6) circle (0.3pt);
    \draw[fill=black] (0.35,-0.55) circle (0.3pt);
    \draw[fill=black] (0.13,-0.62) circle (0.3pt);
    \draw[fill=black] (-3.1,0) circle (0.3pt);
    \draw[fill=black] (-3.05,0.2) circle (0.3pt);
    \draw[fill=black] (-3.05,-0.2) circle (0.3pt);
     \draw[fill=black] (3.1,0) circle (0.3pt);
    \draw[fill=black] (3.05,0.2) circle (0.3pt);
    \draw[fill=black] (3.05,-0.2) circle (0.3pt);
    \draw[fill=black] (-0.37,-1.3) circle (0.3pt);
    \draw[fill=black] (-0.32,-1.15) circle (0.3pt);
    \draw[fill=black] (-0.32,-1.45) circle (0.3pt);
    \draw[fill=black] (-1.95,0.2) circle (0.3pt);
    \draw[fill=black] (-2,0.32) circle (0.3pt);
    \draw[fill=black] (-2.1,0.4) circle (0.3pt);
    \draw[fill=black] (-1.95,-0.2) circle (0.3pt);
    \draw[fill=black] (-2,-0.32) circle (0.3pt);
    \draw[fill=black] (-2.1,-0.4) circle (0.3pt);
    \draw[fill=black] (1.95,0.2) circle (0.3pt);
    \draw[fill=black] (2,0.32) circle (0.3pt);
    \draw[fill=black] (2.1,0.4) circle (0.3pt);
    \draw[fill=black] (1.95,-0.2) circle (0.3pt);
    \draw[fill=black] (2,-0.32) circle (0.3pt);
    \draw[fill=black] (2.1,-0.4) circle (0.3pt);
    \draw[fill=black] (0.2,-3.15) circle (0.3pt);
    \draw[fill=black] (-0.2,-3.15) circle (0.3pt);
    \draw[fill=black] (0,-3.2) circle (0.3pt);
    \draw[fill=black] (-0.12,-2.02) circle (0.3pt);
    \draw[fill=black] (-0.22,-2.05) circle (0.3pt);
    \draw[fill=black] (-0.32,-2.1) circle (0.3pt);
    \draw[fill=black] (0.12,-2.02) circle (0.3pt);
    \draw[fill=black] (0.22,-2.05) circle (0.3pt);
    \draw[fill=black] (0.32,-2.1) circle (0.3pt);
    \end{tikzpicture}
        \begin{tikzpicture}[line cap=round,line join=round,x=1.0cm,y=1.0cm]
\clip(-3.2,-3.5) rectangle (5,2);
        \draw (0,0) circle (0.5cm);
        \draw[->,>=stealth] (0.5,0)--(2,0);
    \draw (2.5,0) circle (0.5cm);
    \draw[rotate around={90:(2.5,0)}] [->,>=stealth] (3,0)--(3.4,0);
     \draw[rotate around={-90:(2.5,0)}] [->,>=stealth](3,0)--(3.4,0);
     \draw[rotate around={120:(2.5,0)}][<-,>=stealth] (3,0)--(3.4,0);
     \draw (1.9,1.03) circle (0.3cm);
     \draw[rotate around={120:(1.9,1.03)},shift={(2.2,1.03)}]\doublefleche;
     \draw[rotate around={-120:(2.5,0)}][<-,>=stealth] (3,0)--(3.4,0);
     \draw (1.9,-1.03) circle (0.3cm);
     \draw[rotate around={-120:(1.9,-1.03)},shift={(2.2,-1.03)}]\doublefleche;
     \draw (-2.5,0) circle (0.5cm);
    \draw[rotate around={90:(-2.5,0)}] [->,>=stealth] (-2,0)--(-1.6,0);
     \draw[rotate around={-90:(-2.5,0)}] [->,>=stealth] (-2,0)--(-1.6,0);
     \draw[rotate around={60:(-2.5,0)}][<-,>=stealth] (-2,0)--(-1.6,0);
     \draw (-1.9,-1.03) circle (0.3cm);
     \draw[rotate around={60:(-1.9,1.03)},shift={(-1.6,1.03)}]\doublefleche;
     \draw[rotate around={-60:(-2.5,0)}][<-,>=stealth] (-2,0)--(-1.6,0);
     \draw (-1.9,1.03) circle (0.3cm);
     \draw[rotate around={-60:(-1.9,-1.03)},shift={(-1.6,-1.03)}]\doublefleche;
     \draw[rotate=180][->,>=stealth] (0.5,0)--(2,0);
        \draw[rotate=-135][<-,>=stealth](0.5,0)--(0.8,0);
        \draw (-0.77,-0.77) circle (0.3cm);
        \draw[rotate around={-135:(-0.77,-0.77)},shift={(-0.47,-0.77)}]\doublefleche;
        \draw[rotate=-45][<-,>=stealth](0.5,0)--(0.8,0);
        \draw (0.77,-0.77) circle (0.3cm);
        \draw[rotate around={-45:(0.77,-0.77)},shift={(1.07,-0.77)}]\doublefleche;
        \draw[rotate=-90][<-,>=stealth](0.5,0)--(1,0);
        \draw (0,-1.3) circle (0.3cm);
        \draw[rotate around={0:(0,-1.3)},shift={(0.3,-1.3))}] \doublefleche;
        \draw[rotate around={-90:(0,-1.3)}][->,>=stealth] (0.3,-1.3)--(0.8,-1.3);
        \draw (0,-2.6) circle (0.5cm);
        \draw[rotate around={0:(0,-2.6)}][->,>=stealth] (0.5,-2.6)--(0.9,-2.6);
        \draw[rotate around={180:(0,-2.6)}][->,>=stealth] (0.5,-2.6)--(0.9,-2.6);
        \draw[rotate around={45:(0,-2.6)}][<-,>=stealth](0.5,-2.6)--(0.8,-2.6);
        \draw (0.77,-1.83) circle (0.3cm);
        \draw [rotate around={45:(0.77,-1.83)},shift={(1.07,-1.83)}] \doublefleche;
        \draw[rotate around={135:(0,-2.6)}][<-,>=stealth](0.5,-2.6)--(0.8,-2.6);
        \draw (-0.77,-1.83) circle (0.3cm);
        \draw [rotate around={135:(-0.77,-1.83)},shift={(-0.47,-1.83)}] \doublefleche;
\draw (0,0.25) node[anchor=north]{$M_{\MB}$};
\draw (-2.5,0.25) node[anchor=north]{$\mathbf{H}$};
\draw (2.5,0.25) node[anchor=north]{$\mathbf{H}$};
\draw (0,-2.35) node[anchor=north]{$\mathbf{H}$};
\draw (0,-1.05) node[anchor=north]{$\scriptstyle{a(t)}$};
\draw (-0.77,-1.58) node[anchor=north]{$\scriptstyle{a(t)}$};
\draw (0.77,-1.58) node[anchor=north]{$\scriptstyle{b(t)}$};
\draw (0.77,-0.52) node[anchor=north]{$\scriptstyle{a(t)}$};
\draw (-0.77,-0.52) node[anchor=north]{$\scriptstyle{a(t)}$};
\draw (1.9,1.28) node[anchor=north]{$\scriptstyle{a(t)}$};
\draw (-1.9,1.28) node[anchor=north]{$\scriptstyle{a(t)}$};
\draw (1.9,-0.78) node[anchor=north]{$\scriptstyle{a(t)}$};
\draw (-1.9,-0.78) node[anchor=north]{$\scriptstyle{a(t)}$};
    \draw[fill=black] (0,0.6) circle (0.3pt);
    \draw[fill=black] (0.2,0.55) circle (0.3pt);
    \draw[fill=black] (-0.2,0.55) circle (0.3pt);
    \draw[fill=black] (-0.25,-0.6) circle (0.3pt);
    \draw[fill=black] (-0.35,-0.55) circle (0.3pt);
    \draw[fill=black] (-0.13,-0.62) circle (0.3pt);
    \draw[fill=black] (0.25,-0.6) circle (0.3pt);
    \draw[fill=black] (0.35,-0.55) circle (0.3pt);
    \draw[fill=black] (0.13,-0.62) circle (0.3pt);
    \draw[fill=black] (-3.1,0) circle (0.3pt);
    \draw[fill=black] (-3.05,0.2) circle (0.3pt);
    \draw[fill=black] (-3.05,-0.2) circle (0.3pt);
     \draw[fill=black] (3.1,0) circle (0.3pt);
    \draw[fill=black] (3.05,0.2) circle (0.3pt);
    \draw[fill=black] (3.05,-0.2) circle (0.3pt);
    \draw[fill=black] (-0.37,-1.3) circle (0.3pt);
    \draw[fill=black] (-0.32,-1.15) circle (0.3pt);
    \draw[fill=black] (-0.32,-1.45) circle (0.3pt);
    \draw[fill=black] (-1.95,0.2) circle (0.3pt);
    \draw[fill=black] (-2,0.32) circle (0.3pt);
    \draw[fill=black] (-2.1,0.4) circle (0.3pt);
    \draw[fill=black] (-1.95,-0.2) circle (0.3pt);
    \draw[fill=black] (-2,-0.32) circle (0.3pt);
    \draw[fill=black] (-2.1,-0.4) circle (0.3pt);
    \draw[fill=black] (1.95,0.2) circle (0.3pt);
    \draw[fill=black] (2,0.32) circle (0.3pt);
    \draw[fill=black] (2.1,0.4) circle (0.3pt);
    \draw[fill=black] (1.95,-0.2) circle (0.3pt);
    \draw[fill=black] (2,-0.32) circle (0.3pt);
    \draw[fill=black] (2.1,-0.4) circle (0.3pt);
    \draw[fill=black] (0.2,-3.15) circle (0.3pt);
    \draw[fill=black] (-0.2,-3.15) circle (0.3pt);
    \draw[fill=black] (0,-3.2) circle (0.3pt);
    \draw[fill=black] (-0.12,-2.02) circle (0.3pt);
    \draw[fill=black] (-0.22,-2.05) circle (0.3pt);
    \draw[fill=black] (-0.32,-2.1) circle (0.3pt);
    \draw[fill=black] (0.12,-2.02) circle (0.3pt);
    \draw[fill=black] (0.22,-2.05) circle (0.3pt);
    \draw[fill=black] (0.32,-2.1) circle (0.3pt);
    \end{tikzpicture}
   \end{minipage}

    \noindent respectively, where the diagrams carry a minus sign if the disc filled with $b(t)$ is before the one filled with $M_{\MA}$ in the application order. 
    Since $(\Psi_0,s_{d+1}\mathbf{H})$ is a $d$-pre-Calabi-Yau morphism, we have that $\sum\mathcal{E}(\mathcalboondox{D'_2})=\sum\mathcal{E}(\mathcalboondox{D_2})-\sum\mathcal{E}(\mathcalboondox{D_3\dprimeind})$. Again, since $(\Psi_0,s_{d+1}\mathbf{H})$ is a $d$-pre-Calabi-Yau morphism, we have that $\sum\mathcal{E}(\mathcalboondox{D'_4})=-\sum\mathcal{E}(\mathcalboondox{D_3\tprime})$ which concludes the proof.
\end{proof}

\begin{corollary}
      Let $(\MA,s_{d+1}M_{\MA})$, $(\MB,s_{d+1}M_{\MB})$ and $(\mathcal{C},s_{d+1}M_{\mathcal{C}})$ be $d$-pre-Calabi-Yau categories and $\MO_{\MA}$, $\MO_{\MB}$, $
    \MO_{\mathcal{C}}$ the sets of objects of the respective underlying graded quivers. Consider pairs of homotopic morphisms 
    \begin{align*}
        &(\Phi_0,s_{d+1}\mathbf{F}),(\Phi_0,s_{d+1}\mathbf{G}) : (\MA,s_{d+1}M_{\MA})\rightarrow (\MB,s_{d+1}M_{\MB})
        \\&\text{ and } (\Phi_0,s_{d+1}\mathbf{H}),(\Phi_0,s_{d+1}\mathbf{I}) : (\MB,s_{d+1}M_{\MB})\rightarrow (\mathcal{C},s_{d+1}M_{\mathcal{C}}).
    \end{align*} 
    \noindent Then, $s_{d+1}\mathbf{H}\circ s_{d+1}\mathbf{F}$ and $s_{d+1}\mathbf{I}\circ s_{d+1}\mathbf{G}$ are homotopic $d$-pre-Calabi-Yau morphisms.
\end{corollary}
\begin{proof}
    It suffices to combine Propositions \ref{prop:htpy-cp-left} and  \ref{prop:htpy-cp-right} together with Proposition \ref{prop:htpy-equiv-rel}.
\end{proof}

\begin{definition}
    Let $(\MA,s_{d+1}M_{\MA})$, $(\MB,s_{d+1}M_{\MB})$ be $d$-pre-Calabi-Yau categories, denote by $\MO_{\MA}$, $\MO_{\MB}$ the sets of objects of the respective underlying graded quivers and consider a map $\MO_{\MA}\rightarrow\MO_{\MB}$. A $d$-pre-Calabi-Yau morphism $(\Phi_0,s_{d+1}\mathbf{F}) : (\MA,s_{d+1}M_{\MA})\rightarrow (\MB,s_{d+1}M_{\MB})$ is a \textbf{\textcolor{ultramarine}{homotopy equivalence}} if there exists a $d$-pre-Calabi-Yau morphism $(\Psi_0,s_{d+1}\mathbf{G}) : (\MB,s_{d+1}M_{\MB})\rightarrow (\MA,s_{d+1}M_{\MA})$ such that the composition $(\Psi_0\circ \Psi_0,s_{d+1}\mathbf{G}\circ s_{d+1}\mathbf{F})$ (resp. $(\Phi_0\circ \Psi_0,s_{d+1}\mathbf{F}\circ s_{d+1}\mathbf{G})$) is homotopic to $\id_{\MA}$ (resp. $\id_{\MB}$).
    In that case, we call the $d$-pre-Calabi-Yau morphism $(\Psi_0,s_{d+1}\mathbf{G})$ a \textbf{\textcolor{ultramarine}{homotopy inverse}} of $(\Phi_0,s_{d+1}\mathbf{F})$.
\end{definition}

\begin{proposition}\label{prop:htpy-equiv-q-iso}
    If a $d$-pre-Calabi-Yau morphism is a homotopy equivalence, then it is a quasi-isomorphism. 
\end{proposition}
\begin{proof}
    It follows directly from Proposition \ref{prop:chain-htpy-strong-htpy}.
\end{proof}

\section{\texorpdfstring{From homotopic pre-Calabi-Yau morphisms to weak homotopic $A_{\infty}$-morphisms}{Relation with A-infinity morphisms}}\label{section:relation-htpy}

We now present the relation between the notion of homotopy for pre-Calabi-Yau morphisms and the notion of weak homotopy for $A_{\infty}$-morphisms. 

\begin{definition}
    Two morphisms $(sm_{\MA\oplus\MB^*},s\varphi_{\MA},s\varphi_{\MB})$ and $(sm'_{\MA\oplus\MB^*},s\psi_{\MA},s\psi_{\MB})$ of the category $\widehat{A}_{\infty}$ are \textbf{\textcolor{ultramarine}{homotopic}} if $(s\varphi_{\MA},s\psi_{\MA})$ and $(s\varphi_{\MB},s\psi_{\MB})$ are pairs of weakly homotopic $A_{\infty}$-morphisms.
\end{definition}

Let $(\MA,s_{d+1}M_{\MA})$, $(\MB,s_{d+1}M_{\MB})$ be $d$-pre-Calabi-Yau categories and consider $d$-pre-Calabi-Yau morphisms $(\Phi_0,s_{d+1}\mathbf{F}),(\Phi_0,s_{d+1}\mathbf{G}): (\MA,s_{d+1}M_{\MA})\rightarrow (\MB,s_{d+1}M_{\MB})$. By Theorem \ref{thm:main-article-1}, $s_{d+1}\mathbf{F}$ induces an $A_{\infty}$-structure $sm_{\mathbf{F}}$ on $\MA\oplus\MB^*[d-1]$ together with $A_{\infty}$-morphisms 
   \begin{equation}
\begin{tikzcd}
&(\MA \oplus \MB^*[d-1],sm_{\mathbf{F}}) \arrow[swap,"s\varphi_{\mathbf{F}}"]{dl} \arrow[swap, "s\psi_{\mathbf{F}}"]{dr}\\
(\MA \oplus \MA^*[d-1],sm_{\MA\oplus\MA^*})&& (\MB \oplus \MB^*[d-1], sm_{\MB\oplus\MB^*})
\end{tikzcd}
\end{equation}
and 
$s_{d+1}\mathbf{G}$ induces an $A_{\infty}$-structure $sm_{\mathbf{G}}$ on $\MA\oplus\MB^*[d-1]$ together with $A_{\infty}$-morphisms 
 \begin{equation}
\begin{tikzcd}
&(\MA \oplus \MB^*[d-1],sm_{\mathbf{G}}) \arrow[swap,"s\varphi_{\mathbf{G}}"]{dl} \arrow[swap,"s\psi_{\mathbf{G}}"]{dr}\\
(\MA \oplus \MA^*[d-1],sm_{\MA\oplus\MA^*})&& (\MB \oplus \MB^*[d-1], sm_{\MB\oplus\MB^*}).
\end{tikzcd}
\end{equation}

\noindent Then, we have the following result.

\begin{theorem}
\label{thm:homotopies-pCY-Ainf}
      If $(\Phi_0,s_{d+1}\mathbf{F}),(\Phi_0,s_{d+1}\mathbf{G}): (\MA,s_{d+1}M_{\MA})\rightarrow (\MB,s_{d+1}M_{\MB})$ are homotopic $d$-pre-Calabi-Yau morphisms, then the induced morphisms $(sm_{\mathbf{F}},s\varphi_{\mathbf{F}},s\psi_{\mathbf{F}})$ and $(sm_{\mathbf{G}},s\varphi_{\mathbf{G}},s\psi_{\mathbf{G}})$ of $\widehat{A}_{\infty}$ are homotopic.
\end{theorem}
\begin{proof}
    We only show that $s\varphi_{\mathbf{F}}$ and $s\varphi_{\mathbf{G}}$ are weakly homotopic $A_{\infty}$-morphisms since the case of $s\psi_{\mathbf{F}}$ and $s\psi_{\mathbf{G}}$ is similar.
    Since $(\Phi_0,s_{d+1}\mathbf{F})$ and $(\Phi_0,s_{d+1}\mathbf{G})$ are homotopic, there exists an element $a(t)+b(t)dt\in \Bar{\mathcal{L}}_d^{\Phi_0}(\MA,\MB)[t,dt]$ such that $sa(0)=s_{d+1}\mathbf{F}$, $sa(1)=s_{d+1}\mathbf{G}$, $a(t)$ is a Maurer-Cartan element of $(\Bar{\mathcal{L}}_d^{\Phi_0}(\MA,\MB),\Bar{\ell}_{M_{\MA},M_{\MB}})$ and 
    \[
    \frac{\partial a}{\partial t}(t)^{\doubar{x}}=\sum\limits_{n=2}^{\infty}\frac{1}{(n-1)!}\bar{\ell}^n\big(a(t),\dots,a(t),b(t)\big)^{\doubar{x}}
    \]
    for every $t\in\Bbbk$, $\doubar{x}\in\doubar{\MO}_{\MA}$.
    Since $a(t)$ is a Maurer-Cartan element for every $t\in\Bbbk$, $sa(t)$ is a $d$-pre-Calabi-Yau morphism $(\MA,s_{d+1}M_{\MA})\rightarrow (\MB,s_{d+1}M_{\MB})$ so by Theorem \ref{thm:main-article-1} it induces an $A_{\infty}$-structure $sm_{a(t)}$ on $\MA\oplus\MB^*[d-1]$ and an $A_{\infty}$-morphism \[s\varphi_{a(t)} : (\MA\oplus\MB^*[d-1],sm_{a(t)})\rightarrow (\MA\oplus\MA^*[d-1],sm_{\MA\oplus\MA^*}).\]   
We will show that $A(t)=(sm_{a(t)},\varphi_{a(t)},sm_{\MA\oplus\MA^*})$ and $B(t)=(sm(t),\varphi(t),0)$ induce a weak homotopy $A(t)+B(t)dt$ between $(sm_{\mathbf{F}},\varphi_{\mathbf{F}},sm_{\MA\oplus\MA^*})$ and $(sm_{\mathbf{G}},\varphi_{\mathbf{G}},sm_{\MA\oplus\MA^*})$, where $sm(t)$ is defined by $\pi_{\MA[1]}\circ sm(t)^{\doubar{x}}=\sum\mathcal{E}(\mathcalboondox{D})$, $\pi_{\MB^*[d]}\circ sm(t)^{\doubar{x}}=\sum\mathcal{E}(\mathcalboondox{D'})$ and $\varphi(t)$ by $\pi_{\MA[1]}\circ s\varphi(t)^{\doubar{x}}=0$ and $\pi_{\MA^*[d]}\circ s\varphi(t)^{\doubar{x}}=\sum\mathcal{E}(\mathcalboondox{D\dprime})$ where the sums are over all the filled diagrams $\mathcalboondox{D}$, $\mathcalboondox{D'}$ and $\mathcalboondox{D\dprime}$ of type $\doubar{x}$ and of the form

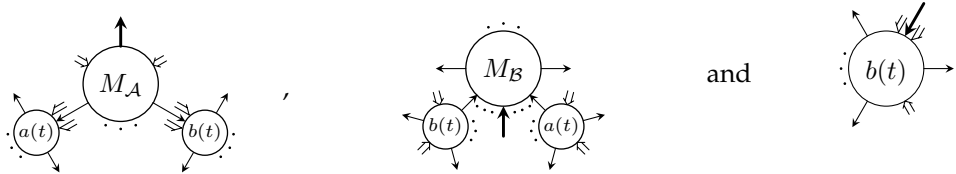
\begin{minipage}{21cm}
     \begin{tikzpicture}[line cap=round,line join=round,x=1.0cm,y=1.0cm]
\clip(-2,-1.3) rectangle (3,1);
      \draw [line width=0.5pt] (0.,0.) circle (0.5cm);
     \shadedraw[rotate=30,shift={(0.5cm,0cm)}] \doublefleche;
     \shadedraw[rotate=150,shift={(0.5cm,0cm)}] \doublefleche;
     \draw[line width=1.1pt,rotate=90][->, >= stealth, >= stealth](0.5,0)--(0.9,0);
     \draw [line width=0.5pt] (1.12,-0.65) circle (0.3cm);
     \shadedraw[shift={(0.86cm,-0.5cm)},rotate=150] \doubleflechescindeeleft;
     \shadedraw[shift={(0.86cm,-0.5cm)},rotate=150] \doubleflechescindeeright;
     \shadedraw[shift={(0.86cm,-0.5cm)},rotate=150] \fleche;
     \draw [rotate around ={60:(1.12,-0.65)}] [->, >= stealth, >= stealth] (1.43,-0.65)--(1.73,-0.65);
     \draw [rotate around ={-120:(1.12,-0.65)}] [->, >= stealth, >= stealth] (1.43,-0.65)--(1.73,-0.65);
     \draw [line width=0.5pt] (-1.12,-0.65) circle (0.3cm);
     \shadedraw[shift={(-0.86cm,-0.5cm)},rotate=30] \doubleflechescindeeleft;
      \shadedraw[shift={(-0.86cm,-0.5cm)},rotate=30] \doubleflechescindeeright;
       \shadedraw[shift={(-0.86cm,-0.5cm)},rotate=30] \fleche;
      \draw [rotate around ={-60:(-1.12,-0.65)}] [->, >= stealth, >= stealth] (-1.43,-0.65)--(-1.73,-0.65);
     \draw [rotate around ={120:(-1.12,-0.65)}] [->, >= stealth, >= stealth] (-1.43,-0.65)--(-1.73,-0.65);
\begin{scriptsize}
\draw [fill=black] (0,-0.6) circle (0.3pt);
\draw [fill=black] (0.2,-0.55) circle (0.3pt);
\draw [fill=black] (-0.2,-0.55) circle (0.3pt);
\draw [fill=black] (1.45,-0.85) circle (0.3pt);
\draw [fill=black] (1.5,-0.67) circle (0.3pt);
\draw [fill=black] (1.33,-0.97) circle (0.3pt);
\draw [fill=black] (-1.45,-0.85) circle (0.3pt);
\draw [fill=black] (-1.5,-0.67) circle (0.3pt);
\draw [fill=black] (-1.33,-0.97) circle (0.3pt);
\end{scriptsize}
\draw (0,0.25) node[anchor=north ] {$M_{\mathcal{A}}$};
\draw (-1.12,-0.4) node[anchor=north ] {$\scriptstyle{a(t)}$};
\draw (1.12,-0.4) node[anchor=north ] {$\scriptstyle{b(t)}$};
\draw (2.2,0) node[anchor=north ] {,};
\end{tikzpicture}
\begin{tikzpicture}[line cap=round,line join=round,x=1.0cm,y=1.0cm]
\clip(-2,-1.5) rectangle (4,1.5);
  \draw [line width=0.5pt] (0.,0.) circle (0.5cm);
     \draw [rotate=0] [->, >= stealth, >= stealth] (0.5,0)--(0.9,0);
     \draw [rotate=180] [->, >= stealth, >= stealth] (0.5,0)--(0.9,0);
     \draw [rotate=-90] [line width=1.1pt,<-, >= stealth, >= stealth] (0.5,0)--(0.9,0);
     \draw [line width=0.5pt] (0.77,-0.77) circle (0.3cm);
     \draw[rotate around={135:(0.77,-0.77)}] [->, >= stealth, >= stealth] (1.07,-0.77)--(1.37,-0.77);
      \draw[rotate around={15:(0.77,-0.77)}] [->, >= stealth, >= stealth] (1.07,-0.77)--(1.37,-0.77);
      \draw[rotate around={-105:(0.77,-0.77)}] [->, >= stealth, >= stealth] (1.07,-0.77)--(1.37,-0.77);
     \shadedraw[rotate around={75:(0.77,-0.77)}, shift={(1.07,-0.77)}] \doublefleche;
     \shadedraw[rotate around={-45:(0.77,-0.77)}, shift={(1.07,-0.77)}] \doublefleche;
     \draw [line width=0.5pt] (-0.77,-0.77) circle (0.3cm);
     \draw[rotate around={45:(-0.77,-0.77)}] [->, >= stealth, >= stealth] (-0.47,-0.77)--(-0.17,-0.77);
      \draw[rotate around={165:(-0.77,-0.77)}] [->, >= stealth, >= stealth] (-0.47,-0.77)--(-0.17,-0.77);
      \draw[rotate around={-75:(-0.77,-0.77)}] [->, >= stealth, >= stealth] (-0.47,-0.77)--(-0.17,-0.77);
     \shadedraw[rotate around={105:(-0.77,-0.77)}, shift={(-0.47,-0.77)}] \doublefleche;
     \shadedraw[rotate around={225:(-0.77,-0.77)}, shift={(-0.47,-0.77)}] \doublefleche;
\begin{scriptsize}
\draw [fill=black] (-0.35,-0.5) circle (0.3pt);
\draw [fill=black] (-0.25,-0.55) circle (0.3pt);
\draw [fill=black] (-0.15,-0.58) circle (0.3pt);
\draw [fill=black] (0.35,-0.5) circle (0.3pt);
\draw [fill=black] (0.25,-0.55) circle (0.3pt);
\draw [fill=black] (0.15,-0.58) circle (0.3pt);
\draw [fill=black] (0,0.6) circle (0.3pt);
\draw [fill=black] (0.2,0.55) circle (0.3pt);
\draw [fill=black] (-0.2,0.55) circle (0.3pt);
\draw [fill=black] (-0.4,-0.7) circle (0.3pt);
\draw [fill=black] (-0.37,-0.83) circle (0.3pt);
\draw [fill=black] (-0.43,-0.95) circle (0.3pt);
\draw [fill=black] (0.4,-0.7) circle (0.3pt);
\draw [fill=black] (0.37,-0.83) circle (0.3pt);
\draw [fill=black] (0.43,-0.95) circle (0.3pt);
\end{scriptsize}
\draw (0,0.25) node[anchor=north ] {$M_{\mathcal{B}}$};
\draw (0.77,-0.52) node[anchor=north ] {$\scriptstyle{a(t)}$};
\draw (-0.77,-0.52) node[anchor=north ] {$\scriptstyle{b(t)}$};
\draw (3,0.2) node[anchor=north ] {and};
\end{tikzpicture}
\begin{tikzpicture}[line cap=round,line join=round,x=1.0cm,y=1.0cm]
\clip(-1,-1.5) rectangle (5.104458484699738,1.5);
    \draw (0,0) circle (0.5cm);
    \draw [rotate=0][->,>=stealth] (0.5,0)--(0.9,0);
    \draw [rotate=120][->,>=stealth] (0.5,0)--(0.9,0);
    \draw [rotate=-120][->,>=stealth] (0.5,0)--(0.9,0);
    \draw[rotate=-60,shift={(0.5,0)}]\doublefleche;
    \draw[rotate=60,shift={(0.5,0)}]\doubleflechescindeeleft;
    \draw[rotate=60,shift={(0.5,0)}]\doubleflechescindeeright;
    \draw[line width=1.1pt,rotate=60][<-,>=stealth] (0.5,0)--(1,0);
    \draw  (0,0.25)node[anchor=north]{$b(t)$};
    \draw[fill=black] (-0.6,0) circle (0.3pt);
    \draw[fill=black] (-0.55,0.2) circle (0.3pt);
    \draw[fill=black] (-0.55,-0.2) circle (0.3pt);
\end{tikzpicture}
\end{minipage}

\noindent respectively, where diagrams of the form $\mathcalboondox{D}$ and $\mathcalboondox{D'}$ contain only one disc filled with $b(t)$ for degree reasons. 

It is clear that $sA(0)=(sm_{\mathbf{F}},s\varphi_{\mathbf{F}},sm_{\MA\oplus\MA^*})$ and $sA(1)=(sm_{\mathbf{G}},s\varphi_{\mathbf{G}},sm_{\MA\oplus\MA^*})$ by definition.
Moreover, for $t\in\Bbbk$, $A(t)$ is a Maurer-Cartan element of $(\mathcal{L}_d^{\Phi_0}(\MA,\MB),\ell)$. Indeed, by construction, $sm_{a(t)}$ and $sm_{\MA\oplus\MA^*}$ are $A_{\infty}$-structures and $s\varphi_{a(t)} : (\MA\oplus\MB^*[d-1],sm_{a(t)})\rightarrow (\MA\oplus\MA^*[d-1],sm_{\MA\oplus\MA^*})$ is an $A_{\infty}$-morphism which ensures that $(sm_{a(t)},\varphi_{a(t)},sm_{\MA\oplus\MA^*})$ is a Maurer-Cartan element for every $t\in\Bbbk$.
It remains to show that 
 \[
    \frac{\partial A}{\partial t}(t)^{\doubar{x}}=\sum\limits_{n=2}^{\infty}\frac{1}{(n-1)!}\ell^n\big(A(t),\dots,A(t),B(t)\big)^{\doubar{x}}
    \]
for every $\doubar{x}\in\doubar{\MO}_{\MA}$. We have that  \[
    \frac{\partial A}{\partial t}(t)=(\Tilde{m}(t),\Tilde{a}(t),0)
    \]
and \begin{equation}
    \begin{split}
        \sum\limits_{n=2}^{\infty}\frac{1}{(n-1)!}\ell^n\big(A(t),\dots,A(t),B(t)\big)^{\doubar{x}}&=\big([sm_{a(t)},sm(t)]_G,\sum\limits_{n=1}^{\infty}sm_{\MA\oplus\MA^*}\circ(\underbrace{\varphi_{a(t)},\dots,\varphi_{a(t)},\varphi(t)}_{\text{n times}})^{\doubar{x}}
    \\&\hskip3.5cm+(\varphi(t)\upperset{G}{\circ} sm_{a(t)})^{\doubar{x}}
    +(\varphi_{a(t)}\upperset{G}{\circ}sm(t))^{\doubar{x}},0\big)
    \end{split}
\end{equation}
for every $t\in\Bbbk$, $\doubar{x}\in\doubar{\MO}_{\MA}$.

We first prove that $\Tilde{m}(t)=[sm_{a(t)},sm(t)]_G$ for every $t\in\Bbbk$. We have that $\pi_{\MA[1]}\circ  \Tilde{m}(t)^{\doubar{x}}=\sum\mathcal{E}(\mathcalboondox{D^1_1})+\sum\mathcal{E}(\mathcalboondox{D^1_2})+\sum\mathcal{E}(\mathcalboondox{D^1_3})+\sum\mathcal{E}(\mathcalboondox{D^1_4})+\sum\mathcal{E}(\mathcalboondox{D^1_5})$ where the sums are over all the filled diagrams $\mathcalboondox{D^1_i}$ for $i\in\llbracket 1,5\rrbracket$ of type $\doubar{x}$ and of the form

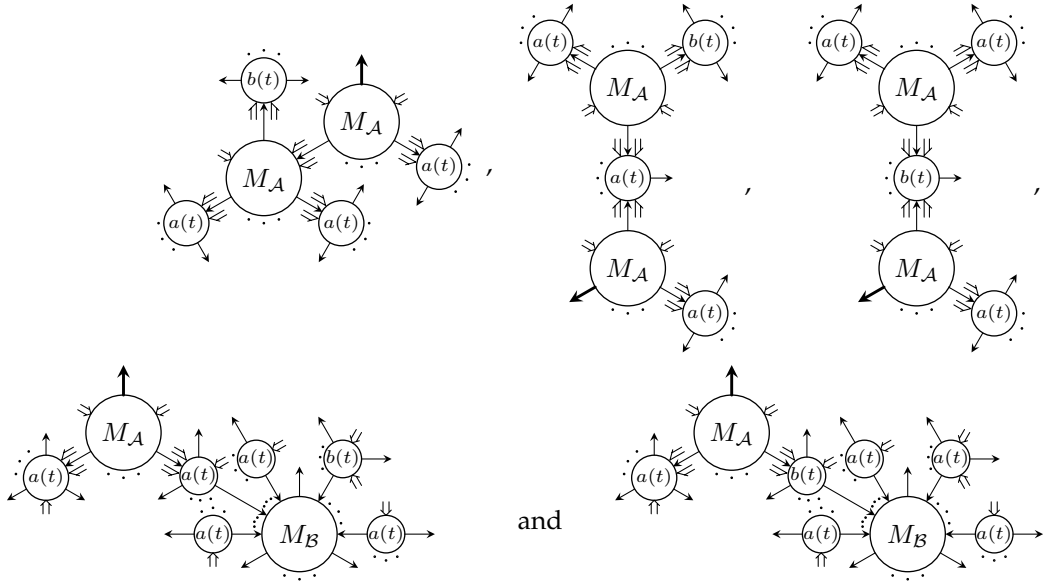
\begin{minipage}{21cm}

\end{minipage}

\noindent respectively. Moreover, we have that \[\pi_{\MB^*[d]}\circ \Tilde{m}(t)^{\doubar{x}}=-\sum\mathcal{E}(\mathcalboondox{D^2_1})-\sum\mathcal{E}(\mathcalboondox{D^2_2})-\sum\mathcal{E}(\mathcalboondox{D^2_3})-\sum\mathcal{E}(\mathcalboondox{D^2_4})-\sum\mathcal{E}(\mathcalboondox{D^2_5})\] where the sums are over all the filled diagrams $\mathcalboondox{D^2_i}$ for $i\in\llbracket 1,5\rrbracket$ of type $\doubar{x}$ and of the form

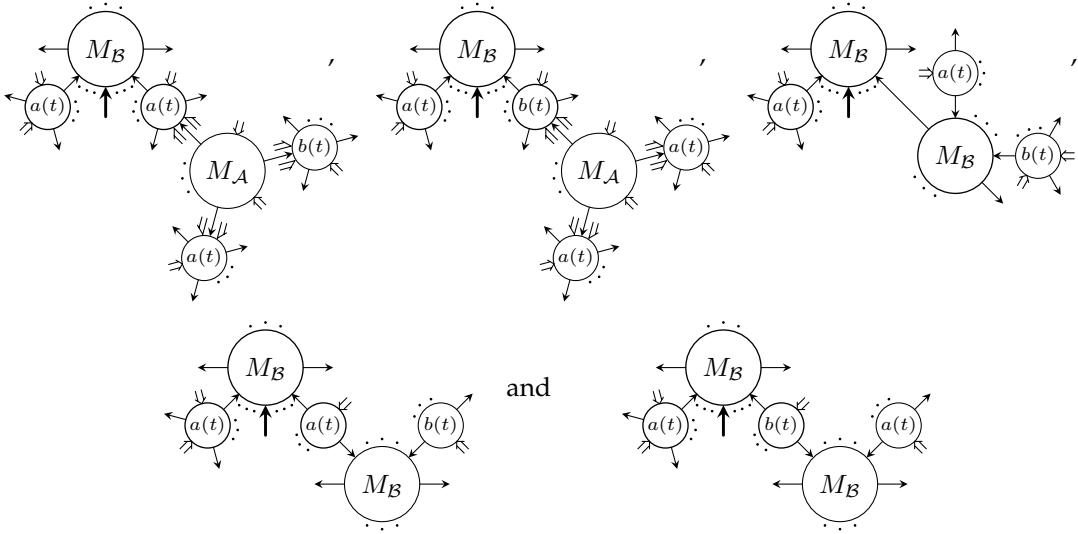
\begin{minipage}{21cm}
    %
\end{minipage}

\noindent respectively.

On the other hand, we have that $[sm_{a(t)},sm]_G=sm_{a(t)}\upperset{G}{\circ}sm-sm\upperset{G}{\circ}sm_{a(t)}$ and by definition $\pi_{\MA[1]}\circ (sm_{a(t)}\upperset{G}{\circ}sm)=\sum\mathcal{E}(\mathcalboondox{D^4_1})+\sum\mathcal{E}(\mathcalboondox{D^4_2})+\sum\mathcal{E}(\mathcalboondox{D^4_3})$ where the sums are over all the filled diagrams $\mathcalboondox{D^4_1}$, $\mathcalboondox{D^4_2}$ and $\mathcalboondox{D^4_2}$ of type $\doubar{x}$ and of the form

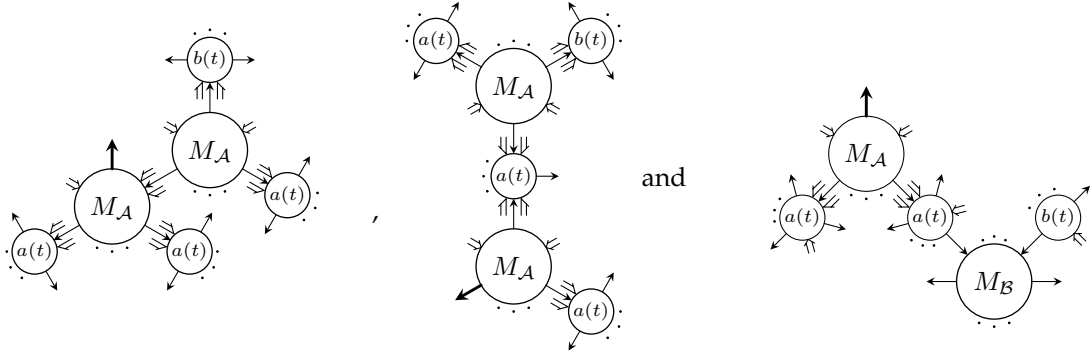
\begin{minipage}{21cm}
%
\end{minipage}

\noindent respectively. Moreover, we have that \[\pi_{\MB^*[d]}\circ (sm_{a(t)}\upperset{G}{\circ}sm)^{\doubar{x}}=-\sum\mathcal{E}(\mathcalboondox{D^4_4})-\sum\mathcal{E}(\mathcalboondox{D^4_5})-\sum\mathcal{E}(\mathcalboondox{D^4_6})\] where the sums are over all the filled diagrams $\mathcalboondox{D^4_i}$ for $i\in\llbracket 3,6\rrbracket$ of type $\doubar{x}$ and of the form

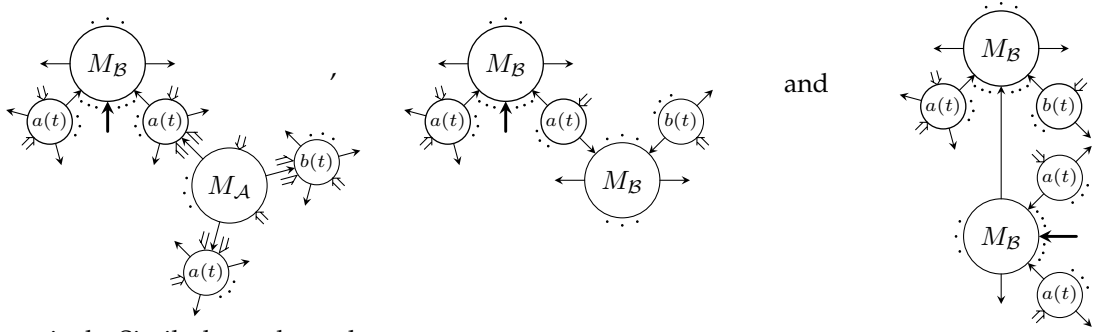
\begin{minipage}{21cm}
    %
\end{minipage}

\noindent respectively. Similarly, we have that \[\pi_{\MA[1]}\circ(sm\upperset{G}{\circ}sm_{a(t)})^{\doubar{x}}=\sum\mathcal{E}(\mathcalboondox{D^5_1})-\sum\mathcal{E}(\mathcalboondox{D^5_2})-\sum\mathcal{E}(\mathcalboondox{D^5_3})-\sum\mathcal{E}(\mathcalboondox{D^5_4})-\sum\mathcal{E}(\mathcalboondox{D^5_5})\] where the sums are over all the filled diagrams $\mathcalboondox{D^5_i}$ for $i\in\llbracket 1,5\rrbracket$ of type $\doubar{x}$ and of the form

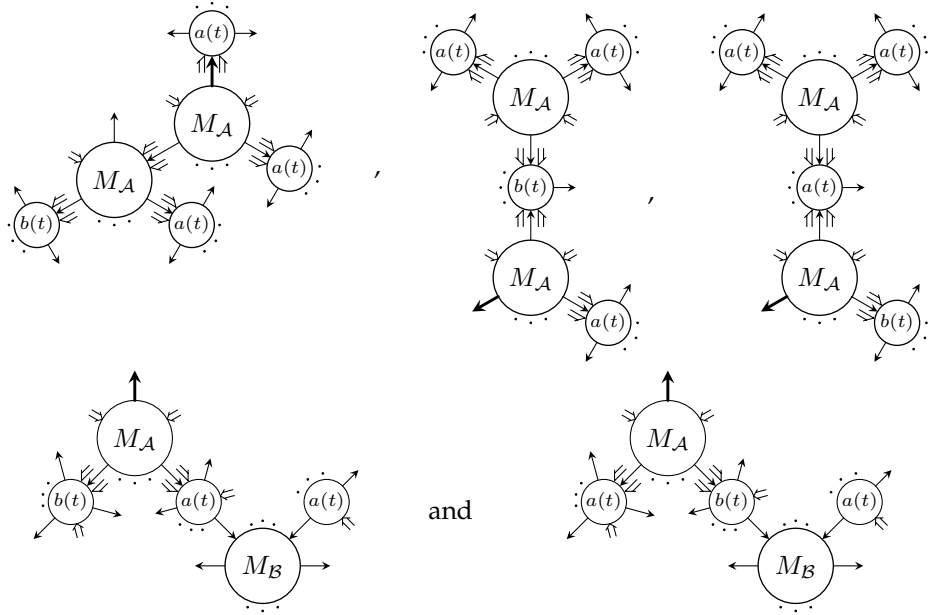
\begin{minipage}{21cm}
    %
\end{minipage}

\noindent respectively and \[\pi_{\MB^*[d]}\circ(sm\upperset{G}{\circ}sm_{a(t)})^{\doubar{x}}=\sum\mathcal{E}(\mathcalboondox{D^5_6})+\sum\mathcal{E}(\mathcalboondox{D^5_7})-\sum\mathcal{E}(\mathcalboondox{D^5_8})+\sum\mathcal{E}(\mathcalboondox{D^5_9})-\sum\mathcal{E}(\mathcalboondox{D^5_{10}})\] where the sums are over all the filled diagrams $\mathcalboondox{D^5_i}$ for $i\in\llbracket 5,10\rrbracket$ of type $\doubar{x}$ and of the form

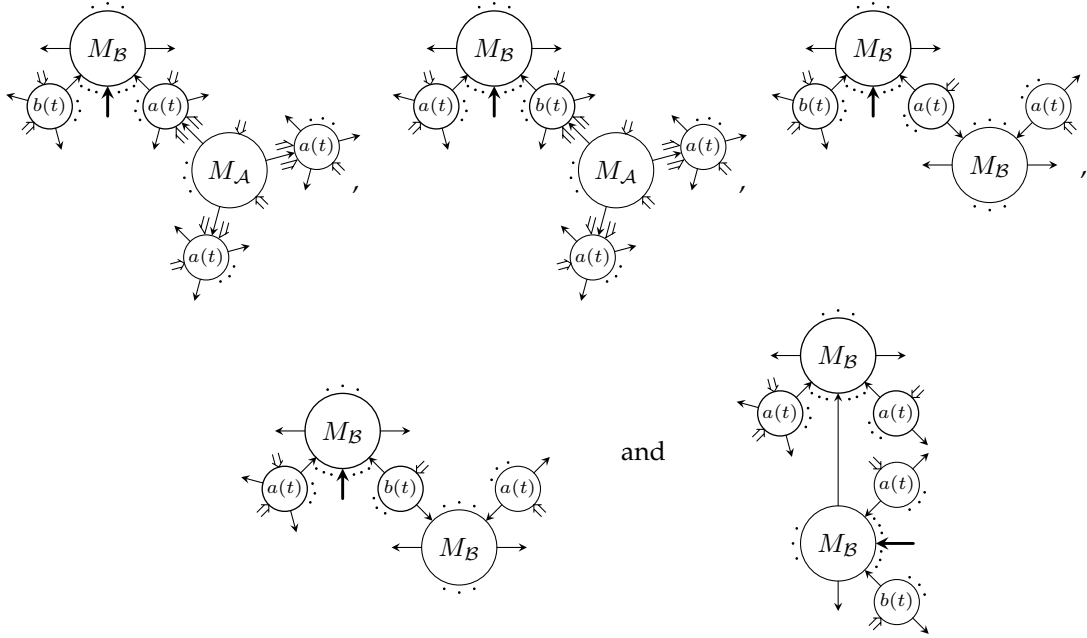
\begin{minipage}{21cm}
        %
\end{minipage}

\noindent respectively. Now, we have that 

\begin{small}
\begin{equation}
    \begin{split}
        &\sum\mathcal{E}(\mathcalboondox{D_2^1})=\sum\mathcal{E}(\mathcalboondox{D_2^4})\text{, }
        \sum\mathcal{E}(\mathcalboondox{D_3^1})=\sum\mathcal{E}(\mathcalboondox{D_2^5})\text{, }
        \sum\mathcal{E}(\mathcalboondox{D_4^1})=\sum\mathcal{E}(\mathcalboondox{D_3^4})\text{, }
        \sum\mathcal{E}(\mathcalboondox{D_5^1})=\sum\mathcal{E}(\mathcalboondox{D_5^5})\text{, }  \\&\sum\mathcal{E}(\mathcalboondox{D_1^2})=\sum\mathcal{E}(\mathcalboondox{D_4^4})\text{, }\sum\mathcal{E}(\mathcalboondox{D_2^2})=\sum\mathcal{E}(\mathcalboondox{D_7^5}) \text{, }       
        \sum\mathcal{E}(\mathcalboondox{D_4^2})=\sum\mathcal{E}(\mathcalboondox{D_5^4})\text{, }\sum\mathcal{E}(\mathcalboondox{D_5^2})=\sum\mathcal{E}(\mathcalboondox{D_9^5})
    \end{split}
\end{equation}
\end{small}
and since $a(t)$ is a Maurer-Cartan element of $(\Bar{\mathcal{L}}_d^{\Phi_0}(\MA,\MB),\Bar{\ell}_{M_{\MA},M_{\MB}})$ for every $t\in \Bbbk$, we have that $\sum\mathcal{E}(\mathcalboondox{D_4^5})=-\sum\mathcal{E}(\mathcalboondox{D_4^5}')+\sum\mathcal{E}(\mathcalboondox{D_4^5}'')$ where the last two sums are over all the filled diagrams $\mathcalboondox{D_4^5}'$ and $\mathcalboondox{D_4^5}''$ of type $\doubar{x}$ and of the form 

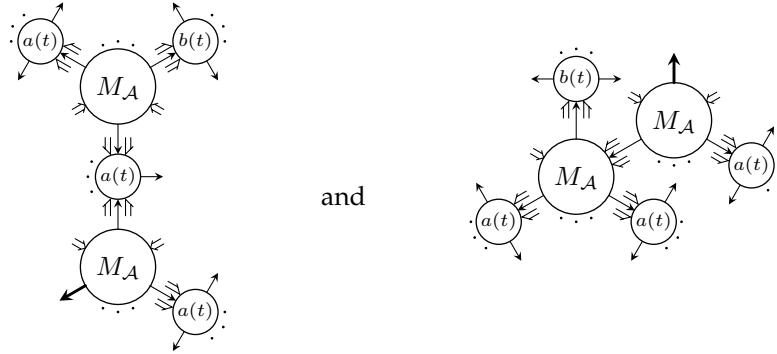
\begin{minipage}{21cm}
\begin{tikzpicture}[line cap=round,line join=round,x=1.0cm,y=1.0cm]
\clip(-4,-2.4) rectangle (4,3);
  \draw [line width=0.5pt] (0.,0.) circle (0.3cm);
     \shadedraw[rotate=90,shift={(0.3cm,0cm)}] \doubleflechescindeeleft;
     \shadedraw[rotate=90,shift={(0.3cm,0cm)}] \doubleflechescindeeright;
     \draw[rotate=90][<-,>=stealth](0.3,0)--(0.7,0);
     \shadedraw[rotate=-90,shift={(0.3cm,0cm)}] \doubleflechescindeeleft;
     \shadedraw[rotate=-90,shift={(0.3cm,0cm)}] \doubleflechescindeeright;
     \draw[rotate=-90][<-,>=stealth](0.3,0)--(0.7,0);
     \draw[->, >= stealth,>=stealth](0.3,0)--(0.6,0);
     \draw [line width=0.5pt] (0,1.2) circle (0.5cm);
    \draw [rotate around ={-30:(0,1.2)}, shift={(0.5,1.2)}]\doublefleche;
     \draw [rotate around ={-150:(0,1.2)}, shift={(0.5,1.2)}]\doublefleche;
      \draw [line width=0.5pt] (-1.03,1.8) circle (0.3cm);
     \draw[rotate around={-30:(-1.03,1.8)}][<-, >= stealth,>=stealth](-0.73,1.8)--(-0.33,1.8);
    \shadedraw[rotate around={-30:(-1.03,1.8)},shift={(-0.73cm,1.8cm)}] \doubleflechescindeeleft;
     \shadedraw[rotate around={-30:(-1.03,1.8)},shift={(-0.73cm,1.8cm)}] \doubleflechescindeeright;
     \draw[rotate around={60:(-1.03,1.8)}][->, >= stealth,>=stealth](-0.73,1.8)--(-0.43,1.8);
     \draw[rotate around={-120:(-1.03,1.8)}][->, >= stealth,>=stealth](-0.73,1.8)--(-0.43,1.8);
     \draw [line width=0.5pt] (1.03,1.8) circle (0.3cm);
     \draw[rotate around={-150:(1.03,1.8)}][<-, >= stealth,>=stealth](1.33,1.8)--(1.73,1.8);
    \shadedraw[rotate around={-150:(1.03,1.8)},shift={(1.33cm,1.8cm)}] \doubleflechescindeeleft;
     \shadedraw[rotate around={-150:(1.03,1.8)},shift={(1.33cm,1.8cm)}] \doubleflechescindeeright;
     \draw[rotate around={-60:(1.03,1.8)}][->, >= stealth,>=stealth](1.33,1.8)--(1.63,1.8);
     \draw[rotate around={120:(1.03,1.8)}][->, >= stealth,>=stealth](1.33,1.8)--(1.63,1.8);
     \draw [line width=0.5pt] (0,-1.2) circle (0.5cm);
      \draw [rotate around ={30:(0,-1.2)}, shift={(0.5,-1.2)}]\doublefleche;
     \draw [rotate around ={150:(0,-1.2)}, shift={(0.5,-1.2)}]\doublefleche;
     \draw[line width=1.1pt,rotate around={-150:(0,-1.2)}][->, >= stealth,>=stealth](0.5,-1.2)--(0.9,-1.2);
      \draw [line width=0.5pt] (1.03,-1.8) circle (0.3cm);
     \draw[rotate around={150:(1.03,-1.8)}][<-, >= stealth,>=stealth](1.33,-1.8)--(1.73,-1.8);
    \shadedraw[rotate around={150:(1.03,-1.8)},shift={(1.33cm,-1.8cm)}] \doubleflechescindeeleft;
     \shadedraw[rotate around={150:(1.03,-1.8)},shift={(1.33cm,-1.8cm)}] \doubleflechescindeeright;
     \draw[rotate around={60:(1.03,-1.8)}][->, >= stealth,>=stealth](1.33,-1.8)--(1.63,-1.8);
     \draw[rotate around={-120:(1.03,-1.8)}][->, >= stealth,>=stealth](1.33,-1.8)--(1.63,-1.8);
    \begin{scriptsize}
\draw [fill=black] (0,1.8) circle (0.3pt);
\draw [fill=black] (0.2,1.75) circle (0.3pt);
\draw [fill=black] (-0.2,1.75) circle (0.3pt);
\draw [fill=black] (0,-1.8) circle (0.3pt);
\draw [fill=black] (0.2,-1.75) circle (0.3pt);
\draw [fill=black] (-0.2,-1.75) circle (0.3pt);
\draw [fill=black] (-0.4,0) circle (0.3pt);
\draw [fill=black] (-0.35,-0.2) circle (0.3pt);
\draw [fill=black] (-0.35,0.2) circle (0.3pt);
\draw [fill=black] (1.4,-2) circle (0.3pt);
\draw [fill=black] (1.42,-1.8) circle (0.3pt);
\draw [fill=black] (1.25,-2.15) circle (0.3pt);
\draw [fill=black] (1.4,2) circle (0.3pt);
\draw [fill=black] (1.42,1.8) circle (0.3pt);
\draw [fill=black] (1.25,2.15) circle (0.3pt);
\draw [fill=black] (-1.4,2.) circle (0.3pt);
\draw [fill=black] (-1.42,1.8) circle (0.3pt);
\draw [fill=black] (-1.25,2.15) circle (0.3pt);
\end{scriptsize}
\draw(0,1.45)node[anchor=north]{$M_{\MA}$};
\draw(0,-0.95)node[anchor=north]{$M_{\MA}$};
\draw(0,0.25)node[anchor=north]{$\scriptstyle{a(t)}$};
\draw(1.03,-1.55)node[anchor=north]{$\scriptstyle{a(t)}$};
\draw(1.03,2.05)node[anchor=north]{$\scriptstyle{b(t)}$};
\draw(-1.03,2.05)node[anchor=north]{$\scriptstyle{a(t)}$};
\draw(3,0.)node[anchor=north]{and};
\end{tikzpicture}
    \begin{tikzpicture}[line cap=round,line join=round,x=1.0cm,y=1.0cm]
\clip(-2,-2.4) rectangle (3.3,3);
   \draw [line width=0.5pt] (0.,0.) circle (0.5cm);
     \shadedraw[rotate=150,shift={(0.5cm,0cm)}] \doublefleche;
      \draw [line width=0.5pt] (0.,1.3) circle (0.3cm);
     \shadedraw[rotate around={-90:(0,1.3)},shift={(0.3cm,1.3cm)}] \doubleflechescindeeleft;
     \shadedraw[rotate around={-90:(0,1.3)},shift={(0.3cm,1.3cm)}] \doubleflechescindeeright;
     \shadedraw[rotate around={-90:(0,1.3)},shift={(0.3cm,1.3cm)}] \fleche;
      \draw[rotate around={180:(0,1.3)}][->, >= stealth, >= stealth](0.3,1.3)--(0.6,1.3);
      \draw[rotate around={0:(0,1.3)}][->, >= stealth, >= stealth](0.3,1.3)--(0.6,1.3);
     \shadedraw[rotate=30,shift={(0.5cm,0cm)}] \doubleflechescindeeleft;
     \shadedraw[rotate=30,shift={(0.5cm,0cm)}] \doubleflechescindeeright;
     \shadedraw[rotate=30,shift={(0.5cm,0cm)}] \fleche;
      \draw [line width=0.5pt] (1.3,0.75) circle (0.5cm);
     \draw [line width=0.5pt] (1.03,-0.6) circle (0.3cm);
    \draw[rotate=-30][->, >= stealth, >= stealth](0.5,0)--(0.9,0);
     \draw [rotate around ={60:(1.03,-0.6)}] [->, >= stealth, >= stealth] (1.33,-0.6)--(1.63,-0.6);
     \draw [rotate around ={-120:(1.03,-0.6)}] [->, >= stealth, >= stealth] (1.33,-0.6)--(1.63,-0.6);
    \draw[rotate=-150][->, >= stealth, >= stealth](0.5,0)--(0.9,0);
    \draw [line width=0.5pt] (-1.03,-0.6) circle (0.3cm);
     \draw [rotate around ={-60:(-1.03,-0.6)}] [->, >= stealth, >= stealth] (-1.33,-0.6)--(-1.63,-0.6);
     \draw [rotate around ={120:(-1.03,-0.6)}] [->, >= stealth, >= stealth] (-1.33,-0.6)--(-1.63,-0.6);
     \draw [rotate around ={-30:(1.3,0.75)}] [->, >= stealth, >= stealth] (1.8,0.75)--(2.2,0.75);
     \draw [line width=1.1pt,rotate around ={90:(1.3,0.75)}] [->, >= stealth, >= stealth] (1.8,0.75)--(2.2,0.75);
     \draw [rotate around ={30:(1.3,0.75)}, shift={(1.8,0.75)}]\doublefleche;
     \draw [rotate around ={150:(1.3,0.75)}, shift={(1.8,0.75)}]\doublefleche;
     \draw [line width=0.5pt] (2.33,0.15) circle (0.3cm);
      \draw [rotate around ={60:(2.33,0.15)}] [->, >= stealth, >= stealth] (2.63,0.15)--(2.93,0.15);
       \draw [rotate around ={-120:(2.33,0.15)}] [->, >= stealth, >= stealth] (2.63,0.15)--(2.93,0.15);
     \shadedraw[rotate around={150:(2.33,0.15)},shift={(2.63cm,0.15cm)}] \doubleflechescindeeleft;
     \shadedraw[rotate around={150:(2.33,0.15)},shift={(2.63cm,0.15cm)}] \doubleflechescindeeright;
    \shadedraw[rotate around={30:(-1.03,-0.6)},shift={(-0.73,-0.6)}] \doubleflechescindeeleft;
     \shadedraw[rotate around={30:(-1.03,-0.6)},shift={(-0.73,-0.6)}] \doubleflechescindeeright;
     \shadedraw[rotate around={150:(1.03,-0.6)},shift={(1.33,-0.6)}] \doubleflechescindeeleft;
     \shadedraw[rotate around={150:(1.03,-0.6)},shift={(1.33,-0.6)}] \doubleflechescindeeright;
\begin{scriptsize}
\draw [fill=black] (0,-0.6) circle (0.3pt);
\draw [fill=black] (0.2,-0.55) circle (0.3pt);
\draw [fill=black] (-0.2,-0.55) circle (0.3pt);
\draw [fill=black] (1.3,0.75-0.6) circle (0.3pt);
\draw [fill=black] (1.5,0.75-0.55) circle (0.3pt);
\draw [fill=black] (1.1,0.75-0.55) circle (0.3pt);
\draw [fill=black] (0,1.7) circle (0.3pt);
\draw [fill=black] (0.2,1.65) circle (0.3pt);
\draw [fill=black] (-0.2,1.65) circle (0.3pt);
\draw [fill=black] (2.7,-0.02) circle (0.3pt);
\draw [fill=black] (2.73,0.18) circle (0.3pt);
\draw [fill=black] (2.55,-0.18) circle (0.3pt);
\draw [fill=black] (1.4,-0.6) circle (0.3pt);
\draw [fill=black] (1.35,-0.81) circle (0.3pt);
\draw [fill=black] (1.2,-0.95) circle (0.3pt);
\draw [fill=black] (-1.4,-0.6) circle (0.3pt);
\draw [fill=black] (-1.35,-0.81) circle (0.3pt);
\draw [fill=black] (-1.2,-0.95) circle (0.3pt);
\end{scriptsize}
\draw(0,0.25)node[anchor=north]{$M_{\MA}$};
\draw(1.3,1)node[anchor=north]{$M_{\MA}$};
\draw(2.33,0.4)node[anchor=north]{$\scriptstyle{a(t)}$};
\draw(0,1.55)node[anchor=north]{$\scriptstyle{b(t)}$};
\draw(1.03,-0.35)node[anchor=north]{$\scriptstyle{a(t)}$};
\draw(-1.03,-0.35)node[anchor=north]{$\scriptstyle{a(t)}$};
\end{tikzpicture}
\end{minipage}

\noindent respectively and $\sum\mathcal{E}(\mathcalboondox{D_6^5})=\sum\mathcal{E}(\mathcalboondox{D_6^5}')+\sum\mathcal{E}(\mathcalboondox{D_6^5}'')$ where the last two sums are over all the filled diagrams $\mathcalboondox{D_6^5}'$ and $\mathcalboondox{D_6^5}''$ of type $\doubar{x}$ and of the form 

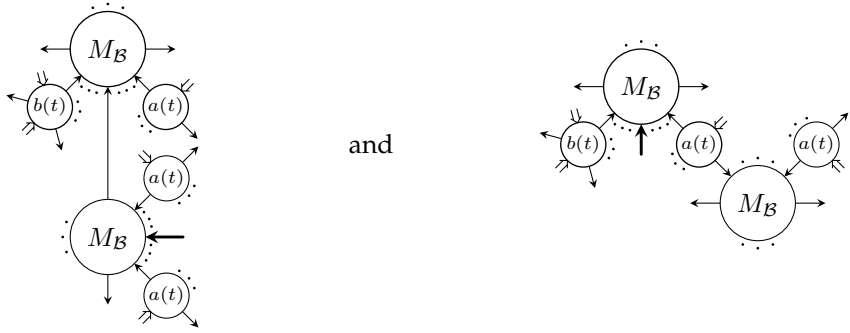
\begin{minipage}{21cm}
\begin{tikzpicture}[line cap=round,line join=round,x=1.0cm,y=1.0cm]
\clip(-3,-4) rectangle (4,1);
  \draw [line width=0.5pt] (0.,0.) circle (0.5cm);
     \draw [rotate=0] [->, >= stealth, >= stealth] (0.5,0)--(0.9,0);
     \draw [rotate=180] [->, >= stealth, >= stealth] (0.5,0)--(0.9,0);
     \draw [rotate=-90] [<-, >= stealth, >= stealth] (0.5,0)--(2,0);
     \draw (0,-2.5) circle (0.5cm);
     \draw[rotate around={-90:(0,-2.5)}][->,>=stealth](0.5,-2.5)--(0.9,-2.5);
      \draw[rotate around={45:(0,-2.5)}][<-,>=stealth](0.5,-2.5)--(0.8,-2.5);
      \draw (0.78,-1.72) circle (0.3cm);
      \draw[rotate around={45:(0.78,-1.72)}][->,>=stealth](1.08,-1.72)--(1.38,-1.72);
      \draw[rotate around={135:(0.78,-1.72)},shift={(1.08,-1.72)}]\doublefleche;
      \draw[rotate around={-45:(0,-2.5)}][<-,>=stealth](0.5,-2.5)--(0.8,-2.5);
      \draw[line width=1.1pt, rotate around={0:(0,-2)}][<-,>=stealth](0.5,-2.5)--(1,-2.5);
      \draw (0.78,-3.28) circle (0.3cm);
      \draw[rotate around={-45:(0.78,-3.28)}][->,>=stealth](1.08,-3.28)--(1.38,-3.28);
      \draw[rotate around={-135:(0.78,-3.28)},shift={(1.08,-3.28)}]\doublefleche;
     \draw [line width=0.5pt] (0.77,-0.77) circle (0.3cm);
     \draw[rotate around={135:(0.77,-0.77)}] [->, >= stealth, >= stealth] (1.07,-0.77)--(1.37,-0.77);
      \draw[rotate around={-45:(0.77,-0.77)}] [->, >= stealth, >= stealth] (1.07,-0.77)--(1.37,-0.77);
     \shadedraw[rotate around={45:(0.77,-0.77)}, shift={(1.07,-0.77)}] \doublefleche;
     \draw [line width=0.5pt] (-0.77,-0.77) circle (0.3cm);
     \draw[rotate around={45:(-0.77,-0.77)}] [->, >= stealth, >= stealth] (-0.47,-0.77)--(-0.17,-0.77);
      \draw[rotate around={165:(-0.77,-0.77)}] [->, >= stealth, >= stealth] (-0.47,-0.77)--(-0.17,-0.77);
      \draw[rotate around={-75:(-0.77,-0.77)}] [->, >= stealth, >= stealth] (-0.47,-0.77)--(-0.17,-0.77);
     \shadedraw[rotate around={105:(-0.77,-0.77)}, shift={(-0.47,-0.77)}] \doublefleche;
     \shadedraw[rotate around={225:(-0.77,-0.77)}, shift={(-0.47,-0.77)}] \doublefleche;
\begin{scriptsize}
\draw [fill=black] (-0.35,-0.5) circle (0.3pt);
\draw [fill=black] (-0.25,-0.55) circle (0.3pt);
\draw [fill=black] (-0.15,-0.58) circle (0.3pt);
\draw [fill=black] (0.35,-0.5) circle (0.3pt);
\draw [fill=black] (0.25,-0.55) circle (0.3pt);
\draw [fill=black] (0.15,-0.58) circle (0.3pt);
\draw [fill=black] (0,0.6) circle (0.3pt);
\draw [fill=black] (0.2,0.55) circle (0.3pt);
\draw [fill=black] (-0.2,0.55) circle (0.3pt);
\draw [fill=black] (-0.4,-0.7) circle (0.3pt);
\draw [fill=black] (-0.37,-0.83) circle (0.3pt);
\draw [fill=black] (-0.43,-0.95) circle (0.3pt);
\draw [rotate around={30:(0.77,-0.77)}][fill=black] (0.4,-0.7) circle (0.3pt);
\draw [rotate around={30:(0.77,-0.77)}][fill=black] (0.37,-0.83) circle (0.3pt);
\draw [rotate around={30:(0.77,-0.77)}][fill=black] (0.43,-0.95) circle (0.3pt);
\draw [fill=black] (-0.55,-2.3) circle (0.3pt);
\draw [fill=black] (-0.55,-2.7) circle (0.3pt);
\draw [fill=black] (-0.6,-2.5) circle (0.3pt);
\draw [fill=black] (0.55,-2.26) circle (0.3pt);
\draw [fill=black] (0.48,-2.17) circle (0.3pt);
\draw [fill=black] (0.58,-2.36) circle (0.3pt);
\draw [rotate around={-45:(0,-2.5)}][fill=black] (0.55,-2.26) circle (0.3pt);
\draw [rotate around={-45:(0,-2.5)}][fill=black] (0.48,-2.17) circle (0.3pt);
\draw [rotate around={-45:(0,-2.5)}][fill=black] (0.58,-2.36) circle (0.3pt);
\draw [fill=black] (1.15,-1.82) circle (0.3pt);
\draw [fill=black] (1.08,-1.95) circle (0.3pt);
\draw [fill=black] (0.95,-2.05) circle (0.3pt);
\draw [fill=black] (1.15,-3.18) circle (0.3pt);
\draw [fill=black] (1.08,-3.05) circle (0.3pt);
\draw [fill=black] (0.95,-2.95) circle (0.3pt);
\end{scriptsize}
\draw (0,0.25) node[anchor=north ] {$M_{\mathcal{B}}$};
\draw (0,-2.25) node[anchor=north ] {$M_{\mathcal{B}}$};
\draw (0.77,-0.52) node[anchor=north ] {$\scriptstyle{a(t)}$};
\draw (-0.77,-0.52) node[anchor=north ] {$\scriptstyle{b(t)}$};
\draw (0.78,-3.03) node[anchor=north ] {$\scriptstyle{a(t)}$};
\draw (0.78,-1.47) node[anchor=north ] {$\scriptstyle{a(t)}$};
\draw (3.5,-1) node[anchor=north ] {and};
\end{tikzpicture}
            \begin{tikzpicture}[line cap=round,line join=round,x=1.0cm,y=1.0cm]
\clip(-3,-3.5) rectangle (4,1);
  \draw [line width=0.5pt] (0.,0.) circle (0.5cm);
     \draw [rotate=0] [->, >= stealth, >= stealth] (0.5,0)--(0.9,0);
     \draw [rotate=180] [->, >= stealth, >= stealth] (0.5,0)--(0.9,0);
     \draw [rotate=-90] [line width=1.1pt,<-, >= stealth, >= stealth] (0.5,0)--(0.9,0);
     \draw [line width=0.5pt] (0.77,-0.77) circle (0.3cm);
     \draw[rotate around={135:(0.77,-0.77)}] [->, >= stealth, >= stealth] (1.07,-0.77)--(1.37,-0.77);
      \draw[rotate around={-45:(0.77,-0.77)}] [->, >= stealth, >= stealth] (1.07,-0.77)--(1.37,-0.77);
     \shadedraw[rotate around={45:(0.77,-0.77)}, shift={(1.07,-0.77)}] \doublefleche;
     \draw (1.55,-1.55) circle (0.5cm);
     \draw[rotate around={180:(1.55,-1.55)}][->,>=stealth](2.05,-1.55)--(2.45,-1.55);
     \draw[rotate around={0:(1.55,-1.55)}][->,>=stealth](2.05,-1.55)--(2.45,-1.55);
     \draw[rotate around={45:(1.55,-1.55)}][<-,>=stealth](2.05,-1.55)--(2.35,-1.55);
     \draw (2.33,-0.77) circle (0.3cm);
     \draw[rotate around={45:(2.33,-0.77)}][->,>=stealth](2.63,-0.77)--(2.93,-0.77);
     \draw[rotate around={-45:(2.33,-0.77)},shift={(2.63,-0.77)}]\doublefleche;
     \draw [line width=0.5pt] (-0.77,-0.77) circle (0.3cm);
     \draw[rotate around={45:(-0.77,-0.77)}] [->, >= stealth, >= stealth] (-0.47,-0.77)--(-0.17,-0.77);
      \draw[rotate around={165:(-0.77,-0.77)}] [->, >= stealth, >= stealth] (-0.47,-0.77)--(-0.17,-0.77);
      \draw[rotate around={-75:(-0.77,-0.77)}] [->, >= stealth, >= stealth] (-0.47,-0.77)--(-0.17,-0.77);
     \shadedraw[rotate around={105:(-0.77,-0.77)}, shift={(-0.47,-0.77)}] \doublefleche;
     \shadedraw[rotate around={225:(-0.77,-0.77)}, shift={(-0.47,-0.77)}] \doublefleche;
\begin{scriptsize}
\draw [fill=black] (-0.35,-0.5) circle (0.3pt);
\draw [fill=black] (-0.25,-0.55) circle (0.3pt);
\draw [fill=black] (-0.15,-0.58) circle (0.3pt);
\draw [fill=black] (0.35,-0.5) circle (0.3pt);
\draw [fill=black] (0.25,-0.55) circle (0.3pt);
\draw [fill=black] (0.15,-0.58) circle (0.3pt);
\draw [fill=black] (0,0.6) circle (0.3pt);
\draw [fill=black] (0.2,0.55) circle (0.3pt);
\draw [fill=black] (-0.2,0.55) circle (0.3pt);
\draw [fill=black] (-0.4,-0.7) circle (0.3pt);
\draw [fill=black] (-0.37,-0.83) circle (0.3pt);
\draw [fill=black] (-0.43,-0.95) circle (0.3pt);
\draw [rotate around={30:(0.77,-0.77)}][fill=black] (0.4,-0.7) circle (0.3pt);
\draw [rotate around={30:(0.77,-0.77)}][fill=black] (0.37,-0.83) circle (0.3pt);
\draw [rotate around={30:(0.77,-0.77)}][fill=black] (0.43,-0.95) circle (0.3pt);
\draw [fill=black] (1.55,-2.15) circle (0.3pt);
\draw [fill=black] (1.75,-2.1) circle (0.3pt);
\draw [fill=black] (1.35,-2.1) circle (0.3pt);
\draw [rotate around={180:(1.55,-1.55)}][fill=black] (1.55,-2.15) circle (0.3pt);
\draw [rotate around={180:(1.55,-1.55)}][fill=black] (1.75,-2.1) circle (0.3pt);
\draw [rotate around={180:(1.55,-1.55)}][fill=black] (1.35,-2.1) circle (0.3pt);
\draw [fill=black] (1.98,-0.65) circle (0.3pt);
\draw [fill=black] (2.05,-0.5) circle (0.3pt);
\draw [fill=black] (2.2,-0.42) circle (0.3pt);
\end{scriptsize}
\draw (0,0.25) node[anchor=north ] {$M_{\mathcal{B}}$};
\draw (1.55,-1.3) node[anchor=north ] {$M_{\mathcal{B}}$};
\draw (0.77,-0.52) node[anchor=north ] {$\scriptstyle{a(t)}$};
\draw (-0.77,-0.52) node[anchor=north ] {$\scriptstyle{b(t)}$};
\draw (2.33,-0.52) node[anchor=north ] {$\scriptstyle{a(t)}$};
\end{tikzpicture}
\end{minipage}

\noindent respectively.
We have that $\sum\mathcal{E}(\mathcalboondox{D_4^5}')=\sum\mathcal{E}(\mathcalboondox{D_3^5})$ and that $\sum\mathcal{E}(\mathcalboondox{D_6^5}'')=\sum\mathcal{E}(\mathcalboondox{D_8^5})$.
Finally, using that $s_{d+1}M_{\MA}$ and $s_{d+1}M_{\MB}$ are $d$-pre-Calabi-Yau structures, we obtain that 
\begin{equation}
    \begin{split}
        &\sum\mathcal{E}(\mathcalboondox{D_1^1})-\sum\mathcal{E}(\mathcalboondox{D_1^4})-\sum\mathcal{E}(\mathcalboondox{D_1^5})+\sum\mathcal{E}(\mathcalboondox{D_4^5}'')=0
        \\ \text{ and }& -\sum\mathcal{E}(\mathcalboondox{D_3^2})+\sum\mathcal{E}(\mathcalboondox{D_6^4})+\sum\mathcal{E}(\mathcalboondox{D_6^5}')-\sum\mathcal{E}(\mathcalboondox{D_{10}^5})=0.
    \end{split}
\end{equation}
Therefore, $\Tilde{m}(t)=[sm_{a(t)},sm(t)]_G$ for every $t\in\Bbbk$.
We now prove that 
\[s\Tilde{a}(t)^{\doubar{x}}=\sum\limits_{n=1}^{\infty}sm_{\MA\oplus\MA^*}\big(\varphi_{a(t)},\dots,\varphi_{a(t)},\varphi(t)\big)^{\doubar{x}}+(s\varphi(t)\upperset{G}{\circ} sm_{a(t)})^{\doubar{x}}
    +(\varphi_{a(t)}\upperset{G}{\circ}sm(t))^{\doubar{x}}\] for every $t\in\Bbbk$, $\doubar{x}\in\doubar{\MO}_{\MA}$. We have that $\pi_{\MA^*[d]}\circ s\Tilde{a}(t)^{\doubar{x}}=\sum\mathcal{E}(\mathcalboondox{D^3_1})+\sum\mathcal{E}(\mathcalboondox{D^3_2})+\sum\mathcal{E}(\mathcalboondox{D^3_3})+\sum\mathcal{E}(\mathcalboondox{D^3_4})+\sum\mathcal{E}(\mathcalboondox{D^3_5})$ where the sums are over all the filled diagrams $\mathcalboondox{D^3_i}$ for $i\in\llbracket 1,5\rrbracket$ of type $\doubar{x}$ and of the form 

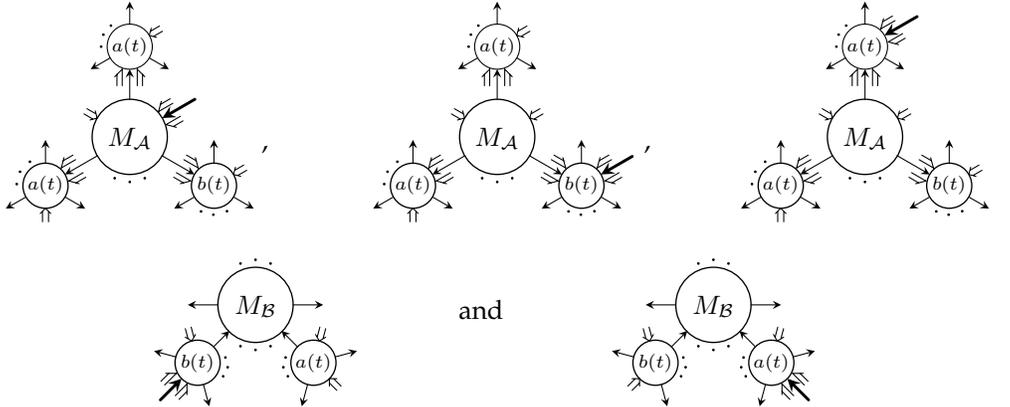
\begin{minipage}{21cm}

\end{minipage}
\noindent respectively and $\pi_{\MA[1]}\circ \Tilde{a}(t)=0$.

On the other hand, we have that $\pi_{\MA^*[d]}\circ \big(\sum\limits_{n=1}^{\infty}sm_{\MA\oplus\MA^*}(\varphi_{a(t)},\dots,\varphi_{a(t)},\varphi)\big)^{\doubar{x}}=\sum\mathcal{E}(\mathcalboondox{D^6_1})$ where the sum is over all the filled diagrams $\mathcalboondox{D_1^6}$ of type $\doubar{x}$ and of the form

     %

Moreover, we have that $\pi_{\MA^*[d]}\circ (\varphi\upperset{G}{\circ} sm_{a(t)})^{\doubar{x}}=\sum\mathcal{E}(\mathcalboondox{D^6_2})+\sum\mathcal{E}(\mathcalboondox{D^6_3})$ where the sums are over all the filled diagrams $\mathcalboondox{D_2^6}$ and $\mathcalboondox{D_2^6}$ of type $\doubar{x}$ and of the form

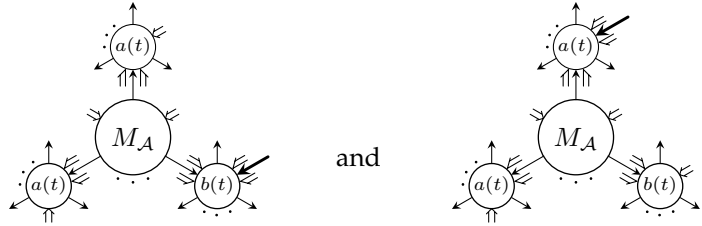
\begin{minipage}{21cm}
             %
\end{minipage}

\noindent respectively and $\pi_{\MA^*[d]}\circ (\varphi_{a(t)}\upperset{G}{\circ}sm)^{\doubar{x}}=\sum\mathcal{E}(\mathcalboondox{D^6_4})+\sum\mathcal{E}(\mathcalboondox{D^6_5})$ where the sums are over all the filled diagrams $\mathcalboondox{D_2^6}$ and $\mathcalboondox{D_2^6}$ of type $\doubar{x}$ and of the form

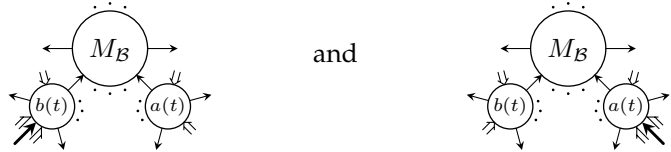
\begin{minipage}{21cm}
        %
\end{minipage}

\noindent respectively. We also have that \[\pi_{\MA[1]}\circ \big(
\sum\limits_{n=1}^{\infty}sm_{\MA\oplus\MA^*}\big(\varphi_{a(t)},\dots,\varphi_{a(t)},\varphi\big)+\varphi\upperset{G}{\circ} sm_{a(t)}
    +\varphi_{a(t)}\upperset{G}{\circ}sm\big)^{\doubar{x}}=0\]
for every $\doubar{x}\in\doubar{\MO}_{\MA}$. 

Therefore, since we have $\sum\mathcal{E}(\mathcalboondox{D^6_1})=\sum\mathcal{E}(\mathcalboondox{D^3_1})$, $\sum\mathcal{E}(\mathcalboondox{D^6_2})=\sum\mathcal{E}(\mathcalboondox{D^3_2})$, $\sum\mathcal{E}(\mathcalboondox{D^6_3})=\sum\mathcal{E}(\mathcalboondox{D^3_4})$, $\sum\mathcal{E}(\mathcalboondox{D^6_4})=\sum\mathcal{E}(\mathcalboondox{D^3_3})$ and $\sum\mathcal{E}(\mathcalboondox{D^6_5})=\sum\mathcal{E}(\mathcalboondox{D^3_5})$, we have that
\[s\Tilde{a}(t)^{\doubar{x}}=\sum\limits_{n=1}^{\infty}sm_{\MA\oplus\MA^*}\big(\varphi_{a(t)},\dots,\varphi_{a(t)},\varphi\big)^{\doubar{x}}+(s\varphi\upperset{G}{\circ} sm_{a(t)})^{\doubar{x}}+(s\varphi_{a(t)}\upperset{G}{\circ}sm)^{\doubar{x}}\] for every $t\in\Bbbk$, $\doubar{x}\in\doubar{\MO}_{\MA}$ which concludes the proof.
\end{proof}

\begin{corollary}
    If $(\Phi_0,s_{d+1}\mathbf{F}) : (\MA,s_{d+1}M_{\MA})\rightarrow (\MB,s_{d+1}M_{\MB})$ is a homotopy equivalence, then so is $s\varphi_{\mathbf{F}}$.
\end{corollary}
\begin{proof}
    Consider $(\Psi_0,s_{d+1}\mathbf{G}) : (\MB,s_{d+1}M_{\MB}) \rightarrow (\MA,s_{d+1}M_{\MA})$ a $d$-pre-Calabi-Yau morphism such that 
    $(\Psi_0\circ \Psi_0,s_{d+1}\mathbf{G}\circ s_{d+1}\mathbf{F})$ (resp. $(\Phi_0\circ \Psi_0,s_{d+1}\mathbf{F}\circ s_{d+1}\mathbf{G})$) is homotopic to $\id_{\MA}$ (resp. $\id_{\MB}$). 
     Since $(\Psi_0,s_{d+1}\mathbf{G})$ is a $d$-pre-Calabi-Yau morphism, it induces an $A_{\infty}$-structure $sm_{\mathbf{G}}$ on $\MB\oplus\MA^*[d-1]$ together with $A_{\infty}$-morphisms 
     \begin{equation}
\begin{tikzcd}
&(\MB \oplus \MA^*[d-1],sm_{\mathbf{G}}) \arrow[swap,"s\varphi_{\mathbf{G}}"]{dl} \arrow[swap,"s\psi_{\mathbf{G}}"]{dr}\\
(\MB \oplus \MB^*[d-1],sm_{\MB\oplus\MB^*})&& (\MA \oplus \MA^*[d-1], sm_{\MA\oplus\MA^*})
\end{tikzcd}
\end{equation}

By Theorem \ref{thm:homotopies-pCY-Ainf}, the image of $(\Psi_0\circ \Psi_0,s_{d+1}\mathbf{G}\circ s_{d+1}\mathbf{F})$ (resp. $(\Phi_0\circ \Psi_0,s_{d+1}\mathbf{F}\circ s_{d+1}\mathbf{G})$) under $\mathcal{P}$ is homotopic to the identity of $\widehat{A}_{\infty}$ associated to $\MA$ (resp. $\MB$).
Moreover, by Theorem \ref{thm:main-article-1}, $\mathcal{P}$ is compatible with partial products which means that $(\Phi_0,s_{d+1}\mathbf{F})$ is a homotopy equivalence with homotopy inverse $(\Psi_0,s_{d+1}\mathbf{G})$.

\end{proof}

\bibliographystyle{alpha}
\bibliography{mybiblio}

\vspace{1cm}

MARION BOUCROT: Univ. Grenoble Alpes, CNRS, IF, 38000 Grenoble, France

\textit{E-mail adress :} \href{mailto:marion.boucrot@univ-grenoble-alpes.fr}{\texttt{marion.boucrot@univ-grenoble-alpes.fr}}

\end{document}